\documentclass[a4paper]{amsart}
\usepackage[T1]{fontenc}
\usepackage[utf8]{inputenc}
\usepackage{lmodern}

\usepackage{enumitem} %<--advanced enumerate, which handles label and ref
\setlist[enumerate,1]{label=\textup{(\arabic*)}}% ensure that enumerations in theorems are upright

\usepackage{mathtools, amssymb}
\usepackage{booktabs}

\usepackage{tikz}
\usetikzlibrary{matrix,arrows,calc}

\newcommand{\triplerightarrow}{%
\tikz[minimum height=0ex]
  \path[->]
   node (a)            {}
   node (b) at (1em,0) {}
  (a.north)  edge (b.north)
  (a.center) edge (b.center)
  (a.south)  edge (b.south);%
}

\tikzset{cd/.style={matrix of math nodes,row sep=2em,column sep=2em,
    text height=1.5ex, text depth=0.5ex}}
\tikzset{cdar/.style={->,auto}}
\tikzset{dar/.style={anchor=mid,double,double equal sign
    distance,-implies}} % style for 2-arrows
\tikzset{equ/.style={double, double equal sign distance}} % style for equality

\tikzset{narrowfill/.style={anchor=mid,inner sep=0pt, fill=white}}% style for nodes with filled background

\usepackage{xcolor}
\usepackage{microtype}

\numberwithin{equation}{section}

\theoremstyle{plain}
\newtheorem{theorem}{Theorem}[section]
\newtheorem{proposition}[theorem]{Proposition}
\newtheorem{lemma}[theorem]{Lemma}
\newtheorem{corollary}[theorem]{Corollary}

\theoremstyle{definition}
\newtheorem{definition}[theorem]{Definition}

\newtheorem{deflemma}[theorem]{Definition and Lemma}
\newtheorem{assumption}[theorem]{Assumption}

\theoremstyle{remark}
\newtheorem{remark}[theorem]{Remark}
\newtheorem{example}[theorem]{Example}

\newcommand {\emptycomment}[1]{}

\DeclareMathOperator{\Ad}{Ad}  % conjugation action of bisections

\DeclareMathOperator{\Kan}{Kan}% Kan condition
\DeclareMathOperator{\Ima}{im} % image

\newcommand{\N}{\mathbb N}

\newcommand{\R}{\mathbb R}

\newcommand*{\nb}    {\nobreakdash}     %no word break after the following hyphen
\newcommand*{\defeq} {\mathrel{\vcentcolon=}} %used for definitions
\newcommand*{\prto}{\twoheadrightarrow}
\newcommand*{\congto}{\overset\sim\to}
\newcommand*{\blank} {\textup{\textvisiblespace}}
\newcommand*{\id}{\textup{id}}%identity map
\newcommand*{\inOb}{\mathrel{\in\in}}%object of a category

\newcommand*{\norm}[1]{\lVert#1\rVert}

\newcommand{\Cat}{\mathcal C}% category C

\newcommand{\Fun}{\mathfrak{Fun}}% 2-category of groupoids and functors
\newcommand{\Anafun}{\mathfrak{Ana}}% 2-category of groupoids and ana-functors
\newcommand{\GrCat}{\mathfrak{Gr}}% quasi-category of groupoids

\newcommand{\Sets}{\mathsf{Sets}}% category of sets
\newcommand{\Top}{\mathsf{Top}}% category of topological spaces
\newcommand{\Hausdorff}{\mathsf{Haus\textsf{-}Top}}% category of topological spaces
\newcommand{\Mfd}{\mathsf{Mfd}}% category of smooth manifolds
\newcommand{\fin}{\mathrm{fin}}% used for finite-dimensional manifolds
\newcommand{\Hil}{\mathrm{Hil}}% used for Hilbert manifolds
\newcommand{\Ban}{\mathrm{Ban}}% used for Banach manifolds
\newcommand{\Fre}{\textup{Fré}}% used for Fréchet manifolds
\newcommand{\lcs}{\mathrm{lcs}}% used for locally convex manifolds

\newcommand{\covers}{\mathcal T}% set of covers in a pretopology
\newcommand{\surj}{\textup{surj}}% surjections in Sets
\newcommand{\limlift}{{\textup{biqu}}}% biquotient maps
\newcommand{\splitsur}{\textup{split}}% split surjections
\newcommand{\locsplit}{\textup{loc.split}}% locally split surjections
\newcommand{\numsplit}{\textup{num.split}}% numerably split surjections
\newcommand{\proper}{\textup{prop}}% proper maps
\newcommand{\open}{\textup{open}}% open maps
\newcommand{\locsect}{\textup{loc.sect.}}% surjections with many local sections
\newcommand{\etale}{\textup{ét}}% etale maps
\newcommand{\subm}{\textup{subm}}% submersions

\newcommand*{\Gr}[1][G]{\mathsf{#1}}% groupoid
\newcommand*{\Act}[1][X]{\mathsf{#1}}% object with groupoid action
\newcommand*{\Base}[1][Z]{\mathsf{#1}}% base of basic action
\newcommand*{\bunp}[1][p]{\mathsf{#1}}% bundle projection
\newcommand*{\s}{\mathsf s}% source map
\newcommand*{\rg}{\mathsf r}% range map
\newcommand*{\pr}{\mathsf{pr}}% coordinate projections
\newcommand*{\mul}{\mathsf m}% multiplication map
\newcommand{\unit}{\mathsf u}% unit map
\newcommand{\invers}{\mathsf i}% inversion map

\newcommand*{\?}{{?}}% test object

\usepackage{hyperref}
\hypersetup{%
  bookmarksnumbered,
  bookmarksopen,
  breaklinks,
  pdfstartview=FitH,
  pdftitle={Groupoids in categories with pretopology},
  pdfauthor={Ralf Meyer and Chenchang Zhu},
  pdfkeywords={Grothendieck topology, cover, groupoid, principal
    bundle, Hilsum-Skandalis morphism;  anafunctor;
  bicategory; comorphism; infinite-dimensional groupoid},
  pdfsubject={Mathematics},
}
\usepackage[lite]{amsrefs} %-load amsrefs after hyperref
\newcommand*{\MRref}[2]{ \href{http://www.ams.org/mathscinet-getitem?mr=#1}{MR \textbf{#1}}}
\newcommand*{\arxiv}[1]{\href{http://www.arxiv.org/abs/#1}{arXiv: #1}}
\renewcommand{\PrintDOI}[1]{\href{http://dx.doi.org/\detokenize{#1}}{doi: \detokenize{#1}}%
  \IfEmptyBibField{pages}{, (to appear in print)}{}}

\begin{document}
\title{Groupoids in categories with pretopology}
\author{Ralf Meyer}
\email{rmeyer2@uni-goettingen.de}
\author{Chenchang Zhu}
\email{zhu@uni-math.gwdg.de}
\address{Mathematisches Institut\\
  Georg-August-Universit\"at G\"ottingen\\
  Bunsenstra\ss e 3--5\\
  37073 G\"ottingen\\
  Germany}

\begin{abstract}
  We survey the general theory of groupoids, groupoid actions,
  groupoid principal bundles, and various kinds of morphisms between
  groupoids in the framework of categories with
  pretopology. The categories of topological spaces, finite or
  infinite dimensional manifolds are examples of such categories.  We
  study extra assumptions on pretopologies that are needed for this
  theory.  We check these extra assumptions in several categories with
  pretopologies.

  Functors between groupoids may be localised at equivalences in two
  ways.  One uses spans of functors, the other bibundles (commuting
  actions) of groupoids.  We show that both approaches give equivalent
  bicategories.  Another type of groupoid morphism, called an actor,
  is closely related to functors between the categories of groupoid
  actions.  We also generalise actors using bibundles, and show that
  this gives another bicategory of groupoids.
\end{abstract}

\maketitle

\tableofcontents

\section{Introduction}
\label{sec:intro}

The notion of groupoid has many variants: topological groupoids,
étale topological groupoids, Lie groupoids of finite and infinite
dimension, algebraic groupoids, and
so on.  A subtle point is that the definition of a groupoid in a
category depends on a notion of ``cover'' because the range and source
maps are assumed to be covers.  A \emph{pretopology} gives a
reasonable notion of cover in a category.

The need to assume range and source maps to be covers is plain for
Lie groupoids: it ensures that the composable pairs of arrows form a
smooth manifold.  The covers
also influence the notion of principal bundle because their bundle
projections are assumed to be covers; this is equivalent to ``local
triviality'' in the sense of the pretopology.  If our category is
that of topological spaces and the covers are the continuous surjections
with local continuous sections, then we get exactly the usual notion
of local triviality for principal bundles; this is why many
geometers prefer this pretopology on topological spaces.  Many
operator algebraists prefer the pretopology of open continuous
surjections instead.

Some authors (like Goehle~\cite{Goehle:Thesis}) make no assumptions on
the range and source maps of a topological groupoid.  Since any map
with a continuous section is a biquotient map, this amounts to the
same as choosing biquotient maps as covers; this is the so called
canonical topology, that is, it is the largest subcanonical
pretopology on the category of Hausdorff topological spaces (see
Section~\ref{sec:Top_biquotient_maps}).  But how much of the usual
theory remains true in this case?

We will see that some aspects of the theory of groupoid Morita
equivalence require extra assumptions on pretopologies.  In
particular, we will show that for the pretopology of biquotient maps,
equivalent groupoids may have non-equivalent categories of actions on
spaces.  We do not know whether Morita equivalence of groupoids is an
equivalence relation when we choose the pretopology of biquotient
maps.  In addition to the theory of Morita equivalence with bibundles,
we also extend the theory of anafunctors that works without
extra assumptions on the pretopology, but is less concrete.

Since we want to work in the abstract setting of groupoids in a
category with pretopology, we develop all the general theory of
groupoids and their actions from scratch.  Most of our results are
known for some types of groupoids, of course.  Since we develop the
theory from scratch, we also take the opportunity to suggest more
systematic notation for various kinds of morphisms of groupoids.  We
study functors, anafunctors, and bibundle functors, actors and
bibundle actors; in addition, there are several kinds of equivalences:
equivalence functors, ana-isomorphisms, anafunctor equivalences, and
bibundle equivalences; and there are the covering bibundle functors,
which are the intersection of bibundle functors and bibundle actors.
The categories of anafunctors and bibundle functors are equivalent,
and ana-isomorphisms and bibundle equivalences are also equivalent
notions; all the other types of morphisms are genuinely different and
are useful in different situations. Several notions of morphisms have
been introduced in the context of Lie groupoids
in~\cite{Pradines:survey08}; our notion of ``actor'' is unrelated to
the one in~\cite{Pradines:survey08}*{A.2}, however.

Roughly speaking, functors and the related anafunctors and
bibundle functors are appropriate if we view groupoids as
generalised spaces, whereas actors and bibundle actors are appropriate
if we view them as generalised symmetries (groups).

Other new names we introduce are \emph{basic actions} and \emph{basic
  groupoids}.  A groupoid action is basic if, together with some
bundle projection, it is a principal action.  A groupoid~\(\Gr\)
is basic if its action on~\(\Gr^0\) is a basic action.  (The name
``principal groupoid'' is already used for something else by operator
algebraists, so a new name is needed.)

Our initial goal was to generalise the bicategory of
topological groupoids with Hilsum--Skandalis morphisms as arrows and
isomorphisms of the latter as \(2\)\nb-arrows.  It is well-known
that this is a bicategory if, say, we use continuous maps
with local continuous sections as covers (see~\cite{Noohi:Foundations_I}).
Our bibundle functors are the analogue of Hilsum--Skandalis morphisms
in a category with pretopology.  To compose bibundle
functors or bibundle actors, we need extra
assumptions on the pretopology.  We introduce these extra assumptions
here and check them for some pretopologies.

\smallskip

We now explain the contents of the sections of the article.

Section~\ref{sec:pretopologies} introduces \emph{pretopologies} on
categories with coproducts.  Here a pretopology is a family of maps
called \emph{covers}, subject to some rather mild axioms.  The more
established notion of a topos is not useful for our purposes:
it does not cover important examples like the category of smooth
manifolds because all finite limits are required to exist in a topos.

All pretopologies are required subcanonical.  We also introduce
three extra assumptions that are only sometimes required.
Assumption~\ref{assum:final} concerns final objects and is needed to
define groups and products of groupoids, as opposed to fibre
products.  Assumption~\ref{assum:local_cover} requires maps that are
``locally covers'' to be covers; the stronger
Assumption~\ref{assum:two-three} requires that if \(f\circ p\)
and~\(p\) are covers, then so must be~\(f\).  We also prove that the
property of being an isomorphism is local (this is also noticed by
\cite{Wolfson:Descent}*{Axiom 4}).  Assumptions \ref{assum:local_cover}
and~\ref{assum:two-three} are needed to compose bibundle functors and
bibundle actors, respectively.

Section~\ref{sec:groupoids} introduces groupoids, functors and
anafunctors.  We first define \emph{groupoids} in a category with
pretopology
in two equivalent ways and compare these two definitions.  One
definition has the unit and inversion maps as data; the other one
only has the multiplication, range and source maps as data and
characterises groupoids by making sense of the elementwise condition
that there should be unique solutions to equations of the form
\(g\cdot x=h\) or \(x\cdot g=h\) for given arrows \(g\) and~\(h\)
with equal range or equal source, respectively.  We explain in
Section~\ref{sec:groupoids} how such elementwise formulas should be
interpreted in a general category, following ideas from synthetic
geometry \cites{Kock:Synthetic, Moerdijk-Reyes:Models}.  We use
elementwise formulas throughout because they clarify statements and
proofs.

We define \emph{functors} between groupoids and \emph{natural
  transformations} between such functors in a category with
pretopology.  We observe that they form a strict \(2\)\nb-category.
We define a ``base-change'' for groupoids: given a cover \(p\colon
X\prto\Gr^0\), we define a groupoid \(p^*(\Gr)\) with object
space~\(X\) and a functor \(p_*\colon p^*(\Gr)\to\Gr\); such functors
are called \emph{hypercovers}, and they are the prototypes of
equivalences.  An \emph{anafunctor} \(\Gr\to\Gr[H]\) is a triple
\((X,p,F)\) with a cover \(p\colon X\prto \Gr^0\) and a functor
\(F\colon p^*(\Gr)\to\Gr[H]\).  An \emph{ana-isomorphism} is an
isomorphism \(p^*(\Gr)\cong q^*(\Gr[H])\) for two covers \(p\colon
X\to\Gr^0\) and \(q\colon X'\to\Gr[H]^0\).  Anafunctors form a
bicategory.  We characterise equivalences in this bicategory by two
different but equivalent criteria: as those anafunctors that lift to
an anafunctor isomorphism, and functors that are almost fully faithful
and almost essentially surjective.  These constructions all work in
any category with a subcanonical pretopology.  If we assume that being
a cover is a local condition, then a functor is almost fully faithful
and almost essentially surjective if and only if it is fully faithful
and essentially surjective.

Section~\ref{sec:actions} introduces groupoid actions, their
transformation groupoids, and actors.  Our convention is that the
anchor map of a groupoid action need not be a cover; this is needed to
associate bibundle functors to functors.  We call an action a \emph{sheaf} if
its anchor map is a cover.  The \emph{transformation groupoid} of a
groupoid action combines the action and the groupoid that is acting.
An \emph{actor} from~\(\Gr\) to~\(\Gr[H]\) is a left action of~\(\Gr\)
on~\(\Gr[H]^1\) that commutes with the right translation action
of~\(\Gr[H]\).  We show that such an actor allows to turn an
\(\Gr[H]\)\nb-action on~\(\Act\) into a \(\Gr\)\nb-action on~\(\Act\)
in a natural way.  Any such functor from \(\Gr\)\nb-actions to
\(\Gr[H]\)\nb-actions with some mild extra conditions
comes from an actor.  This is the type of ``morphism'' that seems most
relevant when we think of a groupoid as a general form of symmetry,
that is, as something that acts on other objects.

Section~\ref{sec:principal_bundles} introduces principal bundles over
groupoids and basic groupoid actions.  A \emph{principal
  \(\Gr\)\nb-bundle} is a (right) \(\Gr\)\nb-action \(\mul\colon
\Act\times_{\s,\Gr^0,\rg}\Gr^1\to\Act\) with a cover \(\bunp\colon
\Act\prto\Base\), such that the map
\begin{equation}
  \label{eq:intro_principality}
  (\mul,\pr_1)\colon
  \Act\times_{\s,\Gr^0,\rg}\Gr^1\to\Act\times_{\bunp,\Base,\bunp}\Act,\qquad
  (x,g)\mapsto (x\cdot g,x),
\end{equation}
is an isomorphism.  Then~\(\Base\) is the orbit space of the right
\(\Gr\)\nb-action, so the bundle projection is unique if it exists.  A
\(\Gr\)\nb-action is \emph{basic} if it is a principal bundle when
taken together with its orbit space projection.  We construct
pull-backs of bundles and relate maps on the base and total spaces of
principal bundles; this is crucial for the composition of bibundle
functors.  We show that the property of being principal is local, that
is, a pull-back of a ``bundle'' along a cover is principal if and only
if the original bundle is.

Principal bundles and basic actions allow us to define bibundle
equivalences, bibundle functors, and bibundle actors.  These are all
given by commuting actions of \(\Gr\) and~\(\Gr[H]\) on some
object~\(\Act\) of the category.  For a \emph{bibundle functor}, we
require that the right \(\Gr[H]\)\nb-action together with the left
anchor map \(\Act\to\Gr^0\) as bundle projection is a principal
bundle; for a \emph{bibundle equivalence}, we require, in addition,
that the left \(\Gr\)\nb-action together with the right anchor map
\(\Act\to\Gr[H]^0\) is a principal bundle.  For a \emph{bibundle
  actor}, we require the right action to be basic -- with arbitrary
orbit space -- and the right anchor
map \(\Act\to\Gr[H]^0\) to be a cover.  If~\(\Act\) is both a bibundle
actor and a bibundle functor, so that both anchor maps are covers,
then we call it a \emph{covering bibundle functor}.

We show that functors give rise to bibundle functors and that
bibundle functors give rise to anafunctors.  The bibundle functor
associated to a functor is a bibundle equivalence if and only if the
functor is essentially surjective and fully faithful.

The composition of bibundle functors, actors, and equivalences is only
introduced in Section~\ref{sec:composition} because it requires extra
assumptions about certain actions being automatically basic.
Section~\ref{sec:assume_covering_actions} discusses these extra
assumptions.  The stronger one,
Assumption~\ref{assum:covering_acts_basically}, requires all actions
of \v{C}ech groupoids to be basic; the weaker one,
Assumption~\ref{assum:covering_acts_basically_weak}, only requires
this for sheaves over \v{C}ech groupoids, that is, for actions of
\v{C}ech groupoids with a cover as anchor map.  The weaker version
of the assumption together with Assumption~\ref{assum:local_cover} on
covers being local suffices to compose bibundle equivalences and
bibundle functors; the stronger form together with
Assumption~\ref{assum:two-three} (the two-out-of-three property for
covers) is needed to compose covering bibundle functors and bibundle
actors.  Under the appropriate assumptions, all these types of
bibundles form bicategories.

We show in Section~\ref{sec:equivalences_in_bibundles} that the
equivalences in these bicategories are precisely the bibundle
equivalences.  Section~\ref{sec:bibundle_versus_vague_2-cat} shows
that the bicategory of bibundle functors is equivalent to the
bicategory of anafunctors.
Section~\ref{sec:characterise_composite} describes the composite of
two bibundle functors by an isomorphism similar
to~\eqref{eq:intro_principality}.  Section~\ref{sec:basic_local} shows
that the assumption needed for the composition of bibundle actors is
equivalent to the assumption that the property of being basic is a
local property of \(\Gr\)\nb-actions.

We show in Section~\ref{sec:decompose_bibundle_actor} that every
bibundle actor is a composite of an actor and a bibundle
equivalence.  Thus the bicategory of bibundle actors is the
smallest one that contains bibundle equivalences and actors and is
closed under the composition of bibundles.  A similar decomposition of
bibundle functors into bibundle equivalences and functors is
constructed in Section~\ref{sec:bibundle_to_vague} when passing from
bibundle functors to anafunctors.

Section~\ref{sec:quasi-category_bibundle_functors} describes the
quasi-categories of bibundle functors, covering bibundle functors, and
bibundle equivalences very succinctly.  The main point are similarities
between the multiplication in a groupoid, a groupoid action, and the
map relating the product of two bibundle functors to the composite.

Section~\ref{sec:examples_covers} considers several examples of
categories with classes of ``covers'' and checks whether these are
pretopologies and satisfy our extra assumptions.
Section~\ref{sec:Sets_surjections} considers sets and surjective maps,
a rather trivial case.  Section~\ref{sec:Top} considers several
classes of ``covers'' between topological spaces; quotient maps do not
form a pretopology; biquotient maps form a subcanonical pretopology
for which Assumption~\ref{assum:covering_acts_basically} fails; open
surjections and several smaller classes of maps are shown to be
subcanonical pretopologies satisfying all our assumptions.

Section~\ref{sec:Mfd} considers smooth manifolds, both of finite and
infinite dimension, with submersions as covers.  These form a
subcanonical pretopology on locally convex manifolds and several
subcategories, which satisfies the assumptions needed for the composition
of bibundle functors and bibundle equivalences.  For Banach manifolds,
we also prove the stronger assumptions that are needed for the
composition of bibundle actors and covering bibundle functors.  It is
unclear whether these stronger assumptions still hold for Fréchet or
locally convex manifolds.  Thus all the theory developed in this
article applies to groupoids in the category of Banach manifolds and
finite-dimensional manifolds, and most of it, but not all, applies
also to Fréchet and locally convex manifolds.

\smallskip

For a category~\(\Cat\), we write \(X\inOb\Cat\) for objects and
\(f\in\Cat\) for morphisms of~\(\Cat\).  Our categories need not be
small, that is, the objects may form a class only.  The set-theoretic
issues that this entails are not relevant to our discussions below and
can be handled easily using universes.  Hence we leave it to the
reader to add appropriate remarks where needed.

\section{Pretopologies}
\label{sec:pretopologies}

The notion of a Grothendieck (pre)topology
from~\cite{Grothendieck:SGA4_Topos} formalises properties that
``covers'' in a category should have.  Usually, a ``cover'' of an
object~\(X\) is a family of maps \(f_\alpha\colon U_\alpha\to X\),
\(\alpha\in A\).  If our category has coproducts, we may replace such
a family of maps by a single map
\[
(f_\alpha)_{\alpha\in A}\colon \bigsqcup_{\alpha\in A} U_\alpha\to X.
\]
This reduces notational overhead, and is ideal for our purposes
because in the following definitions, we never meet covering families
but only single maps that are covers.

\begin{definition}
  \label{def:pretopology}
  Let~\(\Cat\) be a category with coproducts.  A \emph{pretopology}
  on~\(\Cat\) is a collection~\(\covers\) of arrows, called
  \emph{covers}, with the following properties:
  \begin{enumerate}
  \item isomorphisms are covers;
  \item the composite of two covers is a cover;
  \item if \(f\colon Y\to X\) is an arrow in~\(\Cat\) and \(g\colon
    U\prto X\) is a cover, then the fibre product \(Y\times_X U\) exists
    in~\(\Cat\) and the coordinate projection \(\pr_1\colon Y\times_X
    U \prto Y\) is a cover.  Symbolically,
    \begin{equation}
      \label{eq:fibre-product_of_cover}
      \begin{tikzpicture}[baseline=(current bounding box.west)]
        \matrix[cd] (m) {
          & Y \\
          U & X \\
        };
        \begin{scope}[cdar]
          \draw[->>] (m-2-1) -- node {\(g\)} (m-2-2);
          \draw (m-1-2) -- node {\(f\)} (m-2-2);
        \end{scope}
      \end{tikzpicture}
      \qquad\Rightarrow\qquad
      \begin{tikzpicture}[baseline=(current bounding box.west)]
        \matrix[cd] (m) {
          Y\times_{f,X,g} U & Y \\
          U & X \\
        };
        \begin{scope}[cdar]
          \draw[->>] (m-1-1) -- node {\(\pr_1\)} (m-1-2);
          \draw (m-1-1) -- node {\(\pr_2\)} (m-2-1);
          \draw[->>] (m-2-1) -- node {\(g\)} (m-2-2);
          \draw (m-1-2) -- node {\(f\)} (m-2-2);
        \end{scope}
      \end{tikzpicture}
    \end{equation}
  \end{enumerate}
\end{definition}

We use double-headed arrows~\(\prto\) to denote covers.

We always require pretopologies to be subcanonical:

\begin{deflemma}
  \label{def:subcanonical}
  A pretopology~\(\covers\) is \emph{subcanonical} if it satisfies the
  following equivalent conditions:
  \begin{enumerate}
  \item each cover \(f\colon U\prto X\) in~\(\covers\) is a coequaliser,
    that is, it is the coequaliser of some pair of parallel maps
    \(g_1,g_2\colon Z\rightrightarrows U\);
  \item each cover \(f\colon U\prto X\) in~\(\covers\) is the
    coequaliser of \(\pr_1,\pr_2\colon U\times_{f,X,f}
    U\rightrightarrows U\);
  \item for any object~\(W\) and any cover \(f\colon U\prto X\), we have a
    bijection
    \[
    \Cat(X,W)\congto
    \{h\in \Cat(U,W) \mid h\circ \pr_1 = h\circ
    \pr_2\text{ in }\Cat(U\times_{f,X,f} U,W)\},\quad
    g\mapsto g\circ f;
    \]
  \item all the representable functors \(\Cat(\blank,W)\) on~\(\Cat\)
    are sheaves.
  \end{enumerate}
\end{deflemma}

\begin{proof}
  Condition~(3) makes explicit what it means for \(\Cat(\blank,W)\) to
  be a sheaf and for~\(f\) to be the coequaliser of \(\pr_1,\pr_2\),
  so (2)\(\iff\)(3)\(\iff\)(4).  (2) implies~(1) by taking
  \(Z=U\times_{f,X,f} U\) and \(g_i=\pr_i\) for \(i=1,2\).  It remains
  to prove (1)\(\Rightarrow\)(3).

  Let \(g_1,g_2\colon Z\rightrightarrows U\) be as in~(1).  Then
  \(fg_1=fg_2\), so that \((g_1,g_2)\) is a map \(Z\to U\times_{f,X,f}
  U\).  Let \(\pr_i\colon U\times_{f,X,f} U\to U\) for \(i=1,2\) be
  the coordinate projections.  Since \(\pr_i\circ (g_1,g_2)=g_i\) for
  \(i=1,2\), a map \(h\in\Cat(U,W)\) with \(h\circ\pr_1=h\circ\pr_2\)
  also satisfies \(h\circ g_1=h\circ g_2\).  Since we assume~\(f\) to
  be a coequaliser of \(g_1\) and~\(g_2\), such a map~\(h\) is of the
  form \(\tilde{h}\circ f\) for a unique \(\tilde{h}\in\Cat(X,W)\).
  Conversely, any map of the form \(\tilde{h}\circ f\) satisfies
  \((\tilde{h}\circ f)\circ \pr_1=(\tilde{h}\circ f)\circ \pr_2\).
  Thus~(1) implies~(3).
\end{proof}

\begin{definition}
  \label{def:local}
  Let~\(\covers\) be a pretopology on~\(\Cat\) and let~P be a
  property that arrows in~\(\Cat\) may or may not have.  An arrow
  \(f\colon Y\to X\) has~P \emph{locally} if there is a cover
  \(g\colon U\prto Y\) such that \(\pr_2\colon Y\times_{f,X,g} U\to U\)
  has~P.

  The property~P is \emph{local} if an arrow has~P if and only if it
  has~P locally, that is, in the situation
  of~\eqref{eq:fibre-product_of_cover}, \(f\colon Y\to X\) has~P if
  and only if \(\pr_2\colon Y\times_{f,X,g} U\to U\) has~P.
\end{definition}

\subsection{Isomorphisms are local}
\label{sec:isomorphism_local}

\begin{proposition}
  \label{pro:isomorphism_local}
  If the pretopology is subcanonical, then the property of being an
  isomorphism is local.
\end{proposition}

\begin{proof}
  Giorgi Arabidze pointed out a slightly stronger version of this
  result that works in any category~\(\Cat\),
  not necessarily with a pretopology.  Namely, let
  $Y\xrightarrow{f} X$ and $U\xrightarrow{g} X$ be arrows in~\(\Cat\)
  such that the fibre product $U\times_{g, X, f} Y$ exists in~$\Cat$
  and the projection $\pr_1\colon U\times_{g, X, f} Y\to U$ is an
  isomorphism.  If~\(g\)
  is a coequaliser in~\(\Cat\)
  (see Definition and Lemma~\ref{def:subcanonical}) and the projection
  $\pr_2\colon U\times_{g, X, f} Y \to Y$ is an epimorphism, then~$f$
  is an isomorphism.  If~$\Cat$ is a category with a subcanonical
  pretopology and~$g$ is a cover, then the assumptions above are
  satisfied.

  To prove the more general statement above, we take the two maps
  \(u_1,u_2\colon A\rightrightarrows U\)
  that have~\(g\)
  as their coequaliser.  Form
  \(y\defeq \pr_2\circ\pr_1^{-1}\circ u_2\colon A\to U\to U\times_{g,
    X, f} Y\to Y\).
  Since
  \(f\circ y = g\circ \pr_1\circ \pr_1^{-1}\circ u_2= g\circ u_1\),
  there is a map \(b\colon A\to U\times_{g, X, f} Y\)
  with \(\pr_1\circ b=u_1\)
  and \(\pr_2\circ b= y\).
  Since~\(\pr_1\)
  is invertible, \(b=\pr_1^{-1} u_1\),
  so \(\pr_2\pr_1^{-1} u_1 = y = \pr_2\pr_1^{-1} u_2\).
  Since~\(g\)
  is the coequaliser of \(u_1\)
  and~\(u_2\),
  we get a unique map \(c\colon X\to Y\)
  with \(c\circ g = \pr_2\pr_1^{-1}\).
  This satisfies \(f\circ c \circ g = g\).
  Since~\(g\)
  is a coequaliser, it is epic, so this gives \(f\circ c=\id_X\).
  Furthermore, \(c f \pr_2 = c g \pr_1 = \pr_2\)
  gives \(c f=\id_Y\)
  because we assumed~\(\pr_2\)
  to be epic.  Thus~\(c\) is inverse to~\(f\).
\end{proof}

\begin{remark}
  Example~\ref{exa:biquotient_covering_not_basic} provides
  Hausdorff topological spaces \((X,\tau_1)\), \((X,\tau_2)\)
  and~\(Y\) such
  that the identity map \((X,\tau_1)\to (X,\tau_2)\) is a continuous
  bijection but not a homeomorphism, and a quotient map \(f\colon
  (X,\tau_2)\to Y\) such that the map \(f\colon (X,\tau_1)\to Y\) is
  a quotient map as well, and such that the identical map
  \[
  (X,\tau_1) \times_{f,Y,f} (X,\tau_1)
  \to (X,\tau_1) \times_{f,Y,f} (X,\tau_2)
  \]
  is a homeomorphism.  Thus the pull-back of a non-homeomorphism along
  a quotient map may become a homeomorphism.  Quotient maps on
  topological spaces do not form a pretopology, so this does not
  contradict Proposition~\ref{pro:isomorphism_local}.
\end{remark}

\subsection{Simple extra assumptions on pretopologies}
\label{sec:assum_local_cover}

Consider the fibre-product situation
of~\eqref{eq:fibre-product_of_cover}.  Then~\(g\) is a cover by
assumption and~\(\pr_1\) is one by the definition of a pretopology.
If~\(f\) is a cover as well, then so is~\(\pr_2\) by the definition
of a pretopology.  That the converse holds is an extra assumption on
the pretopology, which means that the property of being a cover is
local.

\begin{assumption}
  \label{assum:local_cover}
  The property of being a cover is local, that is, in the
  fibre-product situation of~\eqref{eq:fibre-product_of_cover}, if
  \(g\) and~\(\pr_2\) in~\eqref{eq:fibre-product_of_cover} are
  covers, then so is~\(f\).
\end{assumption}

If \(g\) and~\(\pr_2\) are covers, then so is
\(g\circ\pr_2=f\circ\pr_1\) as a composite of covers, and~\(\pr_1\) is
a cover by definition of a pretopology.  Hence
Assumption~\ref{assum:local_cover} is weaker than the following
\emph{two-out-of-three property}:

\begin{assumption}
  \label{assum:two-three}
  Let \(f\in\Cat(Y,Z)\) and \(p\in\Cat(X,Y)\) be composable.  If
  \(f\circ p\) and~\(p\) are covers, then so is~\(f\).
\end{assumption}

A pretopology is called \emph{saturated} if \(f\) is a cover whenever
\(f\circ p\) is one, without requiring~\(p\) to be a cover as well
(see \cite{Roberts:Anafunctors_localisation}*{Definition 3.11}).  This
stronger assumption fails in many cases where
Assumption~\ref{assum:two-three} holds (see
Example~\ref{exa:non-saturated}), and Assumption~\ref{assum:two-three}
suffices for all our applications.  For the categories of locally
convex and Fr\'echet manifolds with surjective submersions as covers,
Assumption~\ref{assum:local_cover} holds, but we cannot prove
Assumption~\ref{assum:two-three}.  The other examples we consider in
Section~\ref{sec:examples_covers} all verify the stronger
Assumption~\ref{assum:two-three}.

\begin{assumption}
  \label{assum:final}
  There is a final object~\(\star\) in~\(\Cat\), and all maps to it
  are covers.
\end{assumption}

\begin{lemma}
  \label{lem:final_object_products}
  Under Assumption~\textup{\ref{assum:final}}, \(\Cat\) has finite
  products.  If \(f_1\colon U_1\prto X_1\) and \(f_2\colon U_2\prto
  X_2\) are covers, then so is \(f_1\times f_2\colon U_1\times
  U_2\to X_1\times X_2\).
\end{lemma}

\begin{proof}
  The first statement is clear because products are fibre products
  over the final object~\(\star\).  The second one is
  \cite{Henriques:L-infty}*{Lemma 2.7}.
\end{proof}

Assumption~\ref{assum:final} has problems with initial objects.  For
instance, the unique map from the empty set to the one-point set is
not surjective, so Assumption~\ref{assum:final} fails for any
subcanonical pretopology on the category of sets.  We must exclude the
empty set for Assumption~\ref{assum:final} to hold.  We will not use
Assumption~\ref{assum:final} much.

\section{Groupoids in a category with pretopology}
\label{sec:groupoids}

Let~\(\Cat\) be a category with coproducts and let~\(\covers\) be a
subcanonical pretopology on~\(\Cat\).  We define groupoids
in~\((\Cat,\covers)\) in two equivalent ways, with more or less
data.

\subsection{First definition}
\label{sec:groupoids_first_def}

A groupoid in \((\Cat,\covers)\) consists of
\begin{itemize}
\item objects \(\Gr^0\inOb\Cat\) (\emph{objects}) and
  \(\Gr^1\inOb\Cat\) (\emph{arrows}),
\item arrows \(\rg\in\Cat(\Gr^1,\Gr^0)\) (\emph{range}),
  \(\s\in\Cat(\Gr^1, \Gr^0)\) (\emph{source}),
  \(\mul\in\Cat(\Gr^1\times_{\s,\Gr^0,\rg} \Gr^1, \Gr^1)\)
  (\emph{multiplication}), \(\unit\in\Cat(\Gr^0, \Gr^1)\)
  (\emph{unit}), and \(\invers\in\Cat(\Gr^1, \Gr^1)\)
  (\emph{inversion}),
\end{itemize}
such that
\begin{enumerate}
\item \(\rg\) and~\(\s\) are covers;
\item \(\mul\) is associative, that is, \(\rg\circ\mul=\rg\circ
  \pr_1\) and \(\s\circ\mul=\s\circ\pr_2\) for the coordinate
  projections \(\pr_1,\pr_2\colon \Gr^1\times_{\s,\Gr^0,\rg} \Gr^1\to
  \Gr^1\), and the following diagram commutes:
  \begin{equation}
    \label{eq:mul_associative}
    \begin{tikzpicture}[baseline=(current bounding box.west)]
      \matrix[cd,column sep=6em] (m) {
        \Gr^1 \times_{\s,\Gr^0,\rg} \Gr^1 \times_{\s,\Gr^0,\rg} \Gr^1 &
        \Gr^1 \times_{\s,\Gr^0,\rg} \Gr^1 \\
        \Gr^1 \times_{\s,\Gr^0,\rg} \Gr^1 &
        \Gr^1 \\
      };
      \begin{scope}[cdar]
        \draw (m-1-1) -- node {\(\mul \times_{\Gr^0,\rg} \id_{\Gr^1}\)} (m-1-2);
        \draw (m-1-1) -- node[swap] {\(\id_{\Gr^1} \times_{\s,\Gr^0} \mul\)} (m-2-1);
        \draw (m-1-2) -- node {\(\mul\)} (m-2-2);
        \draw (m-2-1) -- node[swap] {\(\mul\)} (m-2-2);
      \end{scope}
    \end{tikzpicture}
  \end{equation}
\item the following equations hold:
  \begin{alignat*}{3}
    \rg\circ \unit &=\id_{\Gr^0},&\quad \rg(1_x)&= x&\quad \forall x&\in \Gr^0,\\
    \s\circ \unit &= \id_{\Gr^0},&\quad \s(1_x)&= x&\quad \forall x&\in \Gr^0,\\
    \mul\circ (\unit\circ\rg,\id_{\Gr^1}) &= \id_{\Gr^1},&\quad
    1_{\rg(g)}\cdot g&=g&\quad \forall g&\in \Gr^1,\\
    \mul\circ (\id_{\Gr^1},\unit\circ\s) &= \id_{\Gr^1},&\quad
    g\cdot 1_{\s(g)}&=g&\quad \forall g&\in \Gr^1,\\
    \s\circ \invers &= \rg,&\quad \s(g^{-1})&=\rg(g)&\quad \forall g&\in \Gr^1,\\
    \rg\circ \invers&=\s,&\quad \rg(g^{-1}) &= \s(g)&\quad \forall g&\in \Gr^1,\\
    \mul\circ (\invers,\id_{\Gr^1}) &= \unit\circ\s,&\quad
    g^{-1}\cdot g&= 1_{\s(g)}&\quad \forall g&\in \Gr^1,\\
    \mul\circ (\id_{\Gr^1},\invers) &= \unit\circ\rg,&\quad
    g\cdot g^{-1}&=1_{\rg(g)}&\quad \forall g&\in \Gr^1.\\
    \intertext{%
      The right two columns interpret the equalities of maps in the
      left column in terms of elements.  The conditions above imply}
    \mul\circ (\unit,\unit) &= \unit,&\quad
    1_x\cdot 1_x&=1_x&\quad \forall x&\in \Gr^0,\\
    \invers^2 &=\id_{\Gr^1},&\quad
    (g^{-1})^{-1} &= g&\quad \forall g&\in \Gr^1,\\
    \mul\circ (\invers\times_{\rg,\Gr^0,\s} \invers) &= \invers \circ
    \mul\circ\sigma,&\quad
    g^{-1}\cdot h^{-1} &= (h\cdot g)^{-1}&\quad
    \forall g,h&\in \Gr^1, \rg(g)=\s(h);
  \end{alignat*}
  here \(\sigma\colon \Gr^1\times_{\rg,\Gr^0,\s} \Gr^1\congto
  \Gr^1\times_{\s,\Gr^0,\rg} \Gr^1\) denotes the flip of the two
  factors.
\end{enumerate}

Most statements and proofs are much clearer in terms of elements.  It
is worthwhile, therefore, to introduce the algorithm that interprets
equations in terms of elements as in the middle and right columns
above as equations of maps as in the left column.

The variables \(x\in \Gr^0\), \(g,h\in \Gr^1\) are interpreted as maps
in~\(\Cat\) from some object \(\?\inOb\Cat\) to \(\Gr^0\)
and~\(\Gr^1\), respectively.  Thus an element~\(x\) of \(X \inOb
\Cat\) is interpreted as a map \(x\colon \?\to X\), and denoted by
\(x\in X\).  The elements of~\(X\) form a category,
which determines~\(X\) by the Yoneda Lemma.

If an elementwise expression~\(A\) is
already interpreted as a map \(A\colon \?\to \Gr^1\), then \(\rg(A)\)
means \(\rg\circ A\colon \?\to \Gr^0\), \(\s(A)\) means \(\s\circ
A\colon \?\to \Gr^0\), and \(A^{-1}\) means \(\invers\circ A\colon
\?\to \Gr^1\); if an elementwise expression~\(A\) is interpreted as
\(A\colon \?\to \Gr^0\), then \(1_A\) means \(\unit\circ A\colon \?\to
\Gr^1\).  If \(A\) and~\(B\) are elementwise expressions that
translate to maps \(A,B\colon \?\to \Gr^1\) with \(\s(A)=\rg(B)\),
that is, \(\s\circ A=\rg\circ B\colon \?\to\Gr^0\), then \(A\cdot B\)
means the composite map
\[
\mul\circ (A,B)\colon \?
\xrightarrow{(A,B)} \Gr^1\times_{\s,\Gr^0,\rg} \Gr^1
\xrightarrow{\mul} \Gr^1.
\]

This algorithm turns the conditions in the middle and right column
above into the conditions in the left column if we let~\(\?\) vary
through all objects of~\(\Cat\).  More precisely, the equation in the
left column implies the interpretation of the condition in the other
two columns for any choice of~\(\?\), and the converse holds for a
suitable choice of~\(\?\).

For instance, in the last condition \(g^{-1}\cdot h^{-1}=(h\cdot
g)^{-1}\), we may choose \(\?=\Gr^1\times_{\rg,\Gr^0,\s} \Gr^1\),
\(g=\pr_1\colon \?\to\Gr^1\), and \(h=\pr_2\colon \?\to\Gr^1\).  The
interpretation of the formula \(g^{-1}\cdot h^{-1} = (h\cdot g)^{-1}\)
for these choices is \(\mul\circ (\invers\times_{\rg,\Gr^0,\s}
\invers) = \invers \circ \mul\circ\sigma\).

As another example, the associativity
diagram~\eqref{eq:mul_associative} is equivalent to the elementwise
statement
\begin{equation}
  \label{eq:mul_associative_elements}
  (g_1\cdot g_2)\cdot g_3=g_1\cdot (g_2\cdot g_3)
  \qquad \forall g_1,g_2,g_3\in \Gr^1,\ \s(g_1)=\rg(g_2),
  \ \s(g_2)=\rg(g_3).
\end{equation}
To go from the elementwise statement to~\eqref{eq:mul_associative},
choose \(\?=\Gr^1\times_{\s,\Gr^0,\rg} \Gr^1\times_{\s,\Gr^0,\rg}
\Gr^1\), \(g_1=\pr_1\), \(g_2=\pr_2\), \(g_3=\pr_3\) (we take one
factor for each of the free variables \(g_1,g_2,g_3\), and implement
the assumptions \(\s(g_1)=\rg(g_2)\), \(\s(g_2)=\rg(g_3)\) by
fibre-product conditions).  Then \(A\defeq g_1\cdot
g_2\) is interpreted as \(\mul\circ (\pr_1,\pr_2)\colon
\Gr^1\times_{\s,\Gr^0,\rg} \Gr^1\times_{\s,\Gr^0,\rg} \Gr^1 \to
\Gr^1\), so \((g_1\cdot g_2)\cdot g_3=A\cdot g_3\) becomes the map
\[
\mul\circ (\mul\circ (\pr_1,\pr_2),\pr_3)
= \mul\circ (\mul\times_{\Gr^0,\rg}\id_{\Gr^1}) \circ (\pr_1,\pr_2,\pr_3)
= \mul\circ (\mul\times_{\Gr^0,\rg}\id_{\Gr^1})
\]
from \(\Gr^1\times_{\s,\Gr^0,\rg} \Gr^1\times_{\s,\Gr^0,\rg} \Gr^1\)
to~\(\Gr^1\).  Similarly, \(g_1\cdot (g_2\cdot g_3)\) is interpreted
as \(\mul\circ (\id_{\Gr^1}\times_{\s,\Gr^0}\mul)\).  Thus \((g_1\cdot
g_2)\cdot g_3=g_1\cdot (g_2\cdot g_3)\) implies
that~\eqref{eq:mul_associative} commutes.  Conversely,
if~\eqref{eq:mul_associative} commutes, then \((g_1\cdot g_2)\cdot
g_3=g_1\cdot (g_2\cdot g_3)\) for any choice of \(\?\inOb\Cat\) and
\(g_1,g_2,g_3\colon \?\to \Gr^1\) with \(\s\circ g_1=\rg\circ g_2\)
and \(\s\circ g_2=\rg\circ g_3\).  Thus~\eqref{eq:mul_associative} is
equivalent to the elementwise statement above.

We may also use elementwise formulas to define maps.  For instance,
suppose that we have already defined maps \(f\colon X\to\Gr^1\) and
\(g\colon Y\to\Gr^1\).  Then there is a unique map \(F\colon
X\times_{\s\circ f,\Gr^0,\rg\circ g} Y \to\Gr^1\) given elementwise by
\(F(x,y) \defeq f(x)\cdot g(y)\) for all \(x\in X\), \(y\in Y\) with
\(\s\circ f(x)=\rg\circ g(y)\).  If a map \(f\in\Cat(X,Z)\) is given
in terms of elements, then it is an isomorphism if and only if every
element~\(z\) of~\(Z\) may be written as \(f(x)\) for a unique \(x\in
X\): the interpretation of this statement in terms of maps is
exactly the Yoneda Lemma.

\begin{remark}
  \label{rem:why_range_cover}
  Why do we need the range and source maps to be covers?  Of course,
  we need some assumption for the fibre product
  \(\Gr^1\times_{\s,\Gr^0,\rg}\Gr^1\) to exist, but there are deeper
  reasons.  Many results in the theory of principal bundles depend on
  the locality of isomorphisms
  (Proposition~\ref{pro:isomorphism_local}), which holds only for
  pull-backs along covers.  The composition of bibundle
  equivalences cannot work unless we require their anchor maps to be
  covers.  The range and source maps of a groupoid must be covers
  because they are the anchor maps for the unit bibundle
  equivalence~\(\Gr^1\) on~\(\Gr\).
\end{remark}

\subsection{Second definition}
\label{sec:groupoids_second_def}

For groupoids in sets, the existence of units and inverses is
equivalent to the existence of unique solutions \(x\in \Gr^1\) to
equations of the form \(x\cdot g=h\) for \(g,h\in \Gr^1\) with
\(\s(g)=\s(h)\) and \(g\cdot x=h\) for \(g,h\in \Gr^1\) with
\(\rg(g)=\rg(h)\).  This leads us to the following equivalent
definition of a groupoid in \((\Cat,\covers)\):

\begin{definition}
  \label{def:groupoid}
  A \emph{groupoid}~\(\Gr\) in \((\Cat,\covers)\) consists of
  \(\Gr^0\inOb\Cat\) (\emph{objects}), \(\Gr^1\inOb\Cat\)
  (\emph{arrows}), \(\rg\in\Cat(\Gr^1,\Gr^0)\)
  (\emph{range}), \(\s\in\Cat(\Gr^1,\Gr^0)\) (\emph{source}), and
  \(\mul\in\Cat(\Gr^1\times_{\s,\Gr^0,\rg} \Gr^1, \Gr^1)\)
  (\emph{multiplication}), such that
  \begin{enumerate}
  \item the maps \(\rg\) and~\(\s\) are covers;
  \item the maps
    \begin{alignat}{2}
      \label{eq:groupoid_basicality_1}
      (\pr_2,\mul)\colon
      \Gr^1\times_{\s,\Gr^0,\rg} \Gr^1&\congto
      \Gr^1\times_{\s,\Gr^0,\s} \Gr^1,
      &\qquad (x,g)&\mapsto (g,x\cdot g),\\
      \label{eq:groupoid_basicality_2}
      (\pr_1,\mul)\colon
      \Gr^1\times_{\s,\Gr^0,\rg} \Gr^1&\congto
      \Gr^1\times_{\rg,\Gr^0,\rg} \Gr^1,
      &\qquad (g,x)&\mapsto (g,g\cdot x),
    \end{alignat}
    are well-defined isomorphisms;
  \item \(\mul\) is associative in the sense of
    \eqref{eq:mul_associative} or, equivalently,
    \eqref{eq:mul_associative_elements}.
  \end{enumerate}
\end{definition}

The first condition implies that the fibre products
\(\Gr^1\times_{\s,\Gr^0,\rg} \Gr^1\), \(\Gr^1\times_{\s,\Gr^0,\s}
\Gr^1\) and \(\Gr^1\times_{\rg,\Gr^0,\rg} \Gr^1\) used above exist
in~\(\Cat\).  The maps in \eqref{eq:groupoid_basicality_1}
and~\eqref{eq:groupoid_basicality_2} are well-defined if and only if
\(\s\circ\mul=\s\circ\pr_2\) and \(\rg\circ\mul=\rg\circ\pr_1\), so
these two equations are assumed implicitly.

If \((\Gr^1,\Gr^0,\rg,\s,\mul,\unit,\invers)\) is a groupoid in the
first sense above, then it is one in the sense of
Definition~\ref{def:groupoid}: the map \((g,h)\mapsto (h\cdot
g^{-1},g)\) is inverse to~\eqref{eq:groupoid_basicality_1}, and the
map \((g,h)\mapsto (g,g^{-1}\cdot h)\) is inverse
to~\eqref{eq:groupoid_basicality_2}.  The converse is proved in
Proposition~\ref{pro:unit_inverse_from_basicality}.

\begin{lemma}
  \label{lem:interpret_invertibility_as_map}
  The invertibility of the map in~\eqref{eq:groupoid_basicality_1}
  is equivalent to the elementwise statement that for all \(g,h\in
  \Gr^1\) with \(\s(g)=\s(h)\) there is a unique \(x\in \Gr^1\) with
  \(\s(x)=\rg(g)\) and \(x\cdot g=h\).  The invertibility of the map
  in~\eqref{eq:groupoid_basicality_2} is equivalent to the
  elementwise statement that for all \(g,h\in \Gr^1\) with
  \(\rg(g)=\rg(h)\) there is a unique \(x\in \Gr^1\) with
  \(\s(g)=\rg(x)\) and \(g\cdot x=h\).
\end{lemma}

\begin{proof}
  We only prove the first statement, the second one is similar.

  We interpret \(g,h,x\) as maps \(g,h,x\colon \?\to \Gr^1\)
  in~\(\Cat\).  The elementwise statement says that for
  \(\?\inOb\Cat\) and \(g,h\colon \?\to \Gr^1\) with \(\s\circ
  g=\s\circ h\) there is a unique \(x\in\Cat(\?,\Gr^1)\) with
  \(\s\circ x=\rg\circ g\) and \(\mul\circ (x,g)=h\).  The maps \(g\)
  and~\(h\) with \(\s\circ g=\s\circ h\) are equivalent to a single
  map \((g,h)\colon \?\to \Gr^1\times_{\s,\Gr^0,\s} \Gr^1\), and the
  conditions \(\s\circ x=\rg\circ g\) and \(\mul\circ (x,g)=h\)
  together are equivalent to \((x,g)\) being a map \(\?\to
  \Gr^1\times_{\s,\Gr^0,\rg} \Gr^1\) with \((\pr_2,\mul)\circ
  (x,g)=(g,h)\).  Hence the elementwise statement is equivalent to the
  statement that for each object~\(\?\) of~\(\Cat\) and each map
  \((g,h)\colon \?\to \Gr^1\times_{\s,\Gr^0,\s} \Gr^1\) there is a
  unique map \(A\colon \?\to \Gr^1\times_{\s,\Gr^0,\rg} \Gr^1\) with
  \((\pr_2,\mul)\circ A=(g,h)\).  This statement means
  that~\((\pr_2,\mul)\) is invertible by the Yoneda Lemma.
\end{proof}

To better understand \eqref{eq:groupoid_basicality_1}
and~\eqref{eq:groupoid_basicality_2}, we view~\(\mul\) as a
\emph{ternary relation} on~\(\Gr^1\).

\begin{definition}
  \label{def:ternary_relation}
  An \emph{\(n\)\nb-ary relation} between \(X_1,\dotsc,X_n\inOb\Cat\)
  is \(R\inOb\Cat\) with maps \(p_j\colon R\to X_j\) for
  \(j=1,\dotsc,n\) such that given \(x_j\in X_j\) for
  \(j=1,\dotsc,n\), there is at most one \(r\in R\) with
  \(p_j(r)=x_j\) for \(j=1,\dotsc,n\).
\end{definition}

The \emph{multiplication relation} defined by~\(\mul\) has
\[
R\defeq \Gr^1\times_{\s,\Gr^0,\rg} \Gr^1,\qquad
p_1\defeq \pr_1,\quad
p_2\defeq\pr_2,\quad
p_3\defeq \mul.
\]

\begin{lemma}
  \label{lem:ternary_nondegenerate}
  The maps in \eqref{eq:groupoid_basicality_1}
  and~\eqref{eq:groupoid_basicality_2} are well-defined isomorphisms
  if and only if the following three maps are well-defined
  isomorphisms:
  \begin{align*}
    (p_1,p_2)&\colon R\to \Gr^1\times_{\s,\Gr^0,\rg} \Gr^1,\\
    (p_2,p_3)&\colon R\to \Gr^1\times_{\s,\Gr^0,\s} \Gr^1,\\
    (p_1,p_3)&\colon R\to \Gr^1\times_{\rg,\Gr^0,\rg} \Gr^1.
  \end{align*}
\end{lemma}

\begin{proof}
  The first isomorphism is the definition of~\(R\), the second one is
  \eqref{eq:groupoid_basicality_1}, the third one
  is~\eqref{eq:groupoid_basicality_2}.
\end{proof}

When we view the multiplication as a relation, then the isomorphisms
\eqref{eq:groupoid_basicality_1}
and~\eqref{eq:groupoid_basicality_2} become similar to the statement
that the multiplication is a partially defined map.

\subsection{Equivalence of both definitions of a groupoid}
\label{sec:groupoid_definitions}

\begin{proposition}
  \label{pro:unit_inverse_from_basicality}
  Let \((\Cat,\covers)\) be a category with a subcanonical
  pretopology.  Let \((\Gr^0,\Gr^1,\rg,\s,\mul)\) be a groupoid
  in~\((\Cat,\covers)\) as in Definition~\textup{\ref{def:groupoid}}.
  Then there are unique maps \(\unit\colon \Gr^0\to \Gr^1\) and
  \(\invers\colon \Gr^1\to \Gr^1\) with the properties of unit map and
  inversion listed above.  Moreover, the multiplication~\(\mul\) is a
  cover.
\end{proposition}

\begin{proof}
  We write down the proof in terms of elements.  Interpreting this as
  explained above gives a proof in a general category with
  pretopology.

  First we construct the unit map.  For any \(g\in \Gr^1\), the
  elementwise interpretation of~\eqref{eq:groupoid_basicality_1} in
  Lemma~\ref{lem:interpret_invertibility_as_map} gives a unique map
  \(\bar{\unit}\colon \Gr^1\to \Gr^1\) with
  \(\s(\bar{\unit}(g))=\rg(g)\) and \(\bar{\unit}(g)\cdot g=g\).
  Hence \(\rg(\bar{\unit}(g))=\rg(\bar{\unit}(g)\cdot g)=\rg(g)\).
  Associativity implies
  \[
  \bar{\unit}(g)\cdot (g\cdot h)
  = (\bar{\unit}(g)\cdot g)\cdot h
  = g\cdot h
  = \bar{\unit}(g\cdot h)\cdot (g\cdot h)
  \]
  for all \(g,h\in \Gr^1\) with \(\s(g)=\rg(h)\).  Since
  \(\bar{\unit}(g\cdot h)\) is unique, this gives \(\bar{\unit}(g\cdot
  h) = \bar{\unit}(g)\) for all \(g,h\in \Gr^1\) with
  \(\s(g)=\rg(h)\), that is, \(\bar{\unit}\circ \mul =
  \bar{\unit}\circ \pr_1\).  The map \((\pr_1,\mul)\) is the
  isomorphism in~\eqref{eq:groupoid_basicality_2}, so we get
  \(\bar{\unit}\circ \pr_1=\bar{\unit}\circ \pr_2\) on
  \(\Gr^1\times_{\rg,\Gr^0,\rg} \Gr^1\).  Since the
  pretopology~\(\covers\) is subcanonical, the cover~\(\rg\) is the
  coequaliser of \(\pr_1,\pr_2\colon \Gr^1\times_{\rg,\Gr^0,\rg}
  \Gr^1\rightrightarrows \Gr^1\).  Hence there is a unique map
  \(\unit\colon \Gr^0\to \Gr^1\) with \(\bar{\unit}=\unit\circ\rg\).
  Since~\(\rg\) is an epimorphism, the equations
  \(\s\circ\bar{\unit}=\s\) and \(\rg\circ\bar{\unit}=\rg\) imply
  \(\s\circ \unit=\id_{\Gr^0}\) and \(\rg\circ \unit=\id_{\Gr^0}\).
  Interpreting \(1_x=\unit(x)\) for \(x\in \Gr^0\), this becomes
  \(\s(1_x)=x=\rg(1_x)\) for all \(x\in \Gr^0\).  The defining
  condition \(\bar{\unit}(g)\cdot g=g\) for \(g\in \Gr^1\) becomes
  \(1_{\rg(g)}\cdot g=g\) for all \(g\in \Gr^1\).  If \(g,h\in \Gr^1\)
  satisfy \(\s(g)=\rg(h)\), then
  \[
  (g\cdot 1_{\s(g)})\cdot h
  = g\cdot (1_{\s(g)}\cdot h)
  = g\cdot (1_{\rg(h)}\cdot h)
  = g\cdot h.
  \]
  Since~\(g\) is the unique map \(x\colon \Gr^1\to \Gr^1\) with
  \(x\cdot h=g\cdot h\) by~\eqref{eq:groupoid_basicality_1}, we also
  get \(g\cdot 1_{\s(g)}=g\) for all \(g\in \Gr^1\).  Hence the unit
  map has all required properties.

  The inverse map \(\invers\colon \Gr^1\to \Gr^1\) is defined as the
  unique map with \(\s(\invers(g))=\rg(g)\) and \(\invers(g)\cdot
  g=1_{\s(g)}\) for all \(g\in \Gr^1\).  Since
  \(\s(g)=\s(1_{\s(g)})\), \eqref{eq:groupoid_basicality_1} provides
  a unique map with this property.  We write \(g^{-1}\) instead
  of~\(\invers(g)\) in the following.  Then \(\s(g^{-1})=\rg(g)\),
  \(g^{-1}\cdot g=1_{\s(g)}\) by definition, and
  \(\rg(g^{-1})=\rg(g^{-1}\cdot g)=\rg(1_{\s(g)})= \s(g)\).
  Associativity gives
  \[
  (g\cdot g^{-1})\cdot g
  = g\cdot (g^{-1}\cdot g)
  = g\cdot 1_{\s(g)}
  = g
  = 1_{\rg(g)}\cdot g.
  \]
  Since the map \(x\colon \Gr^1\to \Gr^1\) with \(x\cdot g =
  1_{\rg(g)}\cdot g\) is unique by~\eqref{eq:groupoid_basicality_1},
  this implies \(g\cdot g^{-1}=1_{\rg(g)}\).  Similarly, associativity
  gives
  \[
  (g\cdot h)^{-1}\cdot (g\cdot h)
  = 1_{\s(g\cdot h)}
  = 1_{\s(h)}
  = (h^{-1}\cdot g^{-1})\cdot (g\cdot h)
  \]
  for all \(g,h\in \Gr^1\) with \(\s(g)=\rg(h)\).  Hence the
  uniqueness of the solution of \(x\cdot (g\cdot h)= 1_{\s(h)}\)
  implies \((g\cdot h)^{-1} = h^{-1}\cdot g^{-1}\) for all \(g,h\in
  \Gr^1\) with \(\s(g)=\rg(h)\).  Similarly,
  \[
  (g^{-1})^{-1}\cdot g^{-1}
  = 1_{\s(g^{-1})}
  = 1_{\rg(g)}
  = g\cdot g^{-1},
  \]
  and the uniqueness of the solution to \(x\cdot g^{-1}=1_{\rg(g)}\)
  gives \((g^{-1})^{-1}=g\) for all \(g\in \Gr^1\).  Thus the
  inversion has all expected properties.

  Since~\(\s\) is a cover, so is the coordinate projection
  \(\pr_2\colon \Gr^1\times_{\s,\Gr^0,\s} \Gr^1\to \Gr^1\).
  Composing~\(\pr_2\) with the invertible map
  in~\eqref{eq:groupoid_basicality_1} gives that \(\mul\colon
  \Gr^1\times_{\s,\Gr^0,\rg} \Gr^1\to \Gr^1\) is a cover.
\end{proof}

\subsection{Examples}

The following examples play an important role for the general theory.

\begin{example}
  \label{exa:0-groupoid}
  An object~\(X\) of~\(\Cat\) is viewed as a groupoid by taking
  \(\Gr^1=\Gr^0=X\), \(\rg=\s=\id_X\), and letting~\(\mul\) be the
  canonical isomorphism \(X\times_X X \congto X\).  A groupoid is
  isomorphic to one of this form if and only if its range or source
  map is an isomorphism.  Such groupoids are called
  \emph{\(0\)\nb-groupoids}.
\end{example}

\begin{example}
  \label{exa:covering_groupoid}
  Let \(p\colon X\prto Y\) be a cover.  Its \emph{\v{C}ech groupoid}
  is the groupoid defined by \(\Gr^0=X\), \(\Gr^1=X\times_{p,Y,p} X\),
  \(\rg(x_1,x_2)=x_1\), \(\s(x_1,x_2)=x_2\) for all \(x_1,x_2\in X\)
  with \(p(x_1)=p(x_2)\), and \((x_1,x_2)\cdot (x_2,x_3)\defeq
  (x_1,x_3)\) for all \(x_1,x_2,x_3\in X\) with
  \(p(x_1)=p(x_2)=p(x_3)\).  The range and source maps are covers by
  construction.  The multiplication is clearly associative.  The
  canonical isomorphisms
  \begin{alignat*}{2}
    \Gr^1\times_{\s,\Gr^0,\rg} \Gr^1
    &\congto X\times_{p,Y,p} X\times_{p,Y,p} X,
    &\qquad (x_1,x_2,x_2,x_3)\mapsto(x_1,x_2,x_3),\\
    \Gr^1\times_{\rg,\Gr^0,\rg} \Gr^1
    &\congto X\times_{p,Y,p} X\times_{p,Y,p} X,
    &\qquad (x_1,x_2,x_1,x_3)\mapsto(x_1,x_2,x_3),\\
    \Gr^1\times_{\s,\Gr^0,\s} \Gr^1
    &\congto X\times_{p,Y,p} X\times_{p,Y,p} X,
    &\qquad (x_1,x_2,x_3,x_2)\mapsto(x_1,x_2,x_3),
  \end{alignat*}
  show that \eqref{eq:groupoid_basicality_1}
  and~\eqref{eq:groupoid_basicality_2} are isomorphisms.  Thus we have
  a groupoid in~\((\Cat,\covers)\).  Its unit and inversion maps are
  given by \(1_x=(x,x)\) and \((x_1,x_2)^{-1}=(x_2,x_1)\).
\end{example}

\begin{example}
  \label{exa:groupoid_base_change}
  Let~\(\Gr\) be a groupoid and let \(p\colon X\prto\Gr^0\) be a
  cover.  We define a groupoid \(p^*\Gr=\Gr(X)\) with objects~\(X\)
  and arrows \(X\times_{p,\Gr^0,\rg} \Gr^1\times_{\s,\Gr^0,p} X\),
  range and source maps \(\pr_1\) and~\(\pr_3\), and with the
  multiplication map defined elementwise by
  \[
  (x_1,g_1,x_2)\cdot (x_2,g_2,x_3)\defeq (x_1,g_1\cdot g_2,x_3)
  \]
  for all \(x_1,x_2,x_3\in X\), \(g_1,g_2\in\Gr^1\) with
  \(p(x_1)=\rg(g_1)\), \(p(x_2)=\s(g_1)=\rg(g_2)\),
  \(p(x_3)=\s(g_2)\); this multiplication is associative.  Since the
  maps \(\s\), \(\rg\), and~\(p\) are covers, the fibre product
  \(X\times_{p,\Gr^0,\rg} \Gr^1\times_{\s,\Gr^0,p} X\) exists and
  the source and range maps of \(\Gr(X)\) are covers.  Straightforward
  computations show that \eqref{eq:groupoid_basicality_1}
  and~\eqref{eq:groupoid_basicality_2} are isomorphisms.  Units
  and inverses are given by \(1_x=(x,1_{p(x)},x)\) and
  \((x_1,g,x_2)^{-1}= (x_2,g^{-1},x_1)\) for all \(x,x_1,x_2\in X\),
  \(g\in\Gr^1\) with \(p(x_1)=\rg(g)\), \(p(x_2)=\s(g)\).

  If~\(\Gr\) is just a space~\(\Gr[Y]\) viewed as a groupoid
  (Example~\ref{exa:0-groupoid}), then \(\Gr(X)\) is the \v{C}ech
  groupoid defined in Example~\ref{exa:covering_groupoid}.

  Let \(p_1\colon X_1\prto X_2\) and \(p_2\colon X_2\prto\Gr^0\) be
  covers.  Then \(p_2\circ p_1\) is a cover and there is a natural
  groupoid isomorphism \(p_1^*(p_2^*\Gr) \cong (p_2\circ
  p_1)^*\Gr\).
\end{example}

\begin{example}
  \label{exa:groups}
  Assume Assumption~\ref{assum:final} about a final
  object~\(\star\).  Then a \emph{group} in~\((\Cat,\covers)\) is a
  groupoid~\(\Gr\) with \(\Gr^0=\star\).  The range and source maps
  \(\Gr^1\to\Gr^0=\star\) are automatically covers by
  Assumption~\ref{assum:final}.  The multiplication is now defined
  on the full product \(\Gr^1\times_\star\Gr^1=\Gr^1\times\Gr^1\),
  and the isomorphisms \eqref{eq:groupoid_basicality_1}
  and~\eqref{eq:groupoid_basicality_2} also take place on the
  product \(\Gr^1\times\Gr^1\).  The unit in a groupoid is a map
  \(\unit\colon \star\to\Gr^1\).  For any object \(\?\inOb\Cat\),
  let \(1\colon \?\to\Gr^1\) be the composite of~\(\unit\) with the
  unique map \(\?\to\star\).  These maps give the \emph{unit
    element} in~\(\Gr^1\).  It satisfies \(1\cdot g=g=g\cdot 1\) for
  any map \(g\colon \?\to\Gr^1\).
\end{example}

\subsection{Functors and natural transformations}
\label{sec:functors}

\begin{definition}
  \label{def:functor}
  Let \(\Gr\) and~\(\Gr[H]\) be groupoids in~\((\Cat,\covers)\).  A
  \emph{functor} from~\(\Gr\) to~\(\Gr[H]\) is given by arrows
  \(F^i\in\Cat(\Gr^i,\Gr[H]^i)\) in~\(\Cat\) for \(i=0,1\) with
  \(\rg_{\Gr[H]}(F^1(g))=F^0(\rg_{\Gr}(g))\) and
  \(\s_{\Gr[H]}(F^1(g))=F^0(\s_{\Gr}(g))\) for all \(g\in \Gr^1\)
  and \(F^1(g_1\cdot g_2) = F^1(g_1)\cdot F^1(g_2)\) for all
  \(g_1,g_2\in \Gr^1\) with \(\s_{\Gr}(g_1)=\rg_{\Gr}(g_2)\).

  The \emph{identity functor} \(\id_{\Gr}\colon \Gr\to\Gr\) on a
  groupoid~\(\Gr\) has \(F^i=\id_{\Gr^i}\) for \(i=0,1\).

  The \emph{product} of two functors \(F_1\colon \Gr\to\Gr[H]\) and
  \(F_2\colon \Gr[H]\to\Gr[K]\) is the functor \(F_2\circ F_1\colon
  \Gr\to\Gr[K]\) given by the composite maps \(F_2^i\circ
  F_1^i\colon \Gr^i\to\Gr[K]^i\) for \(i=0,1\).

  Let \(F_1,F_2\colon \Gr\rightrightarrows\Gr[H]\) be two parallel
  functors.  A \emph{natural transformation} \(F_1\Rightarrow F_2\)
  is \(\Phi\in\Cat(\Gr^0,\Gr[H]^1)\) with
  \(\s(\Phi(x))=F_1^0(x)\), \(\rg(\Phi(x))=F_2^0(x)\) for all
  \(x\in\Gr^0\) and \(\Phi(\rg(g))\cdot F_1^1(g) = F_2^1(g)\cdot
  \Phi(\s(g))\) for all \(g\in\Gr^1\):
  \[
  \begin{tikzpicture}[baseline=(current bounding box.west)]
    \matrix[cd,column sep=4em] (m) {
      F_1^0(\s(g))&F_1^0(\rg(g))\\
      F_2^0(\s(g))&F_2^0(\rg(g))\\
    };
    \begin{scope}[cdar]
      \draw (m-1-1) -- node {\(F_1^1(g)\)} (m-1-2);
      \draw (m-1-1) -- node[swap] {\(\Phi(\s(g))\)} (m-2-1);
      \draw (m-1-2) -- node {\(\Phi(\rg(g))\)} (m-2-2);
      \draw (m-2-1) -- node {\(F_2^1(g)\)} (m-2-2);
    \end{scope}
  \end{tikzpicture}
  \]

  The \emph{identity transformation} \(1_F\colon F\Rightarrow F\) on
  a functor \(F\colon \Gr\to\Gr[H]\) is given by
  \(\Phi(x)=1_{F(x)}\) for all \(x\in\Gr^0\).
  If \(\Phi\colon F_1\Rightarrow F_2\) is a natural transformation,
  then its \emph{inverse}~\(\Phi^{-1}\) is the natural
  transformation \(F_2\Rightarrow F_1\) given by
  \(\invers\circ\Phi\colon g\mapsto \Phi(g)^{-1}\).

  The \emph{vertical product} of natural transformations \(\Phi\colon
  F_1\Rightarrow F_2\) and \(\Psi\colon F_2\Rightarrow F_3\) for three
  functors \(F_1,F_2,F_3\colon \Gr\triplerightarrow\Gr[H]\) is
  the natural transformation \(\Psi\cdot\Phi\colon F_1\Rightarrow
  F_3\) defined by \((\Psi\cdot\Phi)(x) \defeq \Psi(x)\cdot\Phi(x)\)
  for all \(x\in\Gr^0\).

  Let \(F_1,F_1'\colon \Gr\rightrightarrows\Gr[H]\) and
  \(F_2,F_2'\colon \Gr[H]\rightrightarrows\Gr[K]\) be composable
  pairs of parallel functors and let \(\Phi\colon F_1\Rightarrow
  F_1'\) and \(\Psi\colon F_2\Rightarrow F_2'\) be natural
  transformations.  Their \emph{horizontal product} is the natural
  transformation \(\Psi\circ\Phi\colon F_2\circ F_1\Rightarrow
  F_2'\circ F_1'\) defined by
  \[
  (\Psi\circ\Phi)(x)
  \defeq \Psi((F_1')^0(x))\cdot F_2^1(\Phi(x))
  =  (F_2')^1(\Phi(x))\cdot \Psi(F_1^0(x)),
  \]
  \[
  \begin{tikzpicture}[baseline=(current bounding box.west)]
    \matrix[cd,column sep=6em] (m) {
      F_2^0\circ F_1^0(x)&F_2^0\circ (F_1')^0(x)\\
      (F_2')^0\circ F_1^0(x)&(F_2')^0\circ (F_1')^0(x)\\
    };
    \begin{scope}[cdar]
      \draw (m-1-1) -- node {\(F_2^1(\Phi(x))\)} (m-1-2);
      \draw (m-1-1) -- node[swap] {\(\Psi(F_1^0(x))\)} (m-2-1);
      \draw (m-1-2) -- node {\(\Psi((F_1')^0(x))\)} (m-2-2);
      \draw (m-2-1) -- node[swap] {\((F_2')^1(\Phi(x))\)} (m-2-2);
      \draw (m-1-1) -- node[anchor=mid,fill=white]
      {\((\Psi\circ\Phi)(x)\)} (m-2-2);
    \end{scope}
  \end{tikzpicture}
  \]
  for all \(x\in\Gr^0\).
  This diagram commutes because of the naturality of~\(\Psi\)
  applied to~\(\Phi(x)\).
\end{definition}

If \(\Phi\colon F_1\Rightarrow F_2\) is a natural transformation,
then
\[
\Phi^{-1}\cdot \Phi=1_{F_1}\colon F_1\Rightarrow F_1,\qquad
\Phi\cdot\Phi^{-1}=1_{F_2}\colon F_2\Rightarrow F_2,
\]
justifying the name inverse.  Thus all natural transformations
between groupoids are \emph{natural isomorphisms}.

\begin{proposition}
  \label{pro:groupoids_and_functors_and_nattrafo}
  Groupoids, functors, and natural transformations
  in~\((\Cat,\covers)\) with the products defined above form a
  strict \(2\)\nb-category in which all \(2\)\nb-arrows are
  invertible.
\end{proposition}

\begin{proof}
  See~\cites{Baez:Introduction_n, Leinster:Basic_Bicategories,
    Noohi:two-groupoids} for the definition of a strict
  \(2\)\nb-category.  The proof is straightforward.
\end{proof}

\begin{example}
  \label{exa:functor_from_base_change}
  Let~\(\Gr\) be a groupoid and let \(p\colon X\prto\Gr^0\) be a
  cover.  Then~\(p\) and the map \(\pr_2\colon X\times_{p,\Gr^0,\rg}
  \Gr^1\times_{\s,\Gr^0,p} X\to \Gr^1\) form a functor
  \(p_*\colon \Gr(X)\to\Gr\).  Functors of this form are called
  \emph{hypercovers}.
\end{example}

\begin{example}
  \label{exa:base_change_of_functor}
  Let \(\Gr\) and~\(\Gr[H]\) be groupoids, \(F\colon \Gr\to\Gr[H]\)
  a functor, and \(p\colon X\prto\Gr[H]^0\) a cover.  Let
  \(\tilde{X}\defeq \Gr^0\times_{F^0,\Gr[H]^0,p} X\); this exists
  and \(\pr_1\colon \tilde{X}\prto \Gr^0\) is a cover because~\(p\)
  is a cover.  We define a functor \(p^*(F) = F(X)\colon
  \Gr(\tilde{X})\to\Gr[H](X)\) by \(\pr_2\colon \tilde{X}\to X\) on
  objects and by \((x_1,g,x_2)\mapsto
  (\pr_2(x_1),F^1(g),\pr_2(x_2))\) for all \(x_1,x_2\in\tilde{X}\),
  \(g\in\Gr^1\) with \(\pr_1(x_1)=\rg(g)\) and \(\pr_1(x_2)=\s(g)\)
  on arrows.
\end{example}

\begin{example}
  \label{exa:equivalence_id}
  An equivalence from the identity functor on~\(\Gr\) to a functor
  \(F\colon \Gr\to\Gr\) is \(\Phi\in\Cat(\Gr^0,\Gr^1)\) with
  \(\s\circ\Phi=\id_{\Gr^0}\), \(\rg\circ\Phi=F^0\), and
  \begin{equation}
    \label{eq:inner_endomorphism}
    F^1(g) = \Phi(\rg(g)) \cdot g\cdot \Phi(\s(g))^{-1}
    \qquad\text{for all }g\in\Gr^1.
  \end{equation}
  Conversely, let~\(\Phi\) be any section for \(\s\colon
  \Gr^1\prto\Gr^0\).  Then \(F^0=\rg\circ\Phi\)
  and~\eqref{eq:inner_endomorphism} define an endomorphism
  \(\Ad(\Phi)=F\colon \Gr\to\Gr\) such that~\(\Phi\) is a natural
  transformation \(\id_{\Gr^0}\Rightarrow F\).  Functors \(F\colon
  \Gr\to\Gr\) that are naturally equivalent to the identity functor
  are called \emph{inner}.
\end{example}

The horizontal product of \(\Phi_i\colon \id_{\Gr^0}\Rightarrow
\Ad(\Phi_i)\) for \(i=1,2\) is a natural transformation
\[
\Phi_1\circ\Phi_2\colon \id_{\Gr^0}\Rightarrow
\Ad(\Phi_1)\circ\Ad(\Phi_2).
\]
Here
\begin{equation}
  \label{eq:compose_sections}
  \Phi_1\circ\Phi_2(x) \defeq \Phi_1(\rg\circ\Phi_2(x))\cdot \Phi_2(x)
  \qquad\text{for all }x\in\Gr^0.
\end{equation}
This product on sections for~\(\s\) gives a monoid with unit
\(\unit\colon x\mapsto 1_x\).  The map~\(\Ad\) from this monoid to the
submonoid of inner endomorphisms of~\(\Gr\) is a unital homomorphism,
that is, \(\Ad(\Phi_1\circ \Phi_2) = \Ad(\Phi_1)\circ\Ad(\Phi_2)\) and
\(\Ad(1)=\id_{\Gr}\).

\begin{deflemma}
  \label{lem:bisections}
  The following are equivalent for a section~\(\Phi\) of~\(\s\):
  \begin{enumerate}
  \item \(\Phi\) is invertible for the horizontal product~\(\circ\);
  \item \(\Ad(\Phi)\) is an automorphism;
  \item \(\rg\circ\Phi\) is invertible in~\(\Cat\).
  \end{enumerate}
  Sections of~\(\s\) with these equivalent properties are called
  \emph{bisections}.
\end{deflemma}

\begin{proof}
  Let~\(\Phi\) be invertible for~\(\circ\) with
  inverse~\(\Phi^{-1}\).  Then \(\Ad(\Phi^{-1})\) is inverse
  to~\(\Ad(\Phi)\), so \(\Ad(\Phi)\) is an automorphism of~\(\Gr\).
  Then \(\rg\circ\Phi=\Ad(\Phi)^0\in\Cat(\Gr^0,\Gr^0)\) is
  invertible in~\(\Cat\).  Conversely, if \(\alpha\defeq
  \rg\circ\Phi\) is invertible in~\(\Cat\), then \(\Phi^{-1}(x)
  \defeq \Phi(\alpha^{-1}(x))^{-1}\colon \Gr^0\to\Gr^1\) has
  \(\s(\Phi^{-1}(x)) =
  \rg(\Phi(\alpha^{-1}(x)))=\alpha\alpha^{-1}(x)=x\), so it is a
  section, and both \(\Phi\circ\Phi^{-1}= 1_x\) and
  \(\Phi^{-1}\circ\Phi= 1_x\).  We have shown
  (1)\(\Rightarrow\)(2)\(\Rightarrow\)(3)\(\Rightarrow\)(1).
\end{proof}

We may think of a section as a subobject \(\Phi(\Gr^0)\subseteq
\Gr^1\) such that~\(\s\) restricts to an isomorphism
\(\Phi(\Gr^0)\congto\Gr^0\).  Being a bisection means that \emph{both}
\(\s\) and~\(\rg\) restrict to isomorphisms
\(\Phi(\Gr^0)\congto\Gr^0\).  This explains the name ``bisection.''

Isomorphisms of groupoids with natural transformations between them
form a strict \(2\)\nb-groupoid.  Thus the automorphisms of~\(\Gr\)
with natural transformations between them form a strict
\(2\)\nb-group or, equivalently, a crossed module.  The crossed module
combines automorphisms of~\(\Gr\) and bisections because the latter
are the natural transformations from the identity functor to another
automorphism of~\(\Gr\).  The crossed module involves the horizontal
product~\(\circ\) on bisections and the composition~\(\circ\) on
automorphisms, and the group homomorphism~\(\Ad\) from bisections to
automorphisms; another piece of structure is the conjugation action
of automorphisms on bisections, given by conjugation:
\[
F\bullet\Phi\defeq 1_F\circ \Phi\circ 1_{F^{-1}}\colon
\id_{\Gr}=F\circ \id_{\Gr}\circ F^{-1} \Rightarrow F\Ad(\Phi)F^{-1}.
\]
By the definition of the horizontal product, we have
\[
F\bullet\Phi(x) = F^1(\Phi(x))
\qquad\text{for all }x\in\Gr^0.
\]
The crossed module conditions \(\Ad(F\bullet\Phi)=F\Ad(\Phi)F^{-1}\)
and \(\Ad(\Phi_1)\bullet \Phi_2=\Phi_1\circ\Phi_2\circ\Phi_1^{-1}\)
are easy to check.

\subsection{Anafunctors}
\label{sec:vague_functors}

The functor \(p_*\colon \Gr(X)\to\Gr\) from
Example~\ref{exa:functor_from_base_change} is not an isomorphism,
but should be an equivalence of groupoids.  When we formally invert
these functors, we arrive at the following definition:

\begin{definition}
  \label{def:gen_functor}
  Let \(\Gr\) and~\(\Gr[H]\) be groupoids in~\(\Cat\).  An
  \emph{anafunctor} from~\(\Gr\) to~\(\Gr[H]\) is a triple
  \((X,p,F)\), where \(X\inOb\Cat\), \(p\colon X\prto\Gr^0\) is a
  cover, and \(F\colon \Gr(X)\to\Gr[H]\) is a functor,
  with~\(\Gr(X)\) defined in
  Example~\ref{exa:groupoid_base_change}.

  An \emph{isomorphism} between two anafunctors \((X_i,p_i,F_i)\),
  \(i=1,2\), is an isomorphism \(\varphi\in\Cat(X_1,X_2)\) with
  \(p_2\circ\varphi= p_1\) and \(F_2\circ \varphi_*=F_1\), where
  \(\varphi_*\colon \Gr(X_1)\to\Gr(X_2)\) is the functor given on
  objects by~\(\varphi\) and on arrows by \(\varphi_*^1(x_1,g,x_2) =
  (\varphi(x_1),g,\varphi(x_2))\) for all \(x_1,x_2\in X_1\),
  \(g\in\Gr^1\) with \(p_1(x_1)=\rg(g)\), \(p_1(x_2)=\s(g)\).

  The functor \(F\colon \Gr(X)\to\Gr[H]\) is a pair of maps
  \[
  F^0\colon X\to\Gr[H]^0,\qquad
  F^1\colon X\times_{p,\Gr^0,\rg} \Gr^1\times_{\s,\Gr^0,p}
  X\to\Gr[H]^1
  \]
  with \(\rg(F^1(x_1,g,x_2))=F^0(x_1)\),
  \(\s(F^1(x_1,g,x_2))=F^0(x_2)\) for all \(x_1,x_2\in X\),
  \(g\in\Gr^1\) with \(p(x_1)=\rg(g)\), \(\s(g)=p(x_2)\), and
  \(F^1(x_1,g_1,x_2)\cdot F^1(x_2,g_2,x_3)= F^1(x_1,g_1\cdot
  g_2,x_3)\) for all \(x_1,x_2,x_3\in X\), \(g_1,g_2\in\Gr^1\) with
  \(p(x_1)=\rg(g_1)\), \(p(x_2)=\s(g_1)=\rg(g_2)\), and
  \(p(x_3)=\s(g_2)\).
\end{definition}

Pronk~\cite{Pronk:Etendues_fractions} and
Carchedi~\cite{Carchedi:Thesis} and several other authors use spans of
functors like we do, but with weak equivalences instead of
hypercovers.  As a consequence, they use the weak instead of the
strong pull-back to compose spans.  An equivalent notion of anafunctor
is used by Roberts~\cite{Roberts:Anafunctors_localisation}, who also
shows that we get an equivalent bicategory if we invert all weak
equivalences.  The following category of the bicategory of anafunctors
may also be found in~\cite{Roberts:Anafunctors_localisation}.

We are going to compose anafunctors.  Let \(\Gr_1\), \(\Gr_2\)
and~\(\Gr_3\) be groupoids in~\((\Cat,\covers)\) and let
\((X_{ij},p_{ij},F_{ij})\) for \(ij=12\) and \(ij=23\) be anafunctors
from~\(\Gr_i\) to~\(\Gr_j\).  Their product
\((X_{13},p_{13},F_{13})\) is the anafunctor from~\(\Gr_1\)
to~\(\Gr_3\) given by
\begin{align*}
  X_{13}&\defeq X_{12}\times_{F_{12}^0,\Gr_2^0,p_{23}} X_{23},\\
  p_{13}&\defeq p_{12}\circ \pr_1\colon X_{13}\prto X_{12}\prto \Gr^0,\\
  F_{13}&\defeq F_{23}\circ F_{12}(X_{23})\colon \Gr_1(X_{13})\cong
  \Gr_1(X_{12})(X_{23})\to\Gr_2(X_{23})\to \Gr_3.
\end{align*}
Thus \(F_{13}^0(x_1,x_2)= F_{23}^0(x_2)\) for all \(x_1\in X_{12}\),
\(x_2\in X_{23}\) with \(F_{12}^0(x_1)=p_{23}(x_2)\) and
\(F_{13}^1(x_1,x_2,g,x_3,x_4) =
F_{23}^1(x_2,F_{12}^1(x_1,g,x_3),x_4)\) for all \(x_1,x_3\in X_{12}\),
\(x_2,x_4\in X_{23}\), \(g\in\Gr^1\) with
\(F_{12}^0(x_1)=p_{23}(x_2)\), \(F_{12}^0(x_3)=p_{23}(x_4)\),
\(p_{12}(x_1)=\rg(g)\), \(p_{12}(x_3)=\s(g)\).  The map \(\pr_1\colon
X_{13}\to X_{12}\) is a cover because~\(p_{23}\) is one.
Hence~\(p_{13}\) is a cover as a composite of covers.

\begin{lemma}
  \label{lem:composition_vague_functors}
  The composition of anafunctors is associative up to
  isomorphism; the associator comes from the canonical isomorphism
  \[
  (X_{12}\times_{F_{12}^0,\Gr_2^0,p_{23}} X_{23})
  \times_{F_{23}^0\pr_2,\Gr_3^0,p_{34}} X_{34}
  \congto
  X_{12}\times_{F_{12}^0,\Gr_2^0,p_{23}\pr_1} (X_{23}
  \times_{F_{23}^0,\Gr_3^0,p_{34}} X_{34}).
  \]
  The identity functor viewed as an anafunctor with \(X=\Gr^0\)
  and \(p=\id_{\Gr^0}\) is a unit for this composition, up to the
  natural isomorphisms \(Y\times_{\Gr^0} \Gr^0 \cong Y \cong
  \Gr^0\times_{\Gr^0} Y\).
\end{lemma}

\begin{proof}
  The proof is routine.
\end{proof}

Since fibre products are, in general, only associative up to the
canonical isomorphism in the above lemma, anafunctors do not form a
category; together with their isomorphisms, they give a bicategory.
We now incorporate natural transformations into this bicategory.

To simplify notation, we usually ignore the associators in
Lemma~\ref{lem:composition_vague_functors} from now on, assuming that
fibre products are strictly associative.

\begin{definition}
  \label{def:gen_transfor}
  Let \(\Gr\) and~\(\Gr[H]\) be groupoids in~\(\Cat\) and let
  \((X_1,p_1,F_1)\) and \((X_2,p_2,F_2)\) be anafunctors
  from~\(\Gr\) to~\(\Gr[H]\).  Let \(X\defeq
  X_1\times_{p_1,\Gr^0,p_2} X_2\) and let \(p\defeq
  p_1\circ\pr_1=p_2\circ\pr_2\colon X\prto\Gr^0\).  We get functors
  \(F_1\circ (\pr_1)_*\colon \Gr(X)\to\Gr[H]\) and \(F_2\circ
  (\pr_2)_*\colon \Gr(X)\to\Gr[H]\).  An \emph{anafunctor natural
    transformation} from \((X_1,p_1,F_1)\) to \((X_2,p_2,F_2)\) is a
  natural transformation \(\Phi\colon F_1\circ (\pr_1)_*\Rightarrow
  F_2\circ (\pr_2)_*\):
  \[
  \begin{tikzpicture}[baseline=(current bounding box.west)]
    \matrix[cd] (m) {
      X&X_1\\
      X_2&\Gr^0\\
    };
    \begin{scope}[cdar,->>]
      \draw (m-1-1) -- node {\(\pr_1\)} (m-1-2);
      \draw (m-1-1) -- node[swap] {\(\pr_2\)} (m-2-1);
      \draw (m-1-2) -- node {\(p_1\)} (m-2-2);
      \draw (m-2-1) -- node[swap] {\(p_2\)} (m-2-2);
      \draw (m-1-1) -- node[anchor=mid,fill=white]
      {\(p\)} (m-2-2);
    \end{scope}
  \end{tikzpicture}
  \qquad
  \begin{tikzpicture}[baseline=(current bounding box.west)]
    \matrix[cd,column sep=6em] (m) {
      \Gr(X)&\Gr(X_1)\\
      \Gr(X_2)&\Gr[H]\\
    };
    \begin{scope}[cdar]
      \draw (m-1-1) -- node {\((\pr_1)_*\)} (m-1-2);
      \draw (m-1-1) -- node[swap] {\((\pr_2)_*\)} (m-2-1);
      \draw (m-1-2) -- node {\(F_1\)} (m-2-2);
      \draw (m-2-1) -- node[swap] {\(F_2\)} (m-2-2);
    \end{scope}
    \draw[dar] (m-1-2) -- node[fill=white] {\(\Phi\)} (m-2-1);
  \end{tikzpicture}
  \]

  The natural transformation~\(\Phi\) is a map \(\Phi\colon
  X\to\Gr[H]^1\) with \(\s\circ\Phi=F_1^0\circ \pr_1\colon X\to
  \Gr[H]^0\), \(\rg\circ\Phi=F_2^0\circ \pr_2\colon X\to \Gr[H]^0\),
  and
  \[
  \Phi(x_1,x_2)\cdot F_1^1(x_1,g,x_3)
  = F_2^1(x_2,g,x_4) \cdot \Phi(x_3,x_4)
  \]
  for all \(x_1,x_3\in X_1\), \(x_2,x_4\in X_2\), \(g\in\Gr^1\) with
  \(p_1(x_1)=p_2(x_2)=\rg(g)\) and \(p_1(x_3)=p_2(x_4)=\s(g)\).
\end{definition}

\begin{example}
  \label{exa:vague_natural_trafo}
  Let \((X_i,p_i,F_i)\) for \(i=1,2\) be anafunctors from~\(\Gr\)
  to~\(\Gr[H]\) and let \(\varphi\colon X_1\congto X_2\) be an
  isomorphism between them.  Then
  \[
  \Phi(x_1,x_2) \defeq F_1^1(\varphi^{-1}(x_2),1_{p_1(x_1)},x_1)
  = F_2^1(x_2,1_{p_2(x_2)},\varphi(x_1))
  \]
  for all \(x_1\in X_1\), \(x_2\in X_2\) with \(p_1(x_1)=p_2(x_2)\)
  is a natural transformation from \((X_1,p_1,F_1)\) to
  \((X_2,p_2,F_2)\).
\end{example}

\begin{lemma}
  \label{lem:vague_gen_transfor}
  Let \(F_1,F_2\colon \Gr\to\Gr[H]\) be functors and let \(p\colon
  X\prto\Gr^0\) be a cover.  Let \(p_*\colon \Gr(X)\to\Gr\) be
  the induced hypercover.  Then natural
  transformations \(F_1\circ p_*\Rightarrow F_2\circ p_*\) are in
  canonical bijection with natural transformations \(F_1\Rightarrow
  F_2\).
\end{lemma}

\begin{proof}
  A natural transformations \(\Phi\colon F_1\Rightarrow F_2\)
  induces a natural transformation \(\Phi\circ 1_{p_*}\colon
  F_1\circ p_*\Rightarrow F_2\circ p_*\) by the horizontal product;
  explicitly, \(\Phi\circ 1_{p_*}\) is the map \(\Phi\circ p\colon
  X\to\Gr^0\to\Gr[H]^1\).  To show that this canonical map
  \(\Phi\mapsto \Phi\circ 1_{p_*}\) is bijective, we must prove that
  any natural transformation \(\Psi\colon F_1\circ p_*\Rightarrow
  F_2\circ p_*\colon \Gr(X)\rightrightarrows \Gr[H]^1\) factors
  uniquely through~\(p\) and a map \(\Phi\colon \Gr^0\to\Gr[H]^1\);
  the latter is then automatically a natural transformation
  \(F_1\Rightarrow F_2\) because \(p_*^1\colon \Gr(X)^1\to\Gr^1\) is a
  cover, hence an epimorphism, and the conditions for a natural
  transformation hold after composing with~\(p\).

  Since~\(\Psi\) is natural with respect to all arrows
  \((x_1,1_{p(x_1)},x_2)\) in \(\Gr(X)\) for \(x_1,x_2\in X\) with
  \(p(x_1)=p(x_2)\), we get \(\Psi(x_1)=\Psi(x_2)\) if
  \(p(x_1)=p(x_2)\).  Since the pretopology is subcanonical
  and~\(p\) is a cover, \(\Psi\) factors uniquely through \(p\colon
  X\prto\Gr^0\).
\end{proof}

\begin{theorem}
  \label{the:vague_functor_2-category}
  Groupoids in~\((\Cat,\covers)\), anafunctors and natural
  transformations of anafunctors with the composition of
  anafunctors described above form a bicategory with invertible
  \(2\)\nb-arrows.
\end{theorem}

\begin{proof}
  We must compose natural transformations of anafunctors vertically
  and horizontally.  We reduce this to the same constructions for
  ordinary natural transformations: the reduction explains why the
  conditions for a bicategory hold.

  Let \((X_i,p_i,F_i)\) be three anafunctors from~\(\Gr\)
  to~\(\Gr[H]\) and let \(\Phi_{ij}\colon (X_i,p_i,F_i)\Rightarrow
  (X_j,p_j,F_j)\) for \(ij=12,23\) be natural transformations to
  compose vertically.  Let \(X_{ij}\defeq X_i\times_{p_i,\Gr^0,p_j}
  X_j\) and \(p_{ij}\defeq p_i\circ\pr_1= p_j\circ\pr_2\colon
  X_{ij}\prto\Gr^0\) for \(ij=12,13,23\).  Let \(X_{123}\defeq
  X_1\times_{p_1,\Gr^0,p_2} X_2\times_{p_2,\Gr^0,p_3} X_3\) and
  \(p_{123}\defeq p_1\circ \pr_1=p_2\circ\pr_2=p_3\circ\pr_3\colon
  X\prto\Gr^0\).

  The given natural transformations are maps \(\Phi_{12}\colon
  X_{12}\to\Gr[H]^1\) and \(\Phi_{23}\colon X_{23}\to\Gr[H]^1\)
  with, among others, \(\s(\Phi_{12}(x_1,x_2)) = F_1^0(x_1)\),
  \(\rg(\Phi_{12}(x_1,x_2)) = F_2^0(x_2)=\s(\Phi_{23}(x_2,x_3))\),
  \(\rg(\Phi_{23}(x_2,x_3)) = F_3^0(x_3)\) for all \(x_i\in X_i\) with
  \(p_i(x_i)=p_j(x_j)\) for \(i,j\in \{1,2,3\}\).  We define
  \[
  \Phi\colon X_{123}\to \Gr[H]^1,\qquad
  \Phi(x_1,x_2,x_3)\defeq \Phi_{23}(x_2,x_3)\cdot\Phi_{12}(x_1,x_2),
  \]
  which is well-defined because \(\rg(\Phi_{12}(x_1,x_2)) =
  F_2^0(x_2)=\s(\Phi_{23}(x_2,x_3))\); moreover,
  \(\s(\Phi(x_1,x_2,x_3))=\s(\Phi_{12}(x_1,x_2)) = F_1^0(x_1)\) and
  \(\rg(\Phi(x_1,x_2,x_3)) =
  \rg(\Phi_{23}(x_2,x_3)) = F_3^0(x_3)\).  The naturality
  of~\(\Phi_{ij}\) implies that~\(\Phi\) is a natural transformation
  from \(F_1\circ (\pr_1)_*\colon \Gr(X_{123})\to\Gr[H]\) to
  \(F_3\circ (\pr_3)_*\colon \Gr(X_{123})\to\Gr[H]\).
  Lemma~\ref{lem:vague_gen_transfor} shows that~\(\Phi\) factors
  uniquely through \(\pr_{13}\colon X_{123}\prto X_{13}\) and thus
  defines a natural transformation \(\Phi_{13}\colon
  (X_1,p_1,F_1)\Rightarrow (X_3,p_3,F_3)\).  This~\(\Phi_{13}\) is the
  vertical product of \(\Phi_{23}\) and~\(\Phi_{12}\).

  If~\(\Phi_{12}\) comes from an isomorphism of anafunctors
  \(\varphi_{12}\colon X_1\to X_2\) as in
  Example~\ref{exa:vague_natural_trafo}, then the
  composite~\(\Phi_{13}\) above is
  \begin{equation}
    \label{eq:nat_trafo_iso_12}
    \Phi_{13}\colon X_1\times_{p_1,\Gr^0,p_3} X_3
    \xrightarrow[\cong]{\varphi_{12}\times_{\Gr^0,p_3} X_3}
    X_2\times_{p_2,\Gr^0,p_3} X_3 \xrightarrow{\Phi_{23}}
    \Gr[H]^1,
  \end{equation}
  Similarly, if~\(\Phi_{23}\)
  comes from an isomorphism of anafunctors
  \(\varphi_{23}\colon X_2\to X_3\),
  then the vertical product~\(\Phi_{13}\) is
  \begin{equation}
    \label{eq:nat_trafo_iso_23}
    \Phi_{13}\colon X_1\times_{p_1,\Gr^0,p_3} X_3
    \xrightarrow[\cong]{X_1\times_{p_1,\Gr^0} \varphi_{23}^{-1}}
    X_1\times_{p_1,\Gr^0,p_2} X_2 \xrightarrow{\Phi_{12}}
    \Gr[H]^1.
  \end{equation}
  As a consequence, the natural transformation corresponding to the
  identity isomorphism of an anafunctor is a unit for the vertical
  product.

  Since arrows in groupoids are invertible, any natural
  transformation of anafunctors has an inverse: exchange the
  order of the factors in the fibre product and compose~\(\Phi\)
  with the inversion map on~\(\Gr[H]^1\).  This is an inverse for
  the vertical product above.  Thus all \(2\)\nb-arrows are
  invertible.

  When we further compose with a natural transformation
  \(\Phi_{34}\colon (X_3,p_3,F_3)\Rightarrow (X_4,p_4,F_4)\), then
  we may first construct a natural transformation on the fibre
  product~\(X_{1234}\) of all four~\(X_i\).  Then by
  Lemma~\ref{lem:vague_gen_transfor} this factors uniquely through
  the covers \(X_{1234}\prto X_{124}\) and \(X_{1234}\prto
  X_{134}\), and then further through the covers \(X_{124}\prto X_{14}\)
  and \(X_{134}\prto X_{14}\).  The resulting map on~\(X_{14}\) does
  not depend on whether we factor through \(X_{124}\) or~\(X_{134}\)
  first.  Thus the vertical product is associative.

  Now we turn to the horizontal product.  Let~\(\Gr[K]\) be a third
  groupoid in~\((\Cat,\covers)\); let \((X_i,p_i,F_i)\) for
  \(i=1,2\) be anafunctors \(\Gr\to\Gr[H]\) and let
  \((Y_i,q_i,E_i)\) for \(i=1,2\) be anafunctors
  \(\Gr[H]\to\Gr[K]\); let \(\Phi\colon
  (X_1,p_1,F_1)\Rightarrow(X_2,p_2,F_2)\) and \(\Psi\colon
  (Y_1,q_1,E_1)\Rightarrow(Y_2,q_2,E_2)\) be natural
  transformations.  The horizontal product~\(\Psi\circ\Phi\) must be
  a natural transformation of anafunctors \(E_1\circ
  F_1\Rightarrow E_2\circ F_2\).  We are going to reduce it to an
  ordinary horizontal product of natural transformations of
  composable functors.  To make our functors composable, we pass to
  larger fibre products.

  First, let
  \[
  Y\defeq Y_1\times_{q_1,\Gr[H]^0,q_2} Y_2,\qquad
  \bar{E}_i\defeq E_i\circ (\pr_i)_*\colon
  \Gr[H](Y)\to\Gr[H](Y_i)\to \Gr[K];
  \]
  by definition, \(\Psi\) is a natural transformation of functors
  \(\Psi\colon \bar{E}_1\Rightarrow\bar{E}_2\).  Secondly, let
  \[
  Y_iX_j \defeq Y_i\times_{q_i,\Gr[H]^0,F_j^0} X_j
  \]
  for \(i,j\in\{1,2\}\) and
  \[
  YX \defeq Y_1X_1 \times_{\Gr^0} Y_2X_1 \times_{\Gr^0} Y_1X_2
  \times_{\Gr^0} Y_2X_2
  \cong (Y\times_{q,\Gr[H]^0,F_1^0} X_1) \times_{\Gr^0}
  (Y\times_{q,\Gr[H]^0,F_2^0} X_2).
  \]
  Let \(\bar{F}_i\colon \Gr(YX)\to\Gr[H](Y)\) be the composite
  functors \(\Gr(YX)\to \Gr(Y\times_{\Gr^0} X_i)\to\Gr[H](Y)\), where
  the first functor is associated to the coordinate projection
  \(YX\to Y\times_{\Gr^0} X_i\) and the second functor
  is~\(F_i(Y)\).  The natural transformation of anafunctors~\(\Phi\)
  induces a natural transformation of ordinary
  functors \(\bar\Phi\colon \bar{F}_1\Rightarrow\bar{F}_2\).

  Now the natural transformations of ordinary functors \(\Psi\)
  and~\(\bar{\Phi}\) may be composed horizontally in the usual way,
  giving a natural transformation \(\bar{E}_1\circ\bar{F}_1
  \Rightarrow \bar{E}_2\circ\bar{F}_2\).  The composite functor
  \(\bar{E}_i\circ\bar{F}_j\colon \Gr(YX)\to\Gr[K]\) and the
  composite anafunctor \(E_i\circ F_j\) are related as
  follows: the latter is given by a functor \(E_i\circ F_j\colon
  \Gr(Y_iX_j)\to \Gr[K]\), and we get \(\bar{E}_i\circ\bar{F}_j\) by
  composing \(E_i\circ
  F_j\) with the functor \(\Gr(YX)\to\Gr(Y_iX_j)\) associated to the
  coordinate projection \(YX\to Y_iX_j\).  Now
  Lemma~\ref{lem:vague_gen_transfor} shows that the natural
  transformation \(\Psi\circ\bar{\Phi}\) descends to a natural
  transformation of anafunctors \(E_1\circ F_1\Rightarrow
  E_2\circ F_2\).  This defines the horizontal product.

  It is routine to see that the horizontal product is associative and
  commutes with the vertical product (exchange law).  For each proof,
  we replace anafunctors by ordinary functors defined over suitable
  fibre products, and first get the desired equality of natural
  transformations on this larger fibre product.  Then
  Lemma~\ref{lem:vague_gen_transfor} shows that it holds as an
  equality of natural transformations of anafunctors.

  The composition of anafunctors is associative and unital up to
  canonical isomorphisms of anafunctors by
  Lemma~\ref{lem:vague_gen_transfor}.  These isomorphisms give natural
  transformations by Example~\ref{exa:vague_natural_trafo}.  These are
  natural with respect to natural transformations of anafunctors
  and satisfy the usual coherence conditions; this follows from the
  simplifications of the vertical product with isomorphisms in
  \eqref{eq:nat_trafo_iso_12} and~\eqref{eq:nat_trafo_iso_23}.
\end{proof}

\begin{theorem}
  \label{the:vague_functor_localisation}
  The bicategory of anafunctors is the localisation of the
  \(2\)\nb-category
  of groupoids, functors and natural transformations at the class of
  hypercovers, and it is also the localisation at the class of weak equivalences
\end{theorem}

This theorem clarifies in which sense the anafunctor bicategory is the
right bicategory of groupoids to work with.  It follows from results
in~\cite{Roberts:Anafunctors_localisation}: first, Theorem~7.2 says
that the anafunctor bicategory is the localisation at the class of all
weak equivalences; secondly, Proposition 6.4 implies that it makes no
difference whether we localise at hypercovers or weak equivalences.
Instead of a proof, we make the statement of
Theorem~\ref{the:vague_functor_localisation} more concrete.

Let~\(\Fun\)
be the strict \(2\)\nb-category
of groupoids, functors and transformations.  Let \(\Anafun\)
be the bicategory of groupoids, anafunctors and anafunctor natural
transformations.  Let~\(\mathcal{D}\)
be another bicategory.  We embed \(\Fun\to\Anafun\)
in the obvious way, sending a functor \(F\colon \Gr\to\Gr[H]\)
to the anafunctor \((\Gr^0,\id_{\Gr^0},F)\).
We assume now that fibre products in~\(\Cat\)
of the form \(X\times_X Y\)
or \(Y\times_X X\)
are \emph{equal} to~\(Y\),
not just canonically \emph{isomorphic} via the coordinate projection;
this can be arranged by redefining the fibre product in these cases.
This choice ensures that the map from functors to anafunctors above is
part of a \emph{strict} homomorphism of bicategories.  Furthermore, it
implies that the identity functors are strict units in the bicategory
of anafunctors, that is, the unit transformations
\(\id_{\Gr}\circ (X,p,F) \cong (X,p,F)\)
and \((X,p,F) \circ \id_{\Gr[H]} \cong (X,p,F)\)
for an anafunctor \((X,p,F)\colon \Gr\to\Gr[H]\)
are identity \(2\)\nb-arrows.

Since the embedding \(\Fun\hookrightarrow \Anafun\)
is a strict homomorphism, we may restrict a morphism of bicategories
\(E\colon \Anafun\to\mathcal{D}\)
to a morphism \(E|_{\Fun}\colon \Fun\to\mathcal{D}\);
restrict a transformation \(\Phi\colon E_1\to E_2\)
between two morphisms
\(E_1,E_2\colon \Anafun\rightrightarrows\mathcal{D}\)
to a transformation
\(\Phi|_{\Fun}\colon E_1|_{\Fun} \to E_2|_{\Fun}\);
and restrict a modification \(\psi\colon \Phi_1\to\Phi_2\)
between two such transformations
\(\Phi_1,\Phi_2\colon E_1\rightrightarrows E_2\)
to a modification \(\psi|_t\colon \Phi_1|_{\Fun}\to \Phi_2|_{\Fun}\).
Here and below we use the notation
of~\cite{Leinster:Basic_Bicategories} regarding bicategories,
morphisms, transformations, and modifications.

Theorem~\ref{the:vague_functor_localisation} is almost equivalent to
the following three more concrete statements:
\begin{enumerate}
\item Any homomorphism \(\Fun\to\mathcal{D}\)
  that maps hypercovers to equivalences is the restriction of a
  homomorphism \(\Anafun\to\mathcal{D}\).
\item Let \(\bar{E}_1,\bar{E}_2\colon \Anafun\to\mathcal{D}\)
  be homomorphisms.  Any strong transformation
  \(\bar{E}_1|_{\Fun} \to \bar{E}_2|_{\Fun}\)
  is the restriction of a strong transformation
  \(\bar{E}_1\to\bar{E}_2\).
\item Let \(\bar{E}_1,\bar{E}_2\colon \Anafun\to\mathcal{D}\)
  be homomorphisms and let
  \(\bar{\Phi}_1,\bar{\Phi}_2\colon \bar{E}_1\rightrightarrows
  \bar{E}_2\)
  be strong transformations.  Restriction to~\(\Fun\)
  gives a bijection from modifications
  \(\bar\Phi_1\Rightarrow \bar\Phi_2\)
  to modifications
  \(\bar\Phi_1|_{\Fun}\Rightarrow \bar\Phi_2|_{\Fun}\).
\end{enumerate}
Note that we restrict attention to \emph{homo}morphisms and
\emph{strong} transformations, as usual in the theory of localisations
\cites{Pronk:Etendues_fractions, Roberts:Anafunctors_localisation}.
That is, we require the \(2\)\nb-arrows
in~\(\mathcal{D}\)
that appear in the definition of a morphism or a transformation to be
invertible.

Actually, instead of~(2), the localisation theorem
in~\cite{Roberts:Anafunctors_localisation} only says that any strong
transformation \(\Phi\colon \bar{E}_1|_{\Fun}\to \bar{E}_2|_{\Fun}\)
is equivalent to the restriction of a strong transformation
\(\bar{\Phi}\colon \bar{E}_1\to \bar{E}_2\)
by an invertible modification \(t\colon \bar{\Phi}|_{\Fun}\to \Phi\).
We may, however, get rid of~\(t\)
because \(\Fun\)
and~\(\Anafun\)
have the same objects.  Thus we may conjugate~\(\bar{\Phi}\)
by the invertible modification~\(t\)
to get a strong transformation that restricts to~\(\Phi\)
exactly.  Hence the usual form of the localisation theorem gives~(2).
Similarly, instead of~(1), the localisation theorem
in~\cite{Roberts:Anafunctors_localisation} only says that for any
homomorphism \(E\colon \Fun\to\mathcal{D}\)
there is a homomorphism \(\bar{E}\colon \Anafun\to\mathcal{D}\)
such that~\(\bar{E}|_{\Fun}\)
is equivalent to~\(E\),
that is, there are strong transformations between \(E\)
and~\(\bar{E}|_{\Phi}\)
that are inverse to each other up to invertible modifications.  In
particular, this strong transformation gives equivalences
\(E(\Gr)\cong \bar{E}(\Gr)\)
for any groupoid~\(\Gr\);
we may conjugate~\(\bar{E}\)
by these to produce an equivalent homomorphism~\(\bar{E}'\)
so that all the arrows \(\bar{E}'(\Gr) \to E(\Gr)\)
involved in the resulting strong transformation
\(\bar{E}'|_{\Fun} \cong E\)
become identities.  Hence this strong transformation consists only
invertible \(2\)\nb-arrows
\(\bar{E}'(F)\cong E(F)\)
for all functors~\(F\),
subject to some conditions.  We use these invertible \(2\)\nb-arrows
together with the identity \(2\)\nb-arrows
on~\(\bar{E}'(X,p,F)\)
for anafunctors with non-identical~\(p\)
to change~\(\bar{E}'\)
once more to an equivalent homomorphism
\(\bar{E}''\colon \Anafun\to\mathcal{D}\).
Now the strong transformation \(\bar{E}''|_{\Fun} \cong E\)
becomes the identity transformation; that is,
\(\bar{E}''|_{\Fun}= E\).
Thus (1)--(3) are equivalent to the usual form of the localisation
theorem.

If the bicategory~\(\mathcal{D}\)
is strictly unital as well, then it makes sense to restrict attention
to strictly unital homomorphisms.  A homomorphism on~\(\Anafun\)
is strictly unital if and only if its restriction to~\(\Fun\)
is, so assertion~(1) continues to hold for strictly unital
homomorphisms as well.

\subsection{Anafunctor equivalences and ana-isomorphisms}
\label{sec:vague_equivalence_iso}

An arrow \(f\colon \Gr\to\Gr[H]\) in a bicategory is an
\emph{equivalence} if there are an arrow \(g\colon \Gr[H]\to\Gr\) and
invertible \(2\)\nb-arrows \(g\circ f\Rightarrow\id_{\Gr}\) and
\(f\circ g\Rightarrow \id_{\Gr[H]}\).

\begin{definition}
  \label{def:vague_equivalence}
  Equivalences in the bicategory of anafunctors are called
  \emph{anafunctor equivalences}.
\end{definition}

Anafunctor natural transformations may be far from being ``isomorphisms.''
This makes it non-trivial to decide whether a given anafunctor is an
anafunctor equivalence.  We are going to provide necessary and sufficient
conditions for this that are easier to check.

\begin{definition}
  \label{def:vague_isomorphism}
  An \emph{ana-isomorphism} between two groupoids \(\Gr\)
  and~\(\Gr[H]\) in~\((\Cat,\covers)\) is given by an object
  \(X\inOb\Cat\), covers \(p\colon X\prto\Gr^0\) and \(q\colon
  X\prto\Gr[H]^0\), and a groupoid isomorphism \(F\colon p^*(\Gr)
  \congto q^*(\Gr[H])\) with \(F^0=\id_{X}\).  An anafunctor
  from~\(\Gr\) to~\(\Gr[H]\) \emph{lifts to an ana-isomorphism} if it
  is equivalent (through an anafunctor natural transformation) to
  \((X,p,q_*\circ \bar{F})\) for an ana-isomorphism
  \((X,p,q,\bar{F})\).
\end{definition}

\begin{definition}
  \label{def:equivalence_functor}
  A functor \(F\colon \Gr\to\Gr[H]\) is \emph{essentially surjective}
  if the map
  \begin{equation}
    \label{eq:ess_surj}
    \rg_{\Act}\colon \Gr^0\times_{F^0,\Gr[H]^0,\rg} \Gr[H]^1 \to
    \Gr[H]^0,\qquad (g,x)\mapsto \s(x),
  \end{equation}
  is a cover.  It is \emph{almost essentially surjective} if the
  map~\eqref{eq:ess_surj} is a cover locally (see
  Definition~\ref{def:local}).  It is \emph{fully faithful} if the
  fibre product \(\Gr^0\times_{F^0,\Gr[H]^0,\rg}
  \Gr[H]^1\times_{\s,\Gr[H]^0,F^0} \Gr^0\) exists in~\(\Cat\) and the
  map
  \begin{equation}
    \label{eq:fully_faithful}
    \Gr^1 \to \Gr^0\times_{F^0,\Gr[H]^0,\rg}
    \Gr[H]^1\times_{\s,\Gr[H]^0,F^0} \Gr^0,\qquad
    g\mapsto (\rg(g),F^1(g),\s(g)),
  \end{equation}
  is an isomorphism in~\(\Cat\).  It is \emph{almost fully faithful}
  if there is some cover \(q\colon Y\prto \Gr[H]^0\) for which the
  functor~\(q^*(F)\) is fully faithful.
\end{definition}

It is trivial that essentially surjective functors are almost
essentially surjective and fully faithful functors are almost fully
faithful.  The converse holds under
Assumption~\textup{\ref{assum:local_cover}}:

\begin{lemma}
  \label{lem:almost_ess_surjective_ff}
  If Assumption~\textup{\ref{assum:local_cover}} holds, then any
  almost essentially surjective functor is essentially surjective.  An
  essentially surjective and almost fully faithful functor is
  fully faithful.
\end{lemma}

\begin{proof}
  Assumption~\textup{\ref{assum:local_cover}} says that a map that is
  locally a cover is a cover, so an almost essentially surjective
  functor is essentially surjective.

The fibre product
  \(\Gr^0\times_{F^0,\Gr[H]^0,\rg} \Gr[H]^1\) always exists
  because~\(\rg\) is a cover.  If~\(F\) is essentially surjective,
  then \(\s\colon \Gr^0\times_{F^0,\Gr[H]^0,\rg} \Gr[H]^1\prto
  \Gr[H]^0\) is a cover, so the fibre product
  \((\Gr^0\times_{F^0,\Gr[H]^0,\rg} \Gr[H]^1)\times_{\s,\Gr[H]^0,F^0}
  \Gr^0\) exists.  Let \(q\colon Y\prto \Gr[H]^0\) be a cover such
  that~\(q^*(F)\) is fully faithful.  The projection
  \[
  \pr_{234}\colon Y\times_{q,\Gr[H]^0, F^0}
  \Gr^0\times_{F^0,\Gr[H]^0,\rg} \Gr[H]^1\times_{\s,\Gr[H]^0,F^0}
  \Gr^0 \times_{F^0,\Gr[H]^0, q} Y \to
  \Gr^0 \times_{F^0,\Gr[H]^0,\rg} \Gr[H]^1\times_{\s,\Gr[H]^0,F^0}
  \Gr^0
  \]
  is a cover because~\(q\) is a cover.  The pull-back of the
  map~\eqref{eq:fully_faithful} for the functor~\(F\) along this cover
  gives the map~\eqref{eq:fully_faithful} for the functor~\(q^*(F)\),
  which is assumed to be an isomorphism.  Since being an isomorphism
  is a local property by Proposition~\ref{pro:isomorphism_local}, the
  functor~\(F\) is fully faithful if~\(q^*(F)\) is fully faithful.
  (This conclusion holds whenever the fibre product
  \(\Gr^0\times_{F^0,\Gr[H]^0,\rg} \Gr[H]^1\times_{\s,\Gr[H]^0,F^0}
  \Gr^0\) exists.)
\end{proof}

\begin{theorem}
  \label{the:explicit-vague-equivalence}
  Let \((X, p, F)\) be an anafunctor from~\(\Gr\) to~\(\Gr[H]\).
  Then the following are equivalent:
  \begin{enumerate}
  \item \label{e:vague_eq1} \((X,p,F)\) is an anafunctor equivalence;
  \item \label{e:vague_eq2} the functor \(F\colon \Gr(X)\to\Gr[H]\) is
    almost fully faithful and almost essentially surjective;
  \item \label{e:vague_eq3} \((X, p, F)\) lifts to an ana-isomorphism.
  \end{enumerate}
  If Assumption~\textup{\ref{assum:local_cover}} holds, then we may
  replace~\ref{e:vague_eq2} by the condition that~\(F\) is fully
  faithful and essentially surjective -- the usual definition of a
  weak equivalence.
\end{theorem}

The proof of the theorem will occupy the remainder of this subsection.
The last statement about replacing~\ref{e:vague_eq2} is Lemma~\ref{lem:almost_ess_surjective_ff}.  We
will show the implications
\ref{e:vague_eq1}\(\Rightarrow\)\ref{e:vague_eq2}\(\Rightarrow\)\ref{e:vague_eq3}\(\Rightarrow\)\ref{e:vague_eq1}.

We start with the easiest implication:
\ref{e:vague_eq3}\(\Rightarrow\)\ref{e:vague_eq1}.  The main point
here is the following example:

\begin{example}
  \label{exa:base_change_vague_isomorphism}
  Let~\(\Gr\) be a groupoid in~\((\Cat,\covers)\) and let \(p\colon
  X\prto\Gr^0\) be a cover.  The hypercover \(p_*\colon \Gr(X)\to\Gr\)
  is an anafunctor equivalence.  Its quasi-inverse is the anafunctor
  \((X,p,\id_{\Gr(X)})\) from~\(\Gr\) to~\(\Gr(X)\).  The composite
  functor \(\Gr\to\Gr(X)\to\Gr\) is the anafunctor~\((X,p,p_*)\).
  The unit section \(X\to\Gr(X)^1\) gives a natural transformation
  between~\((X,p,p_*)\) and the identity anafunctor
  \((\Gr^0,\id_{\Gr^0},\id_{\Gr})\) on~\(\Gr\).  The other composite
  functor \(\Gr(X)\to\Gr\to\Gr(X)\) is the anafunctor
  \((X\times_{\Gr^0} X,\pr_1,(\pr_2)_*)\).  It is naturally equivalent
  to the identity on~\(\Gr(X)\) because the functors \((\pr_2)_*\)
  and~\((\pr_1)_*\) from \(\Gr(X\times_{\Gr^0} X)\) to~\(\Gr(X)\) are
  naturally equivalent.
\end{example}

If \(F\cong q_*\circ\bar{F}\) for a cover \(q\colon X\to\Gr[H]^0\) and
an isomorphism~\(\bar{F}\), then the functors \(F\colon
\Gr(X)\to\Gr[H]\) and \(p_*\colon \Gr(X)\to\Gr\) are anafunctor equivalences by Example~\ref{exa:base_change_vague_isomorphism}.
Since \((X,p,F)\circ p_* \simeq F\colon \Gr(X)\to \Gr[H]\), we have
\((X,p,F) \simeq F\circ (p_*)^{-1}\), so~\((X,p,F)\) is an anafunctor
equivalence as well; here~\(\simeq\) denotes equivalence of
anafunctors.  Thus~\ref{e:vague_eq3} implies~\ref{e:vague_eq1} in
Theorem~\ref{the:explicit-vague-equivalence}.

Now we show \ref{e:vague_eq2}\(\Rightarrow\)\ref{e:vague_eq3}.  The
following lemma explains why (almost) fully faithful functors are
related to hypercovers.

\begin{lemma}
  \label{lem:hypercover_surjective_fully_faithful}
  Given a functor \(F\colon \Gr(X)\to \Gr[H]\),
  the following statements are equivalent to each other:
  \begin{enumerate}
  \item $F$ is isomorphic to a hypercover;
  \item $F$ is fully faithful and \(F^0\colon X\prto\Gr[H]^0\)
    is a cover;
  \item $F$ is almost fully faithful and \(F^0\colon X\prto\Gr[H]^0\)
    is a cover.
  \end{enumerate}
\end{lemma}

Here ``isomorphic'' means that there is an isomorphism \(\Gr(X)\cong
q^*(\Gr[H])\) intertwining~\(F\) and the hypercover \(q_*\colon
q^*(\Gr[H])\to\Gr[H]\).

\begin{proof}
  If~\(F^0\) is a cover, then \(F\) is essentially surjective,
  so~\(F\) is almost fully faithful if and only if~\(F\) is fully
  faithful by Lemma~\ref{lem:almost_ess_surjective_ff}. Thus we only
  need to show the equivalence between (1) and (2).

  Let~\(F\) be fully faithful and let \(F^0\colon X\prto\Gr[H]^0\) be
  a cover.  Since~\(F\) is fully faithful, the map
  \((\rg,F^1,\s)\colon \Gr(X)^1 \congto X\times_{F^0,\Gr[H]^0,\rg}
  \Gr[H]^1 \times_{\s,\Gr[H]^0,F^0} X\) is an isomorphism.  This
  isomorphism together with the identity map on~\(X\) is a groupoid
  isomorphism from~\(\Gr(X)\) to~\((F^0)^*(\Gr[H])\) that
  intertwines~\(F\) and~\(q_*\).  Conversely, if \(F\cong q_*\),
  then~\(F\) is fully faithful and~\(F^0\) is a cover because~\(q\) is
  one.
\end{proof}

Now assume that \(F\colon \Gr(X)\to\Gr[H]\) is almost fully faithful
and almost essentially surjective.  Then there is a cover \(q\colon
Y\prto\Gr[H]^0\) such that the map
\[
\pr_3\colon X\Gr[H]Y\defeq
X\times_{F^0,\Gr[H]^0,\rg} \Gr[H]^1\times_{\s,\Gr[H]^0,q} Y
\to Y
\]
is a cover.  Then so is \(q\circ\pr_3\colon X\Gr[H]Y\prto \Gr[H]^0\).
The projection \(\pr_1\colon X\Gr[H]Y\to X\) is a cover because \(q\)
and~\(\rg\) are covers.  Hence the anafunctor~\((X,p,F)\) is
equivalent to \((X\Gr[H]Y,p\circ\pr_1,F\circ(\pr_1)_*)\).  This
functor maps \((x,h,y)\mapsto F^0(x)\) on objects and
\((x_1,h_1,y_1,g,x_2,h_2,y_2)\mapsto F^1(x_1,g,x_2)\) on arrows.  It is
equivalent to the functor~\(G\) given by
\[
G^0(x,h,y)\defeq q(y)=\s(h),\qquad
G^1(x_1,h_1,y_1,g,x_2,h_2,y_2)\defeq h_1^{-1}\cdot
F^1(x_1,g,x_2)\cdot h_2;
\]
the natural transformation \(G\Rightarrow F\circ(\pr_1)_*\) is
\((x,h,y)\mapsto h\).  Thus \((X,p,F)\) is equivalent to
\((X\Gr[H]Y,p\circ\pr_1,G)\).  The map \(G^0=q\circ\pr_3\colon
X\Gr[H]Y\prto\Gr[H]^0\) is a cover.

\begin{lemma}
  The functor~\(G\) is almost fully faithful.
\end{lemma}

\begin{proof}
  By assumption, there is a cover \(r\colon Z\prto \Gr[H]^0\) such
  that~\(r^*(F)\) is fully faithful; that is, the map
  \begin{multline*}
    Z\times_{r,\Gr[H]^0, F^0} X\times_{p,\Gr^0,\rg} \Gr^1
    \times_{\s,\Gr^0,p} X \times_{F^0,\Gr[H]^0, r} Z
    \\ \to Z\times_{r,\Gr[H]^0, F^0} X\times_{F^0,\Gr[H]^0,\rg}
    \Gr[H]^1\times_{\s,\Gr[H]^0,F^0} X
    \times_{F^0,\Gr[H]^0, r} Z,\\
    (z_1,x_1,g,x_2,z_2)\mapsto (z_1,x_1,F^1(x_1,g,x_2),x_2,z_2),
  \end{multline*}
  is an isomorphism.  Adding extra variables \(y_1,y_2\in Y\),
  \(h_1,h_2\in\Gr[H]^1\) with \((x_i,h_i,y_i)\in X\Gr[H]Y\) for
  \(i=1,2\) gives an isomorphism
  \[
  (z_1,x_1,h_1,y_1,g,x_2,h_2,y_2,z_2)\mapsto
  (z_1,x_1,h_1,y_1,F^1(x_1,g,x_2),x_2,h_2,y_2,z_2)
  \]
  between the resulting fibre products.  The conjugation map
  \[
  (z_1,x_1,h_1,y_1,h,x_2,h_2,y_2,z_2)\mapsto
  (z_1,x_1,h_1,y_1,h_1^{-1} h h_2,x_2,h_2,y_2,z_2)
  \]
  is an isomorphism between suitable fibre products as well, with
  inverse
  \[
  (z_1,x_1,h_1,y_1,h,x_2,h_2,y_2,z_2)\mapsto
  (z_1,x_1,h_1,y_1,h_1 h h_2^{-1},x_2,h_2,y_2,z_2).
  \]
  Composing these isomorphisms gives the map
  \begin{multline*}
    (z_1,x_1,h_1,y_1,g,x_2,h_2,y_2,z_2)
    \\\mapsto
    (z_1,x_1,h_1,y_1,G^1(x_1,h_1,y_1,g,x_2,h_2,y_2),x_2,h_2,y_2,z_2).
  \end{multline*}
  Since this is an isomorphism, \(G\) is almost fully faithful.
\end{proof}

Since~\(G\) is almost fully faithful and~\(G^0\) is a cover, \(G\) is
isomorphic to a hypercover by
Lemma~\ref{lem:hypercover_surjective_fully_faithful}.  Thus
\((X,p,F)\) lifts to an ana-isomorphism.  This finishes the proof
that \ref{e:vague_eq2}\(\Rightarrow\)\ref{e:vague_eq3} in
Theorem~\ref{the:explicit-vague-equivalence}.

Finally, we show that~\ref{e:vague_eq1} implies~\ref{e:vague_eq2} in
Theorem~\ref{the:explicit-vague-equivalence}; this is the most
difficult part of the theorem.  An anafunctor equivalence is given by
anafunctors \((X, p, F)\colon \Gr \to \Gr[H]\) and \((Y, q, E)\colon
\Gr[H]\to \Gr\) and natural transformations \(\Psi\colon (X, p, F)
\circ (Y, q, E) \Rightarrow\id_{\Gr[H]}\) and \(\Phi\colon (Y, q,
E)\circ (X, p, F) \Rightarrow \id_{\Gr}\), which are automatically
invertible.  More explicitly, \(\Psi\colon Y\times_{E^0, \Gr^0, p} X
\to \Gr[H]^1\) is an anafunctor natural transformation from~\(F\circ p^* E\)
to~\(\id_{\Gr[H]}\); that is, \(\Psi(y,x)\in\Gr[H]^1\) is defined for
all \(y\in Y\), \(x\in X\) with \(E^0(y)=p(x)\) and is an arrow
in~\(\Gr[H]^1\) with \(\s(\Psi(y,x))=F^0(x)\) and
\(\rg(\Psi(y,x))=q(y)\); and for all \((x_1, y_1, h, y_2, x_2) \in
X\times_{p, \Gr^0, E^0} Y \times_{q,\Gr[H]^0,\rg} \Gr[H]^1
\times_{\s,\Gr[H]^0,q} Y\times_{E^0, \Gr^0, p} X\),
\begin{equation}
  \label{eq:Psi_natural}
  \Psi(y_1, x_1) \cdot F^1(x_1, E^1(y_1, h, y_2), x_2)
  = h \cdot \Psi(y_2, x_2);
\end{equation}
and \(\Phi\colon X\times_{F^0, \Gr[H]^0, q} Y \to \Gr^1\) is a natural
transformation from~\(E \circ q^* F\) to~\(\id_{\Gr}\), that is,
\(\Phi(x,y)\in\Gr^1\) is defined for all \(x\in X\), \(y\in Y\)
with \(F^0(x)=q(y)\) and is an arrow in~\(\Gr^1\) with
\(\s(\Psi(x,y))=E^0(y)\) and \(\rg(\Psi(x,y))=p(x)\); and for all
\((y_1, x_1, g, x_2, y_2) \in Y\times_{q, \Gr[H]^0, F^0} X
\times_{p,\Gr^0,\rg} \Gr^1 \times_{\s,\Gr^0,p} X\times_{F^0, \Gr[H]^0,
  q} Y\),
\begin{equation}
  \label{eq:Phi_natural}
  \Phi(x_1, y_1) \cdot E^1(y_1, F^1(x_1, g, x_2), y_2)
  = g \cdot \Phi(x_2, y_2).
\end{equation}

\begin{lemma}
  \label{lem:unique_map_XHY-XGY}
  There is a unique map \(a\colon X\times_{F^0, \Gr[H]^0, \rg}
  \Gr[H]^1 \times_{\s,\Gr[H]^0, q} Y\to \Gr^1\) with
  \begin{equation}
    \label{eq:uniqueness_of_a}
    a(x_1,h,y_2) = \Phi(x_1, y_1) \cdot E^1(y_1,h,y_2)
  \end{equation}
  for all \(x_1\in X\), \(h\in\Gr[H]^1\), \(y_1,y_2\in Y\) with
  \(F^0(x_1)=\rg(h)= q(y_1)\), \(\s(h)=q(y_2)\).

  The following map is an isomorphism:
  \[
  (\pr_1,a,\pr_3)\colon
  X\times_{F^0, \Gr[H]^0, \rg} \Gr[H]^1 \times_{\s,\Gr[H]^0, q} Y
  \congto X\times_{p, \Gr^0, \rg} \Gr^1 \times_{\s, \Gr^0, E^0} Y.
  \]
\end{lemma}

\begin{proof}
  The fibre products \(X\times_{F^0, \Gr[H]^0, \rg} \Gr[H]^1
  \times_{\s,\Gr[H]^0, q} Y\) and \(Y\times_{q,\Gr[H]^0,F^0}
  X\times_{F^0, \Gr[H]^0, \rg} \Gr[H]^1 \times_{\s,\Gr[H]^0, q} Y\)
  exist because \(\rg\) and~\(q\) are covers.  The fibre product
  \(X\times_{p, \Gr^0, \rg} \Gr^1 \times_{\s, \Gr^0, E^0} Y\) exists
  because \(p\) and~\(\s\) are covers.
  Equation~\eqref{eq:uniqueness_of_a} defines a map \(\hat{a}\colon
  Y\times_{q,\Gr[H]^0,F^0} X\times_{F^0, \Gr[H]^0, \rg} \Gr[H]^1
  \times_{\s,\Gr[H]^0, q} Y \to \Gr^1\).  We must check
  that~\(\hat{a}\) factors through the projection~\(\pr_{234}\).
  Since~\(q\) is a cover, so is \(\pr_{234}\).  Since our pretopology
  is subcanonical, \(\hat{a}\) factors through this cover if and only
  if \(\hat{a} \circ \pi_1=\hat{a}\circ\pi_2\), where \(\pi_1,\pi_2\)
  are the two canonical projections
  \begin{multline*}
    (Y\times_{q,\Gr[H]^0,F^0} X\times_{F^0, \Gr[H]^0, \rg} \Gr[H]^1
    \times_{\s,\Gr[H]^0, q} Y) \times_{\pr_{234},\pr_{234}}
    (Y\times_{q,\Gr[H]^0,F^0} X\times_{F^0, \Gr[H]^0, \rg} \Gr[H]^1
    \times_{\s,\Gr[H]^0, q} Y)
    \\ \to
    Y\times_{q,\Gr[H]^0,F^0} X\times_{F^0, \Gr[H]^0, \rg} \Gr[H]^1
    \times_{\s,\Gr[H]^0, q} Y.
  \end{multline*}
  We may identify \(\pi_1,\pi_2\) with the two coordinate projections
  \begin{multline*}
    \pr_{2345},\pr_{1345}\colon
    Y\times_{q,\Gr[H]^0,q} Y\times_{q,\Gr[H]^0,F^0} X\times_{F^0, \Gr[H]^0, \rg} \Gr[H]^1
    \times_{\s,\Gr[H]^0, q} Y
    \\\rightrightarrows Y\times_{q,\Gr[H]^0,F^0} X\times_{F^0, \Gr[H]^0, \rg} \Gr[H]^1
    \times_{\s,\Gr[H]^0, q} Y.
  \end{multline*}
  Hence~\(\hat{a}\) factors through~\(\pr_{234}\) if and only if
  \(\hat{a}(y_1,x_1,h,y_2) = \hat{a}(y_1',x_1,h,y_2)\) for all
  \(y_1,y_1',y_2\in Y\), \(x_1\in X\), \(h\in\Gr[H]^1\) with
  \(q(y_1)=q(y_1')=F^0(x_1)=\rg(h)\), \(\s(h)=q(y_2)\).  Equivalently,
  \[
  \Phi(x_1, y_1) \cdot E^1(y_1,h,y_2)
  = \Phi(x_1, y_1') \cdot E^1(y_1',h,y_2)
  \]
  for \(y_1,y_1',x_1,h,y_2\) as above.  Since
  \(E^1(y_1',h,y_2)=E^1(y_1',1_{q(y_1)},y_1)\cdot E^1(y_1,h,y_2)\),
  and \(q(y_1)=q(y_1')=F^0(x_1)\), this is equivalent to
  \[
  \Phi(x_1, y_1) = \Phi(x_1, y_1') \cdot E^1(y_1',1_{q(y_1)},y_1)
  = \Phi(x_1, y_1') \cdot E^1(y_1',F^1(x_1,1_{p(x_1)},x_1),y_1).
  \]
  This equation is a special case of the naturality
  condition~\eqref{eq:Phi_natural} for~\(\Phi\).  This finishes the
  proof that~\(\hat{a}\) factors through a unique map \(a\colon
  X\times_{F^0, \Gr[H]^0, \rg} \Gr[H]^1 \times_{\s,\Gr[H]^0, q} Y\to
  \Gr^1\).
  Since \(\rg(a(x_1,h,y_2)) = \rg(\Phi(x_1,y_2))=p(x_1)\) and
  \(\s(a(x_1,h,y_2))= \s(E^1(y_1,h,y_2))=E^0(y_2)\), the map
  \((\pr_1,a,\pr_3)\) maps \(X\times_{F^0, \Gr[H]^0, \rg} \Gr[H]^1
  \times_{\s,\Gr[H]^0, q} Y\) to \(X\times_{p, \Gr^0, \rg} \Gr^1
  \times_{\s, \Gr^0, E^0} Y\).

  The property of being an isomorphism is local by
  Proposition~\ref{pro:isomorphism_local}.  Thus the
  map~\((\pr_1,a,\pr_3)\) is an isomorphism if and only if the
  following pull-back along a cover induced by~\(q\) is an
  isomorphism:
  \begin{multline*}
    (\pr_1,\pr_2,\hat{a},\pr_4)\colon
    Y\times_{q,\Gr[H]^0,F^0} X\times_{F^0, \Gr[H]^0, \rg} \Gr[H]^1
    \times_{\s,\Gr[H]^0, q} Y
    \\\congto Y\times_{q,\Gr[H]^0,F^0} X\times_{p, \Gr^0, \rg} \Gr^1
    \times_{\s, \Gr^0, E^0} Y.
  \end{multline*}
  Here we have simplified the pull-back using the construction
  of~\(a\) through~\(\hat{a}\).  Actually, to construct a map in the
  opposite direction, we add even more variables and study the map
  \begin{multline*}
    \beta\colon X \times_{p,\Gr[H]^0, E^0}
    Y\times_{q,\Gr[H]^0,F^0} X\times_{F^0, \Gr[H]^0, \rg} \Gr[H]^1
    \times_{\s,\Gr[H]^0, q} Y \times_{E^0,\Gr^0, p} X
    \times_{F^0,\Gr[H]^0, q} Y
    \\ \to
    X \times_{p,\Gr[H]^0, E^0}
    Y\times_{q,\Gr[H]^0,F^0} X\times_{p, \Gr^0, \rg} \Gr^1
    \times_{\s,\Gr^0, E^0} Y \times_{E^0,\Gr^0, p} X
    \times_{F^0,\Gr[H]^0, q} Y,\\
    (x_3,y_1,x_1,h,y_2,x_2,y_3)\mapsto
    (x_3,y_1,x_1,\Phi(x_1,y_1)\cdot E^1(y_1,h,y_2),y_2,x_2,y_3).
  \end{multline*}
  Since \(p\) and~\(q\) are covers, the fibre products above exist
  and~\(\beta\) is a pull-back of~\((\pr_1,a,\pr_3)\) along a cover.
  By Proposition~\ref{pro:isomorphism_local}, it suffices to prove
  that~\(\beta\) is invertible.

  In the opposite direction, we have the map defined elementwise by
  \[
  \gamma (x_3,y_1,x_1,g,y_2,x_2,y_3)\defeq
  (x_3,y_1,x_1,F^1(x_1,g,x_2)\cdot \Psi(y_2,x_2)^{-1},y_2,x_2,y_3);
  \]
  the product in the fourth entry is well-defined because
  \(\s(F^1(x_1,g,x_2)) = F^0(x_2)
  = \s(\Psi(y_2,x_2))\), and it has range and source
  \[
  \rg(F^1(x_1,g,x_2))= F^0(x_1),\qquad
  \rg(\Psi(y_2,x_2))=q(y_2),
  \]
  respectively, so \((x_3,y_1,x_1, F^1(x_1,g,x_2)\cdot
  \Psi(y_2,x_2)^{-1}, y_2,x_2,y_3)\) belongs to the domain
  of~\(\beta\).  In the following computations, we only consider the
  fourth entry for simplicity because this is the only one that is
  touched by~\(\beta\).  We compute
  \begin{align*}
    \pr_4\circ\beta\circ \gamma &(x_3,y_1,x_1,g,y_2,x_2,y_3)
    \\ &= \Phi(x_1,y_1)\cdot E^1(y_1,F^1(x_1,g,x_2)\cdot
    \Psi^{-1}(y_2,x_2),y_2)
    \\ &= \Phi(x_1,y_1)\cdot E^1(y_1,F^1(x_1,g,x_2),y_3) \cdot
    E^1(y_3,\Psi(y_2,x_2)^{-1},y_2)
    \\ &= g \cdot \Phi(x_2,y_3) \cdot E^1(y_3,\Psi(y_2,x_2)^{-1},y_2),
    \\ \pr_4\circ\gamma\circ \beta &(x_3,y_1,x_1,h,y_2,x_2,y_3)
    \\ &= F^1(x_1, \Phi(x_1,y_1)\cdot E^1(y_1,h,y_2),x_2)\cdot
    \Psi(y_2,x_2)^{-1}
    \\ &= F^1(x_1, \Phi(x_1,y_1),x_3) \cdot F^1(x_3,
    E^1(y_1,h,y_2),x_2)\cdot \Psi(y_2,x_2)^{-1}
    \\ &= F^1(x_1, \Phi(x_1,y_1),x_3) \cdot \Psi(y_1,x_3)^{-1}\cdot h.
  \end{align*}
  The first map multiplies~\(g\) from the right with some complicated
  element~\(\xi\) of~\(\Gr^1\) with \(\rg(\xi)=\s(\xi)=\s(g)\); the
  crucial point is that~\(\xi\) does not depend on~\(g\), so the map
  \(g\mapsto g\cdot \xi^{-1}\) is a well-defined inverse for
  \(\beta\circ\gamma\).  Similarly, the map~\(\gamma\circ\beta\) is
  invertible because it multiplies~\(h\) on the left with some
  element~\(\eta\) of~\(\Gr[H]^1\) with \(\s(\eta)=\rg(\eta)=\rg(h)\)
  that does not depend on~\(h\).
  Thus~\(\beta\) is both left and right invertible, so it is
  invertible.
\end{proof}

\begin{lemma}
  \label{lem:covers}
  The projection \(\pr_3\colon X\Gr[H]Y\defeq X\times_{F^0, \Gr[H]^0,
    \rg} \Gr[H]^1 \times_{\s, \Gr[H]^0, q} Y \to Y\) is a cover.
\end{lemma}

\begin{proof}
  The isomorphism in Lemma~\ref{lem:unique_map_XHY-XGY} allows us to
  replace \(\pr_3\colon X\Gr[H]Y \to Y\) by \(\pr_3\colon X\Gr Y\defeq
  X\times_{p, \Gr^0, \rg} \Gr^1 \times_{\s, \Gr^0, E^0} Y\to Y\).
  Since \(p\colon X\prto \Gr^0\) and \(\s\colon \Gr^1\prto\Gr^0\) are
  covers, so are the induced maps \(\pr_{23}\colon X\Gr Y \prto \Gr
  Y\) and \(\pr_2\colon \Gr Y\prto Y\).  Hence so is their composite
  \(\pr_3\colon X\Gr Y\prto Y\).
\end{proof}

\begin{lemma}
  \label{lem:vage_equivalence_essentially_surjective}
  The functor \(F\colon \Gr(X) \to \Gr[H]\) is almost essentially
  surjective.
\end{lemma}

\begin{proof}
  We consider the following diagram
  \[
  \begin{tikzpicture}[baseline=(current bounding box.west)]
    \matrix[cd] (m) {
      X\times_{F^0, \Gr[H]^0, \rg} \Gr[H]^1
      \times_{\s, \Gr[H], q} Y & Y \\
      X\times_{F^0, \Gr^0, \rg} \Gr[H]^1 &\Gr[H]^0 \\
    };
    \begin{scope}[cdar]
      \draw[->>] (m-1-1) -- node {\(\pr_Y\)} (m-1-2);
      \draw (m-2-1) -- node {\(\s\circ \pr_{\Gr[H]^1}\)} (m-2-2);
      \draw[->>] (m-1-1) -- node[swap] {\(\pr_{12}\)} (m-2-1);
      \draw[->>] (m-1-2) -- node {\(q\)} (m-2-2);
    \end{scope}
  \end{tikzpicture}
  \]
  Lemma~\ref{lem:covers} shows that~\(\pr_Y\) is a cover.  The
  map~\(q\) is a cover by assumption.  Thus \(s\circ\pr_{\Gr[H]^1}\)
  is locally a cover, meaning that~\(F\) is almost essentially
  surjective.
\end{proof}

\begin{lemma}
  \label{lem:vague_equivalence_fully_faithful}
  The functor \(F\colon \Gr(X)\to\Gr[H]\) is almost fully faithful.
\end{lemma}

\begin{proof}
  We claim that~\(q^*(F)\) is fully faithful for the given cover
  \(q\colon Y\prto\Gr[H]^0\), that is, the map
  \begin{multline*}
    \delta\colon
    Y\times_{q,\Gr[H]^0, F^0} X\times_{p,\Gr^0,\rg} \Gr^1
    \times_{\s,\Gr^0,p} X \times_{F^0,\Gr[H]^0,q} Y
    \\\to Y\times_{q,\Gr[H]^0, F^0} X\times_{F^0,\Gr[H]^0,\rg} \Gr[H]^1
    \times_{\s,\Gr[H]^0,q} Y \times_{q,\Gr[H]^0,F^0} X,
    \\ (y_1,x_1,g,x_2,y_2)\mapsto (y_1,x_1,F^1(x_1,g,x_2),y_2,x_2),
  \end{multline*}
  is an isomorphism.  The fibre product on the right exists by
  Lemma~\ref{lem:covers}.  An element of the codomain of~\(\delta\) is
  given by \(y_1,y_2\in Y\), \(x_1,x_2\in X\), \(h\in\Gr[H]^1\) with
  \(q(y_1)=F^0(x_1)=\rg(h)\), \(q(y_2)=F^0(x_2)=\s(h)\).  Hence
  \(E^1(y_1,h,y_2)\), \(\Phi(x_1,y_1)\) and \(\Phi(x_2,y_2)\) are
  well-defined arrows in~\(\Gr^1\), and their ranges and sources match
  so that
  \[
  \epsilon(y_1,x_1,h,x_2,y_2) \defeq
  \Phi(x_1,y_1)\cdot E^1(y_1,h,y_2)\cdot \Phi(x_2,y_2)^{-1} \in \Gr^1
  \]
  is well-defined and has range \(\rg(\Phi(x_1,y_1))=p(x_1)\) and
  source \(\s(\Phi(x_2,y_2))=p(x_2)\).  We get a map
  \begin{multline*}
    \bar\epsilon\colon
    Y\times_{q,\Gr[H]^0, F^0} X\times_{F^0,\Gr[H]^0,\rg} \Gr[H]^1
    \times_{\s,\Gr[H]^0,F^0} X \times_{F^0,\Gr[H]^0,q} Y
    \\\to Y\times_{q,\Gr[H]^0, F^0} X\times_{p,\Gr^0,\rg} \Gr^1
    \times_{\s,\Gr^0,p} X \times_{F^0,\Gr[H]^0,q} Y,
    \\ (y_1,x_1,h,x_2,y_2)\mapsto (y_1,x_1,\epsilon(y_1,x_1,h,x_2,y_2),x_2,y_2).
  \end{multline*}
  The map \(\bar\epsilon\circ\delta\) is the identity on
  \(Y\times_{q,\Gr[H]^0, F^0} X\times_{p,\Gr^0,\rg} \Gr^1
  \times_{\s,\Gr^0,p} X \times_{F^0,\Gr[H]^0,q} Y\) by the naturality
  condition~\eqref{eq:Phi_natural}.  We are going to show that the
  composite \(\delta\circ\bar\epsilon\) is invertible (and hence the
  identity because \(\bar\epsilon\circ\delta\) is the identity).
  Elementwise, \(\delta\circ\bar\epsilon\) maps \((y_1,x_1,h,x_2,y_2)\)
  to \((y_1,x_1,\eta(y_1,x_1,h,x_2,y_2),x_2,y_2)\) with
  \[
  \eta(y_1,x_1,h,x_2,y_2) =
  F^1(x_1,\Phi(x_1,y_1)\cdot E^1(y_1,h,y_2)\cdot \Phi(x_2,y_2)^{-1},x_2).
  \]
  To show that~\(\delta\circ\bar\epsilon\) is an isomorphism, we
  pull~\(\delta\circ\bar\epsilon\) back to the map on
  \[
  X\times_{p,\Gr^0,E^0} Y\times_{q,\Gr[H]^0, F^0} X\times_{F^0,\Gr[H]^0,\rg} \Gr[H]^1
  \times_{\s,\Gr[H]^0,F^0} X \times_{F^0,\Gr[H]^0,q} Y
  \times_{E^0,\Gr^0,p} X
  \]
  that sends \((x_3,y_1,x_1,h,x_2,y_2,x_4)\) to
  \((x_3,y_1,x_1,\eta(y_1,x_1,h,x_2,y_2),x_2,y_2,x_4)\).  Since
  \(p\colon X\to \Gr^0\) is a cover,
  Proposition~\ref{pro:isomorphism_local} shows
  that~\(\delta\circ\bar\epsilon\) is an isomorphism if and only if this
  new map is an isomorphism.  Using the extra variables \(x_3,x_4\),
  we may simplify~\(\eta\):
  \begin{multline*}
    \eta(y_1,x_1,h,x_2,y_2) =
    F^1(x_1,\Phi(x_1,y_1)\cdot E^1(y_1,h,y_2)\cdot
    \Phi(x_2,y_2)^{-1},x_2)
    \\ = F^1(x_1,\Phi(x_1,y_1),x_3)\cdot F^1(x_3, E^1(y_1,h,y_2), x_4)
    \cdot F^1(x_4,\Phi(x_2,y_2)^{-1},x_2)
    \\ = F^1(x_1,\Phi(x_1,y_1),x_3)\cdot \Psi(y_1,x_3)^{-1} \cdot h\cdot
    \Psi(y_2,x_4) \cdot F^1(x_4,\Phi(x_2,y_2)^{-1},x_2).
  \end{multline*}
  This map is invertible because~\(\eta\) multiplies~\(h\) on the left
  and right by expressions that do not depend on~\(h\): the inverse
  maps \((x_3,y_1,x_1,h,x_2,y_2,x_4)\) to
  \begin{multline*}
    (x_3,y_1,x_1,\Psi(y_1,x_3)\cdot F^1(x_1,\Phi(x_1,y_1),x_3)^{-1}
    \cdot h
    \\\times F^1(x_2,\Phi(x_2,y_2),x_4)\cdot \Psi(y_2,x_4)^{-1},x_2,y_2,x_4).
  \end{multline*}
  Thus~\(\delta\circ\bar\epsilon\) is invertible.  It follows
  that~\(\delta\) is invertible, so that~\(F\) is almost
  fully faithful.
\end{proof}

This finishes the proof that \ref{e:vague_eq1} implies~\ref{e:vague_eq2} in
Theorem~\ref{the:explicit-vague-equivalence}, and hence the proof of
the theorem.

\section{Groupoid actions}
\label{sec:actions}

Let~\(\Cat\) be a category with coproducts and~\(\covers\) a
subcanonical pretopology on~\(\Cat\).

\begin{deflemma}
  \label{def:action}
  Let \(\Gr=(\Gr^0,\Gr^1,\rg,\s,\mul)\) be a groupoid
  in~\((\Cat,\covers)\).  A (right) \emph{\(\Gr\)\nb-action}
  in~\(\Cat\) is \(\Act\inOb\Cat\) with \(\s\in\Cat(\Act,\Gr^0)\)
  (\emph{anchor}) and \(\mul\in\Cat(\Act\times_{\s,\Gr^0,\rg} \Gr^1,
  \Act)\) (\emph{action}), denoted multiplicatively as~\(\cdot\), such
  that
  \begin{enumerate}
  \item \(\s(x\cdot g)=\s(g)\) for all \(x\in \Act\), \(g\in \Gr^1\)
    with \(\s(x)=\rg(g)\);
  \item \((x\cdot g_1)\cdot g_2= x\cdot (g_1\cdot g_2)\) for all
    \(x\in \Act\), \(g_1,g_2\in \Gr^1\) with \(\s(x)=\rg(g_1)\),
    \(\s(g_1)=\rg(g_2)\);
  \item \(x\cdot 1_{\s(x)}=x\) for all \(x\in \Act\).
  \end{enumerate}
  This definition does not depend on~\(\covers\).  In the presence of
  (1) and~(2), the third condition is equivalent to
  \begin{enumerate}
  \item[(\(3'\))] \(\mul\colon \Act\times_{\s,\Gr^0,\rg} \Gr^1\to \Act\) is an
    epimorphism;
  \item[(\(3''\))] \(\mul\colon \Act\times_{\s,\Gr^0,\rg} \Gr^1\to \Act\) is a
    cover;
  \item[(\(3'''\))] \((\mul,\pr_2)\colon \Act\times_{\s,\Gr^0,\rg} \Gr^1\to
    \Act\times_{\s,\Gr^0,\s} \Gr^1\), \((x,g)\mapsto (x\cdot g,g)\), is
    invertible.
  \end{enumerate}
  An action where the anchor map is a cover is called a \emph{sheaf}
  over~\(\Gr\).
\end{deflemma}

\begin{proof}
  Conditions (1)--(3) imply \((x\cdot g^{-1})\cdot g= x\cdot
  (g^{-1}\cdot g) = x\cdot 1_{\s(x)}= x\) for all \(x\in \Act\),
  \(g\in \Gr^1\) with \(\s(x)=\s(g)\) and \((x\cdot g)\cdot g^{-1} =
  x\) for \(x\in \Act\), \(g\in \Gr^1\) with \(\s(x)=\rg(g)\).
  Hence the elementwise formula \((x,g)\mapsto (x\cdot g^{-1},g)\)
  gives an inverse for the map in~(\(3'''\)).  Thus (3)
  implies~(\(3'''\)) in the presence of (1) and~(2).

  Assume~(\(3'''\)).  The map \(\pr_1\colon \Act\times_{\s,\Gr^0,\s}
  \Gr^1\to \Act\) is a cover because \(\s\colon \Gr^1\prto\Gr^0\) is a
  cover.  Composing with the isomorphism in~(\(3'''\)) shows
  that~\(\mul\) is a cover.  Covers are epimorphisms because the
  pretopology is subcanonical.  Thus
  \[
  (1)\text{--}(3)\Rightarrow(3''')\Rightarrow(3'')\Rightarrow(3').
  \]

  Since \((x\cdot g)\cdot 1_{\s(g)}= x \cdot (g\cdot 1_{\s(g)}) =
  x\cdot g\) for all \(x\in \Act\), \(g\in \Gr^1\) with
  \(\s(x)=\rg(g)\), the map \(f\colon \Act\to \Act\), \(x\mapsto
  x\cdot 1_{\s(x)}\), satisfies \(f\circ \mul=\mul\).  If~\(\mul\) is
  an epimorphism, this implies \(f=\id_{\Act}\).  Thus \((3')\)
  implies~(3).
\end{proof}

\begin{remark}
  \label{rem:anchor_not_cover}
  Sheaves over a Lie groupoid are equivalent to groupoid actions with
  an étale anchor map
  (see~\cite{Moerdijk-Mrcun:Groupoids_sheaves}*{Section 3.1}).
  Sheaf theory only works well for étale groupoids, however.  If we
  work in the category of smooth manifolds with
  étale maps as covers, then the above definition of a sheaf
  is an equivalent way of defining the
  concept of a sheaf via local sections
  \cite{Moerdijk-Mrcun:Groupoids_sheaves}*{Section 3.1}.  We also
  require the anchor
  map to be surjective, which restricts to sheaves for which all
  stalks are non-empty; this seems a rather mild restriction.  We have
  not yet tried to carry over the usual constructions with
  sheaves to non-étale groupoids using the above definition.

  The class of groupoid actions with a cover as anchor map certainly
  deserves special consideration, but it is not a good idea to require
  this for all actions.  Then we would also have to require this for
  bibundles and, in particular, for bibundle functors.  Hence we could
  only treat covering bibundle functors, where both anchor maps are
  covers.  But the bibundle functor associated to a functor need not
  be covering: this happens if and only if the functor is essentially
  surjective.  As a result, anafunctors and bibundle functors would
  no longer be equivalent.
\end{remark}

\begin{definition}
  \label{def:G-equivariant}
  Let \(\Act\) and~\(\Act[Y]\) be right \(\Gr\)\nb-actions.  A
  \emph{\(\Gr\)\nb-equivariant map} or briefly \emph{\(\Gr\)\nb-map}
  \(\Act\to \Act[Y]\) is \(f\in\Cat(\Act,\Act[Y])\) with
  \(\s(f(x))=\s(x)\) for all \(x\in \Act\) and \(f(x\cdot
  g)=f(x)\cdot g\) for all \(x\in \Act\), \(g\in\Gr^1\) with
  \(\s(x)=\rg(g)\).

  The \(\Gr\)\nb-actions and \(\Gr\)\nb-maps form a category, which we
  denote by~\(\Cat(\Gr)\).  Let \(\Cat_\covers(\Gr) \subseteq
  \Cat(\Gr)\) be the full subcategory of \(\Gr\)\nb-sheaves.
\end{definition}

\begin{definition}
  \label{def:G-invariant}
  Let \(\Act\) be a \(\Gr\)\nb-action and \(\Base\inOb\Cat\).  A map
  \(f\in\Cat(\Act,\Base)\) is \emph{\(\Gr\)\nb-invariant} if
  \(f(x\cdot g)=f(x)\) for all \(x\in\Act\), \(g\in\Gr^1\) with
  \(\s(x)=\rg(g)\).
\end{definition}

\subsection{Examples}
\label{sec:act_examples}

\begin{example}
  \label{exa:action_on_objects}
  Any groupoid~\(\Gr\) acts on~\(\Gr^0\) by \(\s=\id_{\Gr^0}\) and
  \(\rg(g)\cdot g=\s(g)\) for all \(g\in\Gr^1\).
\end{example}

\begin{proposition}
  \label{pro:Cat(Gr)_final}
  \(\Gr^0\) is a final object in \(\Cat(\Gr)\)
  and~\(\Cat_\covers(\Gr)\).
\end{proposition}

\begin{proof}
  For any \(\Gr\)\nb-action~\(\Act\), the source map~\(\s\) is a
  \(\Gr\)\nb-map, and it is the only \(\Gr\)\nb-map \(\Act\to\Gr^0\).
  Since identity maps are covers, \(\Gr^0\) belongs
  to~\(\Cat_\covers(\Gr)\).
\end{proof}

\begin{example}
  \label{exa:action_0-groupoid}
  View \(\Gr[Y]\inOb\Cat\) as a groupoid as in
  Example~\ref{exa:0-groupoid}.  A \(\Gr[Y]\)\nb-action on~\(\Act\) is
  equivalent to a map \(\Act\to \Gr[Y]\), namely, the anchor map of
  the action; the multiplication map \(\Act\times_{\Gr[Y]}
  \Gr[Y]\to\Act\) must be the canonical isomorphism.  A
  \(\Gr[Y]\)\nb-map between actions \(\s_i\colon \Act_i\to\Gr[Y]\),
  \(i=1,2\), of~\(\Gr[Y]\) is a map \(f\colon \Act_1\to\Act_2\) with
  \(\s_2\circ f=\s_1\).  Thus \(\Cat(\Gr[Y])\) is the slice category
  \(\Cat\downarrow \Gr[Y]\) of objects in~\(\Cat\) over~\(\Gr[Y]\).
\end{example}

\begin{example}
  \label{exa:group_action}
  Let~\(\Gr\) be a group in~\((\Cat,\covers)\), that is, \(\Gr^0\) is
  a final object.  Then there is a unique map \(\Act\to\Gr^0\) for any
  \(\Act\inOb\Cat\).  Since this unique map is the only choice for the
  anchor map, a \(\Gr\)\nb-action is given by the multiplication map
  \(\Act\times\Gr^1\to\Act\) alone.  This defines a group action if
  and only if \((x\cdot g_1)\cdot g_2= x\cdot (g_1\cdot g_2)\) for all
  \(x\in \Act\), \(g_1,g_2\in \Gr^1\), and \(x\cdot 1=x\) for all
  \(x\in\Act\); the unit element~\(1\) is defined in
  Example~\ref{exa:groups}.
\end{example}

\subsection{Transformation groupoids}
\label{sec:trafo_groupoids}

\begin{deflemma}
  \label{deflem:transformation_groupoid}
  Let~\(\Act\) be a right \(\Gr\)\nb-action.  The \emph{transformation
    groupoid} \(\Act\rtimes \Gr\) in~\((\Cat,\covers)\) is the
  groupoid with objects~\(\Act\), arrows~\(\Act\times_{\s,\Gr^0,\rg}
  \Gr^1\), range~\(\pr_1\), source~\(\mul\), and multiplication
  defined by
  \[
  (x_1,g_1)\cdot (x_2,g_2) \defeq (x_1,g_1\cdot g_2)
  \]
  for all \(x_1,x_2\in \Act\), \(g_1,g_2\in \Gr^1\) with
  \(\s(x_1)=\rg(g_1)\), \(\s(x_2)=\rg(g_2)\), and \(x_1\cdot
  g_1=x_2\), that is, \(\s(x_1,g_1)=\rg(x_2,g_2)\) in~\(\Act\); then
  \(\s(x_1)=\rg(g_1\cdot g_2)\), so that \((x_1,g_1\cdot g_2)\in
  \Act\times_{\s,\Gr^0,\rg} \Gr^1\).

  This is a groupoid in \((\Cat,\covers)\) with unit map and inversion
  given by \(1_x\defeq (x,1_{\s(x)})\) and \((x,g)^{-1}\defeq (x\cdot
  g,g^{-1})\).  It acts on~\(\Act\) with anchor map \(\s=\id_{\Act}\colon \Act\to
  \Act\), by \(x\cdot (x,g) \defeq x\cdot g\) for all \(x\in \Act\), \(g\in
  \Gr^1\) with \(\s(x)=\rg(g)\).
\end{deflemma}

\begin{proof}
  The range map \(\pr_1\colon \Act\times_{\s,\Gr^0,\rg} \Gr^1\prto
  \Act\) of \(\Act\rtimes\Gr\) is a cover because \(\rg\colon
  \Gr^1\prto \Gr^0\) is a cover; Lemma~\ref{def:action} shows that the
  source map~\(\mul\) of \(\Act\rtimes\Gr\) is a cover as well.  The
  remaining properties are routine computations.  The action of
  \(\Act\rtimes \Gr\) on~\(\Act\) is a special case of
  Example~\ref{exa:action_on_objects}.
\end{proof}

\begin{proposition}
  \label{pro:transformation_groupoid_action}
  Let~\(\Act\) be a \(\Gr\)\nb-action.  An action of the
  transformation groupoid \(\Act\rtimes\Gr\) on an object
  \(\Act[Y]\inOb\Cat\) is equivalent to an action of~\(\Gr\)
  on~\(\Act[Y]\) together with a \(\Gr\)\nb-map \(f\colon
  \Act[Y]\to\Act\).  Furthermore, the following two groupoids are
  isomorphic:
  \[
  \Act[Y]\rtimes(\Act\rtimes\Gr) \cong \Act[Y]\rtimes\Gr.
  \]
  A map \(\Act[Y]\to\Base\) is \(\Act\rtimes\Gr\)-invariant if and
  only if it is \(\Gr\)\nb-invariant and over~\(\Act\), and a map
  between two \(\Act\rtimes\Gr\)\nb-actions is
  \(\Act\rtimes\Gr\)-equivariant if and only if it is
  \(\Gr\)\nb-equivariant.
\end{proposition}

\begin{proof}
  An action of \(\Act\rtimes\Gr\) on~\(\Act[Y]\) is given by an anchor
  map \(f\colon \Act[Y]\to\Act\) and a multiplication map
  \[
  \Act[Y]\times_{f,\Act,\pr_1} (\Act\times_{\s,\Gr^0,\rg}\Gr^1) \to \Act[Y].
  \]
  We compose~\(f\) with \(\s_{\Act}\colon \Act\to\Gr^0\) to get an anchor
  map \(\s_{\Act[Y]}\colon \Act[Y]\to\Gr^0\), and we compose the
  multiplication map with the canonical isomorphism
  \begin{equation}
    \label{eq:trafo_gr_arrows}
    \Act[Y]\times_{\s_{\Act[Y]},\Gr^0,\rg}\Gr^1 \congto
    \Act[Y]\times_{f,\Act,\pr_1} (\Act\times_{\s_{\Act},\Gr^0,\rg}\Gr^1)
  \end{equation}
  to get a \(\Gr\)\nb-action
  on~\(\Act[Y]\): \(y\cdot g\defeq y\cdot (f(y),g)\) for all
  \(y\in\Act[Y]\), \(g\in\Gr^1\) with \(\s_{\Act[Y]}(y)=\rg(g)\).  It is
  routine to check that this is an action of~\(\Gr\).  The map~\(f\)
  is \(\Gr\)\nb-equivariant: \(\s_{\Act}\circ f=\s_{\Act[Y]}\) by
  construction, and \(f(y\cdot g)=\s_{\Act\rtimes\Gr}(f(y),g) =
  f(y)\cdot g\).

  Conversely, given a \(\Gr\)\nb-action on~\(\Act[Y]\) and a
  \(\Gr\)\nb-equivariant map \(f\colon \Act\to\Act[Y]\), we define an
  action of \(\Act\rtimes\Gr\) by taking~\(f\) as the anchor map and
  \(y\cdot (x,g) \defeq y\cdot g\) for all \(y\in\Act[Y]\),
  \(x\in\Act\), \(g\in\Gr^1\) with \(f(y)=x\), \(\s(x)=\rg(g)\).
  These two processes are inverse to each other.  Since both
  constructions are natural, a map is \(\Act\rtimes\Gr\)-invariant if
  and only if it is both \(\Gr\)\nb-invariant and a map over~\(\Act\).

  The transformation groupoids \(\Act[Y]\rtimes(\Act\rtimes\Gr)\)
  and~\(\Act[Y]\rtimes\Gr\) have the same objects~\(\Act[Y]\), and
  isomorphic arrows by~\eqref{eq:trafo_gr_arrows}.  This isomorphism
  also intertwines their range, source and multiplication maps, so
  both transformation groupoids for~\(\Act[Y]\) are isomorphic.  The
  isomorphism~\eqref{eq:trafo_gr_arrows} also implies that
  \(\Act\rtimes\Gr\)- and \(\Gr\)\nb-invariance are equivalent.
\end{proof}

\subsection{Left actions and several commuting actions}
\label{sec:left_actions}

\emph{Left \(\Gr\)\nb-actions} and \(\Gr\)\nb-maps between them are
defined similarly; but we denote anchor maps for left actions by
\(\rg\colon \Act\to \Gr^0\) instead of~\(\s\).

\begin{lemma}
  \label{lem:left_right_action}
  The categories of left and right \(\Gr\)\nb-actions are isomorphic.
\end{lemma}

\begin{proof}
  We turn a right \(\Gr\)\nb-action \((\s,\mul)\) on~\(\Act\) into a left
  \(\Gr\)\nb-action by \(\rg=\s\) and \(g\cdot x\defeq x\cdot g^{-1}\)
  for \(g\in\Gr^1\), \(x\in \Act\) with \(\rg(g^{-1})=\s(g)=\rg(x)\).  A
  left \(\Gr\)\nb-action gives a right one by \(\s=\rg\) and \(x\cdot
  g\defeq g^{-1}\cdot x\) for \(x\in \Act\), \(g\in\Gr^1\) with
  \(\s(x)=\rg(g)=\s(g^{-1})\).  A map is equivariant for left actions
  if and only if it is equivariant for the corresponding right
  actions, so this is an isomorphism of categories.
\end{proof}

\begin{definition}
  \label{def:two_actions}
  Let \(\Gr\) and~\(\Gr[H]\) be groupoids in~\((\Cat,\covers)\).  A
  \emph{\(\Gr,\Gr[H]\)-bibundle} is \(\Act\inOb\Cat\) with a left
  \(\Gr\)\nb-action and a right \(\Gr[H]\)\nb-action such that
  \(\s(g\cdot x)=\s(x)\), \(\rg(x\cdot h)=\rg(x)\), and \((g\cdot
  x)\cdot h=g\cdot (x\cdot h)\) for all \(g\in\Gr^1\), \(x\in \Act\),
  \(h\in\Gr[H]^1\) with \(\s(g)=\rg(x)\), \(\s(x)=\rg(h)\).  Let
  \(\Cat(\Gr,\Gr[H])\) be the category with \(\Gr,\Gr[H]\)-bibundles
  as objects and \emph{\(\Gr,\Gr[H]\)-maps} -- maps \(\Act\to
  \Act[Y]\) that are equivariant for both actions -- as arrows.
\end{definition}

A \(\Gr,\Gr[H]\)-bibundle also has a \emph{transformation groupoid}
\(\Gr\ltimes\Act\rtimes\Gr[H]\) with objects~\(\Act\) and arrows
\[
(\Gr\ltimes\Act\rtimes\Gr[H])^1 \defeq
\Gr^1\times_{\s,\Gr^0,\rg} \Act\times_{\s,\Gr[H]^0,\rg} \Gr[H]^1;
\]
\(\rg(g,x,h)\defeq g\cdot x\), \(\s(g,x,h)\defeq x\cdot h\), and
\[
(g_1,x_1,h_1)\cdot (g_2,x_2,h_2)
\defeq (g_1\cdot g_2,g_2^{-1}\cdot x_1,h_1\cdot h_2)
= (g_1\cdot g_2,x_2\cdot h_1^{-1},h_1\cdot h_2)
\]
for all \(g_1,g_2\in\Gr^1\), \(x_1,x_2\in\Act\),
\(h_1,h_2\in\Gr[H]^1\) with \(\s(g_i)=\rg(x_i)\),
\(\s(x_i)=\rg(h_i)\) for \(i=1,2\) and \(x_1\cdot h_1=g_2\cdot
x_2\), so that \(g_2^{-1}\cdot x_1 = x_2\cdot h_1^{-1}\).

\begin{remark}
  \label{rem:two_actions_versus_product_action}
  Assume Assumption~\ref{assum:final}.  Then the product of two
  groupoids in~\((\Cat,\covers)\) is again a groupoid
  in~\((\Cat,\covers)\); the range and source maps are again covers by
  Lemma~\ref{lem:final_object_products}.  The category
  \(\Cat(\Gr,\Gr[H])\) is isomorphic to \(\Cat(\Gr\times\Gr[H])\):
  actions of \(\Gr\) and~\(\Gr[H]\) determine an action of
  \(\Gr\times\Gr[H]\) by \(x\cdot (g,h)\defeq g^{-1}\cdot x\cdot h\)
  for \(x\in \Act\), \(g\in\Gr^1\), \(h\in\Gr[H]^1\) with
  \(\rg(g)=\s(x)\), \(\rg(x)=\s(h)\).  Here we use that
  \((g,h)\in\Gr^1\times\Gr[H]^1\) is equivalent to \(g\in\Gr^1\) and
  \(h\in\Gr[H]^1\).  The transformation groupoid
  \(\Gr\ltimes\Act\rtimes\Gr[H]\) is naturally isomorphic to the
  transformation groupoid \(\Act\rtimes(\Gr\times\Gr[H])\) via
  \((g,x,h)\mapsto (g\cdot x,g^{-1},h)\).
\end{remark}

More generally, we may consider an object of~\(\Cat\) with several
groupoids acting on the left and several on the right, with all
actions commuting.  We single out \(\Gr,\Gr[H]\)-bibundles because
they are the basis of several bicategories of groupoids.

\subsection{Fibre products of groupoid actions}
\label{sec:fibre_product_actions}

Let \(\Act_1\), \(\Act_2\) and~\(\Act[Y]\) be \(\Gr\)\nb-actions and
let \(f_i\colon \Act_i\to\Act[Y]\) for \(i=1,2\) be \(\Gr\)\nb-maps.
Assume that the fibre product \(\Act\defeq
\Act_1\times_{f_1,\Act[Y],f_2} \Act_2\) exists in~\(\Cat\) (this
happens, for instance, if \(f_1\) or~\(f_2\) is a cover).

\begin{lemma}
  \label{lem:unique_fibre-product_action}
  There is a unique \(\Gr\)\nb-action on~\(\Act\) for which both
  coordinate projections \(\pr_i\colon \Act\to\Act_i\), \(i=1,2\), are
  equivariant.  With this \(\Gr\)\nb-action, \(\Act\) becomes a fibre product in
  the category of \(\Gr\)\nb-actions.
\end{lemma}

\begin{proof}
  We define a \(\Gr\)\nb-action on~\(\Act\) as follows.  The anchor
  map \(\s\colon \Act\to\Gr^0\) is defined by
  \[
  \s(x_1,x_2)\defeq \s(x_1)=\s(f_1(x_1))=\s(f_2(x_2))=\s(x_2),
  \]
  the multiplication by \((x_1,x_2)\cdot g\defeq (x_1\cdot g,x_2\cdot
  g)\) for all \(x_1,x_2\in\Act_1\), \(g\in\Gr^1\) with
  \(f_1(x_1)=f_2(x_2)\) and \(\s(x_1,x_2)=\rg(g)\).  This elementwise
  formula defines a map
  \((\Act_1\times_{f_1,\Act[Y],f_2}\Act_2)\times_{\s,\Gr^0,\rg} \Gr^1
  \to \Act_1\times_{f_1,\Act[Y],f_2}\Act_2\) by taking
  \(\?=(\Act_1\times_{f_1,\Act[Y],f_2}\Act_2)\times_{\s,\Gr^0,\rg}
  \Gr^1\) and \(x_1\), \(x_2\), \(g\) the three coordinate projections
  on~\(\?\).

  Routine computations show that this defines a \(\Gr\)\nb-action
  on~\(\Act\).  The coordinate projections \(\Act\to\Act_i\) for
  \(i=1,2\) are \(\Gr\)\nb-equivariant by construction, and the above
  formulas are clearly the only ones that make this happen.  Finally,
  a map \(h\colon \Act[W]\to\Act\) is a \(\Gr\)\nb-map if and only if
  \(h_i\defeq \pr_i\circ h\) for \(i=1,2\) are \(\Gr\)\nb-maps, and
  two \(\Gr\)\nb-maps \(h_i\colon \Act[W]\to \Act_i\) combine to a
  \(\Gr\)\nb-map \(\Act[W]\to\Act\) if and only if \(f_1\circ
  h_1=f_2\circ h_2\).  Thus~\(\Act\) has the property of a
  fibre product in the category~\(\Cat(\Gr)\).
\end{proof}

\subsection{Actors: another category of groupoids}
\label{sec:actors}

A functor \(\Gr\to\Gr[H]\) between two groupoids does \emph{not}
induce a functor \(\Cat(\Gr[H])\to\Cat(\Gr)\).  (The only general
result is Proposition~\ref{pro:bibundles_act}, which gives a functor
\(\Cat_\covers(\Gr[H])\to \Cat(\Gr)\) on the subcategory of sheaves.)
Here we introduce actors, a different type of groupoid morphisms that
are equivalent to functors \(\Cat(\Gr[H])\to\Cat(\Gr)\) with some
extra properties.  In the context of locally compact groupoids, these
were studied by Buneci (see \cites{Buneci:Morphisms_dynamical,
  Buneci-Stachura:Morphisms_groupoids}) because they induce morphisms
between groupoid \(\textup{C}^*\)\nb-algebras.  In the context of Lie
groupoids and Lie algebroids, these were studied under the name of
comorphisms (see \cite{Chen-Liu:Comorphisms},
\cite{Mackenzie:General_Lie_groupoid_algebroid}*{Definition 4.3.16}
and the references there).  Our notion of ``actor'' is unrelated to
the one in~\cite{Pradines:survey08}*{A.2}.

\begin{definition}
  \label{def:actors}
  Let \(\Gr\) and~\(\Gr[H]\) be groupoids in~\((\Cat,\covers)\).  An
  \emph{actor} from~\(\Gr\) to~\(\Gr[H]\) is a left
  \(\Gr\)\nb-action on~\(\Gr[H]^1\) that commutes with the right
  multiplication action of~\(\Gr[H]\) on~\(\Gr[H]^1\).
\end{definition}

Other authors exchange left and right in the above definition.  The
following generalises
\cite{Buss-Exel-Meyer:InverseSemigroupActions}*{Lemma 4.3}:

\begin{proposition}
  \label{pro:actor_as_functor}
  An actor from~\(\Gr\) to~\(\Gr[H]\) is equivalent to a pair
  consisting of a left action of~\(\Gr\) on~\(\Gr[H]^0\) and a functor
  \(\Gr\ltimes\Gr[H]^0\to \Gr[H]\) that is the identity on objects.
\end{proposition}

\begin{proof}
  Let \(\rg_{\Gr^0}\colon \Gr[H]^1\to\Gr^0\) and \(\mul\colon
  \Gr^1\times_{\s,\Gr^0,\rg} \Gr[H]^1\to\Gr[H]^1\) define an actor
  from~\(\Gr\) to~\(\Gr[H]\).  Since the left and right actions
  on~\(\Gr[H]^1\) commute, we have \(\rg_{\Gr^0}(h_1\cdot
  h_2)=\rg_{\Gr^0}(h_1)\) for all \(h_1,h_2\in\Gr[H]^1\) with
  \(\s(h_1)=\rg_{\Gr[H]^0}(h_2)\).  Thus \(\rg_{\Gr^0}(h) =
  \rg_{\Gr^0}(h\cdot h^{-1}) = \rg_{\Gr^0}(1_{\rg_{\Gr[H]^0}(h)})\).
  With the map \(\rg^0\colon \Gr[H]^0\to\Gr^0\) defined by
  \(\rg^0(x)\defeq \rg_{\Gr^0}(1_x)\) for \(x\in\Gr[H]^0\), this
  becomes \(\rg_{\Gr^0} = \rg^0\circ\rg_{\Gr[H]^0}\).

  The map~\(\rg^0\) is the anchor map of a \(\Gr\)\nb-action
  on~\(\Gr[H]^0\).  The action \(\Gr^1\times_{\s,\Gr^0,\rg^0}
  \Gr[H]^0\to\Gr[H]^0\) is defined elementwise by \(g\cdot x\defeq
  \rg_{\Gr[H]^0}(g\cdot 1_x)\).  We have \(\rg^0(g\cdot x) =
  \rg_{\Gr^0}(1_{g\cdot x}) = \rg_{\Gr^0}(1_{\rg_{\Gr[H]^0}(g\cdot
    1_x)}) = \rg_{\Gr^0}(g\cdot 1_x) = \rg(g)\) and
  \(1_{\rg^0(x)}\cdot x = x\) for all \(x\in\Gr[H]^0\).  Furthermore,
  \(\rg_{\Gr[H]^0}(g\cdot h) = \rg_{\Gr[H]^0}(g\cdot h\cdot h^{-1}) =
  g\cdot \rg_{\Gr[H]^0}(h)\) for all \(g\in\Gr^1\), \(h\in\Gr[H]^1\)
  with \(\s(g)=\rg_{\Gr^0}(h)\).  Hence \((g_1\cdot g_2)\cdot x =
  g_1\cdot (g_2\cdot x)\) for all \(g_1,g_2\in\Gr^1\), \(x\in\Gr[H]^0\)
  with \(\s(g_1)=\rg(g_2)\), \(\s(g_2)(x)=\rg^0(x)\).

  The map
  \[
  F^1\colon \Gr^1\times_{\s,\Gr^0,\rg^0} \Gr[H]^0\to \Gr[H]^1,\qquad
  (g,x)\mapsto g\cdot 1_x,
  \]
  together with the identity map on~\(\Gr[H]^0\) is a functor
  \(F\colon \Gr\ltimes\Gr[H]^0\to \Gr[H]\).  This functor determines
  the actor because \(g\cdot h = (g\cdot 1_{\rg_{\Gr[H]^0}(h)}) \cdot
  h = F(g,\rg_{\Gr[H]^0}(h)) \cdot h\) for all \(g\in\Gr^1\),
  \(h\in\Gr[H]^1\) with \(\s(g)=\rg_{\Gr^0}(h)\).

  Conversely, a left \(\Gr\)\nb-action on~\(\Gr[H]^0\) and a functor
  \(F\colon \Gr\ltimes\Gr[H]^0\to \Gr[H]\) yield an actor by taking
  the anchor map \(\rg^0\circ\rg_{\Gr[H]^0}\colon
  \Gr[H]^1\to\Gr[H]^0\to\Gr^0\) and the multiplication \(g\cdot
  h\defeq F(g,\rg_{\Gr[H]^0}(h))\cdot h\).
\end{proof}

Proposition~\ref{pro:actor_as_functor} shows that actors
\(\Gr\to\Gr[H]\) are usually not functors \(\Gr\to\Gr[H]\), and vice
versa.  The intersection of both types of morphisms is described in
the following example:

\begin{example}
  \label{exa:actor_from_nice_functor}
  Let \(F\colon \Gr\to\Gr[H]\) be a functor with invertible
  \(F^0\in\Cat(\Gr^0,\Gr[H]^0)\).  Then we define an associated
  actor by \(\rg_{\Gr}\colon \Gr[H]^1\to\Gr^0\), \(g\mapsto
  (F^0)^{-1}(\rg_{\Gr[H]}(g))\), and \(g\cdot h\defeq F^1(g)\cdot
  h\) for all \(g\in\Gr^1\), \(h\in\Gr[H]^1\) with
  \(\rg(h)=\s(F^1(g)) = F^0(\s(g))\).  Conversely, if an actor has
  the property that the associated anchor map \(\rg\colon
  \Gr[H]^0\to\Gr^0\) is invertible, then we may identify~$\Gr$ with
  the transformation groupoid $\Gr\ltimes\Gr[H]^0$ and get a functor
  \(\Gr\to\Gr[H]\) by Proposition~\ref{pro:actor_as_functor}.  The
  actor associated to this functor is the one we started with.
\end{example}

\begin{proposition}
  \label{pro:actor_acts}
  An actor from~\(\Gr\) to~\(\Gr[H]\) induces a functor
  \(\Cat(\Gr[H])\to\Cat(\Gr)\) that does not change the underlying
  objects in~\(\Cat\), that is, we have a commutative diagram
  \[
  \begin{tikzpicture}[baseline=(current bounding box.west)]
    \matrix[cd] (m) {
      \Cat( \Gr[H])
    & & \Cat(\Gr), \\
    & \Cat & \\
    };
    \begin{scope}[cdar]
      \draw[->] (m-1-1) -- node {} (m-1-3);
      \draw[->] (m-1-3) -- node {$F_{\Gr}$} (m-2-2);
      \draw[->] (m-1-1) -- node[swap] {$F_{\Gr[H]}$} (m-2-2);
    \end{scope}
  \end{tikzpicture}
  \]
  where $F_{\Gr}$ and $F_{\Gr[H]}$ are the forgetful functors forgetting
  the $\Gr$- and $\Gr[H]$-actions, respectively.
  Furthermore, \(\Gr[H]\)\nb-invariant maps are
  also \(\Gr\)\nb-invariant.  Any functor \(\Cat(\Gr[H])\to\Cat(\Gr)\)
  with these extra properties comes from an actor.  Thus groupoids as
  objects and actors as arrows form a category, and
  \(\Gr\mapsto\Cat(\Gr)\) is a contravariant functor on this category.
\end{proposition}

\begin{proof}
  We use categories of left actions during the proof.  Describe an
  actor as in Proposition~\ref{pro:actor_as_functor} as a
  right \(\Gr\)\nb-action on~\(\Gr[H]^0\) with a functor \(F\colon
  \Gr[H]^0\rtimes\Gr\to\Gr[H]\).  Let~\(\Act\) be a left
  \(\Gr[H]\)\nb-action.  Then we define a left \(\Gr\)\nb-action
  on~\(\Act\) by taking the anchor map \(\rg^0\circ\rg\colon
  \Act\to\Gr[H]^0\to\Gr^0\) and the multiplication \(g\cdot x\defeq
  F^1(g,\rg(x))\cdot x\) for all \(g\in\Gr^1\), \(x\in\Act\) with
  \(\s(g) = \rg^0(\rg(x))\).  This action is the unique one with
  \(\rg_{\Gr^0}(h\cdot x)= \rg_{\Gr^0}(h)\) and \(g\cdot (h\cdot x) =
  (g\cdot h)\cdot x\) for all \(g\in\Gr^1\), \(h\in\Gr[H]^1\),
  \(x\in\Act\) with \(\s(g)=\rg_{\Gr^0}(h)\), \(\s(h) = \rg(x)\).

  Routine computations show that the above formulas define a
  \(\Gr\)\nb-action on~\(\Act\) in a natural way so as to give a
  functor \(\Cat(\Gr[H])\to\Cat(\Gr)\), and that
  \(\Gr[H]\)\nb-invariant maps \(f\colon \Act\to\Base\) remain
  \(\Gr\)\nb-invariant.

  Conversely, let \(F\colon \Cat(\Gr[H])\to\Cat(\Gr)\) be a functor
  that does not change the underlying object of~\(\Cat\) and that
  preserves invariant maps.  That is, \(F\) equips any left
  \(\Gr[H]\)\nb-action~\(\Act\) with a left \(\Gr\)\nb-action, such
  that \(\Gr[H]\)\nb-maps remain \(\Gr\)\nb-maps and
  \(\Gr[H]\)\nb-invariant maps remain \(\Gr\)\nb-invariant.  When we
  apply~\(F\) to the left multiplication action on~\(\Gr[H]^1\), we
  get a left \(\Gr\)\nb-action on~\(\Gr[H]^1\).  We claim that this
  commutes with the right multiplication action of~\(\Gr[H]^1\), so
  that we have an actor, and that this actor induces the
  functor~\(F\).

  Let~\(\Act\) be any left \(\Gr[H]\)\nb-action.  Let~\(\Gr[H]\) act
  on \(\Gr[H]\Act\defeq \Gr[H]^1 \times_{\s,\Gr[H]^0,\rg} \Act\) by
  \(h_1\cdot (h_2,x)\defeq (h_1\cdot h_2,x)\) for all
  \(h_1,h_2\in\Gr[H]^1\), \(x\in \Act\) with \(\s(h_1)=\rg(h_2)\),
  \(\s(h_2)=\rg(x)\).  Equip~\(\Gr[H]\Act\) with the \(\Gr\)\nb-action
  from \(F(\Gr[H]\Act)\).  Routine computations show that the maps
  \(\pr_1\colon \Gr[H]\Act\to \Gr[H]^1\) and \(\mul\colon
  \Gr[H]\Act\to \Act\) are \(\Gr[H]\)\nb-equivariant and the map
  \(\pr_2\colon \Gr[H]\Act\to\Act\) is \(\Gr[H]\)\nb-invariant.  By
  assumption, \(\pr_1\) and~\(\mul\) are \(\Gr\)\nb-equivariant
  and~\(\pr_2\) is \(\Gr\)\nb-invariant.  The statements for \(\pr_1\)
  and~\(\pr_2\) mean that \(g\cdot (h,x) = (g\cdot h,x)\) for all
  \(g\in\Gr^1\), \(h\in\Gr[H]^1\), \(x\in\Act\) with
  \(\s(g)=\rg_{\Gr^0}(h)\), \(\s(h)=\rg(x)\).  The
  \(\Gr\)\nb-equivariance of~\(\mul\) gives \(g\cdot \mul(h, x) =
  \mul(g\cdot (h, x)) = \mul(g\cdot h,x)\) or briefly \(g\cdot (h\cdot
  x) = (g\cdot h)\cdot x\) for \(g,h,x\) as above.  For
  \(\Act=\Gr[H]^1\), this says that the left \(\Gr\)\nb-action
  on~\(\Gr^1\) commutes with the right \(\Gr[H]\)\nb-action by
  multiplication, so that we have an actor.  For general~\(\Act\), we
  get \(g\cdot x= g\cdot (1_{\rg(x)}\cdot x) = (g\cdot
  1_{\rg(x)})\cdot x\), so the \(\Gr\)\nb-action given by~\(F\) is
  the one from this actor.  We also see that the functor~\(F\)
  determines the actor uniquely.

  Since the composite of two functors \(\Cat(\Gr[H])\to\Cat(\Gr)\) and
  \(\Cat(\Gr[K])\to\Cat(\Gr[H])\) that does not change the underlying
  object of~\(\Cat\) and preserves invariant maps again has the same
  properties, it must come from an actor.  Thus actors may be
  composed.  More explicitly, given actors \(\Gr\to\Gr[H]\) and
  \(\Gr[H]\to\Gr[K]\), there is a unique right \(\Gr\)\nb-action
  on~\(\Gr[K]^1\) with \((g\cdot h)\cdot k = g\cdot (h\cdot k)\) for
  all \(k\in \Gr[K]^1\), \(h\in\Gr[H]^1\), \(g\in\Gr^1\) with
  \(\s(g)=\rg_{\Gr[G]^0}(h)\) and \(\s(h)=\rg_{\Gr[H]^0}(k)\).  This
  left action gives the actor from~\(\Gr\) to~\(\Gr[K]\) that
  corresponds to the composite functor \(\Cat(\Gr[K])\to\Cat(\Gr)\).
\end{proof}

\begin{proposition}
  \label{pro:isom_actors}
  An actor $\Gr\to \Gr[H]$ is invertible if and only if it comes from an
  isomorphism between $\Gr$ and $\Gr[H]$.
\end{proposition}

\begin{proof}
  The anchor maps \(\Gr[H]^0\to\Gr^0\) are composed when we compose
  actors.  Since this anchor map is the identity for the identity
  actor, the anchor map \(\Gr[H]^0\to\Gr^0\) for an invertible actor
  must be invertible in~\(\Cat\).  The identification between actors
  \(\Gr\to\Gr[H]\) with invertible anchor map \(\Gr[H]^0\to\Gr^0\) and
  functors acting identically on objects in
  Example~\ref{exa:actor_from_nice_functor} is compatible with the
  compositions on both sides.  Thus invertible actors are the same as
  invertible functors, that is, isomorphisms of groupoids.
\end{proof}

In order to enrich actors to a \(2\)\nb-category, we study
automorphisms of~\(\Gr[H]^1\) that preserve the right \(\Gr[H]\)\nb-action:

\begin{proposition}
  \label{pro:left_multiplier}
  The right \(\Gr[H]\)\nb-maps \(\varphi\colon \Gr[H]^1\to\Gr[H]^1\) are exactly
  the maps of the form \(h\mapsto \Phi(\rg(h))\cdot h\) for a
  section \(\Phi\colon \Gr[H]^0\to\Gr[H]^1\) of~\(\s\).  This map is
  invertible if and only if~\(\Phi\) is a bisection.
\end{proposition}

\begin{proof}
  In the following, ``section'' always means a section of~\(\s\).
  It is clear that left multiplication with a section defines an
  \(\Gr[H]\)\nb-map.  Furthermore, the map from sections to
  \(\Gr[H]\)\nb-maps is a homomorphism for the horizontal
  product~\eqref{eq:compose_sections} of sections, so a bisection
  gives an invertible map on~\(\Gr[H]^1\).  It remains to prove the
  converse.  Let \(\varphi\colon \Gr[H]^1\to\Gr[H]^1\) be an
  \(\Gr[H]\)\nb-map and let \(h\in\Gr[H]^1\).  Then
  \(\s(\varphi(h))=\s(h)\), so
  Lemma~\ref{lem:interpret_invertibility_as_map} gives a unique
  \(\omega_h\) with \(\varphi(h)=\omega_h\cdot h\).  This defines a map
  \(\omega\colon \Gr[H]^1\to\Gr[H]^1\) with \(\s\circ\omega=\rg\) and
  \(\omega(h)\cdot h=\varphi(h)\) for all \(h\in\Gr[H]^1\).  Since
  \(\varphi(h_1\cdot h_2)=\varphi(h_1)\cdot h_2\) for all
  \(h_1,h_2\in\Gr[H]^1\) with \(\s(h_1)=\rg(h_2)\), we get
  \(\omega(h_1\cdot h_2)=\omega(h_1)\).  This implies
  \(\omega=\Phi\circ\rg\) for a map \(\Phi\colon
  \Gr[H]^0\to\Gr[H]^1\), compare the proof of
  Proposition~\ref{pro:unit_inverse_from_basicality}.  We have
  \(\s\circ\Phi=\id_{\Gr[H]^0}\), so~\(\Phi\) is a section, and
  \(\varphi(h)=\Phi(\rg(h))\cdot h\) for all \(h\in\Gr[H]^1\) by
  construction.  The section~\(\Phi\) is uniquely determined
  by~\(\varphi\).

  If~\(\varphi\) is invertible, its inverse is also an \(\Gr[H]\)\nb-map
  and hence associated to some section~\(\Psi\).  Since the map
  \(\Phi\mapsto\varphi\) is injective, \(\Phi\) is invertible in the
  monoid of sections, hence it is a bisection by
  Lemma~\ref{lem:bisections}.
\end{proof}

Let \(\mul_1,\mul_2\) be two actors from~\(\Gr\) to~\(\Gr[H]\), that
is, left actions of~\(\Gr\) on~\(\Gr[H]^1\) that commute with the
right multiplication action.  A \(2\)\nb-arrow from~\(\mul_1\)
to~\(\mul_2\) is a \(\Gr,\Gr[H]\)-bibundle map
\((\Gr[H]^1,\mul_1)\to (\Gr[H]^1,\mul_2)\).  By
Proposition~\ref{pro:left_multiplier}, this is the same as a
section~\(\Phi\) of~\(\Gr[H]\) with
\[
\Phi(\rg(g\cdot_1 h))\cdot (g\cdot_1 h) =
g\cdot_2 (\Phi(\rg(h))\cdot_1 h)
\]
for all \(g\in\Gr^1\), \(h\in\Gr[H]^1\) with
\(\s(g)=\rg_{\Gr,1}(\rg(h))\), where \(\cdot_1\) and~\(\cdot_2\)
come from \(\mul_1\) and~\(\mul_2\), respectively.  Composing the
\(\Gr,\Gr[H]\)-maps
associated to sections \(\Phi_1\) and~\(\Phi_2\) gives the
\(\Gr,\Gr[H]\)-map associated to the section \(\Phi_1\circ\Phi_2\)
defined in~\eqref{eq:compose_sections}.

\begin{proposition}
  \label{pro:actor_2-category}
  There is a strict \(2\)\nb-category with groupoids
  in~\((\Cat,\covers)\) as objects, actors as arrows,
  sections~\(\Phi\) as above as \(2\)\nb-arrows, the usual
  composition of actors and the usual unit actors, and with the
  product~\(\circ\) for sections as vertical product.  The
  equivalences in this \(2\)\nb-category are precisely the
  isomorphisms of groupoids in~\((\Cat,\covers)\).
\end{proposition}

\begin{proof}
  It remains to construct the horizontal product.  Let \(\Gr\),
  \(\Gr[H]\) and~\(\Gr[K]\) be groupoids, let \(\mul_1,\mul_2\colon
  \Gr\rightrightarrows\Gr[H]\) and \(\mul'_1,\mul'_2\colon
  \Gr[H]\rightrightarrows\Gr[K]\) be actors, let \(\Phi\colon
  \Gr[H]^0\to\Gr[H]^1\) and \(\Psi\colon \Gr[K]^0\to\Gr[K]^1\) be
  sections with
  \[
  \Phi(\rg(g\cdot_1 h))\cdot (g\cdot_1 h) =
  g\cdot_2 (\Phi(\rg(h))\cdot_1 h),\qquad
  \Psi(\rg(h\cdot_1 k))\cdot (h\cdot_1 k) =
  h\cdot_2 (\Psi(\rg(k))\cdot_1 k)
  \]
  for all \(g\in\Gr^1\), \(h\in\Gr[H]^1\), \(k\in\Gr[K]^1\) with
  \(\s(g)=\rg_{\Gr,1}(\rg(h))\) and \(\s(h)=\rg_{\Gr[H],1}(\rg(k))\),
  respectively.  Then \(\Psi\bullet\Phi\colon \Gr[K]^0\to\Gr[K]^1\),
  \(x\mapsto \Phi(\rg_{\Gr[H],1}(x))\cdot\Psi(x)\), is a section
  for~\(\Gr[K]\) that gives a \(2\)\nb-arrow from the composite
  \(\mul'_1\circ\mul_1\) to \(\mul'_2\circ\mul_2\).

  An equivalence \(\alpha\colon \Gr\to\Gr[H]\) in this
  \(2\)\nb-category has a quasi-inverse \(\beta\colon \Gr[H]\to\Gr\)
  and invertible \(2\)\nb-arrows
  \(\id_{\Gr}\Rightarrow\beta\circ\alpha\) and
  \(\id_{\Gr[H]}\Rightarrow \alpha\circ\beta\).  Being invertible,
  they correspond to bisections.  Then it follows that the actors
  \(\beta\circ\alpha\) and \(\alpha\circ\beta\) come from inner
  automorphisms of \(\Gr\) and~\(\Gr[H]\), respectively, as in
  Example~\ref{exa:actor_from_nice_functor}.  Thus
  \(\beta\circ\alpha\) and \(\alpha\circ\beta\) are isomorphisms in
  the \(1\)\nb-category of actors.  This implies that~\(\alpha\) is
  an isomorphism in the \(1\)\nb-category of actors.  Now
  Proposition~\ref{pro:isom_actors} shows that~\(\alpha\) is an
  isomorphism of categories.
\end{proof}

\section{Principal bundles}
\label{sec:principal_bundles}

\begin{definition}
  \label{def:principal_bundle}
  Let \(\Gr\) be a groupoid in~\((\Cat,\covers)\).  A
  \emph{\(\Gr\)\nb-bundle} (over~\(\Base\)) is a \(\Gr\)\nb-action
  \((\Act,\s,\mul)\) with a \(\Gr\)\nb-invariant map \(\bunp\colon
  \Act\to\Base\) (\emph{bundle projection}).  A \(\Gr\)\nb-bundle is
  \emph{principal} if
  \begin{enumerate}
  \item \(\bunp\) is a cover;
  \item the following map is an isomorphism:
    \begin{equation}
      \label{eq:principal_bundle}
      (\pr_1,\mul)\colon \Act\times_{\s,\Gr^0,\rg}\Gr^1\congto
      \Act\times_{\bunp,\Base,\bunp} \Act,\qquad
      (x,g)\mapsto (x,x\cdot g).
    \end{equation}
  \end{enumerate}
\end{definition}

The isomorphism~\eqref{eq:principal_bundle} is equivalent to the
elementwise statement that for \(x_1,x_2\in \Act\) with
\(\bunp(x_1)=\bunp(x_2)\), there is a unique \(g\in \Gr^1\) with
\(\s(x_1)=\rg(g)\) and \(x_1\cdot g=x_2\); this translation is proved
like Lemma~\ref{lem:interpret_invertibility_as_map}.

\begin{definition}
  \label{def:orbit_space}
  Let~\(\Gr\) be a groupoid and~\(\Act\) a \(\Gr\)\nb-action
  in~\((\Cat,\covers)\).  The \emph{orbit space projection}
  \(\pi_{\Act}\colon \Act\to \Act/\Gr\) is the coequaliser of the
  two maps \(\pr_1,\mul\colon \Act\times_{\s,\Gr^0,\rg}
  \Gr^1\rightrightarrows \Act\).
\end{definition}

While the coequaliser need not exist in~\(\Cat\), it always exists in
the category of presheaves over~\(\Cat\).

\begin{lemma}
  \label{lem:principal_bundle}
  Let~\(\Act\) with bundle projection \(\bunp\colon \Act\prto \Base\)
  be a principal \(\Gr\)\nb-action.  Then~\(\bunp\) is equivalent to
  the orbit space projection \(\Act\to \Act/\Gr\), which therefore
  exists in~\(\Cat\) and is a cover.  The orbit space projection is
  always \(\Gr\)\nb-invariant, and condition~\((3'')\) in
  Definition~\emph{\ref{def:action}} follows
  from~\eqref{eq:principal_bundle}.
\end{lemma}

\begin{proof}
  Since the pretopology is subcanonical, the cover \(\bunp\colon
  \Act\prto \Base\) is the coequaliser of \(\pr_1,\pr_2\colon
  \Act\times_{\bunp,\Base,\bunp}\Act\rightrightarrows \Act\), and these
  two maps are covers.  By the
  isomorphism~\eqref{eq:principal_bundle}, \(\bunp\) is the
  coequaliser of \(\pr_1,\mul\colon \Act\times_{\s,\Gr^0,\rg}
  \Gr^1\rightrightarrows \Act\) as well, and both \(\pr_1\)
  and~\(\mul\) are
  covers.  Thus~\(\bunp\) is an orbit space projection and~\(\mul\)
  is automatically a cover; the latter is
  Definition~\ref{def:action}.(\(3''\)).
\end{proof}

\subsection{Examples}
\label{sec:example_principal}

\begin{example}
  \label{exa:0-groupoid_principal_bundle}
  Let~\(\Gr[Y]\) be a \(0\)\nb-groupoid (see
  Example~\ref{exa:0-groupoid}), let \(\Act\to\Gr[Y]\) be a
  \(\Gr[Y]\)\nb-action, and let \(\bunp=\id_{\Act}\).  This is a
  principal \(\Gr[Y]\)\nb-bundle.
\end{example}

\begin{example}
  \label{exa:covering_groupoid_principal}
  Let \(\bunp\colon \Act\prto \Base\) be a cover and let~\(\Gr\) be
  its \v{C}ech groupoid (see Example~\ref{exa:covering_groupoid}).
  The canonical action of~\(\Gr\) on its objects~\(\Act\) is given by
  \(\s=\id_{\Act}\) and \(x_1\cdot (x_1,x_2)\defeq x_2\) for all
  \(x_1,x_2\in \Act\) with \(\bunp(x_1)=\bunp(x_2)\) (see
  Example~\ref{exa:action_on_objects}).  Together with the bundle
  projection \(\bunp\colon \Act\prto \Base\), this is a principal
  bundle: the map \(\Act\times_{\id_{\Act},\Act,\pr_1}
  (\Act\times_{\bunp,\Base,\bunp} \Act) \to
  \Act\times_{\bunp,\Base,\bunp} \Act\) defined elementwise by
  \((x_1,(x_1,x_2))\mapsto (x_1,x_2)\) is an isomorphism.
\end{example}

\subsection{Pull-backs of principal bundles}
\label{sec:principal_pull-back}

\begin{proposition}
  \label{pro:pull-back_principal}
  Let \(f\colon \tilde{\Base}\to\Base\) be a map in~\(\Cat\) and let
  \((\Act,\s,\bunp,\mul)\) be a principal \(\Gr\)\nb-bundle
  over~\(\Base\).  Then there is a principal \(\Gr\)\nb-bundle
  \((\tilde{\Act},\tilde{\s},\tilde{\bunp},\tilde{\mul})\)
  over~\(\tilde{\Base}\) with a \(\Gr\)\nb-map \(\hat{f}\colon
  \tilde{\Act}\to\Act\) with
  \(\bunp\circ\hat{f}=f\circ\tilde{\bunp}\).  These are unique in the
  following sense: for any principal \(\Gr\)\nb-bundle
  \((\tilde{\Act}',\tilde{\s}',\tilde{\bunp}',\tilde{\mul}')\)
  over~\(\tilde{\Base}\) and \(\Gr\)\nb-map \(\hat{f}'\colon
  \tilde{\Act}\to\Act\), there is a unique isomorphism of principal
  \(\Gr\)\nb-bundles \(\varphi\colon
  \tilde{\Act}\congto\tilde{\Act}'\) with
  \(\hat{f}'\circ\varphi=\hat{f}\).
\end{proposition}

This principal \(\Gr\)\nb-bundle is called the \emph{pull-back}
of \((\Act,\s,\bunp,\mul)\) along~\(f\).

\begin{proof}
  Let \(\tilde{\Act} \defeq \tilde{\Base}\times_{f,\Base,\bunp}
  \Act\), \(\tilde{\s} \defeq \s\circ\pr_2\colon
  \tilde{\Act}\to\Gr^0\) and \(\tilde{\bunp}\defeq \pr_1\colon
  \tilde{\Act}\to\tilde{\Base}\).  The fibre product~\(\tilde{\Act}\)
  exists and~\(\tilde{\bunp}\) is a cover because~\(\bunp\) is a
  cover.  Let
  \[
  \tilde{\mul}\colon \tilde{\Act} \times_{\tilde{\s},\Gr^0,\rg}\Gr^1
  = \tilde{\Base} \times_{f,\Base,\bunp} \Act \times_{\s,\Gr^0,\rg} \Gr^1
  \xrightarrow{\id_{\tilde{\Base}} \times_{f,\Base} \mul}
  \tilde{\Base} \times_{f,\Base,\bunp} \Act
  = \tilde{\Act};
  \]
  this map is well-defined because \(\bunp\circ\mul=\bunp\circ\pr_1\).

  Routine computations show that \((\tilde{\Act}, \tilde{\s},
  \tilde{\mul})\) is a \(\Gr\)\nb-action and that
  \(\tilde{\bunp}\circ\tilde{\mul}=\tilde{\bunp}\circ\pr_1\).  The
  map~\eqref{eq:principal_bundle} with tildes is the pull-back
  along~\(f\) of the same map without tildes, where we use the
  canonical maps
  from both sides to~\(\Base\).  Since pull-backs of isomorphisms
  remain isomorphisms, \((\tilde{\Act}, \tilde{\Base}, \tilde{\s},
  \tilde{\mul}, \tilde{\bunp})\) is a principal \(\Gr\)\nb-bundle
  over~\(\Base\).  The second coordinate projection \(\hat{f}\colon
  \tilde{\Act}\to\Act\) is a \(\Gr\)\nb-map with \(\bunp\circ\hat{f} =
  f\circ\tilde{\bunp}\).

  Let \((\tilde{\Act}',\tilde{\s}',\tilde{\bunp}',\tilde{\mul}')\)
  be another principal \(\Gr\)\nb-bundle over~\(\tilde{\Base}\) with
  a \(\Gr\)\nb-map \(\hat{f}'\colon \tilde{\Act}'\to\Act\) with
  \(\bunp\circ\hat{f}' = f\circ\tilde{\bunp}'\).  Then
  \(\varphi\defeq (\tilde{\bunp}',\hat{f}')\colon
  \tilde{\Act}'\to\tilde{\Act}\) is a \(\Gr\)\nb-map with
  \(\tilde{\bunp}\circ\varphi= \tilde{\bunp}'\), and it is the only
  \(\Gr\)\nb-map with this extra property.  The proof will be
  finished by showing that~\(\varphi\) is an isomorphism.

  Since \(\tilde{\bunp}'\colon \tilde{\Act}'\prto\tilde{\Base}\) is
  a cover, so is its pull-back along \(\tilde{\bunp}\colon
  \tilde{\Act}\to\tilde{\Base}\).  Identifying
  \(\tilde{\Act}'\times_{\tilde{\bunp}',\tilde{\Base},\tilde{\bunp}}
  \tilde{\Act} \cong \tilde{\Act}'\times_{f\circ
    \tilde{\bunp}',\Base,\bunp}\Act\), this cover becomes the map
  \[
  \tilde{\bunp}'\times_{\Base} \id_{\Act}\colon
  \tilde{\Act}'\times_{f\tilde{\bunp}',\Base,\bunp} \Act
  \prto \tilde{\Act} =
  \tilde{\Base}\times_{f,\Base,\bunp} \Act,\qquad
  (x_1,x_2)\mapsto (\tilde{\bunp}'(x_1),x_2).
  \]
  Since \(\tilde{\bunp}'\times_{\Base} \id_{\Act}\) is a cover, the
  fibre product of the following two maps exists:
  \[
  \tilde{\Act}'\times_{f\tilde{\bunp}',\Base,\bunp} \Act
  \xrightarrow{\tilde{\bunp}'\times_{\Base} \id_{\Act}}
  \tilde{\Act}
  \xleftarrow{\varphi} \tilde{\Act}'.
  \]
  An element of this fibre product is a triple \(x_1\in\tilde{\Act}'\),
  \(x_2\in\Act\), \(x_3\in\tilde{\Act}'\) with
  \(f(\tilde{\bunp}'(x_1))=\bunp(x_2)\), \(\hat{f}'(x_3)=x_2\),
  \(\tilde{\bunp}'(x_3)=\tilde{\bunp}'(x_1)\).  Then \(\bunp(x_2)=
  \bunp(\hat{f}'(x_3)) = f(\tilde{\bunp}'(x_3)) =
  f(\tilde{\bunp}'(x_1)) = \bunp(\hat{f}'(x_1))\).  Since~\(\Act\) is a
  principal \(\Gr\)\nb-bundle over~\(\Base\), there is a unique
  \(g\in\Gr\) with \(\s(x_2)=\rg(g)\) and \(\hat{f}'(x_1)= x_2\cdot
  g\).  Since~\(\tilde{\Act}'\) is a principal \(\Gr\)\nb-bundle
  over~\(\tilde{\Base}\) and \(\tilde{\bunp}'(x_3) =
  \tilde{\bunp}'(x_1)\), there is also a unique \(g'\in\Gr^1\) with
  \(\tilde{\s}'(x_1)=\rg(g')\) and \(x_1\cdot g'=x_3\).  Then
  \(x_2=\hat{f}'(x_3)=\hat{f}'(x_1\cdot g') = \hat{f}'(x_1)\cdot
  g'\), so \(g=g'\).  Thus \(x_3=\hat{f}'(x_1)\cdot g\) is uniquely
  determined by \(x_1\) and~\(x_2\).  That is, the pull-back
  of~\(\varphi\) along the cover \(\tilde{\bunp}'\times_{\Base}
  \id_{\Act}\) is an isomorphism.  This implies that~\(\varphi\) is
  an isomorphism by Proposition~\ref{pro:isomorphism_local}.
\end{proof}

\begin{proposition}
  \label{pro:G-map_versus_base_map}
  Let\/ \(\Act_1\) and\/~\(\Act_2\) be principal \(\Gr\)\nb-bundles over
  \(\Base_1\) and~\(\Base_2\), respectively.  Then a \(\Gr\)\nb-map
  \(f\colon \Act_1\to\Act_2\) induces a map \(f/\Gr\colon
  \Base_1\to\Base_2\).  The map~\(f\) is an isomorphism if and only
  if~\(f/\Gr\) is.  If~\(f/\Gr\) is a cover, so is~\(f\).  Conversely,
  under Assumption~\textup{\ref{assum:local_cover}}, \(f\) is a cover
  if and only if~\(f/\Gr\) is.
\end{proposition}

\begin{proof}
  Since \(\Base_i\) is the orbit space of~\(\Act_i\) and orbit spaces
  are constructed naturally, the \(\Gr\)\nb-map~\(f\) induces a map
  \(f/\Gr\colon \Base_1\to\Base_2\).  Since \(\bunp_2\circ
  f=f/\Gr\circ\bunp_1\), \(\Act_1\) must be the pull-back
  of~\(\Act_2\) along~\(f/\Gr\) by
  Proposition~\ref{pro:pull-back_principal}.  That is, the diagram
  \[
  \begin{tikzpicture}[baseline=(current bounding box.west)]
    \matrix[cd] (m) {
      \Act_1&\Act_2 \\
      \Base_1&\Base_2\\
    };
    \begin{scope}[cdar]
      \draw (m-1-1) -- node {\(f\)} (m-1-2);
      \draw (m-2-1) -- node {\(f/\Gr\)} (m-2-2);
      \draw[->>] (m-1-1) -- node[swap] {\(\bunp_1\)} (m-2-1);
      \draw[->>] (m-1-2) -- node {\(\bunp_2\)} (m-2-2);
    \end{scope}
  \end{tikzpicture}
  \]
  is a fibre-product diagram.  Proposition~\ref{pro:isomorphism_local}
  shows that~\(f\) is an isomorphism if and only if~\(f/\Gr\) is one.
  By the pretology axioms, \(f\) is a cover if~\(f/G\) is, and the
  converse holds under Assumption~\ref{assum:local_cover}.
\end{proof}

\begin{proposition}
  \label{pro:fibre-product_base_above}
  Consider a commuting square of \(\Gr\)\nb-maps between principal
  \(\Gr\)\nb-bundles $\Act_i$ and the corresponding maps between their base
  spaces $\Base_i$:
  \[
  \begin{tikzpicture}[baseline=(current bounding box.west)]
    \matrix[cd] (m) {
      \Act_1&\Act_2 \\
      \Act_3&\Act_4\\
    };
    \begin{scope}[cdar]
      \draw (m-1-1) -- node[swap] {\(\gamma_3\)} (m-2-1);
      \draw (m-1-2) -- node {\(\delta_2\)} (m-2-2);
      \draw (m-1-1) -- node {\(\gamma_2\)} (m-1-2);
      \draw (m-2-1) -- node {\(\delta_3\)} (m-2-2);
    \end{scope}
  \end{tikzpicture}
  \qquad
  \begin{tikzpicture}[baseline=(current bounding box.west)]
    \matrix[cd] (m) {
      \Base_1&\Base_2 \\
      \Base_3&\Base_4\\
    };
    \begin{scope}[cdar]
      \draw (m-1-1) -- node {\(\beta_2\)} (m-1-2);
      \draw (m-2-1) -- node {\(\alpha_3\)} (m-2-2);
      \draw (m-1-1) -- node[swap] {\(\beta_3\)} (m-2-1);
      \draw (m-1-2) -- node {\(\alpha_2\)} (m-2-2);
    \end{scope}
  \end{tikzpicture}
  \]
  The first square is a fibre-product square if and only if the
  second one is.
\end{proposition}

\begin{proof}
  Assume first that the square of~\(\Base_i\) is a fibre-product
  square.  Let \(x_2\in\Act_2\), \(x_3\in\Act_3\) satisfy \(x_4\defeq
  \delta_2(x_2) = \delta_3(x_3)\).  We must show that there is a
  unique \(x_1\in\Act_1\) with \(\gamma_i(x_1)=x_i\) for \(i=2,3\).
  Let \(z_i\defeq \bunp_i(x_i)\in\Base_i\), then \(\alpha_2(z_2) =
  \bunp_4(\delta_2(x_2)) = \bunp_4(\delta_3(x_3)) = \alpha_3(z_3)\).
  Since~\(\Base_1\) is the pull-back of \(\alpha_2,\alpha_3\) by
  assumption, we get a unique \(z_1\in\Base_1\) with
  \(\beta_i(z_1)=z_i\) for \(i=2,3\).
  Proposition~\ref{pro:pull-back_principal} shows that
  \(\Act_1\prto\Base_1\) is the pull-back of \(\Act_4\prto\Base_4\)
  along \(\alpha_i\circ \beta_i\colon \Base_1\to\Base_4\), and
  similarly for \(\Act_i\prto\Base_i\) for \(i=2,3\).  Thus
  \(x_4\in\Act_4\) and \(z_1\in\Act_1\) determine a unique element
  \(x_1\in\Act_1\) with \(\bunp_1(x_1)=z_1\) and
  \(\delta_i\gamma_i(x_1)=x_4\) for \(i=2,3\).  We claim that this is
  the unique element of~\(\Act_1\) with \(\gamma_i(x_1)=x_i\) for
  \(i=2,3\).

  Any element of~\(\Act_1\) with \(\gamma_i(x_1)=x_i\) for
  \(i=2,3\) will also satisfy \(\bunp_1(x_1)=z_1\) and
  \(\delta_i\gamma_i(x_1)=x_4\) for \(i=2,3\), so uniqueness is clear.
  Since \(\Act_i\prto\Base_i\) is the pull-back of
  \(\Act_4\prto\Base_4\) for \(i=2,3\), the element \(x_i\in\Act_i\)
  is uniquely determined by \(z_i=\bunp_i(x_i)\) and
  \(x_4=\delta_i(x_i)\).  Since \(\bunp_i(\gamma_i(x_1)) =
  \beta_i\bunp_1(x_1)= \beta_i(z_1)=z_i = \bunp_i(x_i)\) and
  \(\delta_i(\gamma_i(x_1))=x_4=\delta_i(x_i)\), we get
  \(x_i=\gamma_i(x_1)\) as desired.  Hence we get a bijection between
  \(x_1\in\Act_1\) and pairs \(x_2\in\Act_2\), \(x_3\in\Act_3\) with
  \(\delta_2(x_2)=\delta_3(x_3)\).  This is the elementwise statement
  corresponding to \(\Act_1\cong \Act_2\times_{\Act_4}\Act_3\).

  Now assume, conversely, that \(\Act_1\cong\Act_2\times_{\Act_4}
  \Act_3\).  We are going to show that
  \(\Base_1\cong\Base_2\times_{\Base_4} \Base_3\).  This means that,
  for all \(\?\inOb\Cat\), the map
  \[
  \Cat(\?,\Base_1) \to
  \{(z_2,z_3)\in \Cat(\?,\Base_2)\times\Cat(\?,\Base_3) \mid
  \alpha_2 z_2=\alpha_3 z_3\}
  \]
  is a bijection.  Let \(z_i\colon \?\to\Base_i\) for \(i=2,3,4\) with
  \(\alpha_2z_2=z_4=\alpha_3z_3\) be given.  The \(\Gr\)\nb-maps
  \(\Act_i\to\Act_4\) for \(i=2,3\) give isomorphisms between
  \(\Act_i\prto\Base_i\) and the pull-back of \(\Act_4\to\Base_4\)
  along~\(\alpha_i\).  Pulling back further along the maps~\(z_i\)
  shows that the three pull-backs of the principal \(\Gr\)\nb-bundles
  \(\Act_i\to\Base_i\) along~\(z_i\) for \(i=2,3,4\) are canonically
  isomorphic principal \(\Gr\)\nb-bundles over~\(\?\).  Let
  \(\?\?\prto\?\) be this unique principal \(\Gr\)\nb-bundle
  over~\(\?\).  By Proposition~\ref{pro:pull-back_principal}, the
  maps~\(z_i\) for \(i=2,3,4\) lift uniquely to \(x_i\colon
  \?\?\to\Act_i\).  These liftings still satisfy \(\gamma_2
  x_2=x_4=\gamma_3 x_3\) because the lifting of~\(z_4\) is unique.

  Since \(\Act_1\cong\Act_2\times_{\Act_4} \Act_3\) is a fibre product
  in the category of \(\Gr\)\nb-actions by
  Lemma~\ref{lem:unique_fibre-product_action}, the unique map
  \(x_1\colon \?\?\to\Act_1\) with \(\delta_i x_1=x_i\) for \(i=2,3\)
  is a \(\Gr\)\nb-map.  Thus it induces a map \(z_1\colon
  \?\to\Base_1\) by Proposition~\ref{pro:G-map_versus_base_map}.  The
  naturality of this construction implies \(\beta_i z_1=z_i\) for
  \(i=2,3\).  Any map \(z_1'\colon \?\to\Base_1\) with \(\beta_i
  z_1'=z_i\) for \(i=2,3\) lifts to a map \(x_1'=(z_1',x_i)\colon
  \?\?\to \Base_1\times_{\Base_i} \Act_i\cong \Act_1\) with \(\delta_i
  x_1'=x_i\) for \(i=2,3\) by
  Proposition~\ref{pro:pull-back_principal}.  Since
  \(\Act_1\cong\Act_2\times_{\Act_4} \Act_3\), we must have
  \(x_1'=x_1\), which gives \(z_1'=z_1\).
  Hence there is a unique map \(z_1\colon \?\to\Base_1\) with
  \(\beta_i z_1=z_i\) for \(i=2,3\), as desired.
\end{proof}

\subsection{Locality of principal bundles}
\label{sec:principal_local}

We formulate now what it means for principal bundles to be a
``local'' notion.  Let~\(\Gr\) be a groupoid in~\((\Cat,\covers)\).
Let \(\Act,\Base\inOb\Cat\) and let \(\s\colon \Act\to\Gr^0\),
\(\bunp\colon \Act\to\Base\), and \(\mul\colon
\Act\times_{\s,\Gr^0,\rg} \Gr^1 \to \Act\) be maps.  We pull back
this data along a cover \(f\colon \tilde{\Base}\prto\Base\) as
follows.  Let \(\tilde{\Act} \defeq
\tilde{\Base}\times_{f,\Base,\bunp} \Act\), \(\tilde{\s} \defeq
\s\circ\pr_2\colon \tilde{\Act}\to\Gr^0\), \(\tilde{\bunp}\defeq
\pr_1\colon \tilde{\Act}\to\tilde{\Base}\) and
\[
\tilde{\mul}\colon \tilde{\Act} \times_{\tilde{\s},\Gr^0,\rg}\Gr^1
= \tilde{\Base} \times_{f,\Base,\bunp} \Act \times_{\s,\Gr^0,\rg} \Gr^1
\xrightarrow{\id_{\tilde{\Base}} \times_{f,\Base} \mul}
\tilde{\Base} \times_{f,\Base,\bunp} \Act
= \tilde{\Act};
\]
this map is well-defined if \(\bunp\circ\mul=\bunp\circ\pr_1\).

\begin{proposition}
  \label{pro:principal_is_local}
  The data \((\tilde{\Act}, \tilde{\Base}, \tilde{\s}, \tilde{\mul},
  \tilde{\bunp})\) is a well-defined principal \(\Gr\)\nb-bundle
  over~\(\tilde{\Base}\) if and only if \((\Act, \Base, \s, \mul,
  \bunp)\) is a principal \(\Gr\)\nb-bundle over~\(\Base\).
\end{proposition}

Put in a nutshell, principality for \(\Gr\)\nb-bundles is a local
property.

\begin{proof}
  If \((\Act, \s, \mul, \bunp)\) is a principal \(\Gr\)\nb-bundle
  over~\(\Base\), then Proposition~\ref{pro:pull-back_principal}
  shows that \((\tilde{\Act}, \tilde{\s},
  \tilde{\mul},\tilde{\bunp})\) is a principal \(\Gr\)\nb-bundle
  over~\(\tilde{\Base}\).

  Now assume that \((\tilde{\Act}, \tilde{\s}, \tilde{\mul},
  \tilde{\bunp})\) is a well-defined principal \(\Gr\)\nb-bundle
  over~\(\tilde{\Base}\).  We are going to show that \((\Act, \s,
  \mul, \bunp)\) is a principal \(\Gr\)\nb-bundle over~\(\Base\).
  For~\(\tilde{\mul}\) to be well-defined, we need
  \(\bunp\circ\mul=\bunp\circ\pr_1\) or \(\bunp(x\cdot g)=\bunp(x)\)
  for all \(x\in\Act\), \(g\in\Gr^1\) with \(\s(x)=\rg(g)\).  Since
  \(f\) is a cover, so is the map
  \(\tilde{\Act}\times_{\s,\Gr^0,\rg}
  \Gr^1\prto\Act\times_{\s,\Gr^0,\rg} \Gr^1\) it induces.
  Since~\(\covers\) is subcanonical, this map is an epimorphism.
  Now we get \(\s\circ\mul=\s\circ\pr_2\colon
  \Act\times_{\s,\Gr^0,\rg} \Gr^1\to\Gr^0\) because
  composing \(\s\circ\mul\) and \(\s\circ\pr_2\) with the cover
  \(\tilde{\Act}\times_{\s,\Gr^0,\rg}
  \Gr^1\prto\Act\times_{\s,\Gr^0,\rg} \Gr^1\) gives the same map
  \(\tilde{\s}\circ\mul=\tilde{\s}\circ\pr_2\).  A similar, more
  complicated argument shows that~\(\mul\) is associative
  because~\(\tilde{\mul}\) is associative and both multiplications
  are intertwined by covers \(\tilde{\Act}\times_{\s,\Gr^0,\rg}
  \Gr^1\times_{\s,\Gr^0,\rg} \Gr^1 \prto \Act\times_{\s,\Gr^0,\rg}
  \Gr^1\times_{\s,\Gr^0,\rg} \Gr^1\) and \(\tilde{\Act}\prto\Act\).
  Similarly, we get the unitality condition for the multiplication,
  so \((\Act,\s,\mul)\) is a \(\Gr\)\nb-action.

  It remains to verify that the map
  \[
  (\mul,\pr_1)\colon \Act\times_{\s,\Gr^0,\rg} \Gr^1
  \to \Act \times_{\bunp,\Base,\bunp} \Act
  \]
  is an isomorphism.  Since the coordinate projection
  \(\tilde{\Act}\prto\Act\) is a cover, so is the induced map \(\Act
  \times_{\bunp,\Base,f\circ\tilde{\bunp}} \tilde{\Act} \prto \Act
  \times_{\bunp,\Base,\bunp} \Act\).  The pull-back of
  \((\mul,\pr_1)\) along this map is equivalent to the map
  \begin{equation}
    \label{eq:principal_local_property}
    (\tilde{\mul},\pr_1)\colon
    \tilde{\Act}\times_{\tilde{\s},\Gr^0,\rg} \Gr^1
    \to \tilde{\Act}
    \times_{\tilde{\bunp},\tilde{\Base},\tilde{\bunp}} \tilde{\Act},
  \end{equation}
  where we identify
  \begin{align*}
    \tilde{\Act} \times_{\tilde{\bunp},\tilde{\Base},\tilde{\bunp}} \tilde{\Act}
    &\cong
    \Act \times_{\bunp,\tilde{\Base},f\circ\tilde{\bunp}} \tilde{\Act},\\
    (\Act\times_{\s,\Gr^0,\rg} \Gr^1)\times_{\Act \times_{\bunp,\Base,\bunp} \Act}
    (\Act \times_{\bunp,\Base,f\circ\tilde{\bunp}} \tilde{\Act})
    &\cong \tilde{\Act}\times_{\tilde{\s},\Gr^0,\rg} \Gr^1
  \end{align*}
  in the obvious way.  The map~\eqref{eq:principal_local_property} is
  an isomorphism because~\(\tilde{\Act}\) is principal.  This implies
  that \((\mul,\pr_1)\) is an isomorphism because the property of
  being an isomorphism is local by
  Proposition~\ref{pro:isomorphism_local}.
\end{proof}

\subsection{Basic actions and \texorpdfstring{\v{C}}{C}ech groupoids}
\label{sec:basic_covering}

\begin{definition}
  \label{def:basic_action}
  A groupoid action or sheaf that, together with some bundle
  projection, is
  part of a principal bundle is called \emph{basic}.  A groupoid is
  called \emph{basic} if its canonical action on~\(\Gr^0\) in
  Example~\ref{exa:action_on_objects} is basic.
\end{definition}

We call such actions ``basic'' because they have a well-behaved base
and because ``principal groupoid'' already has a different meaning for
topological groupoids (it means that the action on~\(\Gr^0\) is free).

\begin{proposition}
  \label{pro:covering_groupoid_basic}
  A \(\Gr\)\nb-action~\(\Act\) is basic if and only if the
  transformation groupoid \(\Act\rtimes \Gr\) is isomorphic to a
  \v{C}ech groupoid of a certain cover \(\bunp\colon \Act\prto
  \Base\).
\end{proposition}

\begin{proof}
  Let~\(\Act\) be basic with bundle projection \(\bunp\colon \Act\prto
  \Base\).  Then~\eqref{eq:principal_bundle} provides an isomorphism
  \((\Act\rtimes \Gr)^1 \congto \Act\times_{\bunp,\Base,\bunp} \Act\).
  Together with the identity on objects, this is an isomorphism of
  groupoids from \(\Act\rtimes \Gr\) to the \v{C}ech groupoid
  of~\(\bunp\).  Conversely, let \(\Act\rtimes\Gr\) be isomorphic to
  a \v{C}ech groupoid.  We may assume that the isomorphism is the
  identity on objects, so that the cover whose \v{C}ech groupoid we
  take is a map \(\bunp\colon \Act\prto \Base\).  Since the maps
  \(\pr_1\) and~\(\mul\) in~\eqref{eq:principal_bundle} are the
  range and source maps of \(\Act\rtimes\Gr\), the isomorphism of
  groupoids from \(\Act\rtimes \Gr\) to the \v{C}ech groupoid
  of~\(\bunp\) must be given by \((x,g)\mapsto (x,x\cdot g)\) on
  arrows; thus~\eqref{eq:principal_bundle} is an isomorphism.
\end{proof}

\begin{corollary}
  \label{cor:basic_action_versus_groupoid}
  A \(\Gr\)\nb-action on~\(\Act\) is basic if and only if its
  transformation groupoid \(\Act\rtimes \Gr\) is basic.
\end{corollary}

\begin{proof}
  The criterion in Proposition~\ref{pro:covering_groupoid_basic}
  depends only on \(\Act\rtimes \Gr\).
\end{proof}

\begin{corollary}
  \label{cor:basic_action_trafo_gr}
  An action of a transformation groupoid \(\Act\rtimes\Gr\) is basic
  if and only if the restriction of the action to~\(\Gr\) is basic.
\end{corollary}

\begin{proof}
  Proposition~\ref{pro:transformation_groupoid_action} shows that
  \(\Act[Y]\rtimes (\Act\rtimes\Gr) \cong \Act[Y]\rtimes\Gr\), so the
  assertion follows from
  Corollary~\ref{cor:basic_action_versus_groupoid}.
\end{proof}

\begin{lemma}
  \label{lem:orbit_space_trafo_gr}
  Let~\(\Act[Y]\) carry an action of a transformation groupoid
  \(\Act\rtimes\Gr\), equip it with the resulting \(\Gr\)\nb-action.
  Then \(\Act[Y]/(\Act\rtimes\Gr) \cong \Act[Y]/\Gr\).
\end{lemma}

\begin{proof}
  Proposition~\ref{pro:transformation_groupoid_action} shows that
  the invariant maps for both actions are the same, hence so are the
  orbit spaces.
\end{proof}

\section{Bibundle functors, actors, and equivalences}
\label{sec:HS_equivalences}

Let \(\Gr\) and~\(\Gr[H]\) be groupoids in~\((\Cat,\covers)\).  We
describe several important classes of \(\Gr,\Gr[H]\)\nb-bibundles:

\begin{definition}
  \label{def:HS-morphism}
  A \emph{bibundle equivalence from~\(\Gr\) to~\(\Gr[H]\)} is a
  \(\Gr,\Gr[H]\)\nb-bibundle~\(\Act\) such that both the left and
  right actions are principal bundles with bundle projections
  \(\s\colon \Act\prto \Gr[H]^0\) and \(\rg\colon \Act\prto \Gr^0\),
  respectively.  We call \(\Gr\) and~\(\Gr[H]\) \emph{equivalent} if
  there is a bibundle equivalence from~\(\Gr\) to~\(\Gr[H]\).

  A \emph{bibundle functor from~\(\Gr\) to~\(\Gr[H]\)} is a
  \(\Gr,\Gr[H]\)\nb-bibundle~\(\Act\) such that the right
  \(\Gr[H]\)\nb-action is a principal bundle with bundle projection
  \(\rg\colon \Act\prto \Gr^0\).  A bibundle functor is
  \emph{covering} if the anchor map \(\s\colon \Act\prto\Gr[H]^0\)
  is a cover.

  A \emph{bibundle actor from~\(\Gr\) to~\(\Gr[H]\)} is a
  \(\Gr,\Gr[H]\)\nb-bibundle~\(\Act\) such that the right
  \(\Gr[H]\)\nb-action is basic and a sheaf, that is, \(\s\colon
  \Act\prto \Gr[H]^0\) is a cover.
\end{definition}

Bibundle functors are also called generalised
morphisms\cite{Moerdijk-Mrcun:Groupoids_sheaves}*{Section 2.5} or
Hilsum--Skandalis morphisms~\cite{Mrcun:Thesis}, and bibundle equivalences just
equivalences or Morita equivalences.  We will show later that for
sufficiently nice pretopologies, bibundle functors and actors are
precisely the products of bibundle equivalences with functors and
actors, respectively (see Sections \ref{sec:bibundle_to_vague}
and~\ref{sec:decompose_bibundle_actor}).  This justifies the names
above.

The anchor map \(\rg\colon \Act\prto \Gr^0\) of a bibundle functor
is always a cover because it is the bundle projection of a principal
action.  For an equivalence~\(\Act\), both anchor maps \(\rg\colon
\Act\prto \Gr^0\) and \(\s\colon \Act\prto \Gr[H]^0\) are covers for
the same reason.  Thus bibundle equivalences are covering bibundle
functors.  A bibundle functor is a bibundle actor as well if and
only if it is covering.  The following diagram illustrates the
relations between these notions:
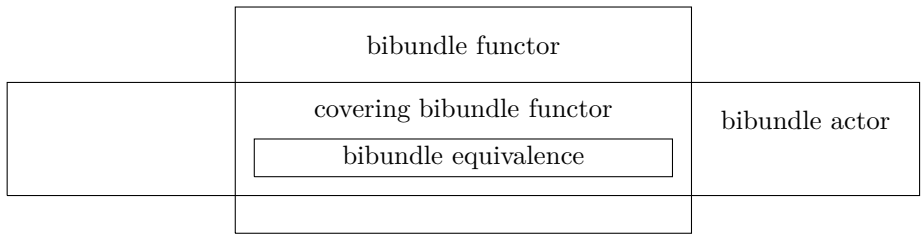
\[
\begin{tikzpicture}
  \draw(0,.5)--(12,.5)--(12,2)--(0,2)--(0,.5);
  \draw(3,0)--(9,0)--(9,3)--(3,3)--(3,0);
  \draw(3.25,.75)--(8.75,.75)--(8.75,1.25)--(3.25,1.25)--(3.25,.75);
  \path (6,1) node {bibundle equivalence};
  \path (6,1.625) node {covering bibundle functor};
  \path (6,2.5) node {bibundle functor};
  \path (10.5,1.5) node {bibundle actor};
\end{tikzpicture}
\]

\begin{example}
  \label{exa:unit_bibundle}
  Let~\(\Gr\) be a groupoid in~\((\Cat,\covers)\).  Then~\(\Gr\) acts
  on~\(\Gr^1\) on the left and right by multiplication.  These actions
  turn~\(\Gr^1\) into a bibundle equivalence from~\(\Gr\) to itself,
  so equivalence of groupoids is a reflexive relation.  First, the two
  actions commute by associativity, so they form a
  \(\Gr,\Gr\)-bibundle; secondly, the right multiplication action
  gives a principal \(\Gr\)\nb-bundle with bundle projection
  \(\s\colon \Gr^1\prto \Gr^0\) by~\eqref{eq:groupoid_basicality_1};
  thirdly, the left action gives a principal \(\Gr\)\nb-bundle with
  bundle projection \(\rg\colon \Gr^1\prto \Gr^0\)
  by~\eqref{eq:groupoid_basicality_2}.

  The assumptions that \(\rg\) and~\(\s\) be covers and the conditions
  \eqref{eq:groupoid_basicality_1}
  and~\eqref{eq:groupoid_basicality_2} in the definition of a groupoid
  are necessary and sufficient for this unit bibundle equivalence to
  work.
\end{example}

Left and right actions of groupoids are equivalent by
Lemma~\ref{lem:left_right_action}.  Left and right
principal bundles are also equivalent in the same way.  Thus a
\(\Gr,\Gr[H]\)\nb-bibundle~\(\Act\) gives an
\(\Gr[H],\Gr\)\nb-bibundle~\(\Act^*\), which is the same object
of~\(\Cat\) with the two anchor maps exchanged and \(h\cdot x\cdot
g\defeq g^{-1}\cdot x\cdot h^{-1}\); this is a bibundle equivalence if
and only if~\(\Act\) is one.  Thus equivalence of groupoids is a
symmetric relation.  Transitivity seems to require an extra assumption
in order to compose bibundle equivalences, see
Section~\ref{sec:composition}.

\begin{example}
  \label{exa:covering_groupoid_equivalent_to_space}
  Let \(\bunp\colon \Act\prto \Base\) be a cover
  in~\((\Cat,\covers)\).  View~\(\Base\) as a groupoid with only
  identity arrows and let~\(\Gr\) be the \v{C}ech groupoid
  of~\(\bunp\).  Then \(\Gr\) and~\(\Base\) are equivalent.

  The equivalence bibundle is~\(\Act\), with the canonical right
  action of~\(\Gr\) on its object space and with the left
  \(\Base\)\nb-action given by the anchor map~\(\bunp\).  The right
  \(\Gr\)\nb-action on~\(\Act\) with~\(\bunp\) as bundle projection
  gives a principal \(\Gr\)\nb-bundle by
  Example~\ref{exa:covering_groupoid_principal}.  The left
  \(\Base\)\nb-action is given simply by its anchor map~\(\bunp\), and
  any such action gives a principal bundle with the identity map
  \(\Act\to\Act\) as bundle projection; the identity is also the right
  anchor map.

  Conversely, if~\(\Act\) is a bibundle equivalence from a
  groupoid~\(\Gr\) to a \(0\)\nb-groupoid~\(\Base\), then~\(\Gr\) is
  isomorphic to the \v{C}ech groupoid of the anchor map
  \(\bunp\colon \Act\prto\Base\).
\end{example}

\begin{example}
  \label{exa:covering_groupoids_equivalent}
  More generally, let \(\bunp_i\colon \Act_i\prto\Base\) for \(i=1,2\)
  be two covers and let \(\Gr_1\) and~\(\Gr_2\) be their \v{C}ech
  groupoids.  Let \(\Act\defeq \Act_1\times_{\bunp_1,\Base,\bunp_2}
  \Act_2\), \(\rg\defeq\pr_1\colon \Act\prto\Act_1=\Gr_1^0\) and
  \(\s\defeq\pr_2\colon \Act\prto\Act_2=\Gr_2^0\), and define the
  multiplication maps by \((x_1,y_1)\cdot (y_1,y_2)\defeq (x_1,y_2)\),
  \((x_1,x_2)\cdot (x_2,y_2)\defeq (x_1,y_2)\) for
  \(x_1,y_1\in\Act_1\), \(x_2,y_2\in\Act_2\) with
  \(\bunp_1(x_1)=\bunp_1(y_1)= \bunp_2(x_2) = \bunp_2(y_2)\).
  Then~\(\Act\) is a bibundle equivalence from~\(\Gr_1\) to~\(\Gr_2\).
  This shows that equivalence of groupoids is a transitive relation
  among \v{C}ech groupoids.
\end{example}

\begin{example}
  \label{exa:HS_0-groupoid}
  Let~\(\Gr[Y]\) be an object of~\(\Cat\) viewed as a groupoid.  An
  action of~\(\Gr[Y]\) is the same as a map to~\(\Gr[Y]\) by
  Example~\ref{exa:action_0-groupoid}.  Thus a \(\Gr,\Gr[Y]\)-bibundle
  is the same as a (left) \(\Gr\)\nb-action~\(\Act\) with a
  \(\Gr\)\nb-invariant map \(f\colon \Act\to\Gr[Y]\).

  A bibundle functor \(\Gr\to\Gr[Y]\) is, up to isomorphism, the same
  as a \(\Gr\)\nb-invariant map \(\Gr^0\to\Gr[Y]\) because of the
  assumption that~\(\rg\) induces an isomorphism
  \(\Act=\Act/\Gr[Y]\congto\Gr^0\).  A bibundle functor
  \(\Gr[Y]\to\Gr\) is the same as a principal \(\Gr\)\nb-bundle
  over~\(\Gr[Y]\).

  A bibundle actor \(\Gr\to\Gr[Y]\) is the same as a
  \(\Gr\)\nb-action~\(\Act\) with a \(\Gr\)\nb-invariant cover
  \(\Act\prto\Gr[Y]\).  A bibundle actor \(\Gr[Y]\to\Gr\) is the same
  as a principal \(\Gr\)\nb-bundle~\(\Act\) over some space~\(\Base\)
  with a map \(\Base\to\Gr[Y]\), such that the anchor map \(\s\colon
  \Act\prto\Gr^0\) is a cover.

  Let \(\Gr[Y]_1\) and \(\Gr[Y]_2\) be two objects of~\(\Cat\) viewed
  as groupoids.  Then a \(\Gr[Y]_1,\Gr[Y]_2\)-bibundle is a span
  \(\Gr[Y]_1\leftarrow\Act\rightarrow \Gr[Y]_2\) in~\(\Cat\).  This is
  \begin{itemize}
  \item a bibundle functor if and only if the map
    \(\Gr[Y]_1\leftarrow\Act\) is invertible;
  \item a bibundle equivalence if and only if both maps
    \(\Gr[Y]_1\leftarrow\Act\rightarrow \Gr[Y]_2\) are isomorphisms;
  \item a bibundle actor if and only if the map
    \(\Act\prto\Gr[Y]_2\) is a cover.
  \end{itemize}
\end{example}

Example~\ref{exa:covering_groupoid_equivalent_to_space} shows that a
groupoid is equivalent to a \(0\)\nb-groupoid if and only if it is
basic.  By Proposition~\ref{pro:covering_groupoid_basic}, a groupoid
is basic if and only if it is isomorphic to a \v{C}ech groupoid.

\subsection{From functors to bibundle functors}
\label{sec:functors_to_bibundles}

Let \(\Gr\) and~\(\Gr[H]\) be groupoids in~\((\Cat,\covers)\) and
let \(F\colon \Gr\to\Gr[H]\) be a functor, given by
\(F^i\in\Cat(\Gr^i,\Gr[H]^i)\) for \(i=0,1\).  We are going to
define an associated bibundle functor~\(\Act_F\) from~\(\Gr\)
to~\(\Gr[H]\) (see also~\cite{Moerdijk-Mrcun:Groupoids_sheaves}).

Since \(\rg\colon \Gr[H]^1\prto\Gr[H]^0\) is a cover, the fibre
product \(\Act_F=\Act\defeq \Gr^0\times_{F^0,\Gr[H]^0,\rg}
\Gr[H]^1\) exists and the coordinate projection \(\pr_1\colon
\Act\prto \Gr^0\) is a cover.  This map is the anchor map for a left
\(\Gr\)\nb-action on~\(\Act\) defined elementwise by \(g\cdot
(x,h)\defeq (\rg(g),F^1(g)\cdot h)\) for all \(g\in \Gr^1\),
\(x\in\Gr^0\), \(h\in \Gr[H]^1\) with \(\s(g)=x\),
\(\rg(h)=F^0(x) = F^0(\s(g))=\s(F^1(g))\).  The map \(\s\colon
\Act\to\Gr[H]^0\), \((x,h)\mapsto \s(h)\), for \(x\in \Gr^0\), \(h\in
\Gr[H]^1\) with \(F^0(x)=\rg(h)\) is the anchor map for a right
\(\Gr[H]\)\nb-action defined elementwise by \((x,h_1)\cdot h_2\defeq
(x,h_1\cdot h_2)\) for all \(x\in \Gr^0\), \(h_1,h_2\in \Gr[H]^1\)
with \(F^0(x)=\rg(h_1)\) and \(\s(h_1)=\rg(h_2)\).

\begin{lemma}
  \label{lem:functor_to_bibundle_functor}
  The \(\Gr,\Gr[H]\)-bibundle~\(\Act\) is a bibundle functor.
\end{lemma}

\begin{proof}
  The left and right actions on~\(\Act\) commute.  For \(x\in \Gr^0\),
  \(h_1,h_2,y\in \Gr[H]^1\) with \(F^0(x)=\rg(h_1)=\rg(h_2)\) we have
  \(\s(h_1)=\rg(y)\) and \((x,h_1)\cdot y = (x,h_2)\) if and only if
  \(y=h_1^{-1}\cdot h_2\).  Hence there is a unique such~\(y\),
  proving the isomorphism~\eqref{eq:principal_bundle} for the
  \(\Gr[H]\)\nb-bundle \(\rg\colon \Act\to\Gr^0\).  Since \(\rg\colon
  \Act\prto\Gr^0\) is a cover as well, the right \(\Gr[H]\)\nb-action
  together with~\(\rg\) is a principal \(\Gr[H]\)\nb-bundle.
\end{proof}

We may generalise the construction above as follows.  Let \(F\colon
\Gr\to\Gr[H]\) be a functor and let~\(\Act[Y]\) be a
bibundle functor from~\(\Gr[H]\) to~\(\Gr[K]\).  We may then
construct a bibundle functor \(F^*(\Act[Y])\) from~\(\Gr\)
to~\(\Gr[K]\).  The bibundle functor~\(\Act_F\) is the special case
where \(\Act[Y]=\Gr[H]^1\) is the identity bibundle functor
on~\(\Gr[H]\).  The only change in the construction is to
replace~\(\Gr[H]^1\) by~\(\Act[Y]\) everywhere.  Thus the underlying
object of~\(\Cat\) is \(F^*(\Act[Y]) \defeq
\Gr^0\times_{F^0,\Gr[H]^0,\rg} \Act[Y]\); this exists in~\(\Cat\)
and the coordinate projection \(\pr_1=\s\colon F^*(\Act[Y])\prto
\Gr^0\) is a cover because \(\rg\colon \Act[Y]\prto \Gr[H]^0\) is a
cover.

The same formulas as above define a left \(\Gr\)\nb-action and a right
\(\Gr[K]\)\nb-action on~\(\Act[Y]\).  The right \(\Gr[K]\)\nb-action
together with \(\pr_1\colon F^*(\Act[Y])\prto\Gr^0\) is a principal
\(\Gr[K]\)\nb-bundle because for \(x\in \Gr^0\), \(y_1,y_2\in
\Act[Y]\), \(k\in\Gr[K]^1\) with \(F^0(x)=\rg(y_1)=\rg(y_2)\), we have
\(\s(y_1)=\rg(k)\) and \((x,y_1)\cdot k = (x,y_2)\) if and only if
\(\s(y_1)=\rg(k)\) and \(y_1\cdot k=y_2\), and this has a unique
solution if \(\rg(y_1)=\rg(y_2)\) because~\(\Act[Y]\) is a bibundle
functor.

\begin{proposition}
  \label{pro:equivalence_functor}
  The bibundle functor associated to~\(F\) is covering if and only
  if~\(F\) is essentially surjective, and an equivalence bibundle if
  and only if~\(F\) is essentially surjective and fully faithful.
\end{proposition}

\begin{proof}
  The right anchor map of~\(\Act_F\) is given by \((x,h)\mapsto
  \s(h)\) for \(x\in\Gr^0\), \(h\in\Gr[H]^1\) with \(F^0(x)=\rg(h)\);
  this is exactly the map that is required to be a cover in order
  for~\(F\) to be essentially surjective.  The bibundle~\(\Act_F\) is
  a bibundle equivalence if and only if the right anchor map is a
  cover and the following map is invertible:
  \[
  \varphi\colon \Gr^1 \times_{\s,\Gr^0,\rg}\Act_F\to \Act_F
  \times_{\s,\Gr[H]^0,\s} \Act_F,\qquad
  (g,x)\mapsto (g\cdot x,x).
  \]
  The domain and codomain of~\(\varphi\) are naturally isomorphic to
  \(\Gr^1\times_{F^0\circ\s,\Gr[H]^0,\rg} \Gr[H]^1\) and
  \(\Gr^0\times_{F^0,\Gr[H]^0,\rg} \Gr[H]^1\times_{\s,\Gr[H]^0,\s}
  \Gr[H]^1\times_{\rg,\Gr[H]^0,F^0} \Gr^0\), respectively,
  and~\(\varphi\) is given elementwise by \(\varphi(g,h) =
  (\rg(g),F^1(g)\cdot h,h,\s(g))\) for all \(g\in\Gr^1\),
  \(h\in\Gr[H]^1\) with \(F^0(\s(g))=\rg(h)\).  If~\(\varphi\) is
  invertible, then the first component of~\(\varphi^{-1}\) applied to
  \((x_1,h,1_{\s(h)},x_2)\) for \(x_1,x_2\in\Gr^0\),
  \(h\in\Gr[H]^1\) with \(F^0(x_1)=\rg(h)\), \(F^0(x_2)=\s(h)\),
  gives an inverse to the map~\(\psi\) in~\eqref{eq:fully_faithful}.
  Conversely, if~\(\psi\) is
  invertible, then so is~\(\varphi\) with \(\varphi^{-1}(x_1,h_1,h_2,x_2)
  = (\psi^{-1}(x_1,h_1\cdot h_2^{-1},x_2),h_2)\).
\end{proof}

\begin{example}
  \label{exa:hypercover_equivalence_functor}
  Let \(p\colon X\prto\Gr^0\) be a cover.  The hypercover \(p_*\colon
  \Gr(X)\to\Gr\) is an equivalence functor.
\end{example}

\subsection{From bibundle functors to anafunctors}
\label{sec:bibundle_to_vague}

Let~\(\Act\) be a bibundle functor between two groupoids \(\Gr\)
and~\(\Gr[H]\) in~\((\Cat,\covers)\).  We are going to turn it into
an anafunctor \(\Gr\to\Gr[H]\).  We take \(\rg\colon
\Act\prto\Gr^0\) as the cover that is part of an anafunctor.  It
remains to define a functor \(F\colon \Gr(\Act)\to\Gr[H]\) using the
\(\Gr,\Gr[H]\)-bibundle~\(\Act\).

\begin{proposition}
  \label{pro:bibundle_to_vague}
  Let~\(\Act\) be a bibundle functor \(\Gr\to\Gr[H]\).  There is a
  natural isomorphism of groupoids \(\Gr(\Act) \cong
  \Gr\ltimes\Act\rtimes\Gr[H]\) that acts identically on objects.
\end{proposition}

\begin{proof}
  Both groupoids \(\Gr\ltimes\Act\rtimes\Gr[H]\) and \(\Gr(\Act)\)
  have object space~\(\Act\).  Their arrow spaces are
  \(\Gr^1\times_{\s,\Gr^0,\rg} \Act\times_{\s,\Gr[H]^0,\rg} \Gr[H]^1\)
  and \(\Act\times_{\rg,\Gr^0,\rg}\Gr^1 \times_{\s,\Gr^0,\rg} \Act\),
  respectively.  We define
  \[
  F^1\colon
  \Gr^1\times_{\s,\Gr^0,\rg} \Act\times_{\s,\Gr[H]^0,\rg} \Gr[H]^1
  \to \Act\times_{\rg,\Gr^0,\rg}\Gr^1 \times_{\s,\Gr^0,\rg} \Act
  \]
  elementwise by \(F^1(g,x,h) \defeq (g\cdot x,g,x\cdot h)\) for all
  \(g\in\Gr^1\), \(x\in\Act\), \(h\in\Gr[H]^1\) with
  \(\s(g)=\rg(x)\), \(\s(x)=\rg(h)\); this is well-defined because
  then \(g\cdot x\) and \(x\cdot h\) are defined and \(\rg(g\cdot
  x)=\rg(g)\) and \(\s(g)=\rg(x)=\rg(x\cdot h)\).  Furthermore,
  \(\rg(F^1(g,x,h))=g\cdot x=\rg(g,x,h)\), \(\s(F^1(g,x,h))= x\cdot
  h=\s(g,x,h)\) and
  \[
  F^1(g_1,x_1,h_1)\cdot F^1(g_2,x_2,h_2)
  = (g_1\cdot x_1,g_1\cdot g_2,x_2\cdot h_2)
  = F^1(g_1\cdot g_2,g_2^{-1}\cdot x_1,h_1\cdot h_2)
  \]
  for all \(g_1,g_2\in\Gr^1\), \(x_1,x_2\in\Act\),
  \(h_1,h_2\in\Gr[H]^1\) with \(\s(g_i)=\rg(x_i)\),
  \(\s(x_i)=\rg(h_i)\) and \(x_1\cdot h_1=g_2\cdot x_2\).
  Thus~\(F^1\) together with the identity on arrows is a functor.
  It remains to show that~\(F^1\) is an invertible map in~\(\Cat\).

  Let \(x_1,x_2\in\Act\), \(g\in\Gr^1\) satisfy \(\rg(x_1)=\rg(g)\),
  \(\s(g)=\rg(x_2)\).  Then \(x\defeq g^{-1}\cdot x_1\) is
  well-defined and \(\rg(x)=\s(g)=\rg(x_2)\).  Since~\(\Act\) with
  bundle projection~\(\rg\) is a principal \(\Gr[H]\)\nb-bundle,
  there is a unique \(h\in\Gr[H]^1\) with \(\s(x)=\rg(h)\) and
  \(x_2=x\cdot h\).  Thus \(F^1(g,x,h)=(x_1,g,x_2)\).  Furthermore,
  the element~\((g,x,h)\) is unique with this property, so~\(F^1\)
  is invertible by the Yoneda Lemma.
\end{proof}

We compose the isomorphism \(\Gr(\Act)\congto
\Gr\ltimes\Act\rtimes\Gr[H]\) in
Proposition~\ref{pro:bibundle_to_vague} with the functor
\[
\Gr\ltimes\Act\rtimes\Gr[H] \to \Gr[H]
\]
that is \(\s\colon \Act\to\Gr[H]\) on objects and the coordinate
projection \(\pr_3\colon (\Gr\ltimes\Act\rtimes\Gr[H])^1 \to
\Gr[H]^1\) on arrows.  This yields a functor \(F_{\Act}\colon
\Gr(\Act)\to\Gr[H]\), given on \(\Gr(\Act)^0=\Act\) by~\(\s\) and on
\(\Gr(\Act)^1 = \Act \times_{\rg,\Gr^0,\rg}\Gr^1 \times_{\s,\Gr^0,\rg}
\Act\) by
\[
F_{\Act}^1(g\cdot x,g,x\cdot h) \defeq h
\]
for all \(g\in\Gr^1\), \(x\in\Act\), \(h\in\Gr[H]^1\) with
\(\s(g)=\rg(x)\), \(\s(x)=\rg(h)\);
Proposition~\ref{pro:bibundle_to_vague} says that any
\((x_1,g,x_2)\in\Act\times_{\rg,\Gr^0,\rg}\Gr^1\times_{\s,\Gr^0,\rg}\Act\)
may be rewritten as \(x_1=g\cdot x\), \(x_2=x\cdot h\) for unique
\(g,x,h\) as above.

The triple \((\Act,\rg,F_{\Act})\) is an anafunctor from~\(\Gr\)
to~\(\Gr[H]\).

\begin{corollary}
  \label{cor:equivalence_bibundle_vague_isomorphism}
  An equivalence bibundle~\(\Act\) from~\(\Gr\) to~\(\Gr[H]\) induces
  an isomorphism \(\Gr(\Act)\congto\Gr[H](\Act)\), that is, a
  ana-isomorphism between \(\Gr\) and~\(\Gr[H]\).
\end{corollary}

\begin{proof}
  If~\(\Act\) is an equivalence bibundle from~\(\Gr\) to~\(\Gr[H]\),
  then we may exchange left and right and get another equivalence
  bibundle~\(\Act^*\) from~\(\Gr[H]\) to~\(\Gr\).
  Proposition~\ref{pro:bibundle_to_vague} applied to \(\Act\)
  and~\(\Act^*\) gives groupoid isomorphisms
  \[
  \rg^*(\Gr)
  = \Gr(\Act)
  \congto \Gr\ltimes\Act\rtimes\Gr[H]
  \congto \Gr[H](\Act)
  = \s^*(\Gr[H]).\qedhere
  \]
\end{proof}

\begin{lemma}
  \label{lem:functor_to_bibundle_to_vague}
  Let \(F\colon \Gr\to\Gr[H]\) be a functor.  This gives rise first to
  a bibundle functor \(\Act_F\colon \Gr\to\Gr[H]\), secondly to an
  anafunctor \((\Act_F,\rg_F,\tilde{F})\) from~\(\Gr\)
  to~\(\Gr[H]\), where \(\rg_F\colon \Act_F\prto\Gr^0\) is the left
  anchor map of~\(\Act_F\).  This anafunctor is equivalent to
  \((\Gr^0,\id_{\Gr^0}, F)\), that is, to~\(F\) viewed as an
  anafunctor.
\end{lemma}

\begin{proof}
  We have \(\Act_F = \Gr^0\times_{F^0,\Gr[H]^0,\rg} \Gr[H]^1\) with
  \(\rg_F=\pr_1\).  The functor \(\tilde{F}\colon
  \Gr(\Act_F)\to\Gr[H]\) is given elementwise by
  \[
  \tilde{F}^0(x,h)\defeq \s(h),\qquad
  \tilde{F}^1\bigl((x_1,h_1),g,(x_2,h_2)\bigr)
  = h_1^{-1}\cdot F^1(g)\cdot h_2
  \]
  for all \(x,x_1,x_2\in\Gr^0\), \(h,h_1,h_2\in\Gr[H]^1\),
  \(g\in\Gr^1\) with \(F^0(x)=\rg(h)\), \(F^0(x_1)=\rg(h_1)\),
  \(F^0(x_2)=\rg(h_2)\), \(x_1=\rg(g)\), \(x_2=\s(g)\).  The map
  \(\Phi\colon \Act_F\to\Gr[H]^1\), \((x,h)\mapsto h\), is a natural
  transformation \(\tilde{F}\Rightarrow F\) because of the following
  commuting diagram in~\(\Gr[H]^1\):
  \[
  \begin{tikzpicture}[baseline=(current bounding box.south east)]
    \matrix[cd,column sep=.5em,text depth=0ex] (m) {
      \rg(h_2) & F^0(x_2)& F^0(\s(g)) &[3em] F^0(\rg(g)) & F^0(x_1) & \rg(h_1) \\
      \s(h_2) & &&&& \s(h_1) \\
    };
    \begin{scope}[cdar]
      \draw (m-1-3) -- node {\(F^1(g)\)} (m-1-4);
      \draw (m-2-1) -- node[swap] {\(\tilde{F}^1\bigl((x_1,h_1),g,(x_2,h_2)\bigr)
        = h_1^{-1}\cdot F^1(g)\cdot h_2\)} (m-2-6);
      \draw (m-2-1) -- node {\(h_2\)} (m-1-1);
      \draw (m-2-6) -- node {\(h_1\)} (m-1-6);
    \end{scope}

    \draw[equ] (m-1-1) -- (m-1-2);
    \draw[equ] (m-1-2) -- (m-1-3);
    \draw[equ] (m-1-4) -- (m-1-5);
    \draw[equ] (m-1-5) -- (m-1-6);
  \end{tikzpicture}
  \qedhere
  \]
\end{proof}

\section{Composition of bibundles}
\label{sec:composition}

The composition of bibundle functors and actors requires an extra
assumption on the pretopology.  We first formulate this assumption in
several closely related ways and then use it to compose the various
types of bibundles.

\subsection{Assumptions on \texorpdfstring{\v{C}}{C}ech groupoid actions}
\label{sec:assume_covering_actions}

Let~\(\covers\) be a subcanonical pretopology on~\(\Cat\).

\begin{assumption}
  \label{assum:covering_acts_basically}
  Any action of a \v{C}ech groupoid of a cover in~\((\Cat,\covers)\) is basic.
\end{assumption}

\begin{assumption}
  \label{assum:covering_acts_basically_weak}
  Any sheaf over a \v{C}ech groupoid of a cover in~\((\Cat,\covers)\) is basic.
\end{assumption}

Assumption~\ref{assum:covering_acts_basically} is obviously stronger
than Assumption~\ref{assum:covering_acts_basically_weak}.  We will use
Assumptions \ref{assum:local_cover}
and~\ref{assum:covering_acts_basically_weak} to compose bibundle
\emph{functors} and \emph{equivalences}.  The stronger Assumptions
\ref{assum:two-three} and~\ref{assum:covering_acts_basically} are
necessary and sufficient to compose bibundle \emph{actors}.

The following two lemmas reformulate the assumptions above.  Recall
that \(\Cat(\Gr)\) and \(\Cat_\covers(\Gr)\) are the categories of
\(\Gr\)\nb-actions and \(\Gr\)\nb-sheaves for a groupoid~\(\Gr\)
in~\((\Cat,\covers)\).

\begin{proposition}
  \label{pro:basic_assum}
  The following are equivalent:
  \begin{enumerate}
  \item Let~\(\Gr\) be a groupoid, \(\Act_1\) and~\(\Act_2\)
    \(\Gr\)\nb-actions and \(f\colon \Act_1\to\Act_2\) a
    \(\Gr\)\nb-map.  If \(\Act_2\) is basic,
    then so is~\(\Act_1\).
  \item Any action of a basic groupoid is basic.
  \item Any action of a \v{C}ech groupoid is basic.
  \item Let \(\bunp\colon \Act\to \Base\) be a cover, let \(\Gr\) be
    its \v{C}ech groupoid, and let~\(\Act[Y]\) be a \(\Gr\)\nb-action.
    Then there are \(\tilde{\Base}\inOb\Cat\) and a map \(f\colon
    \tilde{\Base}\to\Base\) such that
    \(\Act[Y]\cong\tilde{\Base}\times_{f,\Base,\bunp} \Act\)
    with~\(\Gr\) acting on \(\tilde{\Base}\times_{f,\Base,\bunp}
    \Act\) by \((z,x_1)\cdot (x_1,x_2)\defeq (z,x_2)\), as in
    Proposition~\textup{\ref{pro:pull-back_principal}}.
  \item Let \(\bunp\colon \Act\to \Base\) be a cover and let~\(\Gr\)
    be its \v{C}ech groupoid.  Then the functor \(\Cat(\Base)\to
    \Cat(\Gr)\) induced by the equivalence bibundle~\(\Act\)
    between~\(\Gr\) and the \(0\)\nb-groupoid~\(\Base\) is an
    equivalence of categories.
  \end{enumerate}
\end{proposition}

\begin{proof}
  (1)\(\Rightarrow\)(2): Let~\(\Gr\) be a basic groupoid and
  let~\(\Act\) be a \(\Gr\)\nb-action.  Let \(\Act_2=\Gr^0\) and let
  \(f=\s\colon \Act\to\Act_2\) be the anchor map.  Since~\(f\) is a
  \(\Gr\)\nb-map and the \(\Gr\)\nb-action on~\(\Gr^0\) is basic by
  assumption, (1) gives that the \(\Gr\)\nb-action on~\(\Act\) is also
  basic.

  (2) implies~(3) because \v{C}ech groupoids are basic
  (Example~\ref{exa:covering_groupoid_principal}).

  (3) implies~(1): the \(\Gr\)\nb-action on~\(\Act_1\) and the
  map~\(f\) combine to an action of \(\Act_2\rtimes\Gr\)
  on~\(\Act_1\).  The groupoid \(\Act_2\rtimes\Gr\) is isomorphic to
  a \v{C}ech groupoid of a cover by
  Proposition~\ref{pro:covering_groupoid_basic}.  Hence its action
  on~\(\Act_1\) is basic by~(3).
  Corollary~\ref{cor:basic_action_trafo_gr} shows that the action
  of~\(\Gr\) on~\(\Act_1\) is basic as well.

  (4)\(\iff\)(3): The action of the \v{C}ech groupoid of~\(\bunp\)
  on~\(\Act\) is basic by
  Example~\ref{exa:covering_groupoid_principal}.  Hence its
  pull-back along \(f\colon \tilde{\Base}\to\Base\) remains a basic
  action by Proposition~\ref{pro:pull-back_principal}.  Thus the
  actions described in~(4) are basic, so~(4) implies~(3).
  Conversely, let~\(\tilde{\Act}\) be a basic \(\Gr\)\nb-action for
  the \v{C}ech groupoid of \(\bunp\colon \Act\prto\Base\).  Let
  \(\tilde{\bunp}\colon \tilde{\Act}\to\tilde{\Base}\) be the bundle
  projection.  The anchor map \(\hat{f}\defeq \tilde{\s}\colon
  \tilde{\Act}\to\Act=\Gr^0\) of the \(\Gr\)\nb-action
  on~\(\tilde{\Act}\) is a \(\Gr\)\nb-map.  It induces a map
  \(f=\hat{f}/\Gr\colon \tilde{\Base}\to\Base\) by
  Proposition~\ref{pro:G-map_versus_base_map}.
  Proposition~\ref{pro:pull-back_principal} shows that
  \(\tilde{\Act}\) is isomorphic to
  \(\tilde{\Base}\times_{f,\Base,\bunp}\Act\) with the canonical
  action.  Thus~(3) implies~(4).

  A \(\Base\)\nb-action is the same as a map
  \(\tilde{\Base}\to\Base\).  The \v{C}ech groupoid~\(\Gr\) of
  \(\bunp\colon \Act\to\Base\) is equivalent to the
  \(0\)\nb-groupoid~\(\Base\) by
  Example~\ref{exa:covering_groupoid_equivalent_to_space}.  The
  functor \(\Cat(\Base)\to \Cat(\Gr)\) induced by this equivalence
  is, by definition, the pull-back construction described in~(4).
  (This is indeed a special case of the composition of bibundle
  equivalences with actions defined below.)  Propositions
  \ref{pro:pull-back_principal} and~\ref{pro:G-map_versus_base_map}
  yield a bijection between maps
  \(\tilde{\Base}_1\to\tilde{\Base}_2\) and \(\Gr\)\nb-maps
  \(\tilde{\Base}_1\times_{\Base}\Act \to
  \tilde{\Base}_2\times_{\Base}\Act\).  This means that the functor
  \(\Cat(\Base)\to \Cat(\Gr)\) is fully faithful.  Condition~(4)
  means that it is essentially surjective.  A fully faithful functor
  is an equivalence if and only if it is essentially surjective.
\end{proof}

Proposition~\ref{pro:basic_assum}.(5) shows that we need
Assumption~\ref{assum:covering_acts_basically} if we want equivalent
groupoids to have equivalent action categories.

\begin{proposition}
  \label{pro:basic_assum_weak}
  The following are equivalent:
  \begin{enumerate}
  \item Let~\(\Gr\) be a groupoid, \(\Act_1\) and~\(\Act_2\)
    \(\Gr\)\nb-actions and \(f\colon \Act_1\prto\Act_2\) a cover
    that is also a \(\Gr\)\nb-map.  If \(\Act_2\) is basic, then so
    is~\(\Act_1\).
  \item Any sheaf over a basic groupoid is basic.
  \item Any sheaf over a \v{C}ech groupoid is basic.
  \item Let \(\bunp\colon \Act\prto \Base\) be a cover, let~\(\Gr\) be
    its \v{C}ech groupoid.  Let~\(\Act[Y]\) be a \(\Gr\)\nb-sheaf.
    Then there are \(\tilde{\Base}\inOb\Cat\) and a map
    \(f\colon\tilde{\Base}\to\Base\) such that
    \(\Act[Y]\cong\tilde{\Base}\times_{f,\Base,\bunp} \Act\)
    with~\(\Gr\) acting on~\(\tilde{\Base}\times_{f,\Base,\bunp}
    \Act\) by \((z,x_1)\cdot (x_1,x_2)\defeq (z,x_2)\), as in
    Proposition~\textup{\ref{pro:pull-back_principal}}.
  \end{enumerate}
  Assumption~\textup{\ref{assum:local_cover}} is equivalent to the
  statement that the map~\(f\) in \textup{(4)} is automatically a
  cover.
\end{proposition}

\begin{proof}
  The proof is the same as for Proposition~\ref{pro:basic_assum}.
  Assumption~\ref{assum:local_cover} is necessary and sufficient for
  the map~\(f\) in~(4) to be a cover because each fibre-product
  situation~\eqref{eq:fibre-product_of_cover} in which~\(g\) is a
  cover gives rise to a situation as in~(4).
\end{proof}

Let \(\bunp\colon \Act\prto \Base\) be a cover and let~\(\Gr\) be its
\v{C}ech groupoid.  Then the equivalence~\(\Act\) induces a functor
\(\Cat_\covers(\Base)\to \Cat_\covers(\Gr)\) between the
\(0\)\nb-groupoid~\(\Base\) and~\(\Gr\).  By
Proposition~\ref{pro:basic_assum_weak}, Assumptions
\ref{assum:local_cover} and~\ref{assum:two-three}
together are equivalent to the statement that this functor
\(\Cat_\covers(\Base)\to \Cat_\covers(\Gr)\) is an equivalence of
categories.

\begin{proposition}
  \label{pro:smaller_coverings_assum}
  Let \(\covers\) and~\(\covers'\) be pretopologies on~\(\Cat\) with
  \(\covers\subseteq\covers'\).  If~\(\covers'\) is subcanonical,
  then so is~\(\covers\).  If~\(\covers'\) satisfies
  Assumption~\textup{\ref{assum:covering_acts_basically}}, then so
  does~\(\covers\).  If~\(\covers'\) satisfies
  Assumption~\textup{\ref{assum:covering_acts_basically_weak}}, then
  so does~\(\covers\).
\end{proposition}

\begin{proof}
  Let \(\bunp\colon \Act\prto\Base\) be a cover in~\(\covers\) and
  let~\(\Gr\) be its \v{C}ech groupoid.  Let~\(\Act[Y]\) be a
  \(\Gr\)\nb-action.  Since \(\covers\subseteq\covers'\), \(\bunp\)
  is also a cover in~\(\covers'\).  Proposition~\ref{pro:basic_assum}.(4)
  for the pretopology~\(\covers'\) shows that there is a map
  \(f\colon \tilde{\Base}\to\Base\) such that
  \(\Act[Y]\cong\tilde{\Base}\times_{f,\Base,\bunp} \Act\)
  with~\(\Gr\) acting by \((z,x_1)\cdot (x_1,x_2)\defeq (z,x_2)\).
  This action is basic also in~\((\Cat,\covers)\).  The proof of the
  second statement is similar.
\end{proof}

Often a category admits many different pretopologies.
Proposition~\ref{pro:smaller_coverings_assum} shows that we do not
have to check Assumptions \ref{assum:covering_acts_basically}
and~\ref{assum:covering_acts_basically_weak} for all pretopologies,
it suffices to look at a maximal one among the interesting
pretopologies.

\subsection{Composition of bibundle functors and actors}
\label{sec:composition_sub}

\begin{proposition}
  \label{pro:bibundles_act}
  Let \(\Gr\) and~\(\Gr[H]\) be groupoids in~\((\Cat,\covers)\).
  Under Assumption~\textup{\ref{assum:covering_acts_basically}}, a
  bibundle actor~\(\Act\) from~\(\Gr\) to~\(\Gr[H]\) induces a
  functor
  \[
  \Cat(\Gr[H])\to\Cat(\Gr),\qquad
  \Act[Y]\mapsto \Act\times_{\Gr[H]}\Act[Y].
  \]

  Under
  Assumption~\textup{\ref{assum:covering_acts_basically_weak}}, a
  \(\Gr,\Gr[H]\)-bibundle~\(\Act\) with basic action of~\(\Gr[H]\)
  induces a functor
  \[
  \Cat_\covers(\Gr[H])\to\Cat(\Gr),\qquad
  \Act[Y]\mapsto \Act\times_{\Gr[H]}\Act[Y].
  \]

  Both constructions above are natural in~\(\Act\), that is, a
  \(\Gr,\Gr[H]\)-map \(\Act_1\to\Act_2\) induces a natural
  \(\Gr\)\nb-map \(\Act_1\times_{\Gr[H]}\Act[Y] \to
  \Act_2\times_{\Gr[H]}\Act[Y]\).  Assume
  Assumption~\textup{\ref{assum:local_cover}}.  If the
  \(\Gr,\Gr[H]\)-map \(f\colon \Act_1\prto\Act_2\) and the
  \(\Gr[H]\)-map \(g\colon \Act[Y]_1\prto\Act[Y]_2\) are covers,
  then so is the induced map \(f\times_{\Gr[H]} g\colon
  \Act_1\times_{\Gr[H]} \Act[Y]_1\to \Act_2\times_{\Gr[H]}
  \Act[Y]_2\).

  If Assumptions \textup{\ref{assum:covering_acts_basically_weak}}
  and~\textup{\ref{assum:two-three}} hold, then
  \(\Act\times_{\Gr[H]}\Act[Y]\inOb\Cat_\covers(\Gr)\) if~\(\Act\) is
  a \(\Gr,\Gr[H]\)-bibundle functor and
  \(\Act[Y]\inOb\Cat_\covers(\Gr[H])\).
\end{proposition}

\begin{proof}
  The proofs for bibundle functors and actors are almost the same.
  The first difference is that the fibre product
  \[
  \Act[XY]\defeq \Act\times_{\s,\Gr[H]^0,\s} \Act[Y]
  \]
  exists in~\(\Cat\) for different reasons: if~\(\Act\) is a bibundle
  actor, then \(\s\colon \Act\prto\Gr[H]^0\) is a cover, and otherwise
  \(\Act[Y]\inOb\Cat_\covers(\Gr[H])\) means that \(\s\colon
  \Act[Y]\prto\Gr[H]^0\) is a cover.

  There is an obvious left action of~\(\Gr\) on~\(\Act[XY]\) with
  anchor map \(\rg_{\Act[XY]}\defeq \rg\circ\pr_{\Act}\colon
  \Act[XY]\to\Act\to\Gr^0\) and multiplication map
  \[
  \Gr^1\times_{\s,\Gr^0,\rg_{\Act[XY]}} \Act[XY]\to\Act[XY]
  \]
  given elementwise by \(g\cdot (x,y)\defeq (g\cdot x,y)\) for
  \(g\in\Gr^1\), \(x\in\Act\), \(y\in\Act[Y]\) with
  \(\s(g)=\rg(x)\), \(\s(x)=\s(y)\).  We equip~\(\Act[XY]\) with the
  unique right \(\Gr[H]\)\nb-action for which the coordinate
  projections \(\pr_{\Act}\colon \Act[XY]\to\Act\) and
  \(\pr_{\Act[Y]}\colon \Act[XY]\to\Act[Y]\) are \(\Gr[H]\)\nb-maps
  (see Lemma~\ref{lem:unique_fibre-product_action}).  Elementwise,
  we have \(\s(x,y)=\s(x)=\s(y)\) and \((x,y)\cdot h= (x\cdot
  h,y\cdot h)\) for all \(x\in\Act\), \(y\in\Act[Y]\),
  \(h\in\Gr[H]^1\) with \(\s(x)=\s(y)=\rg(h)\).

  The actions of \(\Gr\) and~\(\Gr[H]\) on~\(\Act[XY]\) clearly
  commute, and the map on objects
  \[
  \Cat(\Gr,\Gr[H])\times \Cat(\Gr[H])\to\Cat(\Gr,\Gr[H]),\qquad
  \Act,\Act[Y]\mapsto \Act[XY],
  \]
  is part of a bifunctor; that is, a \(\Gr,\Gr[H]\)-map
  \(\Act_1\to\Act_2\) and an \(\Gr[H]\)\nb-map
  \(\Act[Y]_1\to\Act[Y]_2\) induce a \(\Gr,\Gr[H]\)-map
  \(\Act_1\Act[Y]_1\to\Act_2\Act[Y]_2\).

  The coordinate projection \(\pr_{\Act}\colon \Act[XY]\to\Act\) is an
  \(\Gr[H]\)\nb-map and~\(\Act\) is a basic \(\Gr[H]\)\nb-action by
  assumption.  In the first case,
  Proposition~\ref{pro:basic_assum}.(1) shows that \(\Act[XY]\) is a
  basic \(\Gr[H]\)\nb-action.  In the second case, the map
  \(\pr_{\Act}\colon \Act[XY]\prto\Act\) is a cover because
  \(\s\colon \Act[Y]\prto\Gr[H]^0\) is one; thus
  Proposition~\ref{pro:basic_assum_weak}.(1) shows that \(\Act[XY]\)
  is a basic \(\Gr[H]\)\nb-action.  Let \(\bunp\colon
  \Act[XY]\prto\Base\) be the bundle projection of this basic
  action.  We also write \(\Base=\Act\times_{\Gr[H]} \Act[Y]\).

  Recall that a \(\Gr,\Gr[H]\)\nb-map \(f\colon \Act_1\to\Act_2\)
  and an \(\Gr[H]\)\nb-map \(g\colon \Act[Y]_1\to\Act[Y]_2\) induce
  a \(\Gr,\Gr[H]\)\nb-map \(f\times_{\Gr[H]^0} g\colon
  \Act_1\Act[Y]_1\to\Act_2\Act[Y]_2\).  By
  Proposition~\ref{pro:G-map_versus_base_map}, this induces a map
  \(f\times_{\Gr[H]} g\colon
  \Act_1\times_{\Gr[H]}\Act[Y]_1\to\Act_2\times_{\Gr[H]}\Act[Y]_2\).
  Thus the construction of \(\Act\times_{\Gr[H]} \Act[Y]\) is
  bifunctorial.  If \(f\) and~\(g\) are covers, then so is the
  induced map \(f\times_{\Gr[H]^0} g\); this is a general property
  of pretopologies.  By Proposition~\ref{pro:G-map_versus_base_map},
  the map \(f\times_{\Gr[H]} g\) induced by this on the base spaces
  of principal bundles is a cover as well if
  Assumption~\ref{assum:local_cover} holds.

  It remains to push the \(\Gr\)\nb-action on~\(\Act[XY]\) down to a
  natural \(\Gr\)\nb-action on~\(\Base\).  Let
  \(\Base[W]\defeq\Act/\Gr[H]\).  The anchor map \(\s\colon
  \Act\to\Gr^0\) descends to a map \(\s_{\Base[W]}\colon
  \Base[W]\to\Gr^0\) because it is \(\Gr[H]\)\nb-invariant.
  Proposition~\ref{pro:G-map_versus_base_map} gives a unique map
  \(\pr_{\Act}/\Gr[H]\colon \Base\to\Base[W]\), which we compose
  with~\(\s_{\Base[W]}\) to get a map \(\rg_{\Base}\colon
  \Base\to\Gr^0\).  This is the anchor map of the desired action
  on~\(\Base\).
  It is the unique map with \(\rg_{\Act}\circ\pr_{\Act} =
  \rg_{\Base}\circ\bunp\colon \Act\Act[Y]\to\Gr^0\).

  Pulling back the principal \(\Gr[H]\)\nb-bundle
  \(\Act[XY]\prto\Base\) along \(\pr_2\colon
  \Gr^1\times_{\s,\Gr^0,\rg_{\Base}} \Base\prto \Base\) gives
  \[
  \Gr^1\times_{\s,\Gr^0,\rg_{\Base}} \Base \times_{\Base} \Act[XY]
  \cong \Gr^1\times_{\s,\Gr^0,\rg} \Act[XY]
  \]
  with its usual \(\Gr[H]\)\nb-action.  This is a principal bundle
  over \(\Gr^1\times_{\s,\Gr^0,\rg_{\Base}} \Base\) by
  Proposition~\ref{pro:pull-back_principal}.  By
  Proposition~\ref{pro:G-map_versus_base_map}, the left
  \(\Gr^1\)\nb-action
  \(\Gr^1\times_{\s,\Gr^0,\rg}\Act[XY]\to\Act[XY]\) induces a map
  \(\mul_{\Base}\colon \Gr^1\times_{\s,\Gr^0,\rg_{\Base}}
  \Base\to\Base\) on the bases of these principal bundles.  It is
  routine to see that \(\rg_{\Base}\) and~\(\mul_{\Base}\) define a
  left \(\Gr\)\nb-action on~\(\Base\), using the corresponding facts
  for~\(\Act[XY]\) and passing to orbit spaces by
  Proposition~\ref{pro:G-map_versus_base_map}.  Thus
  \(\Act\times_{\Gr[H]}\Act[Y]\) carries a left \(\Gr\)\nb-action.
  This \(\Gr\)\nb-action is natural in the sense that the maps
  \(f\times_{\Gr[H]} g\) defined above are \(\Gr\)\nb-equivariant.

  If~\(\Act\) is a \(\Gr,\Gr[H]\)-bibundle functor and
  \(\Act[Y]\inOb\Cat_\covers(\Gr[H])\), then \(\pr_{\Act}\colon
  \Act[XY]\prto\Act\) is a cover because \(\s\colon
  \Act[Y]\prto\Gr[H]^0\) is one.  Since \(\rg\colon \Act\prto\Gr^0\)
  for bibundle functors, the composite map \(\Act[XY]\prto\Gr^0\) is a
  cover as well.  The map \(\Act[XY]\prto\Act\times_{\Gr[H]}\Act[Y]\)
  is a cover as the bundle projection of a principal bundle.
  Assumption~\textup{\ref{assum:two-three}} shows that the induced map
  \(\Act\times_{\Gr[H]}\Act[Y]\to\Gr^0\) is a cover, so
  \(\Act\times_{\Gr[H]}\Act[Y]\inOb\Cat_\covers(\Gr)\).
\end{proof}

\begin{proposition}
  \label{pro:invariants_as_composite}
  The quotient~\(\Act/\Gr[H]\) for a \(\Gr,\Gr[H]\)-bibundle~\(\Act\)
  with basic action of\/~\(\Gr[H]\) inherits a natural
  \(\Gr\)\nb-action.  If~\(\Act\) carries an action of a third
  groupoid~\(\Gr[K]\) that commutes with the actions of \(\Gr\)
  and~\(\Gr[H]\), then the induced actions of \(\Gr\)
  and~\(\Gr[K]\) on~\(\Act/\Gr[H]\) commute.
\end{proposition}

\begin{proof}
  The first assertion is shown during the proof of
  Proposition~\ref{pro:bibundles_act}.  We are dealing with
  the special case \(\Act[Y]=\Gr[H]^0\), where \(\Act[XY]\cong
  \Act\), so that \(\Act\times_{\Gr[H]}\Gr[H]^0\cong \Act/\Gr[H]\).
  No extra assumption about actions of \v{C}ech groupoids is needed in
  this particular case.  A (left or right) action of~\(\Gr[K]\)
  commuting with those of \(\Gr\) and~\(\Gr[H]\) induces an action
  on~\(\Act/\Gr[H]\) for the same reasons as for the action
  of~\(\Gr\).  The proof that the induced actions of \(\Gr\)
  and~\(\Gr[K]\) on~\(\Act/\Gr[H]\) still commute is routine and
  omitted.  We only remark that, under Assumption~\ref{assum:final},
  we could replace a pair of commuting actions of \(\Gr\)
  and~\(\Gr[K]\) by a single action of \(\Gr\times\Gr[K]\), which
  would clearly still commute with the action of~\(\Gr[H]\).  This
  action of \(\Gr\times\Gr[K]\) descends to~\(\Act/\Gr[H]\) and may
  then be turned back into two commuting actions of \(\Gr\)
  and~\(\Gr[K]\) on~\(\Act/\Gr[H]\).
\end{proof}

\begin{proposition}
  \label{pro:compose_bibundles}
  Assume Assumptions
  \textup{\ref{assum:covering_acts_basically_weak}}
  and~\textup{\ref{assum:local_cover}}.

  Let \(\Gr\), \(\Gr[H]\) and~\(\Gr[K]\) be groupoids
  in~\((\Cat,\covers)\).  Let \(\Act\) and~\(\Act[Y]\) be
  bibundle functors from~\(\Gr\) to~\(\Gr[H]\) and from~\(\Gr[H]\)
  to~\(\Gr[K]\), respectively.  Then \(\Act\times_{\Gr[H]}\Act[Y]\)
  is a bibundle functor from~\(\Gr\) to~\(\Gr[K]\) in a natural way
  with respect to \(\Gr,\Gr[H]\)-maps \(\Act_1\to\Act_2\) and
  \(\Gr[H],\Gr[K]\)-maps \(\Act[Y]_1\to\Act[Y]_2\).  If both
  \(\Act\) and~\(\Act[Y]\) are bibundle equivalences, then so is
  \(\Act\times_{\Gr[H]} \Act[Y]\).

  Under Assumption~\textup{\ref{assum:two-three}}, if both bibundle
  functors \(\Act\) and~\(\Act[Y]\) are covering, so is
  \(\Act\times_{\Gr[H]} \Act[Y]\).

  Assume Assumptions \textup{\ref{assum:covering_acts_basically}}
  and~\textup{\ref{assum:two-three}}.  Let \(\Act\)
  and~\(\Act[Y]\) be bibundle actors from~\(\Gr\) to~\(\Gr[H]\) and
  from~\(\Gr[H]\) to~\(\Gr[K]\), respectively.  Then
  \(\Act\times_{\Gr[H]}\Act[Y]\) is a bibundle actor from~\(\Gr\)
  to~\(\Gr[K]\) in a natural way with respect to \(\Gr,\Gr[H]\)-maps
  \(\Act_1\to\Act_2\) and \(\Gr[H],\Gr[K]\)-maps
  \(\Act[Y]_1\to\Act[Y]_2\).
\end{proposition}

\begin{proof}
  Since right and left actions of~\(\Gr[H]\) are equivalent,
  Proposition~\ref{pro:bibundles_act} also works for a left
  \(\Gr[H]\)\nb-action on~\(\Act[Y]\).  All cases of composition
  considered fall into one of the two cases of
  Proposition~\ref{pro:bibundles_act}, which provides
  \(\Act\times_{\Gr[H]}\Act[Y]\inOb\Cat\) with a natural \(\Gr\)\nb-action.

  A right \(\Gr[K]\)\nb-action on~\(\Act\times_{\Gr[H]}\Act[Y]\) is
  constructed like the left \(\Gr\)\nb-action in the proof of
  Proposition~\ref{pro:bibundles_act}.  First, \(\Act[XY]\)
  carries a right \(\Gr[K]\)\nb-action that commutes with the
  actions of \(\Gr\) and~\(\Gr[H]\): the anchor map is
  \(\s\circ\pr_{\Act[Y]}\colon \Act[XY]\to\Act[Y]\to\Gr[K]^0\) and
  the multiplication is \(\id_{\Act} \times_{\s,\Gr[H]^0}
  \mul_{\Act[Y]}\colon \Act[XY]
  \times_{\s\circ\pr_{\Act[Y]},\Gr[K]^0,\rg} \Gr[K]^1\to \Act[XY]\).
  Proposition~\ref{pro:invariants_as_composite} shows that this
  \(\Gr[K]\)\nb-action descends from \(\Act[XY]\)
  to~\(\Act\times_{\Gr[H]}\Act[Y]\) and still commutes the induced
  action of~\(\Gr\) on \(\Act\times_{\Gr[H]} \Act[Y]\) commute.

  It remains to show that the \(\Gr,\Gr[K]\)-bibundle
  \(\Act\times_{\Gr[H]}\Act[Y]\) is a bibundle functor, bibundle
  equivalence, covering bibundle functor, or bibundle actor if both
  \(\Act\) and~\(\Act[Y]\) are, under appropriate assumptions
  on~\((\Cat,\covers)\).  We consider the case of bibundle functors
  first.  We must show that \(\rg\colon
  \Act\times_{\Gr[H]}\Act[Y]\to\Gr^0\) is a principal
  \(\Gr[K]\)\nb-bundle.

  It is clear that the anchor map \(\rg\colon
  \Act\times_{\Gr[H]}\Act[Y]\to\Gr^0\) is \(\Gr[K]\)\nb-invariant.
  We claim that it is a cover.  The map \(\pr_{\Act}\colon
  \Act[XY]\prto\Act\) is a cover because \(\rg\colon
  \Act[Y]\prto\Gr[H]^0\) is one.  Hence so is the induced map
  \(\pr_{\Act}/\Gr[H]\colon \Act\times_{\Gr[H]}\Act[Y]\prto
  \Act/\Gr[H]\cong \Gr^0\) by
  Proposition~\ref{pro:G-map_versus_base_map}; here we need
  Assumption~\ref{assum:local_cover}.  This map is the anchor map
  \(\rg\colon \Act\times_{\Gr[H]}\Act[Y]\to\Gr^0\).

  We must also show that the map
  \begin{equation}
    \label{eq:iso_on_XY_over_H}
    (\mul,\pr_1)\colon
    (\Act\times_{\Gr[H]}\Act[Y])\times_{\s,\Gr[K]^0,\rg} \Gr[K]^1
    \to (\Act\times_{\Gr[H]}\Act[Y])\times_{\rg,\Gr^0,\rg}
    (\Act\times_{\Gr[H]}\Act[Y])
  \end{equation}
  is an isomorphism.
  Since \(\rg\colon \Act[Y]\prto\Gr[H]^0\) is a principal
  \(\Gr[K]\)\nb-bundle, its pull-back along \(\s\colon
  \Act\prto\Gr[H]^0\) is a principal \(\Gr[K]\)\nb-bundle
  \(\pr_{\Act}\colon \Act[XY]\prto\Act\).  This means that the map
  \begin{equation}
    \label{eq:iso_on_XY}
    \Act[XY]\times_{\s,\Gr[K]^0,\rg} \Gr[K]^1 \xrightarrow{(\mul,\pr_1)}
    \Act[XY]\times_{\pr_{\Act},\Act,\pr_{\Act}} \Act[XY]
  \end{equation}
  is an isomorphism.

  Pulling \(\Act[XY]\prto\Act\times_{\Gr[H]}\Act[Y]\) back along
  \(\rg\colon \Gr[K]^1\prto\Gr[K]^0\) gives a principal
  \(\Gr[H]\)\nb-bundle \(\Act[XY]\times_{\s,\Gr[K]^0,\rg} \Gr[K]^1\prto
  (\Act\times_{\Gr[H]}\Act[Y])\times_{\s,\Gr[K]^0,\rg} \Gr[K]^1\).
  Proposition~\ref{pro:fibre-product_base_above} shows that
  \(\Act[XY]\times_{\pr_{\Act},\Act,\pr_{\Act}} \Act[XY]\) is a
  principal \(\Gr[H]\)\nb-bundle over \((\Act\times_{\Gr[H]}
  \Act[Y]) \times_{\rg,\Gr^0,\rg} (\Act\times_{\Gr[H]} \Act[Y])\).
  Hence~\eqref{eq:iso_on_XY_over_H} is the map on the base spaces
  induced by the \(\Gr[H]\)\nb-equivariant
  isomorphism~\eqref{eq:iso_on_XY}.
  Thus~\eqref{eq:iso_on_XY_over_H} is an isomorphism as well by
  Proposition~\ref{pro:G-map_versus_base_map}.  This finishes the
  proof that \(\Act\times_{\Gr[H]}\Act[Y]\) is a bibundle functor if
  \(\Act\) and~\(\Act[Y]\) are.  We have used Assumptions
  \ref{assum:covering_acts_basically_weak}
  and~\ref{assum:local_cover}.

  If both \(\Act\) and~\(\Act[Y]\) are bibundle equivalences, then the
  same
  arguments as above for the right \(\Gr[K]\)\nb-action apply to the
  left \(\Gr\)\nb-action on~\(\Act\times_{\Gr[H]} \Act[Y]\).  They
  show that \(\s\colon \Act\times_{\Gr[H]} \Act[Y]\prto \Gr[K]^0\) is
  a principal \(\Gr\)\nb-bundle.  Thus \(\Act\times_{\Gr[H]}\Act[Y]\)
  is a bibundle equivalence if \(\Act\) and~\(\Act[Y]\) are, under
  Assumptions \ref{assum:covering_acts_basically_weak}
  and~\ref{assum:local_cover}.

  Assume now that \(\s\colon \Act\prto\Gr[H]^0\) and \(\s\colon
  \Act[Y]\prto\Gr[K]^0\) are covers.  Then so are
  \(\pr_{\Act[Y]}\colon \Act[XY]\prto\Act[Y]\) and the composite map
  \(\Act[XY]\prto\Act[Y]\prto\Gr[K]^0\).  Since the bundle projection
  \(\Act[XY]\prto \Act\times_{\Gr[H]}\Act[Y]\) is a cover as well,
  Assumption~\ref{assum:two-three} gives that the anchor map
  \(\Act\times_{\Gr[H]}\Act[Y]\to \Gr[K]^0\) is a cover.  Thus a
  product of covering bibundle functors is again a covering bibundle
  functor, under Assumptions \ref{assum:two-three}
  and~\ref{assum:covering_acts_basically_weak}.  Furthermore, the
  right anchor map for a product of two bibundle actors is a cover.
  It remains to show that the right action on the product of two
  bibundle actors is basic.

  Since the \(\Gr[K]\)\nb-action on~\(\Act[Y]\) is basic, we may
  form its base \(\Act[W]\cong\Act[Y]/\Gr[K]\) and get a principal
  \(\Gr[K]\)\nb-bundle \(\bunp\colon \Act[Y]\prto\Act[W]\).
  Proposition~\ref{pro:invariants_as_composite} shows that the
  \(\Gr[H]\)\nb-action on~\(\Act[Y]\) descends to~\(\Act[W]\).

  Now we may also form the \(\Gr\)\nb-action
  \(\Act\times_{\Gr[H]}\Act[W]\), and the cover~\(\bunp\) induces a
  cover \(q\colon \Act\times_{\Gr[H]}\Act[Y] \prto
  \Act\times_{\Gr[H]}\Act[W]\) by
  Proposition~\ref{pro:bibundles_act}; here we only need
  Assumption~\ref{assum:local_cover}.  The map
  \begin{equation}
    \label{eq:compose_bibundles_K_basic}
    (\mul,\pr_1)\colon
    \Act[XY] \times_{\s,\Gr[K]^0,\rg} \Gr[K]^1 \to
    \Act[XY] \times_{\Act[XW]} \Act[XY]
  \end{equation}
  is an isomorphism because \(\bunp\colon \Act[Y]\prto\Act[W]\) is a
  principal \(\Gr[K]\)\nb-bundle, which we have pulled back
  to~\(\Act[XW]\).  The spaces in~\eqref{eq:compose_bibundles_K_basic}
  are total spaces of principal \(\Gr[H]\)\nb-bundles over
  \((\Act\times_{\Gr[H]}\Act[Y]) \times_{\s,\Gr[K]^0,\rg} \Gr[K]^1\)
  and \((\Act\times_{\Gr[H]}\Act[Y])
  \times_{\Act\times_{\Gr[H]}\Act[W]}
  (\Act\times_{\Gr[H]}\Act[Y])\), respectively.
  Proposition~\ref{pro:bibundles_act} shows that the induced
  map \((\mul,\pr_1)\colon (\Act\times_{\Gr[H]}\Act[Y])
  \times_{\s,\Gr[K]^0,\rg} \Gr[K]^1\to (\Act\times_{\Gr[H]}\Act[Y])
  \times_{\Act\times_{\Gr[H]}\Act[W]} (\Act\times_{\Gr[H]}\Act[Y])\)
  on the bases is also an isomorphism.  Hence the induced
  \(\Gr[K]\)\nb-action on \(\Act\times_{\Gr[H]}\Act[Y]\) is basic
  with bundle projection~\(q\).
\end{proof}

\begin{remark}
  \label{rem:induced_orbit_space}
  The proof above shows that \((\Act\times_{\Gr[H]} \Act[Y])/\Gr[K]
  \cong \Act \times_{\Gr[H]} (\Act[Y]/\Gr[K])\) whenever the actions
  of~\(\Gr[H]\) on~\(\Act\) and of~\(\Gr[K]\) on~\(\Act[Y]\) are basic
  and suitable assumptions about~\((\Cat,\covers)\) ensure
  that \(\Act\times_{\Gr[H]} \Act[Y]\) and \(\Act \times_{\Gr[H]}
  (\Act[Y]/\Gr[K])\) exist and \(q\colon \Act\times_{\Gr[H]} \Act[Y]
  \to \Act \times_{\Gr[H]} (\Act[Y]/\Gr[K])\) is a cover.
\end{remark}

\begin{proposition}
  \label{pro:compose_bibundle_associative}
  The compositions defined above under suitable assumptions are
  associative in the following sense:
  \begin{enumerate}
  \item For bibundle functors or bibundle actors \(\Act\colon \Gr
    \to\Gr[H]\), \(\Act[Y]\colon \Gr[H] \to\Gr[K]\), and
    \(\Act[Z]\colon \Gr[K]\to\Gr[L]\), there is a natural
    bibundle isomorphism \((\Act\times_{\Gr[H]}\Act[Y])\times_{\Gr[K]}
    \Act[Z] \congto \Act\times_{\Gr[H]}(\Act[Y]\times_{\Gr[K]}
    \Act[Z])\); these associators for four composable bibundle
    functors or actors make the usual pentagon diagram commute.

  \item For bibundle actors \(\Act\colon \Gr\to\Gr[H]\) and
    \(\Act[Y]\colon \Gr[H]\to\Gr[K]\), and a
    \(\Gr[K]\)\nb-action~\(\Act[Z]\), there is a natural
    \(\Gr\)-equivariant isomorphism
    \((\Act\times_{\Gr[H]}\Act[Y])\times_{\Gr[K]} \Act[Z] \congto
    \Act\times_{\Gr[H]}(\Act[Y]\times_{\Gr[K]} \Act[Z])\).
  \end{enumerate}
\end{proposition}

\begin{proof}
  Literally the same argument proves both statements.  The usual
  fibre product \(\Act[XYZ] \defeq \Act\times_{\s,\Gr[H]^0,\rg}
  \Act[Y]\times_{\s,\Gr[K]^0,\rg} \Act[Z]\) is associative up to very
  canonical isomorphisms, which we drop from our notation to simplify.
  This space carries commuting actions of \(\Gr[H]\) and~\(\Gr[K]\) by
  \((x,y,z)\cdot h=(x\cdot h,h^{-1}\cdot y,z)\) and \((x,y,z)\cdot
  k=(x,y\cdot k,k^{-1}\cdot z)\) for \(x\in\Act\), \(y\in\Act[Y]\),
  \(z\in\Act[Z]\), \(h\in\Gr[H]^1\), \(k\in\Gr[K]^1\) with
  \(\s(x)=\rg(y)=\rg(h)\), \(\s(y)=\rg(z)=\rg(k)\).  We get
  \(\Act\times_{\Gr[H]}(\Act[Y]\times_{\Gr[K]} \Act[Z])\) from
  \(\Act[XYZ]\) by first taking the base space \(\Act\times_{\Gr[H]^0}
  (\Act[Y]\times_{\Gr[K]} \Act[Z])\) for the basic
  \(\Gr[K]\)\nb-action on~\(\Act[XYZ]\) and then taking the base
  space of the resulting basic \(\Gr[H]\)\nb-action on
  \(\Act\times_{\Gr[H]^0} (\Act[Y]\times_{\Gr[K]} \Act[Z])\).
  Remark~\ref{rem:induced_orbit_space} implies that
  \(\Act\times_{\Gr[H]}(\Act[Y]\times_{\Gr[K]} \Act[Z])\) is
  naturally isomorphic to
  \((\Act\times_{\Gr[H]}(\Act[Y]\times_{\Gr[K]^0} \Act[Z]))/\Gr[K]\).
  The latter is, in turn, naturally isomorphic to
  \(((\Act\times_{\Gr[H]}\Act[Y])\times_{\Gr[K]^0} \Act[Z])/\Gr[K]\),
  which is \((\Act\times_{\Gr[H]}\Act[Y])\times_{\Gr[K]} \Act[Z]\) by
  definition.  This provides the desired associator.

  The associator is characterised uniquely by the statement that it is
  lifted by the associator \((\Act\times_{\s,\Gr[H]^0,\rg}
  \Act[Y])\times_{\s,\Gr[K]^0,\rg} \Act[Z] \to
  \Act\times_{\s,\Gr[H]^0,\rg} (\Act[Y]\times_{\s,\Gr[K]^0,\rg}
  \Act[Z])\).  Since the latter associators clearly make the
  associator pentagon commute, the same holds for the induced maps on
  the quotients.
\end{proof}

\begin{remark}
  \label{rem:final_object_action_actor}
  Assume Assumption~\ref{assum:final} about a final object~\(\star\).
  Then a left \(\Gr\)\nb-action~\(\Act\) is the same as a bibundle
  actor from~\(\Gr\) to~\(\star\).  The functor
  \(\Cat(\Gr[H])\to\Cat(\Gr)\) defined by a bibundle actor~\(\Act\)
  from~\(\Gr\) to~\(\Gr[H]\) in Proposition~\ref{pro:bibundles_act} is
  equivalent to the composition of bibundle actors
  \(\Gr\to\Gr[H]\to\star\) under this identification.  Thus the second
  statement in Proposition~\ref{pro:compose_bibundle_associative}
  becomes a special case of the first one under
  Assumption~\ref{assum:final}.
\end{remark}

\begin{proposition}
  \label{pro:compose_bibundle_unital}
  The bibundle equivalence~\(\Gr^1\) from~\(\Gr\) to itself is a unit
  both for the composition of bibundle functors and bibundle actors,
  up to natural isomorphisms \(\Act\times_{\Gr} \Gr^1 \cong \Act\) and
  \(\Gr^1\times_{\Gr} \Act[Y] \cong \Act[Y]\) for a bibundle functor
  or actor~\(\Act\) to~\(\Gr\) and a bibundle
  functor or actor~\(\Act[Y]\) out of\/~\(\Gr\).  The functor
  \(\Cat(\Gr)\to\Cat(\Gr)\) induced by the  bibundle
  equivalence~\(\Gr^1\) is naturally isomorphic to the identity
  functor.  It is routine to check that this natural construction is
  also equivariant with respect to the appropriate left or right
  actions.
\end{proposition}

\begin{proof}
  The bibundle equivalence~\(\Gr^1\) is constructed in
  Example~\ref{exa:unit_bibundle}.  Pulling the principal
  left \(\Gr\)\nb-bundle \(\s\colon \Gr^1\prto \Gr^0\) back along the
  anchor map \(\s\colon \Act\prto\Gr^0\) gives a principal left
  \(\Gr\)\nb-bundle \(\Act\times_{\s,\Gr^0,\s} \Gr^1\prto \Act\); here
  the action of~\(\Gr\) is by \(g_1\cdot (x,g_2) \defeq (x,g_1\cdot
  g_2)\).  The map \(\Act\times_{\s,\Gr^0,\s} \Gr^1\to
  \Act\times_{\s,\Gr^0,\rg} \Gr^1\), \((x,g)\mapsto (x\cdot
  g^{-1},g)\), is an isomorphism with inverse \((x,g)\mapsto (x\cdot
  g,g)\).  It intertwines the \(\Gr\)\nb-action
  on~\(\Act\times_{\s,\Gr^0,\s} \Gr^1\) with the \(\Gr\)\nb-action on
  \(\Act\times_{\s,\Gr^0,\rg} \Gr^1\) by \(g_1\cdot (x,g_2) \defeq
  (x\cdot g_1^{-1},g_1\cdot g_2)\).  These isomorphic
  \(\Gr\)\nb-actions have isomorphic orbit spaces, and the orbit space
  for the latter is \(\Act\times_{\Gr} \Gr^1\).  Hence the
  multiplication map \(\Act\times_{\s,\Gr^0,\rg}\Gr^1\to\Act\),
  \((x,g)\mapsto x\cdot g\), induces an isomorphism \(\Act\times_{\Gr}
  \Gr^1 \congto \Act\).  Similarly, the multiplication map
  \(\Gr^1\times_{\s,\Gr^0,\rg}\Act[Y]\to\Act[Y]\), \((g,y)\mapsto
  g\cdot y\), induces an isomorphism \(\Gr^1\times_{\Gr} \Act[Y]
  \congto \Act[Y]\).
\end{proof}

Putting together the results above gives the following theorem:

\begin{theorem}
  \label{the:bibundle_two-categories}
  Assume Assumptions \textup{\ref{assum:local_cover}}
  and~\textup{\ref{assum:covering_acts_basically}}.  Then there is a
  bicategory with groupoids in~\((\Cat,\covers)\) as objects,
  bibundle functors from~\(\Gr\) to~\(\Gr[H]\) as arrows, maps of
  \(\Gr,\Gr[H]\)-bibundles as \(2\)\nb-arrows between arrows
  \(\Gr\to\Gr[H]\), \(\times_{\Gr}\) in reverse order as composition,
  \(\Gr^1\)~as unit arrow on~\(\Gr\), the associator and unitors from
  Propositions \textup{\ref{pro:compose_bibundle_associative}}
  and~\textup{\ref{pro:compose_bibundle_unital}}, and composition of
  maps as vertical product of \(2\)\nb-arrows; the horizontal product
  of \(2\)\nb-arrows \(f\colon \Act_1\Rightarrow\Act_2\) for
  \(\Act_1,\Act_2\colon \Gr\rightrightarrows\Gr[H]\) and \(g\colon
  \Act[Y]_1\Rightarrow\Act[Y]_2\) for \(\Act[Y]_1,\Act[Y]_2\colon
  \Gr[H]\rightrightarrows\Gr[K]\) is \(f\times_{\Gr[H]} g\colon
  \Act_1\times_{\Gr[H]} \Act[Y]_1 \to \Act_2\times_{\Gr[H]}
  \Act[Y]_2\).  All \(2\)\nb-arrows in this bicategory are invertible.

  Assume Assumptions \textup{\ref{assum:two-three}}
  and~\textup{\ref{assum:covering_acts_basically}}.  Then almost all
  of the above holds for bibundle actors, except that \(2\)\nb-arrows
  need no longer be invertible.  Furthermore, the covering bibundle
  functors form a sub-bicategory of both.
\end{theorem}

\begin{proof}
  It is a direct consequence of the propositions above that bibundle
  functors, bibundle equivalences, bibundle actors, and covering
  bibundle functors are the arrows in bicategories, with the
  \(2\)\nb-arrows and compositions as asserted.

  We show that \(2\)\nb-arrows between bibundle functors are
  equivalences.  Let \(\Act_1\) and~\(\Act_2\) be bibundle functors
  from~\(\Gr\) to~\(\Gr[H]\) and let \(f\colon
  \Act_1\Rightarrow\Act_2\) be a \(\Gr,\Gr[H]\)-map.  Then the induced
  map \(f/\Gr[H]\colon \Act_1/\Gr[H]\to\Act_2/\Gr[H]\) is the identity
  map on~\(\Gr^0\) by \(\Gr\)\nb-equivariance and because
  \(\Act_i/\Gr[H]\cong \Gr^0\) by assumption.  Since~\(f/\Gr[H]\) is
  an isomorphism, so is~\(f\) by
  Proposition~\ref{pro:G-map_versus_base_map}.  This argument fails
  for bibundle actors because their definition does not
  determine~\(\Act_i/\Gr\).
\end{proof}

\begin{remark}
  \label{rem:two-three_actor_necessary}
  Assume Assumption~\ref{assum:final}.  Then Assumptions
  \ref{assum:covering_acts_basically} and~\ref{assum:two-three} are
  necessary to compose bibundle actors.

  By Remark~\ref{rem:final_object_action_actor}, we may identify
  \(\Cat(\Gr)\) for a groupoid~\(\Gr\) with the category of bibundle
  actors \(\Gr\to\star\) and \(2\)\nb-arrows between them.  Let
  \(\bunp\colon \Act\prto\Base\) be a cover and let~\(\Gr\) be its
  \v{C}ech groupoid.  The functor \(\Cat(\Base)\to\Cat(\Gr)\),
  \(\Act[Y]\mapsto \Act[Y]\times_{\Base}\Act\), in
  Proposition~\ref{pro:basic_assum}.5 composes with the bibundle
  equivalence~\(\Act\) from~\(\Base\) to~\(\Gr\).  Exchanging left
  and right in~\(\Act\) gives a bibundle equivalence~\(\Act^*\)
  from~\(\Gr\) to~\(\Base\), such that \(\Act^*\times_{\Base}
  \Act\cong \Gr^1\) and \(\Act\times_{\Gr} \Act^*\cong \Base\).  That
  is, \(\Act^*\) is inverse to~\(\Act\).  If we can compose bibundle
  actors, we can compose them with bibundle equivalences as well,
  and the latter composition must be an equivalence.  Hence the
  existence of a composition of bibundle actors implies
  Proposition~\ref{pro:basic_assum}.5 and thus
  Assumption~\ref{assum:covering_acts_basically}.

  Now consider also a map \(f\colon \Base\to\Act[Y]\) such that
  \(f\circ\bunp\colon \Act\prto\Act[Y]\) is a cover.  Then~\(\Act\)
  with its usual left action of~\(\Gr\) and right action
  of~\(\Act[Y]\) by~\(f\circ\bunp\) is a covering bibundle functor
  from~\(\Gr\) to~\(\Act[Y]\), and hence a bibundle actor.  Composing
  it with the equivalence~\(\Act^*\) gives the bibundle functor
  \(\Base\to\Gr\to\Act[Y]\) from the map~\(f\).  This is a bibundle
  actor if and only if it is a covering bibundle functor if and only
  if~\(f\) is a cover.  Thus Assumption~\ref{assum:two-three} is
  necessary for composites of bibundle actors to exist, and also for
  composites of covering bibundle functors to remain covering bibundle
  functors.
\end{remark}

\subsection{Bibundle functors versus anafunctors}
\label{sec:bibundle_versus_vague_2-cat}

\begin{theorem}
  \label{the:bibundle_to_vague_functor}
  The construction of anafunctors from bibundle functors in
  Section~\textup{\ref{sec:bibundle_to_vague}} is part of an
  equivalence from the bicategory of bibundle functors to the
  bicategory of anafunctors; this equivalence is the identity on
  objects.
\end{theorem}

This theorem requires Assumptions \ref{assum:local_cover}
and~\ref{assum:covering_acts_basically_weak} in order for the
bicategory of bibundle functors to be defined.  Related results for
Lie groupoids and topological groupoids (with suitable covers) are
proved by Pronk~\cite{Pronk:Etendues_fractions} and
Carchedi~\cite{Carchedi:Thesis}.  Pronk develops a general calculus of
fractions in bicategories for such purposes.

\begin{proof}
  Let \(\Gr\) and~\(\Gr[H]\) be groupoids in~\((\Cat,\covers)\).  We
  show first that the groupoid of anafunctors from~\(\Gr\)
  to~\(\Gr[H]\) with natural transformations of anafunctors as
  arrows is equivalent to the groupoid of bibundle functors
  from~\(\Gr\) to~\(\Gr[H]\) with \(\Gr,\Gr[H]\)-maps between them
  as arrows.  Secondly, we check that this equivalence is compatible
  with the composition of arrows in the appropriate weak sense.

  The anafunctor associated to a bibundle functor~\(\Act\)
  from~\(\Gr\) to~\(\Gr[H]\) is \((\Act,\rg,F_{\Act})\), where
  \(F_{\Act}^0\defeq \s\colon \Act\to\Gr[H]^0\) on objects and
  \(F_{\Act}^1\colon \Act\times_{\rg,\Gr^0,\rg} \Gr^1
  \times_{\s,\Gr^0,\rg} \Act\to\Gr[H]^1\) is determined by the
  elementwise formula \(F_{\Act}^1(g\cdot x,g,x\cdot h) = h\) for
  all \(x\in\Act\), \(g\in\Gr^1\), \(h\in\Gr[H]^1\) with
  \(\s(g)=\rg(x)\), \(\s(x)=\rg(h)\) (see the proof of
  Proposition~\ref{pro:bibundle_to_vague}).

  Conversely, let \((X,p,F)\) be an anafunctor from~\(\Gr\)
  to~\(\Gr[H]\).  We have built bibundle functors from functors in
  Section~\ref{sec:functors_to_bibundles}.  In particular,
  \(p_*\colon \Gr(X)\to\Gr\) gives a bibundle functor
  \(X\times_{p,\Gr^0,\rg}\Gr^1\) from \(\Gr(X)\) to~\(\Gr\), and
  \(F\colon \Gr(X)\to\Gr[H]\) gives a bibundle functor
  \(X\times_{F^0,\Gr[H]^0,\rg} \Gr[H]^1\) from~\(\Gr(X)\)
  to~\(\Gr[H]\).  The bibundle functor
  \(X\times_{p,\Gr^0,\rg}\Gr^1\) from \(\Gr(X)\) to~\(\Gr\) is a
  bibundle equivalence
  (Example~\ref{exa:hypercover_equivalence_functor}); exchanging the
  left and right actions on
  \(X\times_{p,\Gr^0,\rg}\Gr^1\) gives a quasi-inverse bibundle
  functor \((X\times_{p,\Gr^0,\rg}\Gr^1)^*\) from~\(\Gr\)
  to~\(\Gr(X)\).  We map the anafunctor \((X,p,F)\) to the
  bibundle functor
  \[
  \beta(X,p,F)\defeq
  (X\times_{p,\Gr^0,\rg}\Gr^1)^* \times_{\Gr(X)}
  (X\times_{F^0,\Gr[H]^0,\rg} \Gr[H]^1).
  \]
  This is the orbit space of the right action of~\(\Gr(X)\) on the
  fibre product
  \[
  (X\times_{p,\Gr^0,\rg}\Gr^1)^* \times_{X}
  (X\times_{F^0,\Gr[H]^0,\rg} \Gr[H]^1)
  \cong \Gr^1\times_{\s,\Gr^0,p} X \times_{F^0,\Gr[H]^0,\rg} \Gr[H]^1
  \]
  with anchor map \(\s(g,x,h)\defeq x\) and
  \[
  (g_1,x_1,h)\cdot (x_1,g_2,x_2) \defeq (g_1\cdot g_2, x_2,
  F^1(x_1,g_2,x_2)^{-1}\cdot h)
  \]
  for all \(g_1,g_2\in\Gr^1\), \(x_1,x_2\in\Act\), \(h\in\Gr[H]^1\)
  with \(\s(g_1)=p(x_1)=\rg(g_2)\), \(F^0(x_1)=\rg(h)\),
  \(p(x_2)=\s(g_2)\); the isomorphism above involves the inversion
  in~\(\Gr^1\), which exchanges the source and range maps.  The
  \(\Gr,\Gr[H]\)-action on \(\beta(X,p,F)\)
  is induced by the obvious \(\Gr,\Gr[H]\)-action
  \[
  g_1\cdot (g_2,x,h_1)\cdot h_2 \defeq (g_1\cdot g_2,x,h_1\cdot h_2)
  \]
  on \(\Gr^1\times_{\s,\Gr^0,p} X \times_{F^0,\Gr[H]^0,\rg}
  \Gr[H]^1\).

  We claim that these two maps between bibundle functors and
  anafunctors are inverse to each other up to natural \(2\)\nb-arrows.
  First we start with a bibundle functor~\(\Act\), turn it into an
  anafunctor \((\Act,\rg,F_{\Act})\) and then turn that into a
  bibundle functor \((\Gr^1\times_{\s,\Gr^0,\rg} \Act
  \times_{\s,\Gr[H]^0,\rg} \Gr[H]^1)/\Gr(X)\).  We claim that the
  map
  \[
  \Gr^1\times_{\s,\Gr^0,\rg} \Act \times_{\s,\Gr[H]^0,\rg}
  \Gr[H]^1\to \Act,\qquad
  (g,x,h)\mapsto g\cdot x\cdot h,
  \]
  descends to an isomorphism \(\beta(\Act,\rg,F_{\Act})\congto\Act\);
  it is clear that this isomorphism is a \(\Gr,\Gr[H]\)-map.  We
  must show that \((g_1,x_1,h_1),(g_2,x_2,h_2)\in
  \Gr^1\times_{\s,\Gr^0,\rg} \Act \times_{\s,\Gr[H]^0,\rg}
  \Gr[H]^1\) satisfy \(g_1\cdot x_1\cdot h_1= g_2\cdot x_2\cdot
  h_2\) if and only if there is \((x_3,g_3,x_4)\in\Gr(X)^1\) with
  \(\s(g_1,x_1,h_1)=\rg(x_3,g_3,x_4)\) and \((g_1,x_1,h_1)\cdot
  (x_3,g_3,x_4) = (g_2,x_2,h_2)\).  First, \(\s(g_i,x_i,h_i) =
  x_i\), \(\rg(x_3,g_3,x_4)=x_3\) and \(\s((g_1,x_1,h_1)\cdot
  (x_3,g_3,x_4)) = \s(x_3,g_3,x_4)=x_4\), so we must have
  \(x_1=x_3\) and \(x_2=x_4\).  Moreover, the \(\Gr^1\)\nb-entry in
  \((g_1,x_1,h_1)\cdot (x_3,g_3,x_4)\) is \(g_1\cdot g_3\), so we
  must have \(g_3=g_1^{-1}\cdot g_2\).  Since
  \(\s(g_1^{-1})=\rg(g_1)=\rg(g_1\cdot x_1\cdot h_1)\) and
  \(\rg(g_2) = \rg(g_2\cdot x_2\cdot h_2)\), this product is
  well-defined if \(g_1\cdot x_1\cdot h_1 = g_2\cdot x_2\cdot h_2\).
  Furthermore, \(\rg(g_1^{-1}\cdot g_2)=\s(g_1)=\rg(x_1)\) and
  \(\s(g_1^{-1}\cdot g_2)=\s(g_2)=\rg(x_2)\), so that
  \((x_1,g_1^{-1}\cdot g_2,x_2)\in\Gr(X)^1\).  Unravelling the
  definition of~\(F_{\Act}\), we see that \((g_1,x_1,h_1)\cdot
  (x_1,g_1^{-1}\cdot g_2,x_2) = (g_2,x_2,h_2)\) if and only if
  \(g_1\cdot x_1\cdot h_1=g_2\cdot x_2\cdot h_2\).  This proves the
  natural isomorphism of bibundle functors \(\beta(\Act,\rg,F_{\Act})
  \cong \Act\).

  Next we start with an anafunctor \((X,p,F)\), take the
  associated bibundle functor
  \[
  \Act[Y]\defeq (\Gr^1\times_{\s,\Gr^0,p} X
  \times_{F^0,\Gr[H]^0,\rg} \Gr[H]^1)/\Gr(X),
  \]
  and construct an anafunctor from that.  This anafunctor
  contains the cover \(q\colon \Act[Y]\prto\Gr^0\) induced by the map
  \(\Gr^1\times_{\s,\Gr^0,p} X \times_{F^0,\Gr[H]^0,\rg} \Gr[H]^1\to
  \Gr^0\), \((g,x,h)\mapsto \rg(g)\), and a functor \(E\colon
  \Gr(\Act[Y])\to\Gr[H]\).  This functor is described elementwise by
  \(E^0([g,x,h]) = \s(h)\) for \(g\in\Gr^1\), \(h\in\Gr[H]^1\),
  \(x\in X\) with \(\s(g)=p(x)\), \(F^0(x)=\rg(h)\) and
  \[
  E^1([g_1,x_1,h_1],g,[g_2,x_2,h_2]) \defeq
  h_1^{-1}\cdot F^1(x_1,g_1^{-1}\cdot g\cdot g_2,x_2)\cdot h_2
  \]
  for \(g,g_1,g_2\in\Gr^1\), \(h_1,h_2\in\Gr[H]^1\), \(x_1,x_2\in
  X\) with \(\s(g_i)=p(x_i)\), \(F^0(x_i)=\rg(h_i)\) for \(i=1,2\)
  and \(\rg(g)=\rg(g_1)\), \(\s(g)=\rg(g_2)\); here \([g,x,h]\)
  stands for the image of \((g,x,h)\in \Gr^1\times_{\s,\Gr^0,p} X
  \times_{F^0,\Gr[H]^0,\rg} \Gr[H]^1\) under the quotient map
  to~\(\Act[Y]\).  The constructions above show that there is a
  unique functor~\(E\) with these properties.

  We describe a canonical natural transformation~\(\Psi\)
  from~\((\Act[Y],q,E)\) to~\((X,p,F)\); this is a map from
  \(X\times_{p,\Gr^0,q} \Act[Y]\) to~\(\Gr[H]^1\) with suitable
  properties.  Elementwise, \(\Psi\) is determined uniquely by
  \[
  \Psi(\tilde{x},[g,x,h]) = F^1(\tilde{x},g,x)\cdot h\colon
  \s(h) \xrightarrow{h} \rg(h) = F^0(x)
  \xrightarrow{F^1(\tilde{x},g,x)} F^0(\tilde{x})
  \]
  for \(\tilde{x},x\in X\), \(g\in\Gr^1\), \(h\in\Gr[H]^1\) with
  \(\s(g)=p(x)\), \(F^0(x)=\rg(h)\), \(p(\tilde{x})=\rg(g)\).  We
  must check that this is well-defined and that it gives a natural
  transformation.

  For well-definedness, we use that \(X \times_{p,\Gr^0,q}
  \Act[Y]\) is the orbit space of the \(\Gr(X)\)-action on
  \(X\times_{p,\Gr^0,\rg} \Gr^1 \times_{\s,\Gr^0,p} X
  \times_{F^0,\Gr[H]^0,\rg} \Gr[H]^1\) on the last three legs; this
  follows as in the construction of the left \(\Gr\)\nb-action on
  \(\Act[X]\times_{\Gr} \Act[Y]\) in the proof of
  Proposition~\ref{pro:bibundles_act}.  It is clear that~\(\Psi\) is
  well-defined as a map \(X\times_{p,\Gr^0,\rg} \Gr^1
  \times_{\s,\Gr^0,p} X \times_{F^0,\Gr[H]^0,\rg}
  \Gr[H]^1\to\Gr[H]^1\).  This map is \(\Gr(X)\)-invariant because
  \[
  F^1(\tilde{x},g_1\cdot g_2,x_2)\cdot F^1(x_1,g_2,x_2)^{-1}\cdot h
  = F^1(\tilde{x},g_1,x_1)\cdot h
  \]
  for \(\tilde{x},x_1,x_2\in X\), \(g_1,g_2\in\Gr^1\),
  \(h\in\Gr[H]^1\) with \(p(\tilde{x})=\rg(g_1)\),
  \(\s(g_1)=p(x_1)=\rg(g_2)\), \(F^0(x_1)=\rg(h)\),
  \(\s(g_2)=p(x_2)\), so that \((x_1,g_2,x_2)\in\Gr(X)\) and
  \((\tilde{x},g_1,x_1,h) \in X\times_{p,\Gr^0,\rg} \Gr^1
  \times_{\s,\Gr^0,p} X \times_{F^0,\Gr[H]^0,\rg} \Gr[H]^1\).

  Naturality of~\(\Psi\) follows because for
  \((\tilde{x}_i,g_i,x_i,h_i)\in X\times_{p,\Gr^0,\rg} \Gr^1
  \times_{\s,\Gr^0,p} X \times_{F^0,\Gr[H]^0,\rg} \Gr[H]^1\) for
  \(i=1,2\), \(g\in\Gr^1\), with \(\rg(g)=p(\tilde{x}_1)\),
  \(\s(g)=p(\tilde{x}_2)\), the following diagram in~\(\Gr[H]\)
  commutes:
  \[
  \begin{tikzpicture}
    \matrix[cd,column sep = 0pt] (m) {
      \s(h_1) &[1.5em] \rg(h_1) &[.5em] F^0(x_1) &[6.7em] F^0(\tilde{x}_1) \\
      \s(h_2) & \rg(h_2) & F^0(x_2) & F^0(\tilde{x}_2) \\
    };
    \begin{scope}[cdar]
      \draw (m-1-1) -- node {\(h_1\)} (m-1-2);
      \draw (m-2-1) -- node[swap] {\(h_2\)} (m-2-2);
      \draw (m-1-3) -- node {\(F^1(\tilde{x}_1,g_1,x_1)\)} (m-1-4);
      \draw (m-2-3) -- node[swap] {\(F^1(\tilde{x}_2,g_2,x_2)\)} (m-2-4);
      \draw (m-2-1) -- node {\(h_1^{-1}F^1(x_1,g_1^{-1}gg_2,x_2)h_2\)}
      (m-1-1);
      \draw (m-2-3) -- node[swap] {\(F^1(x_1,g_1^{-1}gg_2,x_2)\)} (m-1-3);
      \draw (m-2-4) -- node[swap] {\(F^1(\tilde{x_1},g,\tilde{x}_2)\)}
      (m-1-4);
    \end{scope}
    \draw[equ] (m-1-2) -- (m-1-3);
    \draw[equ] (m-2-2) -- (m-2-3);
  \end{tikzpicture}
  \]
  the two rows give~\(\Psi\) and the left and right vertical maps
  give \(E^1\) and~\(F^1\) for an arrow in \(\Gr(X\times_{\Gr^0}
  \Act[Y])\).

  We have now constructed an equivalence between the groupoids of
  bibundle functors and anafunctors from \(\Gr\) to~\(\Gr[H]\).
  To check that we have an equivalence of bicategories, it
  remains to show that the map from anafunctors to bibundle
  functors is compatible with composition of anafunctors up to
  natural \(\Gr,\Gr[H]\)-maps.  The identity (ana)functor on a
  groupoid~\(\Gr\) is clearly mapped to the unit bibundle functor
  on~\(\Gr\).

  The first step is to show that the composition of bibundle
  functors gives the usual composition for ordinary functors.  Let
  \(F_2\colon \Gr\to\Gr[H]\) and \(F_1\colon \Gr[H]\to\Gr[K]\) be
  functors.  Let \(\beta(F_2)\) and~\(\beta(F_1)\) be the associated
  bibundle functors.  We claim that \(\beta(F_1\circ F_2)\) is
  naturally isomorphic to \(\beta(F_2)\times_{\Gr[H]} \beta(F_1)\) as
  a \(\Gr,\Gr[K]\)-action.  The underlying objects are
  \(\Gr^0\times_{F_1^0\circ F_2^0,\Gr[K]^0,\rg} \Gr[K]^1\) for
  \(\beta(F_1\circ F_2)\) and
  \[
  (\Gr^0\times_{F_2^0,\Gr[H]^0,\rg} \Gr[H]^1)
  \times_{\Gr[H]} (\Gr[H]^0\times_{F_1^0,\Gr[K]^0,\rg} \Gr[K]^1)
  \]
  for \(\beta(F_2)\times_{\Gr[H]} \beta(F_1)\); the latter is the
  orbit space of a canonical \(\Gr[H]\)-action on
  \[
  (\Gr^0\times_{F_2^0,\Gr[H]^0,\rg} \Gr[H]^1)
  \times_{\Gr[H]^0} (\Gr[H]^0\times_{F_1^0,\Gr[K]^0,\rg} \Gr[K]^1)
  \cong \Gr^0\times_{F_2^0,\Gr[H]^0,\rg} \Gr[H]^1
  \times_{F_1^0\circ\s,\Gr[K]^0,\rg} \Gr[K]^1.
  \]
  The isomorphism is induced by the map
  \[
  \alpha\colon \Gr^0\times_{F_2^0,\Gr[H]^0,\rg} \Gr[H]^1
  \times_{F_1^0\circ\s,\Gr[K]^0,\rg} \Gr[K]^1\to
  \Gr^0\times_{F_1^0\circ F_2^0,\Gr[K]^0,\rg} \Gr[K]^1,\qquad
  (x,h,k)\mapsto (x,F_1^1(h)\cdot k),
  \]
  for \(x\in\Gr^0\), \(h\in\Gr[H]^1\), \(k\in\Gr[K]^1\) with
  \(F_2^0(x)= \rg(h)\), \(F_1^0(\s(h))=\s(F_1^1(h))=\rg(k)\).  Since
  \(\rg\colon \Gr[H]^1\prto\Gr[H]^0\) is a cover, so is the coordinate
  projection \(\Gr^0\times_{F_2^0,\Gr[H]^0,\rg}\Gr[H]^1 \prto \Gr^0\).
  The pull-back of this cover along the coordinate projection
  \(\Gr^0\times_{F_1^0\circ F_2^0,\Gr[K]^0,\rg} \Gr[K]^1\to\Gr^0\) is
  a cover as well, and this pull-back is exactly~\(\alpha\).
  Thus~\(\alpha\) is a cover.  We must also show that~\(\alpha\)
  induces an isomorphism
  \[
  (\Gr^0\times_{F_2^0,\Gr[H]^0,\rg} \Gr[H]^1)
  \times_{\Gr[H]} (\Gr[H]^0\times_{F_1^0,\Gr[K]^0,\rg} \Gr[K]^1)
  \to \Gr^0\times_{F_1^0\circ F_2^0,\Gr[K]^0,\rg} \Gr[K]^1.
  \]
  We show that~\(\alpha\) is the bundle projection of a principal
  \(\Gr[H]\)-bundle.  This means that if \((x_i,h_i,k_i)\in
  \Gr^0\times_{F_2^0,\Gr[H]^0,\rg} \Gr[H]^1
  \times_{F_1^0\circ\s,\Gr[K]^0,\rg} \Gr[K]^1\) for \(i=1,2\), then
  there is \(h\in\Gr[H]^1\) with
  \[
  (x_2,h_2,k_2)=
  (x_1,h_1,k_1)\cdot h \defeq
  (x_1,h_1\cdot h,F_1^1(h)^{-1}\cdot k_1)
  \]
  if and only if \(\alpha(x_1,h_1,k_1)=\alpha(x_2,h_2,k_2)\).
  Indeed, the only possible choice is \(h=h_1^{-1}\cdot h_2\), and
  this does the trick if and only if \(x_1=x_2\) and
  \(F_1^1(h_1)\cdot k_1= F_1^1(h_2)\cdot k_2\).  This provides an
  isomorphism
  \begin{equation}
    \label{eq:compose_functor_bibundle_functor}
    \beta(F_2)\times_{\Gr[H]} \beta(F_1) \cong \beta(F_1\circ F_2)
  \end{equation}
  for two functors \(F_1\) and~\(F_2\).

  Now consider anafunctors \((X,p,F)\) from~\(\Gr\) to~\(\Gr[H]\)
  and \((Y,q,E)\) from~\(\Gr[H]\) to~\(\Gr[K]\).  Their composition is
  the anafunctor \((X\times_{F^0,\Gr[H]^0,q} Y,p\circ\hat{q},
  E\circ\hat{F})\), where \(\hat{q}=\pr_1\colon
  X\times_{F^0,\Gr[H]^0,q} Y\to X\) and \(\hat{F}\colon
  \Gr(X\times_{F^0,\Gr[H]^0,q} Y) \to \Gr[H](Y)\) is induced by~\(F\).
  The bibundle functors associated to \((X,p,F)\) and \((Y,q,E)\) are
  \(\beta(F)\times_{\Gr(X)} \beta(p_*)^*\) and
  \(\beta(E)\times_{\Gr[H](Y)} \beta(q_*)^*\), respectively; here~\(^*\)
  means to exchange left and right actions for a bibundle equivalence.
  We have an equality of functors \(q_*\circ \hat{F} = F\circ
  \hat{q}_*\) by construction.  Hence the
  isomorphisms~\eqref{eq:compose_functor_bibundle_functor} above
  induce first an isomorphism
  \(\beta(\hat{q}_*)^{-1}\times_{\Gr(X\times_{\Gr[H]^0} Y)}
  \beta(\hat{F}) \cong \beta(F)\times_{\Gr[H]} \beta(q_*)^{-1}\) and then
  an isomorphism
  \begin{align*}
    \beta(X,p,F) \times_{\Gr[H]} \beta(Y,q,E)
    &\defeq
    \beta(p_*)^{-1} \times_{\Gr(X)} \beta(F)
    \times_{\Gr[H]} \beta(q_*)^{-1}
    \times_{\Gr[H](Y)} \beta(E)
    \\ &\cong
    \beta(p_*)^{-1} \times_{\Gr(X)} \beta(\hat{q}_*)^{-1}
    \times_{\Gr(X\times_{\Gr[H]^0} Y)} \beta(\hat{F})
    \times_{\Gr[H](Y)} \beta(E)
    \\ &\cong
    \beta((p\circ \hat{q})_*)^{-1}
    \times_{\Gr(X\times_{\Gr[H]^0} Y)} \beta(E\circ \hat{F})
    \\ &= \beta\bigl((Y,q,E)\circ (X,p,F)\bigr).
  \end{align*}
  Thus the compositions of anafunctors and bibundle functors agree
  up to the \(2\)\nb-arrows above.
  It remains to show that these isomorphisms of bibundle functors
  are natural with respect to natural transformations of
  anafunctors and that they satisfy the coherence conditions needed for
  a functor between bicategories
  (see~\cite{Leinster:Basic_Bicategories}).  All this is straightforward
  computation and left to the reader.
\end{proof}

\subsection{Decomposing bibundle actors}
\label{sec:decompose_bibundle_actor}

Let \(\Gr\) and~\(\Gr[H]\) be groupoids and let~\(\Act\) be a
bibundle actor from~\(\Gr\) to~\(\Gr[H]\).  We will
decompose~\(\Act\) into an ordinary actor from~\(\Gr\) to an
auxiliary groupoid~\(\Gr[K]\) and a bibundle equivalence
from~\(\Gr[K]\) to~\(\Gr[H]\):

\begin{proposition}
  \label{pro:decompose_bibundle_actor}
  Assume Assumptions \textup{\ref{assum:local_cover}}
  and~\textup{\ref{assum:covering_acts_basically_weak}}.  Any
  bibundle actor is a composite of an actor and a bibundle
  equivalence.
\end{proposition}

The groupoid~\(\Gr[K]\) is defined using only the right
\(\Gr[H]\)\nb-action on~\(\Act\), in such a way that~\(\Act\) is a
bibundle equivalence from~\(\Gr[K]\) to~\(\Gr\).  The construction
goes back to Ehresmann; for locally compact groupoids, this
construction is contained
in~\cite{Muhly-Renault-Williams:Equivalence}*{p.~5}.

\begin{proof}
  The fibre product \(\Act\times_{\s,\Gr[H]^0,\s}\Act\) exists and the
  coordinate projections
  \[
  \pr_1,\pr_2\colon
  \Act\times_{\s,\Gr[H]^0,\s}\Act\rightrightarrows\Act
  \]
  are covers because \(\s\colon \Act\prto\Gr[H]^0\) is a cover by
  assumption.  We let \(\Gr[K]^0\defeq \Act/\Gr[H]\) and
  \(\Gr[K]^1\defeq (\Act\times_{\s,\Gr[H]^0,\s}\Act)/\Gr[H]\),
  where~\(\Gr[H]\) acts diagonally on
  \(\Act\times_{\s,\Gr[H]^0,\s}\Act\), that is, \((x_1,x_2)\cdot
  h\defeq (x_1\cdot h,x_2\cdot h)\) if \(x_1,x_2\in\Act\),
  \(h\in\Gr[H]^1\) with \(\s(x_1)=\s(x_2)=\rg(h)\).
  Assumption~\ref{assum:covering_acts_basically_weak} ensures that
  this action is basic because \(\pr_1\) and~\(\pr_2\) are
  \(\Gr[H]\)\nb-maps and the \(\Gr[H]\)\nb-action on~\(\Act\) is
  basic.  The range and source maps \(\rg,\s\colon
  \Gr[K]^1\rightrightarrows\Gr[K]^0\) are induced by the coordinate
  projections \(\pr_1\) and~\(\pr_2\) above and are covers by
  Proposition~\ref{pro:G-map_versus_base_map}; here we need
  Assumption~\ref{assum:local_cover}.  Using
  Proposition~\ref{pro:fibre-product_base_above}, we may identify
  \begin{align*}
    \Gr[K]^1\times_{\s,\Gr[K]^0,\rg} \Gr[K]^1
    &\cong ((\Act\times_{\s,\Gr[H]^0,\s}\Act)\times_{\Act}
    (\Act\times_{\s,\Gr[H]^0,\s}\Act))/\Gr[H]\\
    &\cong (\Act\times_{\s,\Gr[H]^0,\s}\Act\times_{\s,\Gr[H]^0,\s}\Act)/\Gr[H].
  \end{align*}
  The multiplication map of~\(\Gr[K]^1\) is the map on the base
  induced by the \(\Gr[H]\)\nb-map \((\pr_1,\pr_3)\colon
  \Act\times_{\s,\Gr[H]^0,\s}\Act\times_{\s,\Gr[H]^0,\s}\Act \to
  \Act\times_{\s,\Gr[H]^0,\s}\Act\).
  Proposition~\ref{pro:fibre-product_base_above} also implies the
  isomorphisms \eqref{eq:groupoid_basicality_1}
  and~\eqref{eq:groupoid_basicality_2} for~\(\Gr[K]\), so
  that~\(\Gr[K]\) is a groupoid in~\((\Cat,\covers)\).

  We are going to construct a canonical actor from~\(\Gr\)
  to~\(\Gr[K]\).  Let the groupoid~\(\Gr\) act on
  \(\Act\times_{\s,\Gr[H]^0,\s}\Act\) by \(g\cdot (x_1,x_2)\defeq
  (g\cdot x_1,x_2)\) for all \(g\in\Gr^1\), \(x_1,x_2\in\Act\) with
  \(\s(g)=\rg(x_1)\) and \(\s(x_1)=\s(x_2)\); this is well-defined
  because \(\s(g\cdot x_1)=\s(x_1)\).  This \(\Gr\)\nb-action commutes
  with the diagonal \(\Gr[H]\)\nb-action and hence descends to a
  \(\Gr\)\nb-action on the \(\Gr[H]\)\nb-orbit space~\(\Gr[K]^1\) by
  Proposition~\ref{pro:invariants_as_composite}.  The resulting
  \(\Gr\)\nb-action on~\(\Gr[K]^1\) commutes with the right
  multiplication action because \(g\cdot [x_1,x_3]=[g\cdot x_1,x_3] =
  (g\cdot [x_1,x_2])\cdot [x_2,x_3]\) for all \(g\in\Gr^1\),
  \(x_1,x_2,x_3\in\Act\) with \(\s(g)=\rg(x_1)\),
  \(\s(x_1)=\s(x_2)=\s(x_3)\).

  We are going to construct a left \(\Gr[K]\)\nb-action on~\(\Act\)
  such that~\(\Act\) becomes a bibundle equivalence from~\(\Gr[K]\)
  to~\(\Gr[H]\).  Its anchor map is the bundle projection
  \(\bunp\colon \Act\prto\Act/\Gr[H]\).  The action
  \((\Act\times_{\s,\Gr[H]^0,\s}\Act)/\Gr[H]
  \times_{\bunp,\Act/\Gr[H],\bunp} \Act \to \Act\) is induced by the
  map \(\Act\times_{\s,\Gr[H]^0,\s}\Act
  \times_{\bunp,\Act/\Gr[H],\bunp} \Act \to \Act\) that is given
  elementwise by \((x_1,x_2)\cdot (x_2\cdot h)\defeq x_1\cdot h\) for
  all \(x_1,x_2\in\Act\), \(h\in\Gr[H]^1\) with
  \(\s(x_1)=\s(x_2)=\rg(h)\); this defines a map on
  \(\Act\times_{\s,\Gr[H]^0,\s}\Act \times_{\bunp,\Act/\Gr[H],\bunp}
  \Act\) because there is \(h\in\Gr[H]^1\) with \(x_3=x_2\cdot h\)
  whenever \(x_2,x_3\in\Act\) satisfy \(\bunp(x_2)=\bunp(x_3)\).  We
  have \((\Act\times_{\s,\Gr[H]^0,\s}\Act)/\Gr[H]
  \times_{\bunp,\Act/\Gr[H],\bunp} \Act \cong
  (\Act\times_{\s,\Gr[H]^0,\s}\Act \times_{\bunp,\Act/\Gr[H],\bunp}
  \Act)/\Gr[H]\), where \(\Gr[H]\) acts on
  \(\Act\times_{\s,\Gr[H]^0,\s}\Act \times_{\bunp,\Act/\Gr[H],\bunp}
  \Act\) by \((x_1,x_2,x_3)\cdot h\defeq (x_1\cdot h,x_2\cdot
  h,x_3)\).  The multiplication map defined above is
  \(\Gr[H]\)\nb-invariant for this action, hence it descends to a map
  \(\Gr[K]^1\times_{\s,\Act/\Gr[H],\bunp}\Act\to\Act\).

  Let \(\Gr[H]\)~act on \(\Gr[K]^1\times_{\s,\Gr[K]^0,\bunp} \Act\) by
  \((k,x)\cdot h\defeq (k,x\cdot h)\).  This action is basic with
  \[
  (\Gr[K]^1\times_{\s,\Gr[K]^0,\bunp} \Act)/\Gr[H]
  \cong \Gr[K]^1\times_{\s,\Gr[K]^0,\bunp} (\Act/\Gr[H])
  \cong \Gr[K]^1,
  \]
  by Remark~\ref{rem:induced_orbit_space}.  Thus the
  \(\Gr[H]\)\nb-map
  \begin{equation}
    \label{eq:K-action_on_X}
    \Gr[K]^1\times_{\s,\Gr[K]^0,\bunp} \Act \to \Act\times_{\s,\Gr[H]^0,\s}\Act,
    \qquad (k,x)\mapsto (k\cdot x,x),
  \end{equation}
  induces an isomorphism \((\Gr[K]^1\times_{\s,\Gr[K]^0,\bunp}
  \Act)/\Gr[H] \cong (\Act\times_{\s,\Gr[H]^0,\s}\Act)/\Gr[H]\) on the
  bases of principal \(\Gr[H]\)\nb-bundles.  Hence the map
  in~\eqref{eq:K-action_on_X} is an isomorphism as well by
  Proposition~\ref{pro:G-map_versus_base_map}.  Thus \(\s\colon
  \Act\prto\Gr[H]^0\) and the \(\Gr[K]\)\nb-action on~\(\Act\) form a
  principal \(\Gr[K]\)\nb-bundle.  Thus we have turned~\(\Act\) into a
  bibundle equivalence from~\(\Gr[K]\) to~\(\Gr[H]\).

  The actor from~\(\Gr\) to~\(\Gr[K]\) is also a bibundle actor.  Its
  composite with the bibundle equivalence~\(\Act\) from~\(\Gr[K]\)
  to~\(\Gr[H]\) is~\(\Act\) as a right \(\Gr[H]\)\nb-action because
  \(\Gr[K]^1\times_{\Gr[K]} \Act \cong \Act\) for any bibundle
  actor~\(\Act\) from~\(\Gr[K]\) to~\(\Gr[H]\).  The induced action
  of~\(\Gr\) on the composite is the given one because \(g\cdot
  [x,x]\cdot x= [g\cdot x,x]\cdot x= g\cdot x\) for all \(g\in\Gr^1\),
  \(x\in\Act\), where we interpret \([x,x]\in\Gr[K]^1\).  This shows
  that the actor and the bibundle equivalence constructed above
  compose to the given bibundle actor~\(\Act\).
\end{proof}

We conclude that the category of bibundle actors is the smallest one
that contains both bibundle equivalences and actors and where the
composition is given by~\(\times_{\Gr[H]}\).

Now let~\(\Act\) be a covering bibundle functor.  Equivalently,
\(\Act\) is a bibundle actor and the range map induces an
isomorphism \(\Act/\Gr[H]\congto \Gr^0\).  In the above
construction, this means that \(\Gr[K]^0=\Gr^0\), so that the actor
from~\(\Gr\) to~\(\Gr[K]\) acts identically on objects.  Such actors
are exactly the ones that are also functors, by
Proposition~\ref{pro:actor_as_functor}.

\subsection{The symmetric imprimitivity theorem}
\label{sec:symmetric_imprimitivity}

\begin{theorem}
  \label{the:symmetric_imprimitivity}
  Let \(\Gr\) and~\(\Gr[H]\) be groupoids and~\(\Act\) a
  \(\Gr,\Gr[H]\)-bibundle with basic actions of \(\Gr\)
  and~\(\Gr[H]\).  This induces a left action of \(\Gr\)
  on~\(\Act/\Gr[H]\) and a right action of~\(\Gr[H]\)
  on~\(\Gr\backslash\Act\), and~\(\Act\) becomes a bibundle
  equivalence from \(\Gr\ltimes(\Act/\Gr[H])\) to
  \((\Gr\backslash\Act)\rtimes\Gr[H]\).
\end{theorem}

\begin{proof}
  The induced \(\Gr\)\nb-action on~\(\Act/\Gr[H]\) exists by
  Proposition~\ref{pro:invariants_as_composite}.  The bundle
  projection \(\s'\colon \Act\prto\Act/\Gr[H]\) is a \(\Gr\)\nb-map.
  Hence the \(\Gr\)\nb-action on~\(\Act\) and~\(\s'\) combine to a
  left action of \(\Gr\ltimes(\Act/\Gr[H])\) on~\(\Act\) with anchor
  map~\(\s'\) by
  Proposition~\ref{pro:transformation_groupoid_action}.  This left
  action is basic with the same orbit space~\(\Gr\backslash\Act\) by
  Corollary~\ref{cor:basic_action_trafo_gr} and
  Lemma~\ref{lem:orbit_space_trafo_gr} because the \(\Gr\)\nb-action
  on~\(\Act\) is basic.  Exchanging left and right, the same
  arguments give an \(\Gr[H]\)\nb-action
  on~\(\Gr\backslash\Act\) and a basic right action of
  \((\Gr\backslash\Act)\rtimes\Gr[H]\) on~\(\Act\) with anchor map
  \(\rg'\colon \Act\prto\Gr\backslash\Act\) and orbit
  space~\(\Act/\Gr[H]\).  The actions of \(\Gr\ltimes(\Act/\Gr[H])\)
  and \((\Gr\backslash\Act)\rtimes\Gr[H]\) on~\(\Act\) commute
  because the actions of \(\Gr\) and~\(\Gr[H]\) commute, \(\s'\) is
  \(\Gr[H]\)\nb-invariant, and~\(\rg'\) is \(\Gr\)\nb-invariant.
  Hence they form a bibundle equivalence.
\end{proof}

This theorem generalises the symmetric imprimitivity theorem for
actions of locally compact groups by Green and
Rieffel~\cite{Rieffel:Applications_Morita} and provides many
important examples of bibundle equivalences between transformation
groupoids.

\begin{example}
  \label{exa:imprimitivity_groups}
  Assume Assumption~\ref{assum:final} and let~\(\Gr\) be a group
  in~\(\Cat\).  Let \(\Gr[H]\hookrightarrow\Gr\)
  and~\(\Gr[K]\hookrightarrow\Gr\) be ``closed subgroups''
  of~\(\Gr\); we mean by this that the restrictions of the
  multiplication actions on~\(\Gr\) to \(\Gr[H]\) and~\(\Gr[K]\) are
  basic.  Since the left and right multiplication actions commute,
  \(\Gr\) is an \(\Gr[H],\Gr[K]\)-bibundle.  We get an induced
  bibundle equivalence \(\Gr[H]\ltimes (\Gr/\Gr[K]) \sim
  (\Gr[H]\backslash\Gr) \rtimes \Gr[K]\).
\end{example}

\subsection{Characterising composites of bibundle functors}
\label{sec:characterise_composite}

\begin{proposition}
  \label{pro:characterise_composition}
  Let \(\Gr\), \(\Gr[H]\), \(\Gr[K]\) be groupoids in~\(\Cat\).  Let
  \(\Act\colon \Gr\to\Gr[H]\), \(\Act[Y]\colon \Gr[H]\to\Gr[K]\), and
  \(\Act[W]\colon \Gr\to\Gr[K]\) be bibundle functors.  Then
  isomorphisms \(\Act\times_{\Gr[H]}\Act[Y] \congto\Act[W]\) are in
  canonical bijection with \(\Gr,\Gr[K]\)-maps \(\mul\colon
  \Act\times_{\s,\Gr[H]^0,\rg} \Act[Y] \to \Act[W]\) with
  \(\mul(x\cdot h,y)=\mul(x,h\cdot y)\) for all \(x\in\Act\), \(h\in
  \Gr[H]^1\), \(y\in\Act[Y]\) with \(\s(x)=\rg(h)\),
  \(\s(h)=\rg(y)\).  If~\(\mul\) is such a map, then~\(\mul\) is a
  cover and the following map is an isomorphism:
  \begin{equation}
    \label{eq:compose_bibundle_functor_basicality}
    (\pr_1,\mul)\colon \Act\times_{\s,\Gr[H]^0,\rg} \Act[Y]
    \to \Act\times_{\rg,\Gr^0,\rg} \Act[W],\qquad
    (x,y)\mapsto (x,\mul(x,y)).
  \end{equation}
\end{proposition}

\begin{proof}
  Any \(\Gr,\Gr[K]\)\nb-map \(\Act\times_{\Gr[H]}\Act[Y]\to\Act[W]\)
  is an isomorphism by Theorem~\ref{the:bibundle_two-categories}.
  Since \(\Act\times_{\Gr[H]}\Act[Y]\) is the orbit space of an
  \(\Gr[H]\)\nb-action on \(\Act\times_{\s,\Gr[H]^0,\rg}\Act[Y]\), a
  \(\Gr,\Gr[K]\)\nb-map \(\Act\times_{\Gr[H]}\Act[Y]\to\Act[W]\) is
  equivalent to an \(\Gr[H]\)\nb-invariant \(\Gr,\Gr[K]\)\nb-map
  \(\Act\times_{\s,\Gr[H]^0,\rg}\Act[Y]\to\Act[W]\).  It remains to
  show that~\eqref{eq:compose_bibundle_functor_basicality} is always
  an isomorphism.

  Since~\(\Act\) is a principal \(\Gr[H]\)\nb-bundle over~\(\Gr^0\),
  its pull-back \(\Act\times_{\rg,\Gr^0,\rg} \Act[W]\) along
  \(\rg\colon \Act[W]\prto\Gr^0\) is a principal
  \(\Gr[H]\)\nb-bundle over~\(\Act[W]\) with bundle
  projection~\(\pr_2\) by Proposition~\ref{pro:pull-back_principal};
  here~\(\Gr[H]\) acts on
  \(\Act\times_{\rg,\Gr^0,\rg} \Act[W]\) by \(\s(x,w)=\s(x)\) and
  \((x,w)\cdot h\defeq (x\cdot h,w)\) for all \(x\in\Act\),
  \(w\in\Act[W]\), \(h\in\Gr[H]^1\) with \(\rg(x)=\rg(w)\),
  \(\s(x)=\rg(h)\).  By definition of the composition,
  \(\Act\times_{\s,\Gr[H]^0,\rg}\Act[Y]\prto\Act\times_{\Gr[H]}\Act[Y]\)
  is a principal \(\Gr[H]\)\nb-bundle as well.  We compute that
  \((\pr_1,\mul)\) is an \(\Gr[H]\)\nb-map:
  \begin{multline*}
    (\pr_1,\mul)((x,y)\cdot h)
    = (\pr_1,\mul)(x\cdot h,h^{-1}\cdot y)
    \\= (x\cdot h,\mul(x\cdot h,h^{-1}\cdot y))
    = (x\cdot h,\mul(x,y))
    = (x,\mul(x,y))\cdot h
  \end{multline*}
  for all \(x\in\Act\), \(y\in\Act[Y]\), \(h\in\Gr[H]^1\) with
  \(\s(x)=\rg(y)=\rg(h)\).
  Proposition~\ref{pro:G-map_versus_base_map} shows
  that~\eqref{eq:compose_bibundle_functor_basicality} is an
  isomorphism if and only if the induced map
  \(\Act\times_{\Gr[H]}\Act[Y]\to\Act[W]\) is an isomorphism.

  Since
  \(\rg\colon \Act\prto\Gr^0\) is a cover, so is the induced map
  \(\pr_2\colon \Act\times_{\rg,\Gr^0,\rg}\Act[W]\prto\Act[W]\).
  Composing this with the
  isomorphism~\eqref{eq:compose_bibundle_functor_basicality} shows
  that~\(\mul\) is a cover.
\end{proof}

We have seen some special cases of the
isomorphism~\eqref{eq:compose_bibundle_functor_basicality} before.
First, the condition \eqref{eq:groupoid_basicality_2} in the
definition of a groupoid is equivalent to \(\Gr^1\times_{\Gr}
\Gr^1\cong \Gr^1\).  Secondly, the condition
\eqref{eq:principal_bundle} for a principal bundle is equivalent to
\(\Act\times_{\Gr} \Gr^1 \cong \Act\) for a principal
\(\Gr\)\nb-bundle \(\Act\prto\Base\) viewed as a bibundle functor
from~\(\Base\) to~\(\Gr\).  We will use this similarity to simplify
the description of the quasi-category of groupoids and bibundle
functors in Section~\ref{sec:quasi-category_bibundle_functors}.

Condition~\((3''')\) in Definition~\ref{def:action} is also similar
and corresponds to a variant of
Proposition~\ref{pro:characterise_composition} where we characterise
the composite of a covering bibundle functor with a groupoid action.
This variant is also true, and the isomorphism in~\((3''')\)
in Definition~\ref{def:action} is equivalent to \(\Gr^1\times_{\Gr}
\Act\cong\Act\) for any left \(\Gr\)\nb-action~\(\Act\).

\begin{remark}
  \label{rem:characterise_composition}
  The proof of Proposition~\ref{pro:characterise_composition} does
  not use Assumption~\ref{assum:covering_acts_basically}.  For any
  subcanonical pretopology~\((\Cat,\covers)\), if \(\Gr\),
  \(\Gr[H]\), \(\Gr[K]\) are groupoids, \(\Act\colon \Gr\to\Gr[H]\),
  \(\Act[Y]\colon \Gr[H]\to\Gr[K]\) and \(\Act[W]\colon
  \Gr\to\Gr[K]\) bibundle functors, and \(\mul\colon
  \Act\times_{\s,\Gr[H]^0,\rg}\Act[Y]\to\Act[W]\) an
  \(\Gr[H]\)\nb-invariant \(\Gr,\Gr[K]\)-map and a cover such
  that~\eqref{eq:compose_bibundle_functor_basicality} is an
  isomorphism, then \(\mul\) is the bundle projection of a principal
  \(\Gr[H]\)\nb-bundle, so that \(\Act\times_{\Gr[H]}\Act[Y]\)
  exists and is isomorphic to~\(\Act[W]\).
\end{remark}

\subsection{Locality of basic actions}
\label{sec:basic_local}

We now reformulate Assumption~\ref{assum:covering_acts_basically}
in a way similar to the locality of principal bundles in
Proposition~\ref{pro:principal_is_local}.

Let~\(\Gr\) be a groupoid, \(\Act\) a right \(\Gr\)\nb-action,
\(f\colon\tilde{\Base}\prto\Base\) a cover,
\(\Base,\tilde{\Base}\inOb\Cat\), and let \(\varphi\colon
\Act\to\Base\) be some \(\Gr\)\nb-invariant map.  Then we may pull
back~\(\Act\) along~\(f\) to a \(\Gr\)\nb-action~\(\tilde{\Act}\): let
\(\tilde{\Act}\defeq \tilde{\Base}\times_{f,\Base,\varphi} \Act\)
(this exists because~\(f\) is a cover) and define the
\(\Gr\)\nb-action on~\(\tilde{\Act}\) by \(\s(z,x)\defeq \s(x)\) and
\((z,x)\cdot g\defeq (z,x\cdot g)\) for all \(z\in \tilde{\Base}\),
\(x\in\Act\), \(g\in\Gr^1\) with \(f(z)=\varphi(x)\),
\(\s(x)=\rg(g)\).

\begin{proposition}
  \label{pro:basic_action_local}
  In the above situation, if~\(\Act\) is a basic \(\Gr\)\nb-action,
  then so is~\(\tilde{\Act}\).  Under
  Assumption~\textup{\ref{assum:two-three}},
  Assumption~\textup{\ref{assum:covering_acts_basically}} is
  equivalent to the following converse statement: in the above
  situation, \(\Act\)~is basic if~\(\tilde{\Act}\) is.
\end{proposition}

\begin{proof}
  If~\(\Act\) is a basic \(\Gr\)\nb-action, let~\(\Base_0\) be its
  base and \(\bunp\colon \Act\prto\Base_0\) its bundle projection.
  There is a unique map \(\varphi_0\colon \Base_0\to\Base\) with
  \(\varphi_0\circ\bunp=\varphi\) because~\(\varphi\) is
  \(\Gr\)\nb-invariant and~\(\bunp\) is the orbit space projection
  by Lemma~\ref{lem:principal_bundle}.  Let \(\tilde{\Base}_0\defeq
  \tilde{\Base}\times_{f,\Base,\varphi_0} \Base_0\).
  Then~\(\tilde{\Act}\) is naturally isomorphic to the pull-back
  \(\tilde{\Base}_0\times_{\pr_2,\Base_0,\bunp} \Act\).  The induced
  \(\Gr\)\nb-action on this together with the projection
  to~\(\tilde{\Base}_0\) is a principal bundle by
  Proposition~\ref{pro:pull-back_principal}.  Thus~\(\tilde{\Act}\)
  is basic if~\(\Act\) is.

  Conversely, assume that~\(\tilde{\Act}\) is a basic
  \(\Gr\)\nb-action.  First, we also assume that Assumptions
  \ref{assum:two-three} and~\ref{assum:covering_acts_basically} hold.
  Since~\(f\) is a cover, so is \(\pr_2\colon \tilde{\Act}\prto\Act\).
  The map~\(\pr_2\) is a \(\Gr\)\nb-map as well, so the
  \(\Gr\)\nb-action and~\(\pr_2\) give a right
  \(\Act\rtimes\Gr\)\nb-action on~\(\tilde{\Act}\) by
  Proposition~\ref{pro:transformation_groupoid_action}.  This action
  is basic by Corollary~\ref{cor:basic_action_trafo_gr} and because
  the \(\Gr\)\nb-action on~\(\tilde{\Act}\) is assumed basic.  The
  advantage of the right
  \(\Act\rtimes\Gr\)\nb-action over the right \(\Gr\)\nb-action is
  that its anchor map~\(\pr_2\) is a cover.

  Let~\(\Gr[H]\) be the \v{C}ech groupoid of~\(f\).  It acts
  on~\(\tilde{\Act}\) on the left with anchor map~\(\pr_1\) and action
  \((z_1,z_2)\cdot (z_2,x)\defeq (z_1,x)\) for all
  \(z_1,z_2\in\tilde{\Act[Z]}\), \(x\in\Act\) with
  \(f(z_1)=f(z_2)=\varphi(x)\).  This action commutes with the right
  \(\Act\rtimes\Gr\)\nb-action.  Thus~\(\tilde{\Act}\) is a bibundle
  actor from~\(\Gr[H]\) to~\(\Gr\).
  Proposition~\ref{pro:compose_bibundles} allows us to compose
  bibundle actors.  In particular, we may compose~\(\tilde{\Act}\)
  with the bibundle equivalence~\(\tilde{\Base}\) from~\(\Base\)
  to~\(\Gr[H]\) (see
  Example~\ref{exa:covering_groupoid_equivalent_to_space}).  The
  composite is \(\Gr[H]\backslash \tilde{\Act}\cong \Act\) because the
  pull-back of the principal \(\Gr[H]\)\nb-bundle
  \(\tilde{\Base}\prto\Base\) along~\(\varphi\) gives a principal
  \(\Gr[H]\)\nb-bundle \(\tilde{\Act}\prto\Act\).

  Since the isomorphism is the coordinate projection, the induced
  \(\Act\rtimes\Gr\)\nb-action on \(\tilde{\Act}\prto\Act\) is the one
  we started with.  Since this composite of bibundle actors is again a
  bibundle actor from~\(\Base\) to~\(\Act\rtimes\Gr\), we conclude
  that~\(\Act\) is a basic right \(\Act\rtimes\Gr\)\nb-action.  Hence
  it is a basic right \(\Gr\)\nb-action by
  Corollary~\ref{cor:basic_action_trafo_gr}.  Thus being basic is a
  local property of groupoid actions.

  Now assume that basic actions are local.  Let \(f\colon
  \tilde{\Base}\prto\Base\) be a cover and let~\(\Gr\) be its
  \v{C}ech groupoid.  Let~\(\Act\) be a \(\Gr\)\nb-action.  We want
  to show that~\(\Act\) is basic as well, by proving that it is
  locally basic.  Composing the anchor map \(\s\colon
  \Act\to\tilde{\Base}\) with~\(f\) gives a \(\Gr\)\nb-invariant map
  \(\varphi\defeq f\circ\s\colon \Act\to\Base\).  The pull-back of
  the \(\Gr\)\nb-action~\(\Act\) along~\(f\) gives
  \(\tilde{\Act}\defeq \Act\times_{\varphi,\Base,f} \tilde{\Base}\)
  with \(\Gr\)\nb-action by \((x,z)\cdot (z_1,z_2)=(x\cdot
  (z_1,z_2),z)\) for all \(x\in\Act\), \(z,z_1,z_2\in\tilde{\Base}\)
  with \(\varphi(x)=f(z)=f(z_1)=f(z_2)\), \(z_1=\s(x)\).

  The multiplication \(\mul\colon \Act\times_{\s,\Gr^0,\rg}\Gr^1
  \prto\Act\) is a cover by~\((3'')\) in Definition~\ref{def:action}.
  Identifying
  \[
  \Act\times_{\s,\Gr^0,\rg}\Gr^1
  \defeq \Act\times_{\s,\tilde{\Base},\pr_1}
  (\tilde{\Base}\times_{f,\Base,f} \tilde{\Base})
  \cong \Act\times_{f\circ\s,\Base,f} \tilde{\Base},
  \]
  this cover becomes \(q\colon \tilde{\Act}\prto\Act\), \((x,z)\mapsto
  x\cdot (\s(x),z)\).  We claim that~\(q\) is the bundle projection of
  a principal \(\Gr\)\nb-bundle.  If
  \((x_1,z_1),(x_2,z_2)\in\tilde{\Act}\) with \(x_1\cdot (\s(x_1),z_1)
  = x_2\cdot (\s(x_2),z_2)\), then \(f(\s(x_1))= f(z_1)\),
  \(f(\s(x_2)) = f(z_2)\), and \(z_1= \s(x_1\cdot (\s(x_1),z_1)) =
  \s(x_2\cdot (\s(x_2),z_2)) = z_2\).  Thus
  \begin{multline*}
    (x_1,z_1)\cdot (\s(x_1),\s(x_2))
    = (x_1,z_1)\cdot (\s(x_1),z_1)\cdot (z_2,\s(x_2))
    \\= (x_2,z_2)\cdot (\s(x_2),z_2)\cdot (z_2,\s(x_2))
    = (x_2,z_2).
  \end{multline*}
  Furthermore, if \((x_1,z_1)\cdot (z_3,z_4)= (x_2,z_2)\), then
  \(z_3=\s(x_1)\) and \(z_4=\s(x_2)\), so \((\s(x_1),\s(x_2))\) is the
  unique element of~\(\Gr^1\) with this property.  Thus the
  map~\eqref{eq:principal_bundle} is an isomorphism and
  \(\tilde{\Act}\prto\Act\) is a principal \(\Gr\)\nb-bundle.
  Since~\(\tilde{\Act}\) is the pull-back of the
  \(\Gr\)\nb-action~\(\Act\) along the cover~\(f\), any
  \(\Gr\)\nb-action is locally basic.  Therefore, if all locally basic
  actions are basic, then
  Assumption~\ref{assum:covering_acts_basically} follows.
\end{proof}

\subsection{Equivalences in bibundle functors and actors}
\label{sec:equivalences_in_bibundles}

\begin{proposition}
  \label{pro:bibundle_equivalence_invertible}
  Let\/ \(\Gr\) and\/~\(\Gr[H]\) be groupoids and let\/ \(\Act\colon
  \Gr\to\Gr[H]\) be a bibundle equivalence in\/~\((\Cat,\covers)\).
  Then there are canonical isomorphisms of bibundle equivalences\/
  \(\Act\times_{\Gr[H]}\Act^*\cong\Gr^1\) and/\
  \(\Act^*\times_{\Gr}\Act\cong\Gr[H]^1\).
\end{proposition}

\begin{proof}
  Since \(\rg\colon \Act\prto\Gr[H]^0\) is a left principal
  \(\Gr\)\nb-bundle, the map \(\Gr^1\times_{\s,\Gr^0,\rg} \Act \to
  \Act\times_{\s,\Gr[H]^0,\s}\Act\), \((g,x)\mapsto (x,g\cdot x)\), is
  an isomorphism.  The inverse is of the form \((x_1,x_2)\mapsto
  (x_1,\mul(x_1,x_2))\) for a map \(\mul\colon
  \Act\times_{\s,\Gr[H]^0,\s}\Act\to\Gr^1\).  Identifying
  \(\Act\times_{\s,\Gr[H]^0,\s}\Act\cong
  \Act\times_{\s,\Gr[H]^0,\rg}\Act^*\), the criterion of
  Proposition~\ref{pro:characterise_composition} applied to~\(\mul\)
  shows that \(\Act\times_{\Gr[H]} \Act^*\cong\Gr^1\).  Explicitly,
  the isomorphism is induced by~\(\mul\), so it maps the class of
  \((x_1,x_2)\in\Act\times_{\s,\Gr[H]^0,\rg}\Act^*\) to the unique
  \(g\in\Gr^1\) with \(g\cdot x_2=x_1\).

  The same reasoning for~\(\Act^*\) instead of~\(\Act\) gives
  \(\Act^*\times_{\Gr} \Act\cong\Gr[H]^1\).
\end{proof}

\begin{theorem}
  \label{the:equivalences_in_bibundle_functors}
  The equivalences in the bicategory of bibundle functors and
  bibundle actors are exactly the bibundle equivalences.  Furthermore,
  if \(\Act[Y]\colon \Gr[H]\to\Gr\) is quasi-inverse to \(\Act\colon
  \Gr\to\Gr[H]\), then \(\Act[Y]\cong\Act^*\) is obtained
  from~\(\Act\) by exchanging left and right actions.
\end{theorem}

We require Assumptions \ref{assum:local_cover}
and~\ref{assum:covering_acts_basically_weak} for the bibundle
functor case and Assumptions \ref{assum:two-three}
and~\ref{assum:covering_acts_basically} for the bibundle actor case
because otherwise the bicategories in question need not be
defined (Theorem~\ref{the:bibundle_two-categories}).

\begin{proof}
  The bicategories of bibundle functors and anafunctors are equivalent
  by Theorem~\ref{the:bibundle_to_vague_functor}.  By
  Theorem~\ref{the:explicit-vague-equivalence}, the equivalences in
  the bicategory of anafunctors are exactly those anafunctors that
  lift to ana-isomorphisms.  Since the equivalence to bibundle
  functors maps ana-isomorphisms to bibundle equivalences, the
  equivalences in the bicategory of bibundle functors are exactly the
  bibundle equivalences.

  Without Assumption~\ref{assum:local_cover} something goes wrong at
  this point because a fully faithful functor that is only almost
  essentially surjective but not essentially surjective is an
  equivalence in the bicategory of anafunctors, but the
  associated bibundle functor is not a bibundle equivalence by
  Proposition~\ref{pro:equivalence_functor}.

  It is clear that bibundle equivalences are equivalences in the
  bicategory of bibundle actors.  Conversely, let~\(\Act\) be
  an equivalence in this bicategory.
  Proposition~\ref{pro:decompose_bibundle_actor} decomposes~\(\Act\)
  as a composite of a bibundle equivalence and an actor.  Thus it
  suffices to show that an actor that is an equivalence in the
  bicategory of bibundle actors is already an isomorphism of
  categories.

  Let
  \(\mul\colon \Gr^1\times_{\s,\Gr^0,\rg_{\Gr}}\Gr[H]^1 \to\Gr[H]^1\)
  be an actor that is an equivalence in the bicategory of bibundle
  actors.  We write~\({}_{\mul}\Gr[H]^1\)
  for the bibundle actor described by~\(\mul\).
  Let~\(\Act[Y]\) be its quasi-inverse.  Thus
  \({}_{\mul}\Gr[H]^1\times_{\Gr[H]}\Act[Y]\cong \Gr^1\) as
  \(\Gr,\Gr\)\nb-bibundles.  Since there is a natural isomorphism of
  right \(\Gr\)\nb-actions
  \(\Gr[H]^1\times_{\Gr[H]}\Act[Y]\cong\Act[Y]\), we get
  \(\Act[Y]\cong\Gr^1\) as a right \(\Gr\)\nb-action.
  Thus~\(\Act[Y]\) is an actor.  When we view actors as bibundle
  actors, the \(2\)\nb-arrows are exactly the \(\Gr,\Gr[H]\)-maps
  \(\Gr[H]^1\to\Gr[H]^1\).  These are the same \(2\)\nb-arrows that we
  already used in the bicategory of actors.  Any equivalence in this
  bicategory is an isomorphism of categories by
  Proposition~\ref{pro:actor_2-category}.  Since the \(2\)\nb-arrows
  are the same in both bicategories, we conclude that our original
  actor is an isomorphism of categories.
\end{proof}

\section{The quasi-category of bibundle functors}
\label{sec:quasi-category_bibundle_functors}

Let~\(\Cat\) be a category with coproducts and with a
pretopology~\(\covers\) that satisfies Assumptions
\ref{assum:local_cover}
and~\ref{assum:covering_acts_basically_weak}.
Theorem~\ref{the:bibundle_two-categories} shows that groupoids
in~\((\Cat,\covers)\) with bibundle functors as arrows and
\(\Gr,\Gr[H]\)-maps as \(2\)\nb-arrows form a bicategory
with invertible \(2\)\nb-arrows, that is, a $(2, 1)$-category.  Its
nerve is a quasi-category (an $(\infty, 1)$-category) in
the sense of Boardman--Vogt~\cite{Boardman-Vogt:Homotopy_algebraic}
and André Joyal~\cite{Joyal:Quasi-categories}, that is, a
simplicial set satisfying all inner Kan conditions; furthermore,
coming from a bicategory, it satisfies inner unique Kan conditions
in dimensions above~\(3\).  This quasi-category contains essentially
the same information as the bicategory.

We are going to describe this quasi-category in a more elementary way,
without mentioning groupoids and bibundle functors.  The main point is
the similarity between the conditions \eqref{eq:principal_bundle},
\eqref{eq:groupoid_basicality_2}
and~\eqref{eq:compose_bibundle_functor_basicality} in the definitions
of groupoids, principal bundles, and the composition of bibundle
functors.  An \(n\)\nb-simplex in~\(\GrCat\) consists of
\begin{itemize}
\item \(X_i\inOb\Cat\) for \(0\le i\le n\);
\item \(X_{ij}\inOb\Cat\) for \(0\le i\le j\le n\);
\item \(\rg_{ij}\in\Cat(X_{ij},X_i)\) and \(\s_{ij}\in\Cat(X_{ij},
  X_j)\) for \(0\le i\le j\le n\);
\item \(\mul_{ijk}\in\Cat(X_{ij}\times_{\s_{ij},X_j,\rg_{jk}} X_{jk},
  X_{ik})\) for \(0\le i\le j\le k\le n\);
\end{itemize}
such that the following conditions hold:
\begin{enumerate}
\item \label{cond:qc_cover1} \(\rg_{ij}\) is a cover for all \(0\le
  i\le j\le n\); hence the
  domain~\(X_{ij}\times_{\s_{ij},X_j,\rg_{jk}} X_{jk}\)
  of~\(\mul_{ijk}\) is well-defined;
\item \(\s_{ii}\) is a cover for all \(0\le i\le n\);
\item \label{cond:qc_preassoc} \(\rg_{ik}\circ \mul_{ijk} =
  \rg_{ij}\circ \pr_1\) and \(\s_{ik}\circ \mul_{ijk} = \s_{jk}\circ
  \pr_2\) on \(X_{ij}\times_{X_j} X_{jk}\) for all \(0\le i\le j\le
  k\le n\); briefly, \(\rg(x\cdot y)=\rg(x)\) and \(\s(x\cdot
  y)=\s(y)\);
\item \label{cond:qc_assoc} (associativity) for \(0\le i\le j\le k\le l\le n\), the
  following diagram commutes:
  \[
  \begin{tikzpicture}
    \matrix[cd,column sep=7em] (m) {
      X_{ij}\times_{X_j} X_{jk}\times_{X_k} X_{kl} &
      X_{ij}\times_{X_j} X_{jl} \\
      X_{ik}\times_{X_k} X_{kl} & X_{il} \\
    };
    \begin{scope}[cdar]
      \draw (m-1-1) -- node {\(\id_{X_{ij}}\times_{X_j} \mul_{jkl}\)} (m-1-2);
      \draw (m-1-1) -- node[swap] {\(\mul_{ijk}\times_{X_k} \id_{X_{kl}}\)} (m-2-1);
      \draw (m-2-1) -- node {\(\mul_{ikl}\)} (m-2-2);
      \draw (m-1-2) -- node {\(\mul_{ijl}\)} (m-2-2);
    \end{scope}
  \end{tikzpicture}
  \]
  briefly, \((x\cdot y)\cdot z=x\cdot (y\cdot z)\);
\item \label{cond:qc_basic_1} the following maps are isomorphisms for
  all \(0\le i\le j\le k\le n\) with \(i=j\) or \(j=k\):
  \[
  (\pr_1,\mul_{ijk})\colon
  X_{ij}\times_{\s_{ij},X_j,\rg_{jk}} X_{jk} \to
  X_{ij}\times_{\rg_{ij},X_i,\rg_{ik}} X_{ik},\qquad
  (x,y)\mapsto (x,x\cdot y),
  \]
\item \label{cond:qc_basic_2} the following maps are isomorphisms if
  \(0\le i\le j= k\le n\):
  \[
  (\mul_{ijk},\pr_2)\colon
  X_{ij}\times_{\s_{ij},X_j,\rg_{jk}} X_{jk} \to
  X_{ik}\times_{\s_{ik},X_k,\s_{jk}} X_{jk},\qquad
  (x,y)\mapsto (x\cdot y,y).
  \]
\end{enumerate}
An order-preserving map \(\varphi\colon
\{0,\dotsc,n\}\to\{0,\dotsc,m\}\) induces \(\varphi^*\colon
\GrCat_m\to\GrCat_n\): take
\(X_{\varphi(i)}\), \(X_{\varphi(i)\varphi(j)}\),
\(\rg_{\varphi(i)\varphi(j)}\), \(\s_{\varphi(i)\varphi(j)}\),
\(\mul_{\varphi(i)\varphi(j)\varphi(k)}\).  Thus~\(\GrCat_\bullet\)
is a simplicial set.

Why is this the nerve of the bicategory of bibundle functors?

To begin with, for each \(0\le i\le n\), there is a groupoid~\(\Gr_i\)
with objects~\(X_i\) and arrows~\(X_{ii}\), range~\(\rg_{ii}\),
source~\(\s_{ii}\), and multiplication~\(\mul_{iii}\).  The
conditions above for equal indices amount to the conditions of
Definition~\ref{def:groupoid}.

Next, \(X_{il}\) with left anchor map \(\rg_{il}\colon X_{il}\to
X_i\), right anchor map \(\s_{il}\colon X_{il}\to X_l\), left
action~\(\mul_{iil}\) and right action~\(\mul_{ill}\) is a
\(\Gr_i,\Gr_l\)\nb-bibundle; the associativity
condition~\ref{cond:qc_assoc} for \(i=j=k\le l\) says that the left
action is associative, the one for \(i\le j=k=l\) says that the right
action is associative, the one for \(i=j\le k=l\) says that both
actions commute; \ref{cond:qc_basic_1} for \(i=j\le k\) says that the
left \(\Gr_i\)\nb-action satisfies~\((3''')\) in
Definition~\ref{def:action}, and~\ref{cond:qc_basic_2} for \(i\le
j=k\) says that the right \(\Gr_j\)\nb-action satisfies
condition~\((3''')\) in Definition~\ref{def:action}.

The conditions \ref{cond:qc_cover1} for \(i\le k\) and
\ref{cond:qc_basic_1} for \(i\le j=k\) say that the right
\(\Gr_k\)\nb-action on~\(X_{ik}\) with bundle projection
\(\rg_{ik}\colon X_{ik}\prto X_i\) is a principal bundle;
thus~\(X_{ik}\) is a bibundle functor from~\(\Gr_i\) to~\(\Gr_k\).  By
construction, \(X_{ii}\) is the identity bibundle functor
on~\(\Gr_i\).

The associativity conditions~\ref{cond:qc_assoc} for \(0\le i = j\le
k\le l\le n\), \(0\le i\le j=k\le l\le n\), and \(0\le i\le j\le
k=l\le n\) and~\ref{cond:qc_preassoc} say that the maps
\(m_{ijk}\colon X_{ij}\times_{X_j} X_{jk}\to X_{ik}\) are
\(\Gr_j\)\nb-invariant \(\Gr_i,\Gr_k\)\nb-maps.  Hence they induce
isomorphisms of \(\Gr_i,\Gr_k\)\nb-bibundles
\[
\dot{\mul}_{ijk}\colon X_{ij}\times_{\Gr_j} X_{jk} \congto X_{ik}
\]
by Proposition~\ref{pro:characterise_composition}.
Proposition~\ref{pro:characterise_composition} also shows that the
maps~\(\mul_{ijk}\) are covers for all \(0\le i\le j\le k\le n\) and
that the map in~\ref{cond:qc_basic_1} is an isomorphism also for
\(i<j<k\).

By construction, \(\dot{\mul}_{iik}\colon \Gr_i^1\times_{\Gr_i}
X_{ik}\to X_{ik}\) and \(\dot{\mul}_{ikk}\colon X_{ik}\times_{\Gr_k}
\Gr_k^1\to X_{ik}\) are the canonical isomorphisms.  Finally, for \(0\le i\le
j\le k\le l\le n\), the associativity condition~\ref{cond:qc_assoc}
is equivalent to the following commutative diagram of isomorphisms
of bibundle functors:
\[
\begin{tikzpicture}
  \matrix[cd, column sep=7em] (m) {
    X_{ij} \times_{\Gr_j} X_{jk} \times_{\Gr_k} X_{kl} &
    X_{ij} \times_{\Gr_j} X_{jl} \\
    X_{ij} \times_{\Gr_j} X_{jl} &
    X_{il} \\
  };
  \begin{scope}[cdar]
    \draw (m-1-1) -- node {\(\id_{X_{ij}}\times_{\Gr_j} \dot{\mul}_{jkl}\)} (m-1-2);
    \draw (m-1-1) -- node[swap] {\(\dot{\mul}_{ijk}\times_{\Gr_k} \id_{X_{kl}}\)} (m-2-1);
    \draw (m-2-1) -- node {\(\dot{\mul}_{ikl}\)} (m-2-2);
    \draw (m-1-2) -- node {\(\dot{\mul}_{ijl}\)} (m-2-2);
  \end{scope}
\end{tikzpicture}
\]

Rewriting our \(n\)\nb-simplices in terms of the groupoids~\(\Gr_i\),
bibundle functors \(X_{ij}\colon \Gr_i\to\Gr_j\) and isomorphisms of
bibundle functors \(\dot{\mul}_{ijk}\colon
X_{ij}\times_{\Gr_j}X_{jk}\to X_{ik}\) gives us precisely the
definition of the nerve of the bicategory of bibundle functors.  It is
well-known that the nerve of a bicategory satisfies the Kan condition
\(\Kan(2,1)\) and the unique Kan conditions \(\Kan!(n,j)\) for \(n\ge
3\) and \(1\le j\le n-1\) (and also for \(n\ge 5\) and \(j=0\) or
\(j=n\)).  Of course, this only holds under Assumptions
\ref{assum:local_cover} and~\ref{assum:covering_acts_basically}
because otherwise we do not have a bicategory.

The quasi-category associated to the sub-bicategory of covering
bibundle functors is defined similarly; in addition, we require the
maps~\(\s_{ij}\) to be covers for all \(0\le i\le j\le n\).

The quasi-category associated to the sub-bicategory of bibundle
equivalences is defined similarly, with the following extra
conditions:
\begin{itemize}
\item the maps~\(\s_{ij}\) should be covers for all \(0\le i\le j\le
  n\);
\item the maps in \ref{cond:qc_basic_1} and~\ref{cond:qc_basic_2}
  should be isomorphisms if \(0\le i\le j\le k\le n\) and \(i=j\) or
  \(j=k\).
\end{itemize}
With this symmetric form of the above axioms, we may reverse the
order of \(0,\dotsc,n\) and thus show that \(\s_{ij}\colon
X_{ij}\prto X_j\) is a principal \(\Gr_i\)\nb-bundle.

Proposition~\ref{pro:characterise_composition} shows that the maps in
\ref{cond:qc_basic_1} and~\ref{cond:qc_basic_2} are automatically
isomorphisms if \(0\le i<j<k\le n\) if this is assumed for \(i=j\) or
\(j=k\).  So for the quasi-category of bibundle equivalences, we may
require \ref{cond:qc_basic_1} and~\ref{cond:qc_basic_2} for all \(0\le
i\le j\le k\le n\).  This stronger condition is less technical and
provides a very satisfactory description of the quasi-category of
groupoids in~\((\Cat,\covers)\) with bibundle equivalences.  Since
bibundle equivalences are invertible up to \(2\)\nb-arrows, this is a
weak \(2\)\nb-groupoid.  Thus the corresponding quasi-category also
satisfies the outer Kan conditions \(\Kan(2,j)\) for \(j=0,2\) and
\(\Kan!(n,j)\) for \(n\ge 3\), \(j=0,n\).

\section{Examples of categories with pretopology}
\label{sec:examples_covers}

In this section, we discuss pretopologies on different
categories and check whether they satisfy our extra assumptions.  We
also describe basic actions in each case.  We begin with a
trivial pretopology on an arbitrary category.

\begin{example}
  \label{exa:trivial_pretopology}
  Let~\(\Cat\) be any category and let~\(\covers\) be the class of
  all isomorphisms in~\(\Cat\).  This is a subcanonical pretopology
  satisfying Assumptions \ref{assum:local_cover}
  and~\ref{assum:two-three}.  All groupoids in \((\Cat,\covers)\)
  are \(0\)\nb-groupoids by Example~\ref{exa:0-groupoid}, so there
  are no interesting examples of groupoids.
  Assumptions \ref{assum:covering_acts_basically}
  and~\ref{assum:covering_acts_basically_weak} are trivial.
  Assumption~\ref{assum:final} on final objects usually fails.
\end{example}

In this section, unlike the previous ones, \emph{elementwise
  statements are meant naively, with elements of sets in the usual
  sense}.  Since all our examples are concrete categories, these naive
elements suffice.

\subsection{Sets and surjections}
\label{sec:Sets_surjections}

Let~\(\Sets\) be the category of sets.

\begin{proposition}
  \label{pro:Sets_surjections}
  The class~\(\covers_\surj\) of surjective maps in~\(\Sets\) is a
  subcanonical, saturated pretopology satisfying Assumptions
  \textup{\ref{assum:local_cover}}, \textup{\ref{assum:two-three}},
  \textup{\ref{assum:covering_acts_basically}},
  and~\textup{\ref{assum:covering_acts_basically_weak}}.  The subcategory
  of non-empty sets satisfies the same assumptions and also
  Assumption~\textup{\ref{assum:final}}.
\end{proposition}

Thus all the theory developed above applies to groupoids
in~\((\Sets,\covers_\surj)\).  We are going to prove this and
characterise basic actions in~\((\Sets,\covers_\surj)\).

The category~\(\Sets\) is complete and therefore closed under fibre
products.  In the fibre-product
situation~\eqref{eq:fibre-product_of_cover}, if \(g\colon U\prto X\)
is surjective, then so is the induced map \(\pr_1\colon
Y\times_{f,X,g} U\prto Y\).  The isomorphisms in~\(\Sets\) are the
bijections and belong to~\(\covers_\surj\).  Let \(f_1\colon X\to
Y\) and \(f_2\colon Y\to Z\) be composable maps.  If \(f_1\)
and~\(f_2\) are surjective, then so is \(f_2\circ f_1\).  Thus
surjective maps in~\(\Sets\) form a pretopology.

If \(f_2\circ f_1\) is surjective, then so is~\(f_2\).  Hence the
pretopology~\(\covers_\surj\) is saturated, and Assumptions
\ref{assum:two-three} and~\ref{assum:local_cover} follow.  A map
in~\(\Sets\) is a coequaliser if and only if it is surjective.  Thus
the pretopology~\(\covers_\surj\) is subcanonical, and it is the
largest subcanonical pretopology on~\(\Sets\).  Any set with a single
element is a final object, and any map to it from a non-empty set is
surjective, hence a cover.  Hence Assumption~\ref{assum:final} holds
for non-empty sets, but not if we include the empty set.  Since a
fibre-product of non-empty sets along a surjective map is never empty,
surjections still form a pretopology on the subcategory of non-empty
sets, and it still satisfies the same assumptions.  Since
Assumption~\ref{assum:final} does not play an important role for our
theory, it is a matter of taste whether one should remove the empty
set or not.

A groupoid in~\((\Sets,\covers_\surj)\) is just a groupoid in the
usual sense; the surjectivity of the range and source maps is no
condition because the unit map is a one-sided inverse.
Since~\(\Sets\) has arbitrary colimits, any
\(\Gr\)\nb-action~\(\Act\) has an orbit set~\(\Act/\Gr\)
(Definition~\ref{def:orbit_space}).  The canonical map
\(\Act\prto\Act/\Gr\) is always surjective.

\begin{definition}
  \label{def:free}
  A groupoid action in~\(\Sets\) is \emph{free} if for each
  \(x\in\Act\), the only \(g\in\Gr^1\) with \(\s(x)=\rg(g)\) and
  \(x\cdot g=x\) is \(g=1_{\s(x)}\).
\end{definition}

\begin{proposition}
  \label{pro:Sets_basic}
  A groupoid action in~\((\Sets,\covers_\surj)\) is basic if and
  only if it is free.
\end{proposition}

\begin{proof}
  Let~\(\Gr\) be a groupoid and~\(\Act\) a \(\Gr\)\nb-action
  in~\((\Sets,\covers_\surj)\).  The canonical map \(\bunp\colon
  \Act\prto\Act/\Gr\) is a cover.  We have \(\bunp(x_1)=\bunp(x_2)\)
  for \(x_1,x_2\in\Act\) if and only if there is \(g\in\Gr^1\) with
  \(\s(x_1)=\rg(g)\) and \(x_1\cdot g=x_2\).  Thus the map
  \(\Act\times_{\s,\Gr^0,\rg}\Gr^1\to
  \Act\times_{\bunp,\Act/\Gr,\bunp}\Act\) is surjective.  It is
  injective if and only if the action is free.
\end{proof}

Hence a groupoid~\(\Gr\) in~\((\Sets,\covers_\surj)\) is basic if
and only if all \(g\in\Gr^1\) with \(\s(g)=\rg(g)\) are units.
Equivalently, \(\Gr\) is an equivalence relation on~\(\Gr^0\).

We may now finish the proof of
Proposition~\ref{pro:Sets_surjections} by showing that any action of
a \v{C}ech groupoid in~\(\Sets\) is basic.  Let~\(\Gr\) be a
\v{C}ech groupoid and let~\(\Act\) be a \(\Gr\)\nb-action.  If
\(x\in\Act\), \(g\in\Gr^1\) satisfy \(\s(x)=\rg(g)\) and \(x\cdot g
= x\), then \(\s(g) = \s(x\cdot g) = \s(x) = \rg(g)\).  Then
\(g=1_{\s(x)}\) because~\(\Gr\) is a \v{C}ech groupoid.  Thus the
\(\Gr\)\nb-action~\(\Act\) is free; it is basic by
Proposition~\ref{pro:Sets_basic}.  This verifies Assumptions
\ref{assum:covering_acts_basically}
and~\ref{assum:covering_acts_basically_weak}
for~\((\Sets,\covers_\surj)\).

\subsection{Pretopologies on the category of topological spaces}
\label{sec:Top}

Let~\(\Top\) be the category of topological spaces and continuous
maps.  This category is complete and cocomplete.  In particular, all
fibre products exist and any groupoid action has an orbit space.
There are several classes of maps in~\(\Top\) that give candidates
for pretopologies:
\begin{enumerate}
\item quotient maps
\item biquotient maps (also called limit lifting maps)
\item maps with global continuous sections
\item maps with local continuous sections
\item maps with local continuous sections and partitions of unity
\item closed surjections
\item proper surjections
\item open surjections
\item surjections with many continuous local sections
\item étale surjections
\end{enumerate}
We shall see that quotient maps and closed surjections do not form
pretopologies.  The other classes of maps all give subcanonical
pretopologies that satisfy Assumptions \ref{assum:local_cover}
and~\ref{assum:two-three}.  Assumption~\ref{assum:final} about final
objects fails for proper maps and étale maps and holds for the
remaining pretopologies, if we remove the empty topological space.
The pretopologies (2)--(5) are saturated, (7)--(10) are not saturated.
Assumption~\ref{assum:covering_acts_basically} and hence also
Assumption~\ref{assum:covering_acts_basically_weak} hold in cases
(3)--(5) and (8)--(10).
Assumption~\ref{assum:covering_acts_basically} fails for biquotient
maps; we do not know whether
Assumption~\ref{assum:covering_acts_basically_weak} holds for
biquotient maps.

\subsubsection{Quotient maps}
\label{sec:Top_quotient_maps}

\begin{lemma}
  \label{lem:coequaliser_Top_quotient}
  A continuous map \(f\colon X\to Y\) is a coequaliser in~\(\Top\)
  if and only if it is a \emph{quotient map}, that is, \(f\) is
  surjective and \(A\subseteq Y\) is open if and only if
  \(f^{-1}(A)\) is open.
\end{lemma}

\begin{proof}
  Let \(g_1,g_2\colon Z\rightrightarrows X\) be two maps.
  Let~\(\sim\) be the equivalence relation on~\(X\) generated by
  \(g_1(z)\sim g_2(z)\) for all \(z\in Z\).  Equip \(X/{\sim}\) with
  the quotient topology.  The coequaliser of \(g_1,g_2\) is the
  canonical map \(X\to X/{\sim}\).  Thus coequalisers are quotient
  maps.  Conversely, let \(f\colon X\to Y\) be a quotient map and
  define an equivalence relation~\(\sim\) on~\(X\) by \(x_1\sim
  x_2\) if and only if \(f(x_1)=f(x_2)\).  Since~\(f\) is a quotient
  map, the induced map \(X/{\sim} \to Y\) is a homeomorphism.
  Viewed as a subset of \(X\times X\), we have \({\sim} =
  X\times_{f,Y,f} X\).  Thus~\(f\) is a coequaliser of
  \(\pr_1,\pr_2\colon X\times_{f,Y,f} X\rightrightarrows X\).
\end{proof}

There are a quotient map \(f\colon X\to Y\)
and a topological space~\(Z\)
such that \(f\times\id_Z\)
is no longer a quotient map, see \cite{Michael:Bi-quotient}*{Example
  8.4}.  Thus the pull-back of~\(f\)
along the quotient map \(\pr_1\colon Y\times Z\to Y\)
is not a quotient map any more.  This shows that quotient maps do
\emph{not} form a pretopology on~\(\Top\).
The spaces in \cite{Michael:Bi-quotient}*{Example 8.4} are Hausdorff
spaces, so it does not help to restrict from~\(\Top\)
to the full subcategory~\(\Hausdorff\) of Hausdorff spaces.

\subsubsection{Biquotient maps}
\label{sec:Top_biquotient_maps}

\begin{definition}[\cite{Michael:Bi-quotient}]
  \label{def:biquotient}
  Let \(f\colon X\to Y\) be a continuous surjection.  It is a
  \emph{biquotient} map if for every \(y\in
  Y\) and every open covering~\(\mathcal{U}\) of~\(f^{-1}(y)\)
  in~\(X\), there are finitely many \(U\in\mathcal{U}\) for which
  the subsets~\(f(U)\) cover some neighbourhood of~\(y\) in~\(Y\).
\end{definition}

\begin{definition}[\cites{Hajek:Quotient, Hajek:Quotient_corrections}]
  \label{def:limit_lift}
  The map~\(f\) is \emph{limit lifting} if every convergent net
  in~\(Y\) lifts to a convergent net in~\(X\).  More precisely, let
  \((I,\le)\) be a directed set and let \((y_i)_{i\in I}\) be a net
  in~\(Y\) converging to some \(y\in Y\).  A \emph{lifting} of this
  convergent net is a directed set \((J,\le)\) with a surjective
  order-preserving map \(\varphi\colon J\to I\) and a
  net~\((x_j)_{j\in J}\)
  in~\(X\) with \(f(x_j)=y_{\varphi(j)}\) for all \(j\in J\),
  converging to some \(x\in X\) with \(f(x)=y\).
\end{definition}

\begin{proposition}[\cite{Michael:Bi-quotient}]
  Biquotient maps are the same as limit lifting maps.
\end{proposition}

We will use both characterisations of these maps where convenient.

\begin{proposition}
  \label{pro:Top_biquotient}
  The biquotient maps~\(\covers_\limlift\) form a subcanonical,
  saturated pretopology that satisfies Assumptions
  \textup{\ref{assum:local_cover}}, \textup{\ref{assum:two-three}},
  but not Assumption~\textup{\ref{assum:covering_acts_basically}}.  It
  satisfies Assumption~\textup{\ref{assum:final}} if we remove the
  empty space from the category.
\end{proposition}

\begin{proof}
  It is clear that isomorphisms are limit lifting.  Let \(f_1\colon
  X\to Y\) and \(f_2\colon Y\to Z\) be composable maps.  If \(f_1\)
  and~\(f_2\) are limit lifting, then so is \(f_2\circ f_1\) by
  lifting in two steps; and if \(f_2\circ f_1\) is limit lifting,
  then so is~\(f_2\) by first lifting a convergent net along
  \(f_2\circ f_1\)
  and then applying the continuous map~\(f_1\).  Thus we have a
  saturated pretopology, and Assumptions \ref{assum:local_cover}
  and~\ref{assum:two-three} follow.

  Next we claim that pull-backs inherit the property of being limit
  lifting.  In the fibre-product
  situation~\eqref{eq:fibre-product_of_cover}, assume that \(g\colon
  U\prto X\) is limit lifting.  Let \((y_i)_{i\in I}\) be a net
  in~\(Y\) that converges to some \(y\in Y\).  Then \(f(y_i)\)
  converges to~\(f(y)\) in~\(X\).  Since~\(g\) is limit lifting,
  there is an order-preserving map \(\varphi\colon J\to I\) and a
  net \((u_j)_{j\in J}\) in~\(U\) with \(g(u_j)=f(y_{\varphi(j)})\)
  for all \(j\in J\), converging to some \(u\in U\) with
  \(g(u)=f(y)\).  Then \((y_{\varphi(j)},u_j)\) is a net in
  \(Y\times_{f,X,g} U\) that converges towards \((y,u)\).  Thus
  \(\pr_1\colon Y\times_{f,X,g} U\prto Y\) is limit lifting.  This
  completes the proof that~\(\covers_\limlift\) is a pretopology.
  Limit lifting maps are quotient maps, hence coequalisers, so this
  pretopology is subcanonical.

  A space with a single element is a final object in~\(\Top\), and any
  map from a non-empty space to it is limit lifting.  Thus
  Assumption~\ref{assum:final} holds if we exclude the empty
  space.  The following counterexample to
  Assumption~\ref{assum:covering_acts_basically} will finish the proof
  of Proposition~\ref{pro:Top_biquotient}.
\end{proof}

\begin{example}
  \label{exa:biquotient_covering_not_basic}
  We first recall \cite{Michael:Bi-quotient}*{Example 8.4}.
  Let~\(X\) be the disjoint union of \(X_1=(0,1)\) and
  \(X_2=\{0,\frac{1}{2},\frac{1}{3},\frac{1}{4},\ldots\}\).
  Let~\(\tau_1\) be the obvious topology on~\(X\): \(X_1\)
  and~\(X_2\) are open and carry the subspace topologies
  from~\(\R\).  We will later introduce another topology~\(\tau_2\)
  on~\(X\).

  Let~\(Y\) be the quotient space of~\(X\) by the relation that
  identifies \(\frac{1}{n}\) in \(X_1\) and~\(X_2\), and let
  \(f\colon X\to Y\) be the quotient map.  This is a quotient map
  that is not limit lifting.  As a set, \(Y=[0,1)\).  The topology
  on~\(Y\) is the usual one on the subset~\((0,1)\); a subset~\(U\)
  of~\(Y\) with \(0\in U\) is a neighbourhood of~\(0\) if and only
  if there are \(n_0\in\N\) and \(\epsilon_n>0\) for \(n\ge n_0\)
  with \(\bigcup_{n\ge n_0} (\frac{1}{n}-\epsilon_n,
  \frac{1}{n}+\epsilon_n)\subseteq U\).

  We now define another topology~\(\tau_2\) on~\(X\) for which the
  map \(f\colon (X,\tau_2)\to Y\) is limit lifting.  The subset
  \(X_1 \sqcup X_2\setminus\{0\}\) is open in both topologies
  \(\tau_1\) and~\(\tau_2\), and both restrict to the same topology
  on \(X_1 \sqcup X_2\setminus\{0\}\).  A subset \(U\subseteq X\) is
  a \(\tau_2\)\nb-neighbourhood of~\(0\) if \(U\cap X_2\) is a
  neighbourhood of~\(0\) and there are \(n_0\in\N\) and
  \(\epsilon_n>0\) for all \(n\ge n_0\) such that \(\bigcup_{n\ge
    n_0} \bigl(\frac{1}{n}-\epsilon_n, \frac{1}{n}+\epsilon_n\bigr)
  \setminus \bigl\{\frac{1}{n}\bigr\}\subseteq U\).  This uniquely
  determines the topology~\(\tau_2\).

  On the open subset \((0,1)\subseteq Y\), the inclusion of~\(X_1\)
  into~\(X\) gives a continuous section for \(f|_{f^{-1}(0,1)}\);
  hence a net in~\(Y\) that converges to some element of~\((0,1)\)
  lifts to a net in \(X_1\subseteq X\) that converges in both
  \((X,\tau_1)\) and \((X,\tau_2)\).  Let now \((y_i)_{i\in I}\) be
  a net in~\(Y\) converging to~\(0\).  We lift it to a net in~\(X\)
  by lifting~\(y_i\) to \(y_i\in X_2\) if \(y_i\in X_2\), and to
  \(y_i\in X_1\) if \(y_i\notin X_2\).  This net in~\(X\) converges
  to \(0\in X_2\) in the topology~\(\tau_2\).

  The topology~\(\tau_1\) is finer than~\(\tau_2\), that is, the
  identity map on~\(X\) is a continuous map \((X,\tau_1)\to
  (X,\tau_2)\).  We claim that the identity map
  \begin{equation}
    \label{eq:pull-back_to_homeo}
    (X,\tau_1) \times_{f,Y,f} (X,\tau_1)
    \to (X,\tau_1) \times_{f,Y,f} (X,\tau_2)
  \end{equation}
  is a homeomorphism.  That is, the topologies \(\tau_1\)
  and~\(\tau_2\) pull back to the same topology on
  \((X,\tau_1)\times_{f,Y,f} X\) along the quotient map \(f\colon
  (X,\tau_1)\to Y\).  This also shows that the locality of
  isomorphisms, Proposition~\ref{pro:isomorphism_local}, fails for
  quotient maps in~\(\Top\); this is okay because they do not form a
  pretopology.

  To prove that the map in~\eqref{eq:pull-back_to_homeo} is a
  homeomorphism, we show that any net in \(X\times_{f,Y,f}
  X\subseteq X\times X\) that converges for \(\tau_1\times\tau_2\)
  also converges for the topology \(\tau_1\times\tau_1\).  A net in
  \(X\times_Y X\) is a pair of nets \((x_{1i})\) and \((x_{2i})\)
  with \(f(x_{1i})=f(x_{2i})\).  We assume that \((x_{1i},x_{2i})\)
  is \(\tau_1\times\tau_2\)-convergent to \((x_1,x_2)\).
  Equivalently, \(x_{1i}\) has \(\tau_1\)\nb-limit~\(x_1\) and
  \((x_{2i})\) has \(\tau_2\)\nb-limit~\(x_2\).  We must show
  that~\(x_{2i}\) also converges to~\(x_2\) in the
  topology~\(\tau_1\).  This is clear if \(x_2\neq 0\in X_2\)
  because the topologies \(\tau_1\) and~\(\tau_2\) agree away
  from~\(0\).  Thus we may assume \(x_1=x_2=0\).  Since~\(X_2\) is
  \(\tau_1\)\nb-open, we have \(x_{1i}\in X_2\) for almost
  all~\(i\); hence also \(f(x_{2i})=f(x_{1i})\in X_2\).  Since the
  points~\(\frac{1}{n}\) in~\(X_1\) are excluded from the
  \(\tau_2\)\nb-neighbourhoods of~\(0\) and \(x_{2i}\) converges
  to~\(0\) in~\(\tau_2\), this implies \(x_{2i}\in X_2\) for almost
  all~\(i\) as well.  Since \(f|_{X_2}\) is injective, we get
  \(x_{1i}=x_{2i}\) for almost all~\(i\), so~\((x_{2i})\)
  converges in~\(\tau_1\).

  Now let~\(\Gr\) be the \v{C}ech groupoid of the limit lifting map
  \((X,\tau_2)\to Y\).  The action of~\(\Gr\) on its orbit
  space~\((X,\tau_2)\) is basic in \((\Top,\covers_\limlift)\) by
  Example~\ref{exa:covering_groupoid_principal}.  The same action
  on~\((X,\tau_1)\) is still a continuous action: the anchor map
  \((X,\tau_1)\to (X,\tau_2)\) is continuous, and the action map
  \[
  (X,\tau_1)\times_{f,X,f} (X,\tau_2) \cong
  (X,\tau_1)\times_X \bigl((X,\tau_2)\times_{f,X,f} (X,\tau_2)\bigr)
  \xrightarrow{\pr_3} (X,\tau_1)
  \]
  is continuous because of the
  homeomorphism~\eqref{eq:pull-back_to_homeo}.  This action is,
  however, not basic because the orbit space projection
  \((X,\tau_1)\to (X,\tau_1)/\Gr\cong Y\) is not
  in~\(\covers_\limlift\).  If this action were basic, it would
  contradict Proposition~\ref{pro:G-map_versus_base_map} because the
  \(\Gr\)\nb-actions on \((X,\tau_1)\) and \((X,\tau_2)\) have the
  same orbit space~\(Y\).
\end{example}

We do not know whether
Assumption~\ref{assum:covering_acts_basically_weak} holds
for~\((\Top,\covers_\limlift)\).  A counterexample would have to be of
different nature because a continuous bijection that is also a
(bi)quotient map is already a homeomorphism.

\begin{proposition}
  \label{pro:biquotient_largest}
  The pretopology of biquotient maps is the largest subcanonical
  pretopology on~\(\Top\).
\end{proposition}

\begin{proof}
  We claim that a continuous map \(f\colon X\to Y\) is biquotient if
  and only if \(\pr_1\colon Z\times_{g,Y,f} X\to Z\) is a quotient
  map for any map \(g\colon Z\to Y\).  We have already seen that
  biquotient maps form a pretopology; in particular, if~\(f\) is
  biquotient, then \(\pr_1\colon Z\times_{g,Y,f} X \to Z\) is
  biquotient and thus quotient.  The converse remains to be proved.
  It implies the proposition because the covers of a subcanonical
  pretopology on~\(\Top\) must be quotient maps by
  Lemma~\ref{lem:coequaliser_Top_quotient}, hence biquotient by our
  claim.

  Assume that \(\pr_1\) is a quotient map for any map \(g\colon Z\to
  Y\), although~\(f\) is not biquotient; we want to arrive at a
  contradiction.  By assumption, there are \(y_\infty\in Y\) and an
  open covering~\(\mathcal{U}\) of~\(f^{-1}(y_\infty)\) such that for
  any finite set \(F\subseteq\mathcal{U}\), \(\bigcup_{U\in F} f(U)\)
  is not a neighbourhood of~\(y_\infty\).  Let~\(I\) be the set of
  pairs~\((F,V)\) for a finite subset \(F\subseteq\mathcal{U}\) and an
  open neighbourhood~\(V\) of~\(y_\infty\) in~\(Y\).  We order~\(I\)
  by \((F_1,V_1)\le (F_2,V_2)\) if \(F_1\subseteq F_2\) and
  \(V_1\supseteq V_2\); this gives a directed set.  By assumption,
  \(V\setminus \bigcup_{U\in F} f(U)\neq\emptyset\) for all \((F,V)\in
  I\); we choose~\(y_{F,V}\) in this difference.  Since \(y_{F,V}\in
  V\), \(\lim y_{F,V}=y_\infty\).

  Let \(I^+\defeq I\sqcup\{\infty\}\), topologised so that any
  subset of~\(I\) is open and a subset~\(W\) containing~\(\infty\)
  is open if and only if there is \(i\in I\) with \(j\in W\) for all
  \(j\ge i\).  Let
  \[
  Z\defeq \{(i,y)\in I^+\times Y \mid y=y_i\}
  \]
  be the graph of the function \(I^+\ni i\mapsto y_i\in Y\), equipped
  with the subspace topology; let \(g\colon Z\to Y\) be the second
  coordinate projection.  We identify
  \[
  Z\times_{g,Y,f} X\cong \{(i,x)\in I^+\times X\mid f(x)=y_i\},
  \qquad (i,y,x)\mapsto (i,x),
  \]
  for \((i,y)\in Z\), \(x\in X\) with \(f(x)=g(i,y)=y\).  This
  bijection is a homeomorphism for the subspace topologies from
  \(Z\times X\subseteq I^+\times Y\times X\) and \(I^+\times X\).

  The subset \(S\defeq Z\setminus\{(\infty,y_\infty)\}\) in~\(Z\) is
  not closed because \(\lim_{i\in I} (i,y_i)=(\infty,y_\infty)\).
  Since the coordinate projection \(Z\times_{g,Y,f} X\to Z\) is
  assumed to be a quotient map, the preimage~\(\hat{S}\) of~\(S\) is
  not closed in \(Z\times_{g,Y,f} X\).  Hence there is a net
  \((i_j,x_j)_{j\in J}\) in~\(\hat{S}\) that converges in
  \(Z\times_{g,Y,f} X\) towards some point in the complement
  of~\(\hat{S}\), that is, of the form \((\infty,x_\infty)\) with
  \(x_\infty\in X\) such that \(f(x_\infty)=y_\infty\).  That is,
  \(\lim i_j=\infty\) in~\(I^+\) and \(\lim x_j=x_\infty\) in~\(X\).
  We have
  \(f(x_j)=y_{i_j}\) because \((i_j,x_j)\in Z\times_{g,Y,f} X\).

  Since~\(\mathcal{U}\) is an open covering of~\(f^{-1}(y_\infty)\)
  and \(f(x_\infty)=y_\infty\), there is \(U\in\mathcal{U}\) with
  \(x_\infty\in U\).  Then also \(x_j\in U\) for sufficiently large
  \(j\in J\).  Since \(i_j\to\infty\), there is \(j_0\in J\) with
  \(i_j\ge (\{U\},Y)\) for all \(j\ge j_0\); thus \(y_{i_j}\notin
  f(U)\) for \(j\ge j_0\).  Then \(x_j\notin U\) for \(j\ge j_0\), a
  contradiction.  Hence the coordinate projection \(Z\times_{g,Y,f}
  X\to Z\) cannot be a quotient map.
\end{proof}

\begin{remark}
  \label{rem:triquotient}
  There are several other notions of improved quotient maps, such as
  triquotient maps (see \cites{Michael:Complete_tri-quotient,
    Clementino-Hofmann:Limit_stability}).  Triquotient maps also form
  a subcanonical pretopology satisfying Assumptions
  \ref{assum:local_cover} and~\ref{assum:two-three}.  The biquotient
  map \((X,\tau_2)\to Y\) is a triquotient map as well, however.  Thus
  the same counterexample as above shows that
  Assumption~\ref{assum:covering_acts_basically} fails for the
  pretopology of triquotient maps.  The surjective open maps form the
  largest pretopology on~\(\Top\) for which we know
  Assumption~\ref{assum:covering_acts_basically}.
\end{remark}

Now we consider groupoids in \((\Top,\covers_\limlift)\).  Any map
with a continuous section is limit lifting.  Since the unit map is a
section for the range and source maps in a groupoid, the condition
that \(\rg\) and~\(\s\) be limit lifting is redundant in the first
definition of a groupoid.  In the second definition of a groupoid, it
is enough to assume that \(\rg\) and~\(\s\) are both quotient maps:
this suffices to construct a continuous unit map, which then implies
that \(\rg\) and~\(\s\) are limit lifting.

The category~\(\Top\) has arbitrary colimits, so~\(\Act/\Gr\) exists
for any \(\Gr\)\nb-action~\(\Act\); it is the set of
orbits~\(\Act/\Gr\) with the quotient topology.

\begin{proposition}
  \label{pro:Top_basic}
  Let~\(\covers\) be any subcanonical pretopology on~\(\Top\).  Let
  \(\Gr\) be a groupoid and~\(\Act\) a \(\Gr\)\nb-action in
  \((\Top,\covers)\).

  The \(\Gr\)\nb-action~\(\Act\) is basic if and only if it satisfies
  the following conditions:
  \begin{enumerate}
  \item the \(\Gr\)\nb-action~\(\Act\) is free;
  \item the map \(\Act\times_{\bunp,\Act/\Gr,\bunp} \Act\to\Gr^1\)
    that maps \(x_1,x_2\in\Act\) in the same orbit to the unique
    \(g\in\Gr^1\) with \(\s(x_1)=\rg(g)\) and \(x_1\cdot g = x_2\)
    is continuous;
  \item the orbit space projection \(\Act\to\Act/\Gr\) is
    in~\(\covers\).
  \end{enumerate}
  The first two conditions hold for any \(\Gr\)\nb-action of a
  \v{C}ech groupoid of a cover in~\((\Top,\covers)\).
\end{proposition}

\begin{proof}
  By Lemma~\ref{lem:principal_bundle}, the bundle projection for a
  basic action must be the orbit space projection \(\bunp\colon
  \Act\prto\Act/\Gr\).  For a basic action, this must be a cover.  The
  continuous map \(\Act\times_{\s,\Gr^0,\rg} \Gr^1 \to
  \Act\times_{\bunp,\Act/\Gr,\bunp}\Act\), \((x,g)\mapsto (x,x\cdot g)\),
  is surjective by the definition of~\(\Act/\Gr\).  It is bijective
  if and only if it is injective if and only if the action is free
  (see Proposition~\ref{pro:Sets_basic}).  In this case, the inverse
  map sends \((x_1,x_2)\in\Act^2\) with \(\bunp(x_1)=\bunp(x_2)\) to
  \((x_1,g)\) for the unique \(g\in\Gr^1\) with \(\s(x_1)=\rg(g)\)
  and \(x_1\cdot g=x_2\).  This inverse map is continuous if and
  only if the map \((x_1,x_2)\mapsto g\) in~(2) is continuous.  Hence
  the action is basic if and only if (1)--(3) hold.

  Let~\(\Gr\) be the \v{C}ech groupoid of a cover \(f\colon
  \Act[Y]\to\Base\).  Let \((x_1,x_2)\in\Act\) be in the same
  \(\Gr\)\nb-orbit.  Then any \(g\in\Gr^1\) with \(\s(x_1)=\rg(g)\)
  and \(x_1\cdot g=x_2\) will satisfy \(\s(g)=\s(x_1\cdot
  g)=\s(x_2)\), hence \(g=(\s(x_1),\s(x_2))\).  This is unique and
  depends continuously on \(x_1,x_2\), giving (1) and~(2)
  above.
\end{proof}

Hence the only way that \v{C}ech groupoid actions in
\((\Top,\covers)\) may fail to be basic is by the orbit space
projection not being a cover.

\begin{example}
  \label{exa:free_not_principal}
  We construct a free groupoid action in~\((\Top,\covers_\limlift)\)
  that satisfies~(3) in Proposition~\ref{pro:Top_basic},
  but not~(2).  Let~\(\R_d\) be~\(\R\) with the discrete
  topology.  The translation action of~\(\R_d\) on~\(\R\) with the usual
  topology is free.  Its orbit space has only one point, so the map
  \(\R\to \R/\R_d\) is a cover.  Condition~(2) is violated, so the
  action is not basic.
\end{example}

The transformation groupoid of the non-basic action in
Example~\ref{exa:biquotient_covering_not_basic} satisfies (1) and~(2)
in Proposition~\ref{pro:Top_basic}, but not~(3).

\subsubsection{Hausdorff orbit spaces}
\label{sec:Hausdorff_orbit}

We now replace~\(\Top\) by its full subcategory~\(\Hausdorff\) of
Hausdorff topological spaces.  The proof of
Proposition~\ref{pro:biquotient_largest} still works in this
subcategory, showing that the biquotient maps form the largest
subcanonical pretopology on~\(\Hausdorff\).  A groupoid
in~\((\Hausdorff,\covers_\limlift)\) is the same as a groupoid
in~\((\Top,\covers_\limlift)\) with Hausdorff spaces \(\Gr^0\)
and~\(\Gr^1\).  A \(\Gr\)\nb-action for a Hausdorff groupoid on a
Hausdorff space is basic in \((\Hausdorff,\covers_\limlift)\) if and
only if it is basic in~\((\Top,\covers_\limlift)\) \emph{and the
  orbit space is Hausdorff}.  This is not automatic:

\begin{example}
  \label{exa:covering_non-Hausdorff}
  Let \(X\defeq [0,1] \sqcup [0,1]\) and let~\(Y\) be the quotient
  of~\(X\) by the relation that identifies the two copies of
  \((0,1]\).  This is a non-Hausdorff, locally Hausdorff space.  The
  quotient map \(X\to Y\) is open and hence biquotient.  Therefore,
  its \v{C}ech groupoid is basic in~\((\Top,\covers_\limlift)\)
  (Example~\ref{exa:covering_groupoid_principal}).  This \v{C}ech
  groupoid is also a groupoid in~\((\Hausdorff,\covers_\limlift)\)
  because its object space \(\Gr^0=X\) and arrow space \(\Gr^1=
  X\times_Y X\subseteq X\times X\) are both Hausdorff.  Since its
  orbit space is non-Hausdorff, it is \emph{not} basic
  in~\((\Hausdorff,\covers_\limlift)\).
\end{example}

\begin{proposition}
  \label{pro:Hausdorff_quotient}
  Let \(f\colon X\to Y\) be a biquotient map.  The space~\(Y\) is
  Hausdorff if and only if the subset \(X\times_{f,Y,f} X\subseteq
  X\times X\) is closed.
\end{proposition}

\begin{proof}
  If~\(Y\) is Hausdorff and \(f\colon X\to Y\) is continuous, then
  \[
  X\times_{f,Y,f} X = \{(x_1,x_2)\in X\times X\mid f(x_1)=f(x_2)\}
  \]
  is closed: if \((x_1,x_2)\in X\times X\) with \(f(x_1)\neq
  f(x_2)\), then there are neighbourhoods~\(U_i\) of \(f(x_i)\)
  in~\(Y\) for \(i=1,2\) with \(U_1\cap U_2=\emptyset\)
  because~\(Y\) is Hausdorff; then \(f^{-1}(U_1)\times f^{-1}(U_2)\)
  is a neighbourhood of \((x_1,x_2)\) in~\(X\times X\) that does not
  intersect~\(X\times_{f,Y,f} X\).

  It remains to show that~\(Y\) is Hausdorff if~\(X\times_{f,Y,f}
  X\) is closed in \(X\times X\) and~\(f\) is biquotient.  We choose
  \(y, y'\in Y\) with \(y\neq y'\) and try to separate them by
  neighbourhoods.  If \(x,x'\in X\) satisfy \(f(x)=y\) and
  \(f(x')=y'\), then \((x,x')\notin X\times_{f,Y,f} X\), so there
  are open neighbourhoods \(U_{x,x'}\) and~\(U'_{x,x'}\) of \(x\)
  and~\(x'\), respectively, with \(U_{x,x'}\times U'_{x,x'}\cap
  X\times_{f,Y,f} X=\emptyset\); that is, \(f(U_{x,x'})\cap
  f(U'_{x,x'})=\emptyset\).  Choose such neighbourhoods for all
  \((x,x')\in f^{-1}(y)\times f^{-1}(y')\).  For fixed~\(x\), the
  open sets~\(U'_{x,x'}\) for \(x'\in f^{-1}(y')\)
  cover~\(f^{-1}(y')\).  Since~\(f\) is biquotient, there is a
  finite set \(A_x\subseteq f^{-1}(y')\) such that \(V_x\defeq
  \bigcup_{x'\in A_x} f(U'_{x,x'})\) is a neighbourhood of~\(y'\)
  in~\(Y\).

  Let \(U_x\defeq \bigcap_{x'\in A_x} U_{x,x'}\).  This is an open
  neighbourhood of~\(x\) because~\(A_x\) is finite.  We have
  \(f(U_x)\cap f(U'_{x,x'})=\emptyset\) for all \(x'\in A_x\) and
  hence \(f(U_x)\cap V_x=\emptyset\).  Since~\(f\) is biquotient and
  the open subsets~\(U_x\) cover \(f^{-1}(y)\), there is a finite
  subset \(B\subseteq f^{-1}(y)\) such that \(\bigcup_{x\in B}
  f(U_x)\) is a neighbourhood of~\(y\) in~\(Y\).  The finite
  intersection \(\bigcap_{x\in B} V_x\) is a neighbourhood of~\(y'\)
  in~\(Y\).  These two neighbourhoods separate \(y\) and~\(y'\).
\end{proof}

\begin{corollary}
  \label{cor:Hausdorff_diagonal_closed}
  A topological space~\(X\) is Hausdorff if and only if the diagonal
  is a closed subset in~\(X\times X\).
\end{corollary}

\begin{proof}
  Apply Proposition~\ref{pro:Hausdorff_quotient} to the identity map.
\end{proof}

The following definitions are needed to characterise when the orbit
space of a basic action in~\((\Top,\covers_\limlift)\) is Hausdorff.

\begin{definition}[\cite{Bourbaki:Topologie_generale}*{I.10.1},
  ``application propre'']
  \label{def:proper_map}
  A map \(f\colon X \to Y\) of topological spaces is \emph{proper}
  if \(f\times\id_Z\colon X\times Z \to Y\times Z\) is closed for
  any topological space~\(Z\).
\end{definition}

\begin{definition}
  \label{def:proper_action}
  An action of a topological groupoid~\(\Gr\) on a topological
  space~\(\Act\) is \emph{proper} if the map
  \begin{equation}
    \label{eq:action_map_proper}
    \Act\times_{\s,\Gr^0,\rg} \Gr^1\to \Act\times \Act,\qquad
    (x,g)\mapsto (x,x\cdot g),
  \end{equation}
  is proper.  A groupoid~\(\Gr\) is \emph{proper} if its canonical
  action on~\(\Gr^0\) is proper, that is, the map \((\rg,\s)\colon
  \Gr^1\to \Gr^0\times\Gr^0\) is proper.
\end{definition}

A map from a Hausdorff space~\(X\) to a Hausdorff, locally compact
space~\(Y\) is proper if and only if preimages of compact subsets
are compact (\cite{Bourbaki:Topologie_generale}*{Proposition 7 in
  I.10.3}).

\begin{proposition}
  \label{pro:Hausdorff_basic}
  Let~\(\covers\) be any subcanonical pretopology on~\(\Hausdorff\).
  Let\/~\(\Gr\) be a groupoid and\/~\(\Act\) a \(\Gr\)\nb-action
  in~\((\Hausdorff,\covers)\).  The \(\Gr\)\nb-action~\(\Act\) is
  basic in \((\Hausdorff,\covers)\) if and only if it satisfies the
  following conditions:
  \begin{enumerate}
  \item the \(\Gr\)\nb-action~\(\Act\) is free;
  \item the \(\Gr\)\nb-action~\(\Act\) is proper;
  \item the orbit space projection \(\Act\to\Act/\Gr\) belongs
    to~\(\covers\).
  \end{enumerate}
  The first two conditions hold for any \(\Gr\)\nb-action of a
  \v{C}ech groupoid of a cover in~\((\Hausdorff,\covers)\).
\end{proposition}

\begin{proof}
  Since~\(\covers_\limlift\) is the largest subcanonical pretopology
  on~\(\Hausdorff\), all covers in~\(\covers\) are biquotient.  Thus
  a basic \(\Gr\)\nb-action in~\((\Hausdorff,\covers)\) is still
  basic in~\((\Top,\covers_\limlift)\); furthermore, its orbit
  space~\(\Act/\Gr\) is Hausdorff, and the orbit space projection
  \(\bunp\colon \Act\to\Act/\Gr\) is in~\(\covers\).  Conversely,
  these three conditions imply that the action is basic
  in~\((\Hausdorff,\covers)\).

  Proposition~\ref{pro:Top_basic} describes when the action is basic
  in~\((\Top,\covers_\limlift)\).  An injective continuous map is
  proper if and only if it is closed, if and only if it is a
  homeomorphism onto a closed subset (see
  \cite{Bourbaki:Topologie_generale}*{I.10.1, Proposition 2}).  Thus
  a free action is proper if and only if the bijection
  \(\Act\times_{\s,\Gr^0,\rg}
  \Gr^1\to\Act\times_{\bunp,\Act/\Gr,\bunp}\Act\) is a homeomorphism
  and \(\Act\times_{\bunp,\Act/\Gr,\bunp}\Act\) is closed
  in~\(\Act\times\Act\).  The first part of this is the second
  condition in Proposition~\ref{pro:Top_basic}, and the closedness
  of \(\Act\times_{\bunp,\Gr^0,\bunp}\Act\) in~\(\Act\times\Act\) is
  equivalent to~\(\Act/\Gr\) being Hausdorff by
  Proposition~\ref{pro:Hausdorff_quotient} because~\(\bunp\) is
  biquotient.  Now it is routine to see that (1)--(3)
  above characterise basic actions in~\((\Hausdorff,\covers)\).

  The first two conditions in Proposition~\ref{pro:Top_basic} are
  automatic for actions of \v{C}ech groupoids
  in~\((\Top,\covers_\limlift)\).  It remains to show that
  \(\Act\times_{\bunp,\Act/\Gr,\bunp}\Act\) is closed
  in~\(\Act\times\Act\) for any action of a \v{C}ech groupoid
  in~\((\Hausdorff,\covers_\limlift)\).  Let \(Y\) and~\(Z\) be
  Hausdorff spaces and let \(f\colon Y\prto Z\) be a
  biquotient map with \v{C}ech groupoid~\(\Gr\).  Thus
  \(Y=\Gr^0\) and the anchor map on~\(\Act\) is a map \(\s\colon
  \Act\to Y\).  Since~\(Z\) is Hausdorff,
  \[
  \Act\times_Z \Act = \{(x_1,x_2)\in\Act\times\Act\mid
  f(\s(x_1))=f(\s(x_2))\}
  \]
  is a closed subset of~\(\Act\times\Act\).  If
  \((x_1,x_2)\in\Act\times_Z \Act\), then
  \((\s(x_1),\s(x_2))\in\Gr^1\), so \(x_2'\defeq x_1\cdot
  (\s(x_1),\s(x_2))\in\Act\) is defined.  This is the unique element
  in the orbit of~\(x_1\) with \(\s(x_2')=\s(x_2)\).  Thus
  \[
  \Act\times_{\Act/\Gr}\Act = \{(x_1,x_2)\in\Act\times_Z\Act\mid
  x_1\cdot(\s(x_1),\s(x_2)) = x_2\}.
  \]
  Since~\(\Act\) is Hausdorff and the two maps sending \((x_1,x_2)\)
  to \(x_1\cdot(\s(x_1),\s(x_2))\) and~\(x_2\) are continuous, this
  subset is closed in \(\Act\times_Z\Act\) and hence
  in~\(\Act\times\Act\).
\end{proof}

We may also restrict to other subcategories of topological spaces,
still with biquotient maps as covers:

\begin{proposition}
  \label{pro:preserve_biquotient}
  Let \(f\colon X\to Y\) be biquotient.  The following properties
  are inherited by~\(Y\) if~\(X\) has them:
  \begin{enumerate}
  \item locally quasi-compact;
  \item having a countable base;
  \item sequential: a subset~\(A\) is closed if and only if it
    is closed under taking limits of convergent sequences in~\(A\);
  \item Fréchet: the closure of a subset~\(A\) is the set of
    all points that are limits of convergent sequences in~\(A\);
  \item \(k\)\nb-space: a subset is closed if and only if
    \(A\cap K\) is relatively compact in~\(K\) for any quasi-compact
    subset~\(K\).
  \end{enumerate}
\end{proposition}

\begin{proof}
  That biquotient images inherit the first two properties is stated
  in \cite{Michael:Bi-quotient}*{Proposition 3.4}.  The remaining
  statements follow from~\cite{Michael:Quintuple}, which
  characterises the images of various classes of ``nice'' spaces
  under five classes of ``quotient maps.''  If a property~P
  characterises the biquotient images of some particular class of
  ``nice'' spaces, then~P is inherited by biquotient images because
  composites of biquotient maps are again biquotient.  Composites of
  biquotient maps with quotient, hereditarily quotient, or countably
  biquotient maps are also again of the same sort.  Hence all the
  properties in rows 3--6 of the table in
  \cite{Michael:Quintuple}*{p.~93} are preserved under biquotient
  images.  This includes the properties of being sequential,
  Fréchet, or a \(k\)\nb-space given above, and many less familiar
  properties.
\end{proof}

The classes of spaces in the first row in the table in
\cite{Michael:Quintuple}*{p.~93} are, however, not closed under
biquotient images.  This includes the class of locally compact,
paracompact spaces and the class of metrisable spaces.
\cite{Michael:Quintuple} shows that any locally compact space is a
biquotient image of a locally compact paracompact space (it is a
biquotient image of the disjoint union of its compact subsets),
while any ``bisequential'' space is a biquotient image of a
metrisable space.

Since orbit space projections of basic actions in~\((\Top,\covers)\)
are biquotient, the orbit space~\(\Act/\Gr\) of a basic action
in~\((\Top,\covers)\) will inherit the properties listed in
Proposition~\ref{pro:preserve_biquotient} from~\(\Act\).
Let~\(\Cat\subseteq\Top\) be the full subcategory described by this
property, say, the full subcategory of \(k\)\nb-spaces, and
let~\(\covers\) be some pretopology on~\(\Top\).  It follows that a
groupoid action in~\((\Cat,\covers\cap\Cat)\) is basic if and only
if it is basic in~\((\Top,\covers)\); the latter is characterised by
Proposition~\ref{pro:Top_basic}.  If we add Hausdorffness, then a
groupoid action in
\((\Cat\cap\Hausdorff,\covers\cap\Cat\cap\Hausdorff)\) is basic if
and only if it is basic in \((\Hausdorff,\covers\cap\Hausdorff)\),
which is characterised by Proposition~\ref{pro:Hausdorff_basic}.

What if we want to restrict to a subcategory that is not closed
under biquotient images, say, the category of metrisable topological
spaces?  An action of a metrisable topological groupoid on a
metrisable topological space is basic with respect to a
pretopology~\(\covers\) contained in~\(\covers_\limlift\) if and
only if it satisfies the conditions in Proposition
\ref{pro:Hausdorff_basic} or~\ref{pro:Top_basic} and, in addition,
\(\Act/\Gr\) is metrisable.

\begin{example}
  Let \(Y\)~be bisequential and not metrisable.  Let \(f\colon
  X\prto Y\) be a biquotient map with \(X\)~metrisable.  Then
  \(X\times_{f,Y,f} X\subseteq X\times X\) is again metrisable.
  Hence the \v{C}ech groupoid of~\(f\) is a groupoid in the category
  of metrisable topological spaces with biquotient maps as covers.
  It is not basic in the category of metrisable spaces, however,
  because its orbit space is not metrisable.
\end{example}

\subsubsection{Covers defined by continuous sections}
\label{sec:Top_sections}

Let \(f\colon X\to Y\) be a continuous map.

\begin{definition}
  \label{def:continuous_sections}
  A \emph{global continuous section} for~\(f\) is a continuous map
  \(\sigma\colon Y\to X\) with \(f\circ\sigma=\id_Y\).
  Let~\(\covers_\splitsur\) be the class of all continuous maps with
  a global continuous section.

  A \emph{local continuous section} for~\(f\) at \(y\in Y\) is a
  pair \((U,\sigma)\) consisting of a neighbourhood~\(U\) of~\(y\)
  and a continuous map \(\sigma\colon U\to X\) with
  \(f\circ\sigma=\id_U\).  We call~\(f\) \emph{locally split}
  if such local continuous sections exist at all \(y\in Y\).
  Let~\(\covers_\locsplit\) be the class of all locally split
  continuous maps.

  We call~\(f\)
  \emph{numerably split} if there is a \emph{numerable} open
  covering~\(\mathcal{U}\)
  of~\(Y\)
  with local sections defined on all elements of~\(\mathcal{U}\);
  an open covering is numerable if it has a subordinate partition of
  unity (see~\cite{Dold:Partitions}).  Let~\(\covers_\numsplit\)
  be the class of numerably split maps.
\end{definition}

By definition,
\[
\covers_\splitsur\subseteq \covers_\numsplit \subseteq \covers_\locsplit.
\]
A local continuous section \(\sigma\colon U\to X\) for \(f\colon
X\to Y\) at \(y\in Y\) is determined by its image \(S\defeq
\sigma(U)\subseteq X\).  This is a subset on which~\(f\) restricts
to a homeomorphism onto a neighbourhood of~\(y\) in~\(Y\), and
any such subset gives a unique local continuous section for~\(f\)
near~\(y\).  The subset~\(S\) is also called a \emph{slice}.

\begin{proposition}
  \label{pro:Top_numsplit_pretopology}
  \(\covers_\splitsur\), \(\covers_\numsplit\) and
  \(\covers_\locsplit\) are subcanonical, saturated pretopologies on
  \(\Top\) and~\(\Hausdorff\) satisfying Assumptions
  \textup{\ref{assum:local_cover}}, \textup{\ref{assum:two-three}},
  \textup{\ref{assum:covering_acts_basically}},
  and~\textup{\ref{assum:covering_acts_basically_weak}}.  They
  satisfy~\textup{\ref{assum:final}} if we exclude the empty
  topological space.
\end{proposition}

\begin{proof}
  Let \(f\colon Y\to X\) be a continuous map and let \(g\colon U\prto
  X\) belong to \(\covers_\splitsur\), \(\covers_\numsplit\) or
  \(\covers_\locsplit\), respectively.  Let \(W\subseteq X\) be the
  domain of a local continuous section \(\sigma\colon W\to U\)
  for~\(g\).  Then \((\id,\sigma\circ f)\colon f^{-1}(W)\to Y\times_X
  U\) is a local continuous section for the coordinate projection
  \(\pr_1\colon Y\times_X U \prto Y\).  If \(W=X\), then we get a
  global continuous section.  If the domains of such local sections
  cover~\(X\), then the subsets \(f^{-1}(W)\) cover~\(Y\).  If
  \((\psi_W)_{W\in\mathcal{U}}\) is a partition of unity on~\(X\) subordinate
  to a cover by domains of local continuous sections, then
  \((\psi_W\circ f)_{W\in\mathcal{U}}\) provides such a partition of
  unity for the domains of the induced local continuous sections for
  \(\pr_1\colon Y\times_X U \prto Y\).  Thus the map~\(\pr_1\) belongs
  to \(\covers_\splitsur\), \(\covers_\numsplit\) or
  \(\covers_\locsplit\), respectively, if~\(g\) does.

  It is clear that homeomorphisms are in~\(\covers_\locsplit\) and
  hence in the other classes as well.  If \(f_2\circ f_1\) is defined
  and belongs to~\(\covers_\locsplit\), then we may get local sections
  for~\(f_2\) by composing local sections for \(f_2\circ f_1\) with
  the continuous map~\(f_1\).  Since this does not change the domain,
  we may use the same partition of unity for~\(f_2\) as for~\(f_2\circ
  f_1\) in the numerably split case.  Hence all three pretopologies
  are saturated, which implies Assumptions \ref{assum:local_cover}
  and~\ref{assum:two-three}.  It is also clear that the constant map
  from any non-empty space to the one-point space has global
  continuous sections, which gives Assumption~\ref{assum:final} for
  all three pretopologies if we exclude the empty space.  A map with
  local continuous sections must be surjective, and we may use local
  continuous sections to lift convergent nets.  Hence
  \(\covers_\locsplit\subseteq\covers_\limlift\), so that the
  pretopologies are subcanonical.

  Let \(f\colon \Act\prto\Base\) be a cover in one of our three
  pretopologies, let~\(\Gr\) be its \v{C}ech groupoid, and
  let~\(\Act[Y]\) be a \(\Gr\)\nb-action.  Let \(\bunp\colon
  \Act[Y]\to\Act[Y]/\Gr\) be the orbit space projection.  By
  Proposition~\ref{pro:Top_basic}, the \(\Gr\)\nb-action is
  basic if and only if~\(\bunp\) is a cover.  The anchor map
  \(\s\colon \Act[Y]\to\Act\) of the \(\Gr\)\nb-action induces a
  continuous map \(\s/\Gr\colon \Act[Y]/\Gr \to \Act/\Gr=\Base\).

  Let~\(\mathcal{U}\) be an open cover of~\(\Base\) by domains of local
  sections \(\sigma_U\colon U\to\Act\) for \(U\in\mathcal{U}\); in the
  case of the pretopology~\(\covers_\numsplit\), we assume a partition
  of unity~\((\psi_U)_{U\in\mathcal{U}}\) subordinate to this cover,
  and in the case of the pretopology~\(\covers_\splitsur\), we assume
  \(\mathcal{U}=\{\Base\}\).  The subsets \(\s^{-1}(U)\) for
  \(U\in\mathcal{U}\) form an open cover~\(\s^*(\mathcal{U})\)
  of~\(\Act[Y]/\Gr\) because~\(\s/\Gr\) is continuous; in the case
  of~\(\covers_\numsplit\),
  \((\psi_U\circ(\s/\Gr))_{U\in\mathcal{U}}\) is a partition of unity
  on~\(\Act[Y]/\Gr\) subordinate to~\(\s^*(\mathcal{U})\),
  and in the case of~\(\covers_\splitsur\), \(\s^*(\mathcal{U})=
  \{\Act[Y]/\Gr\}\).  It remains to construct local continuous
  sections \(\sigma'_U\colon \s^{-1}(U)\to\Act[Y]\) for
  \(U\in\mathcal{U}\).

  Let \(y\in\Act[Y]\) represent \([y]\in\s^{-1}(U)\subseteq
  \Act[Y]/\Gr\).  Then \(\s(y)\in f^{-1}(U)\subseteq \Act\) and
  \(g_y\defeq (\s(y),\sigma_U(f(\s(y)))) \in\Gr^1\) has
  \(\rg(g_y)=\s(y)\), so that \(y\cdot g_y\) is defined
  in~\(\Act[Y]\).  This is the unique element in the
  \(\Gr\)\nb-orbit of~\(y\) with \(\s(y\cdot g_y)\in\sigma_U(U)\).
  Hence \(y\cdot g_y\) depends only on \([y]\in\Act[Y]/\Gr\), so
  that we get a continuous map \(\sigma'_U\colon \s^{-1}(U)\to
  \Act[Y]\), \([y]\mapsto y\cdot g_y\).  This is a continuous
  section for~\(\bunp\) on~\(\s^{-1}(U)\) as needed.
\end{proof}

Thus our whole theory applies to groupoids in~\(\Top\) for the three
pretopologies above.  The range and source maps of a
groupoid are automatically in~\(\covers_\splitsur\) because the unit
map provides a global continuous section.  Hence our pretopologies
do not restrict the class of topological groupoids.  They give
different notions of principal bundles, however, and thus different
bicategories of bibundle functors, equivalences, and actors.

Let~\(\Gr\) be a topological groupoid (without condition on the
range and source maps) and let~\(\Act\) be a \(\Gr\)\nb-action that
satisfies (1) and~(2) in Proposition~\ref{pro:Top_basic}.
Let \(\bunp\colon \Act\to\Act/\Gr\) be its orbit space projection.
This is a principal \(\Gr\)\nb-bundle for~\(\covers\) if and only if
\(\bunp\in\covers\).  We now discuss what this means.

\begin{lemma}
  \label{lem:trivial_G-bundle}
  The map \(\bunp\) is in~\(\covers_\splitsur\) if and only
  if~\(\Act\) is isomorphic to a pull-back of the principal
  \(\Gr\)\nb-bundle \(\rg\colon \Gr^1\prto\Gr^0\) along some map
  \(\Act/\Gr\to\Gr^0\).
\end{lemma}

\begin{proof}
  By definition, \(\bunp\in\covers_\splitsur\) if and only if it has
  a global continuous section \(\sigma\colon \Act/\Gr\to\Act\).  We
  claim that the map
  \begin{equation}
    \label{eq:trivialisation_principal_bundle}
    \Act/\Gr\times_{\s\circ\sigma,\Gr^0,\rg} \Gr^1 \to \Act,\qquad
    ([x],g)\mapsto \sigma[x]\cdot g,
  \end{equation}
  is a homeomorphism; it is clearly a continuous bijection, and the
  inverse map sends \(y\in\Act\) to \(([y],g)\) where~\(g\) is the
  unique element of~\(\Gr^1\) with \(\s(\sigma[y])=\rg(g)\) and
  \(\sigma[y]\cdot g=y\); the map \(y\mapsto ([y],g)\) is continuous
  by~(2) in Proposition~\ref{pro:Top_basic}.

  The homeomorphism~\eqref{eq:trivialisation_principal_bundle} is
  \(\Gr\)\nb-equivariant for the \(\Gr\)\nb-action on
  \(\Act/\Gr\times_{\s\circ\sigma,\Gr^0,\rg} \Gr^1\) defined by
  \(([x],g_1)\cdot g_2 = ([x],g_1\cdot g_2)\).  Hence the principal
  \(\Gr\)\nb-bundle \(\Act\prto\Act/\Gr\) is isomorphic to the
  pull-back of the principal \(\Gr\)\nb-bundle \(\rg\colon
  \Gr^1\prto\Gr^0\) along the continuous map \(\s\circ\sigma\colon
  \Act/\Gr\to\Gr^0\) (Proposition~\ref{pro:pull-back_principal}).
  Conversely, such pull-backs are principal \(\Gr\)\nb-bundles
  in~\((\Top,\covers_\splitsur)\) because \(\rg\colon
  \Gr^1\prto\Gr^0\) is a principal \(\Gr\)\nb-bundle
  in~\((\Top,\covers_\splitsur)\).
\end{proof}

Thus the principal
\(\Gr\)\nb-bundles in~\((\Top,\covers_\splitsur)\) are precisely the
\emph{trivial} \(\Gr\)\nb-bundles in the following sense:

\begin{definition}
  \label{def:trivial_G-bundle}
  A \(\Gr\)\nb-bundle \(\bunp\colon \Act\prto\Base\) is
  \emph{trivial} if it is isomorphic to a pull-back of the
  \(\Gr\)\nb-bundle \(\rg\colon \Gr^1\prto\Gr^0\) along some map
  \(\Base\to\Gr^0\).

  A \(\Gr\)\nb-bundle \(\bunp\colon \Act\prto\Base\) is
  \emph{locally trivial} if there is an open
  covering~\(\mathcal{U}\) of~\(\Base\) such that the restriction
  \(\bunp|_U\colon \bunp^{-1}(U)\prto U\) is a trivial
  \(\Gr\)\nb-bundle for each \(U\in\mathcal{U}\).
\end{definition}

If \(\covers=\covers_\locsplit\), then there is an open
covering~\(\mathcal{U}\) of~\(\Act/\Gr\) such that trivialisations
as in~\eqref{eq:trivialisation_principal_bundle} are defined on the
\(\Gr\)\nb-invariant subsets \(\bunp^{-1}(U)\) for
\(U\in\mathcal{U}\).  Lemma~\ref{lem:trivial_G-bundle} shows that
the principal \(\Gr\)\nb-bundles in~\((\Top,\covers_\locsplit)\) are
precisely the \emph{locally trivial} \(\Gr\)\nb-bundles.

Let \(\bunp\colon \Act\prto\Base\) be a locally trivial bundle.  We
have an open covering~\(\mathcal{U}\) of~\(\Base\) and continuous
maps \(\varphi_U\colon U\to\Gr^0\) and local trivialisations
\(\tau_U\colon \bunp^{-1}(U)\to U\times_{\varphi_U,\Gr^0,\rg}
\Gr^1\) for all \(U\in\mathcal{U}\).  These local trivialisations
are not unique and therefore do not agree on intersections \(U_1\cap
U_2\) for \(U_1,U_2\in\mathcal{U}\).  Since any \(\Gr\)\nb-map
between fibres of \(\rg\colon \Gr^1\prto\Gr^0\) is given by left
multiplication with some \(g\in\Gr^1\), we get continuous maps
\(g_{U_1,U_2}\colon U_1\cap U_2\to \Gr^1\) with \(\varphi_{U_2}(z) =
\rg(g_{U_1,U_2}(z))\), \(\varphi_{U_1}(z) = \s(g_{U_1,U_2}(z))\), and
\(\tau_{U_2}\circ\tau_{U_1}^{-1}(z,g) = (z,g_{U_1,U_2}(z)\cdot g)\) for
all \((z,g)\in (U_1\cap U_2)\times_{\varphi_{U_1},\Gr^0,\rg}
\Gr^1\).  These maps~\(g_{U_1,U_2}\) satisfy the usual cocycle
conditions.

If \(\covers=\covers_\numsplit\), then principal bundles are locally
trivial and, in addition, the open covering by trivialisation charts
is numerable.  This suffices to construct a continuous
\emph{classifying map} \(\varphi\colon \Act/\Gr\to B\Gr\) such that
\(\Act\prto\Act/\Gr\) is the pull-back of the universal principal
bundle \(E\Gr\prto B\Gr\) along~\(\varphi\), even if the base
space~\(\Act/\Gr\) is not paracompact.  Classifying spaces for
topological groupoids were introduced by
Haefliger~\cite{Haefliger:Feuilletages}, following
Segal~\cite{Segal:Classifying};
Bracho~\cite{Bracho:Haefliger_linear} proves that
\(\Gr\)\nb-principal bundles over locally compact topological spaces
are classified by continuous maps to~\(B\Gr\).  The local
compactness assumption should not be needed here, as shown by the
example of principal bundles over topological groups, but we have
not found a proof of this in the literature.

\begin{remark}
  \label{rem:sections_orbit_projection}
  If~\(\Act\) is Hausdorff, then continuous local sections for the
  orbit space projection \(\Act\to\Act/\Gr\) imply that~\(\Act/\Gr\)
  is locally Hausdorff, that is, every point has a Hausdorff open
  neighbourhood; but it does not yet imply that~\(\Act/\Gr\) is
  Hausdorff.  Indeed, Example~\ref{exa:covering_non-Hausdorff} is also
  a basic groupoid in~\((\Top,\covers_\locsplit)\).
\end{remark}

\subsubsection{Closed and proper maps}
\label{sec:Top_proper_maps}

A map between topological spaces is \emph{closed} if it maps closed
subsets to closed subsets.  A map \(f\colon X\to Y\) is proper if
and only if it is closed and \(f^{-1}(y)\subseteq X\) is
quasi-compact for all \(y\in Y\); this is shown in
\cite{Bourbaki:Topologie_generale}*{I.10.2, Théor\`eme 1}.  In
particular, there are closed maps that are not proper.  Thus closed
maps are not hereditary for pull-backs and do not form a
pretopology.

\begin{proposition}
  \label{pro:Top_proper_pretopology}
  The proper maps form a subcanonical
  pretopology~\(\covers_\proper\) on~\(\Top\).  It satisfies
  Assumptions \textup{\ref{assum:local_cover}} and
  \textup{\ref{assum:two-three}}, but not
  Assumption~\textup{\ref{assum:final}}.
\end{proposition}

\begin{proof}
  It is clear that homeomorphisms are proper and that composites of
  proper or closed maps are again proper or closed, respectively.
  Let \(g\colon U\to X\) be proper and let \(f\colon Y\to X\) be any
  map.  Let~\(Z\) be another topological space.  Then \(g_*\defeq
  \id_Y \times g\times \id_Z\colon Y\times U\times Z\to Y\times
  X\times Z\) is closed.  So is its restriction to any subset of the
  form \(g_*^{-1}(S)\) for \(S\subseteq Y\times X\times Z\).
  If~\(S\) is the product of the graph of~\(f\) and~\(Z\), then this
  restriction is the map \(\pr_1\times\id_Z\colon Y\times_{f,X,g} U
  \times Z \to Y\times Z\).  Since this is closed for any~\(Z\),
  \(\pr_1\colon Y\times_{f,X,g} U\to Y\) is proper.

  Proper surjections are biquotient maps by
  \cite{Michael:Bi-quotient}*{Proposition 3.2}.  Hence they are
  coequalisers, so the pretopology of proper maps is subcanonical.

  Let \(f_1\colon X\to Y\) and \(f_2\colon Y\to Z\) be continuous
  surjections.  If \(f_2\circ f_1\) is closed, then so is~\(f_2\):
  for a closed subset \(S\subseteq Y\), \(f_2(S) = (f_2\circ
  f_1)f_1^{-1}(S)\) because~\(f_1\) is surjective, and this is
  closed because~\(f_1\) is continuous and \(f_2\circ f_1\) is
  closed.  Hence the proper maps satisfy
  Assumption~\ref{assum:two-three} and thus
  Assumption~\ref{assum:local_cover}.  The map from a space~\(X\) to
  the one-point space is proper if and only if~\(X\) is
  quasi-compact.  Hence Assumption~\ref{assum:final} fails even if we
  exclude the empty space.
\end{proof}

For a groupoid in~\((\Top,\covers_\proper)\), the source and range
maps must be proper.  This restriction is so strong that we seem to
get a rather useless class of groupoids.  In particular, proper
topological groupoids need not be groupoids
in~\((\Top,\covers_\proper)\).

\subsubsection{Open surjections}
\label{sec:Top_open}

This pretopology is used for topological groupoids
in~\cite{Pronk:Etendues_fractions}, and is implicitly used by most of
the literature on groupoid \(\textup{C}^*\)\nb-algebras
for locally compact groupoids.

A map between topological spaces is \emph{open} if it maps open
subsets to open subsets.  The following criterion for open maps is
similar to but subtly different from the definition of limit lifting
maps:

\begin{proposition}[\cite{Williams:crossed-products}*{Proposition
    1.15}]
  \label{pro:open_lifting_criterion}
  A continuous surjection \(f\colon X\to Y\) between
  topological spaces is open if and only if, for any \(x\in X\), a
  convergent net~\((y_i)\) in~\(Y\) with \(\lim y_i=f(x)\) lifts to
  a net in~\(X\) converging to~\(x\).
\end{proposition}

\begin{proposition}
  \label{pro:Top_open_surjections}
  The class~\(\covers_\open\) of surjective, open maps in~\(\Top\)
  is a subcanonical pretopology that satisfies Assumptions
  \textup{\ref{assum:local_cover}}, \textup{\ref{assum:two-three}},
  \textup{\ref{assum:covering_acts_basically}}
  and~\textup{\ref{assum:covering_acts_basically_weak}}, but it is
  not saturated.  It satisfies Assumption~\textup{\ref{assum:final}}
  if we exclude the empty space.
\end{proposition}

The proof that~\(\covers_\open\) is a subcanonical pretopology
satisfying Assumptions \ref{assum:local_cover},
\ref{assum:two-three} and~\ref{assum:final} is similar to the proof
for~\(\covers_\limlift\), using
Proposition~\ref{pro:open_lifting_criterion}, and left to the
reader.  The following simple counterexample shows that it is not
saturated:

\begin{example}
  \label{exa:non-saturated}
  Let \(f\colon X\to Y\) be a continuous map that is not open.  Let
  \(f_2\defeq (f,\id_X)\colon X\sqcup Y\to Y\) and let \(f_1\colon
  Y\to X\sqcup Y\) be the inclusion map.  Then \(f_2\circ f_1\) is
  the identity map on~\(Y\), hence open.  But~\(f_2\) is not open
  because~\(f\) is not open.  Hence the
  pretopology~\(\covers_\open\) is not saturated.
\end{example}

Groupoids in~\((\Top,\covers_\open)\) must have open range and
source maps.  This is no serious restriction for operator
algebraists because they need Haar systems to construct groupoid
\(\textup{C}^*\)\nb-algebras, and Haar systems cannot exist unless
the range and source maps are open.  The benefit of this assumption
is that it forces the orbit space projections \(\Act\to\Act/\Gr\)
for all \(\Gr\)\nb-actions to be open:

\begin{proposition}
  \label{pro:open_orbit_space_projection}
  Let\/~\(\Gr\) be a groupoid and\/~\(\Act\) a \(\Gr\)\nb-action
  in~\((\Top,\covers_\open)\).  The orbit space projection
  \(\bunp\colon \Act\to \Act/\Gr\) is open.  All actions of \v{C}ech
  groupoids in~\((\Top,\covers_\open)\) are basic.
\end{proposition}

\begin{proof}
  Let \(U\subseteq\Act\) be open.  Then \(\bunp^{-1}(\bunp(U))\) is
  the set of all \(x\cdot g\) with \(x\in U\), \(g\in \Gr^1\),
  \(\s(x)=\rg(g)\).  This is \(\mul\bigl((U\times \Gr^1)\cap
  (\Act\times_{\s,\Gr^0,\rg} \Gr^1)\bigr)\), which is open because
  \((U\times \Gr^1)\cap (\Act\times_{\s,\Gr^0,\rg} \Gr^1)\) is open
  in \(\Act\times_{\s,\Gr^0,\rg} \Gr^1\) and~\(\mul\) is a cover
  (open surjection) by
  Proposition~\ref{pro:unit_inverse_from_basicality}.
  Thus~\(\bunp(U)\) is open in~\(\Act/\Gr\), and~\(\bunp\) is open.
  Proposition~\ref{pro:Top_basic} characterises basic actions by
  three conditions, of which two are automatic for actions of
  \v{C}ech groupoids.  The third condition is that~\(\bunp\) should
  be a cover, which \emph{is} automatic for the
  pretopology~\(\covers_\open\).  Thus all \v{C}ech groupoid actions
  in~\((\Top,\covers_\open)\) are basic, verifying Assumptions
  \ref{assum:covering_acts_basically}
  and~\ref{assum:covering_acts_basically_weak}.
\end{proof}

This finishes the proof of
Proposition~\ref{pro:Top_open_surjections}.

\begin{corollary}
  \label{cor:basic_actions_open}
  Let~\(\Gr\) be a groupoid and~\(\Act\) a \(\Gr\)\nb-action
  in~\((\Top,\covers_\open)\).  The following are equivalent:
  \begin{enumerate}
  \item the action is basic with Hausdorff base space~\(\Act/\Gr\);
  \item the action is free and proper.
  \end{enumerate}
\end{corollary}

\begin{proof}
  This follows from the proof of
  Proposition~\ref{pro:Hausdorff_basic}, which still works even if
  \(\Gr\) or~\(\Act\) are non-Hausdorff.
\end{proof}

The proof of Proposition~\ref{pro:Hausdorff_basic} shows that an
action is free and proper if and only if it satisfies the first two
conditions in Proposition~\ref{pro:Top_basic} and
\(\Act\times_{\bunp,\Act/\Gr,\bunp}\Act\) is closed in
\(\Act\times\Act\).  This is exactly how Henri Cartan defines
principal fibre bundles for a topological group~\(G\) in
\cite{Cartan:Espaces_Fibres}*{condition (FP) on page 6-05}:
\begin{quotation}%\selectlanguage{french}
  Un espace fibré principal~\(E\) est un espace topologique~\(E\),
  o\`u op\`ere un groupe topologique~\(G\) (appelé \emph{groupe
    structural}), de mani\`ere que soit rempli l'axiome suivant~:

  (FP)~: le graphe \(C\) de la relation d'équivalence définie
  par~\(G\) est une partie \emph{fermée} de \(E\times E\)~; pour
  chaque point \((x,y)\in C\) il existe un \(s\in G\) et un seul tel
  que \(x\cdot s=y\)~; en outre, en désignant par~\(u\)
  l'application de~\(C\) dans~\(G\) ainsi définie, on suppose
  que~\(u\) est une fonction \emph{continue}.
\end{quotation}
Thus our notion of principal bundle in~\((\Top,\covers_\open)\) is
the same as Henri Cartan's.  Palais~\cite{Palais:Slices} and later
authors (\cites{Clark:Types_principal, Huef:Transformation_Fell})
call a locally compact group(oid) action \emph{Cartan} if it
satisfies a condition that, for free actions, is equivalent to the
continuity of the map \((x,x\cdot g)\mapsto g\), allowing
non-Hausdorff orbit spaces.  Cartan himself, however, requires the
orbit space to be Hausdorff because he requires the orbit
equivalence relation to be closed
(Proposition~\ref{pro:Hausdorff_quotient}).

\subsubsection{Covers with many local continuous sections}
\label{sec:Top_other_covers}

\begin{definition}
  \label{def:local_section}
  A continuous map \(f\colon X\to Y\) \emph{has many local continuous
    sections} if it is surjective and for all \(x\in X\) there is an
  open neighbourhood \(U\subseteq Y\) of~\(f(x)\) and a continuous map
  \(\sigma\colon U\to X\) with \(\sigma(f(x))=x\) and
  \(f(\sigma(y))=y\) for all \(y\in U\).  Let~\(\covers_\locsect\) be
  the class of continuous maps with many local continuous sections.
\end{definition}

The difference between this pretopology and~\(\covers_\locsplit\) is
that we require local sections with a given \(x\in X\) in the image.
This forces the map~\(f\) to be open, so
\(\covers_\locsect\subseteq\covers_\open\), whereas
\(\covers_\locsplit\not\subseteq \covers_\open\).

\begin{proposition}
  \label{pro:Top_locsect}
  The class~\(\covers_\locsect\) is a subcanonical pretopology
  on~\(\Top\) that satisfies Assumptions
  \textup{\ref{assum:local_cover}}, \textup{\ref{assum:two-three}},
  \textup{\ref{assum:covering_acts_basically}},
  and~\textup{\ref{assum:covering_acts_basically_weak}}, but is not
  saturated.  It satisfies Assumption~\textup{\ref{assum:final}} if we
  exclude the empty space.
\end{proposition}

\begin{proof}
  The proof is very similar to the proof for the
  pretopology~\(\covers_\locsplit\)
  (Proposition~\ref{pro:Top_numsplit_pretopology}) and therefore
  left to the reader.  The idea of Example~\ref{exa:non-saturated}
  also shows that~\(\covers_\locsect\) is not saturated.  Since
  \(\covers_\locsect\) is contained both in \(\covers_\locsplit\)
  and~\(\covers_\open\) and both are subcanonical and have the
  property that actions of \v{C}ech groupoids are basic, Assumptions
  \ref{assum:covering_acts_basically} and
  hence~\ref{assum:covering_acts_basically_weak} follow from
  Proposition~\ref{pro:smaller_coverings_assum}.
\end{proof}

The topological groupoids in~\((\Top,\covers_\locsect)\) are those
topological groupoids for which~\(\rg\) or, equivalently, \(\s\) has
many local continuous sections.  This happens for both maps once it
happens for one because of the continuous inversion on~\(\Gr^1\).

\begin{proposition}
  \label{pro:many_versus_some_local_sections}
  Let~\(\Gr\) be a groupoid and~\(\Act\) a \(\Gr\)\nb-action
  in~\((\Top,\covers_\locsect)\).  Then~\(\Act\) is basic
  in~\((\Top,\covers_\locsect)\) if and only if it is basic
  in~\((\Top,\covers_\locsplit)\).
\end{proposition}

\begin{proof}
  Since \(\rg\colon \Gr^1\prto\Gr^0\) is a principal
  \(\Gr\)\nb-bundle in~\((\Top,\covers_\locsect)\), so are its
  pull-backs, that is, all trivial bundles.  Since the existence of
  many local continuous sections is a local condition, this extends
  to locally trivial \(\Gr\)\nb-bundles.  Since the principal
  \(\Gr\)\nb-bundle in~\((\Top,\covers_\locsplit)\) are precisely
  the locally trivial \(\Gr\)\nb-bundles, they are principal
  in~\((\Top,\covers_\locsect)\) as well.  The same then holds for
  basic actions.
\end{proof}

Roughly speaking, if the map \(\rg\colon \Gr^1\prto\Gr^0\) has many
local continuous sections, then we may shift local continuous
sections for the quotient map \(\Act\to\Act/\Gr\) so that a particular
point in a fibre is in the image of the local section.

Proposition~\ref{pro:many_versus_some_local_sections} implies that
the bicategories of bibundle functors, equivalences, and
actors for~\((\Top,\covers_\locsect)\) are full
sub-bicategories of the corresponding bicategories
for~\((\Top,\covers_\locsplit)\); we only restrict from all
topological groupoids to those where the range and source maps have
many local continuous sections.

\subsubsection{Étale surjections}
\label{sec:Top_etale}

A continuous map \(f\colon X\to Y\) is called \emph{étale} (or local
homeomorphism) if
for all \(x\in X\) there is an open neighbourhood~\(U\) such that
\(f(U)\) is open and \(f|_U\colon U\to f(U)\) is a homeomorphism for
the subspace topologies on \(U\) and~\(f(U)\) from \(X\) and~\(Y\),
respectively.  This implies that~\(f\) is open.
Let~\(\covers_\etale\) be the class of étale surjections.

\begin{proposition}
  \label{pro:Top_etale}
  The class~\(\covers_\etale\) is a subcanonical pretopology
  on~\(\Top\) that satisfies Assumptions
  \textup{\ref{assum:local_cover}}, \textup{\ref{assum:two-three}},
  \textup{\ref{assum:covering_acts_basically}}
  and~\textup{\ref{assum:covering_acts_basically_weak}}; it is
  not saturated and Assumption~\textup{\ref{assum:final}} fails.
\end{proposition}

\begin{proof}
  The proof is very similar to that of
  Proposition~\ref{pro:Top_numsplit_pretopology} and therefore left to
  the reader.  Unless~\(X\) is discrete, the constant map from~\(X\)
  to a point is not étale, so Assumption~\ref{assum:final} fails
  for~\(\covers_\etale\) even if we exclude the empty space.  The idea
  of Example~\ref{exa:non-saturated} shows that~\(\covers_\etale\) is
  not saturated.
\end{proof}

The groupoids in~\((\Top,\covers_\etale)\) are precisely the étale
topological groupoids.

\begin{proposition}
  \label{pro:basic_for_etale_cover}
  An action of an étale topological groupoid is basic
  in~\((\Top,\covers_\etale)\) if and only if it satisfies the first
  two conditions in Proposition~\textup{\ref{pro:Top_basic}}.  A
  groupoid
  action in~\((\Hausdorff,\covers_\etale)\) is basic if and only if
  it is free and proper.
\end{proposition}

\begin{proof}
  Let~\(\Gr\) be an étale topological groupoid and let~\(\Act\) be
  a \(\Gr\)\nb-action in~\(\Top\) that satisfies the first two
  conditions in Proposition~\ref{pro:Top_basic}.  We must prove that
  the canonical map \(\bunp\colon \Act\to\Act/\Gr\) is étale,
  that is, the third condition in Proposition~\ref{pro:Top_basic} is
  automatic for the pretopology~\(\covers_\etale\).  The
  characterisation of basic actions
  in~\((\Hausdorff,\covers_\etale)\) then follows from
  Proposition~\ref{pro:Hausdorff_basic}.

  The first two conditions in Proposition~\ref{pro:Top_basic} say
  that the map
  \[
  \Act\times_{\s,\Gr^0,\rg} \Gr^1 \to
  \Act\times_{\bunp,\Act/\Gr,\bunp} \Act,
  \qquad
  (x,g)\mapsto (x,x\cdot g),
  \]
  is a homeomorphism.  Since~\(\Gr\) is étale, the set of units
  \(u(\Gr^0)\) is open in~\(\Gr^1\).  Its image in
  \(\Act\times_{\bunp,\Act/\Gr,\bunp} \Act\) is the diagonal
  \(\{(x_1,x_2)\in\Act\times \Act \mid x_1=x_2\}\).  Hence any
  \(x\in\Act\) has an open neighbourhood \(U\subseteq\Act\) such that
  \((U\times U)\cap (\Act\times_{\bunp,\Act/\Gr,\bunp} \Act)\) is the
  diagonal in~\(U\).  This means that for \(x_1,x_2\in U\),
  \(\bunp(x_1)=\bunp(x_2)\) only if \(x_1=x_2\).  Thus~\(\bunp\) is
  injective on the open subset \(U\subseteq\Act\).  Since~\(\Gr\) is
  an étale groupoid, its range and source maps are open.
  Hence~\(\bunp\) is open by
  Proposition~\ref{pro:open_orbit_space_projection}.  Its restriction
  to~\(U\) is injective, open and continuous, hence a homeomorphism
  onto an open subset of~\(\Act/\Gr\).
\end{proof}

\subsection{Smooth manifolds of finite and infinite dimension}
\label{sec:Mfd}

We consider the following categories of smooth manifolds of increasing
generality:
\begin{center}
  \begin{tabular}{ll}
    \(\Mfd_\fin\)&finite-dimensional manifolds;\\
    \(\Mfd_\Hil\)&Hilbert manifolds;\\
    \(\Mfd_\Ban\)&Banach manifolds;\\
    \(\Mfd_\Fre\)&Fréchet manifolds;\\
    \(\Mfd_\lcs\)&locally convex manifolds.
  \end{tabular}
\end{center}
Such manifolds are Hausdorff topological spaces that are locally
homeomorphic to finite-dimensional vector spaces, Hilbert spaces,
Banach spaces, Fréchet spaces, or locally convex topological vector
spaces, respectively.  Paracompactness is a standard assumption for
finite-dimensional manifolds, but not for infinite-dimensional ones.
The morphisms between all these types of manifolds are smooth maps,
meaning maps given in local charts by smooth maps between topological
vector spaces.  Banach manifolds are treated in Lang's
textbook~\cite{Lang:Fundamentals_diffgeo}, Fréchet manifolds in
\cite{Hamilton:Nash-Moser}*{Section~I.4}, and locally convex manifolds
in \cite{Nikolaus-Sachse-Wockel:String}*{Appendix~A} and
\cite{Wockel-Zhu:Integrating}*{Appendix~B}.  The covers are, in each
case, the \emph{surjective submersions} in the following sense:

\begin{definition}[see \cite{Hamilton:Nash-Moser}*{Definition 4.4.8}
  and \cite{Nikolaus-Sachse-Wockel:String}*{Appendix A}]
  \label{def:submersion}
  Let \(X\) and~\(Y\) be locally convex manifolds.  A smooth map is a
  \emph{submersion} if for each \(x\in X\), there is an open
  neighbourhood~\(V\) of~\(x\) in~\(X\) such that \(U\defeq f(V)\) is
  open in~\(Y\), and there are a smooth manifold~\(W\) and a
  diffeomorphism \(V\cong U\times W\) that intertwines~\(f\) and the
  coordinate projection \(\pr_1\colon U\times W\to U\).
\end{definition}

Choosing \(U\) small enough, we may achieve that \(V\), \(U\)
and~\(W\) are locally convex topological vector spaces.  So
submersions are maps that locally look like projections onto direct
summands in locally convex topological vector spaces.

\begin{proposition}
  \label{pro:surj_subm_pretopology}
  Surjective submersions form a subcanonical pretopology on
  \(\Mfd_\lcs\), \(\Mfd_\Fre\), \(\Mfd_\Ban\), \(\Mfd_\Hil\) and
  \(\Mfd_\fin\).  In each case, Assumptions
  \textup{\ref{assum:local_cover}}
  and~\textup{\ref{assum:covering_acts_basically_weak}} hold, and
  Assumption~\textup{\ref{assum:final}} holds if we exclude the empty
  manifold.  The pretopology is not saturated.
\end{proposition}

\begin{proof}
  That surjective submersions form a pretopology on locally convex
  manifolds is shown in \cites{Nikolaus-Sachse-Wockel:String,
    Wockel-Zhu:Integrating}.  We briefly recall the argument below.
  It also works in the subcategories \(\Mfd_\fin\), \(\Mfd_\Hil\),
  \(\Mfd_\Ban\), and \(\Mfd_\Fre\); more generally, we may use
  manifolds based on any class~\(\Cat\) of topological vector spaces
  (not necessarily locally convex) that is closed under taking finite
  products and closed subspaces (this ensures that fibre products of
  manifolds locally modelled on topological vector spaces in~\(\Cat\)
  still have local models in~\(\Cat\)).

  It is routine to check that isomorphisms are surjective submersions
  and that composites of surjective submersions are again surjective
  submersions.  If \(f_i\colon M_i\to N\) for \(i=1,2\) are smooth
  maps and~\(f_1\) is a submersion, then the fibre-product
  \(M_1\times_N M_2\) is a smooth submanifold of \(M_1\times M_2\),
  which satisfies the universal property of a fibre product.  With our
  definition of smooth submersion, the proof is rather trivial, see
  \cite{Hamilton:Nash-Moser}*{Theorem 4.4.10} for the case of
  Fréchet manifolds and
  \cite{Nikolaus-Sachse-Wockel:String}*{Proposition A.3} for the
  general case.  The proof also shows that \(\pr_2\colon M_1\times_N
  M_2\to M_2\) is a (surjective) submersion if~\(f_1\) is one,
  so~\(\covers_\subm\) is a pretopology.  Our categories of smooth
  manifolds do not admit fibre products in general, so we must be
  careful about their representability.

  Next we show that~\(\covers_\subm\) is subcanonical.  A submersion
  \(f\colon X\to Y\) is open because it is locally open.  So our
  covers are open surjections and hence quotient maps of the
  underlying topological spaces, since open surjections form a
  subcanonical pretopology by
  Proposition~\ref{pro:Top_open_surjections}.  Thus any smooth map
  \(X\to Z\) that equalises the coordinate projections \(X\times_Y
  X\rightrightarrows X\) factors through a continuous map \(f'\colon
  Y\to Z\).  Since any submersion~\(f\) admits smooth local sections,
  this map~\(f'\) is smooth if and only if~\(f\) is smooth.
  Hence~\(f\) is the equaliser of the coordinate projections
  \(X\times_Y X\rightrightarrows X\) also in the category of smooth
  manifolds, so the pretopology~\(\covers_\subm\) is subcanonical.

  The final object is the zero-dimensional one-point manifold, and any
  map to it from a non-empty manifold is a surjective submersion.  So
  Assumption~\ref{assum:final} is trivial for all our categories of
  smooth manifolds if we exclude the empty manifold.  The idea of
  Example~\ref{exa:non-saturated} shows that~\(\covers_\subm\) is not
  saturated on any of our categories.

  Now we check Assumption~\ref{assum:local_cover}.  Let \(f\colon
  X\prto U\) be a surjective submersion and let \(g\colon Y\to U\) be
  a smooth map.  We have already shown that surjective submersions
  form a pretopology.  So the fibre product \(W\defeq X\times_{f,U,g}
  Y\) exists and the projection \(\pr_2\colon W\prto Y\) is a
  surjective submersion.  Assume that \(\pr_1\colon W\prto X\) is a
  surjective submersion as well.  We are going to show that~\(g\) is a
  surjective submersion.

  It is clear that~\(g\) is surjective.  The submersion condition is
  local, so we must check it near some \(y\in Y\).  Since~\(f\) is
  surjective, there is some \(x\in X\) with \(f(x)=g(y)\), that is,
  \((x,y)\in W\).  Since~\(f\) is a submersion, there are
  neighbourhoods~\(X_0\) of \(x\) in~\(X\) and~\(U_0\) of~\(f(x)\)
  in~\(U\), a smooth manifold~\(\Xi_0\), and a diffeomorphism
  \(\alpha\colon \Xi_0\times U_0 \congto X_0\) such that
  \(f\circ\alpha(\xi,u) = u\) for all \(\xi\in\Xi_0\), \(u\in U_0\).
  Let \((\xi_0,u_0)=\alpha^{-1}(x)\).  The preimage \(W_0\defeq
  \pr_1^{-1}(X_0)\) in~\(W\) is an open submanifold.  It is
  diffeomorphic to the fibre product \((\Xi_0\times U_0) \times_{U_0}
  Y \cong \Xi_0\times Y_0\) with \(Y_0\defeq g^{-1}(U_0)\)
  because~\(W\) is a fibre product.  The restriction of~\(\pr_1\)
  to~\(W_0\) is still a surjective submersion \(W_0\to X_0\)
  because~\(\pr_1\) is one and the condition for submersions is local.
  Hence we have improved our original fibre-product situation to one
  where~\(f\) is the projection map \(\pr_1\colon \Xi_0\times U_0\to
  U_0\).  Thus \(W_0= \Xi_0\times Y_0\) is an ordinary product and
  \(\pr_1\colon W_0\to X_0\) becomes the map
  \[
  \id\times g_0\colon \Xi_0\times Y_0 \to \Xi_0\times U_0,\qquad
  (\xi,y_0)\mapsto (\xi,g_0(y_0)),
  \]
  where \(g_0\colon Y_0\to U_0\) is the restriction of~\(g\).

  The map \(\id\times g_0\) is a submersion by assumption.  Hence
  there are neighbourhoods~\(W_1\) around~\((\xi_0,y_0)\)
  in~\(\Xi_0\times Y_0\) and~\(X_1\) around~\((\xi_0,u_0)\)
  in~\(\Xi_0\times U_0\), a smooth manifold~\(\Omega_1\), and a
  diffeomorphism \(\beta\colon W_1\congto \Omega_1\times X_1\) that
  intertwines~\(\id\times g_0\) and the coordinate projection
  \(\Omega_1\times X_1\to X_1\).  We may shrink~\(X_1\) to be of
  product type: \(X_1=\Xi_1\times U_1\) for open submanifolds
  \(\Xi_1\subseteq \Xi_0\) and \(U_1\subseteq U_0\) and
  replace~\(W_1\) by the preimage of \(\Omega_1\) times the
  new~\(X_1\), so we may assume without loss of generality
  that~\(X_1\) is of product type.  If \((\xi_1,y_1)\in W_1\), then
  \[
  \beta(\xi_1,y_1)=(\omega_1(\xi_1,y_1),\xi_1,g_0(y_1))
  \]
  for some function \(\omega_1\colon W_1\to \Omega_1\): the second two
  components of~\(\beta\) are \(\id\times g_0\) by construction.  The
  subset
  \[
  Y_1\defeq \{y\in Y_0\mid (\xi_0,y)\in W_1\}
  \]
  is open in~\(Y_0\).  We claim that \(\gamma(y_1) \defeq
  (\omega_1(\xi_0,y_1),g_0(y_1))\) defines a diffeomorphism
  from~\(Y_1\) onto \(\Omega_1\times U_1\); indeed, if
  \((\omega_1,u_1)\in \Omega_1\times U_1\), then
  \(\beta^{-1}(\omega_1,\xi_0,u_1)\) must be of the form
  \((\xi_0,y_1)\) for some \(y_1\in Y_0\) with \((\xi_0,y_1)\in W_1\)
  because the \(\Xi_0\)\nb-component of \(\beta(\xi_1,y_1)\)
  is~\(\xi_1\).  Hence the smooth map
  \(\beta^{-1}|_{\Omega_1\times\{\xi_0\}\times U_1}\) is inverse
  to~\(\gamma\).  The restriction of~\(g_0\) to~\(Y_1\) becomes the
  coordinate projection \(\Omega_1\times U_1 \to U_1\) by
  construction.  Hence \(g\colon Y\to U\) satisfies the submersion
  condition near~\(y\).  This finishes the proof that
  Assumption~\ref{assum:local_cover} holds for~\(\covers_\subm\).

  Now we check Assumption~\ref{assum:covering_acts_basically_weak}.
  Let \(\bunp\colon X\prto Z\) be a surjective submersion.  Let~\(\Gr\)
  be its \v{C}ech groupoid.  Let~\(Y\) be a sheaf over~\(\Gr\), that
  is, a \(\Gr\)\nb-action that has a surjective submersion \(\s\colon
  Y\prto X\) as anchor map.  We claim that the \(\Gr\)\nb-action
  on~\(Y\) is basic.  That is, the orbit space~\(Y/\Gr\) is a smooth
  manifold, the orbit space projection \(Y\to Y/\Gr\) is a surjective
  submersion, and the shear map \(Y\times_{\s,\Gr^0,\rg} \Gr^1 \to
  Y\times_{Y/G} Y\), \((y,g)\mapsto (y,y\cdot g)\), is a
  diffeomorphism.  Since surjective submersions are open, \(G\)
  becomes a \v{C}ech groupoid in \((\Hausdorff,\covers_\open)\).  We
  know that Assumption~\ref{assum:covering_acts_basically} holds in
  that case, so~\(Y/G\) is a Hausdorff topological space, the orbit
  space projection \(Y\to Y/G\) is a continuous open surjection, and
  the shear map is a homeomorphism.

  Since \(\bunp\) and~\(\s\) are covers, so is \(\bunp\circ\s\), so
  the fibre product \(Y\times_Z Y\) is a smooth submanifold of~\(Y\).
  If \((y_1,y_2)\in Y\times_Z Y\), then \((\s(y_1),\s(y_2))\in
  X\times_Z X=\Gr^1\), with range and source \(\s(y_1)\)
  and~\(\s(y_2)\), respectively.  Thus there is a well-defined smooth
  map
  \[
  \psi\colon Y\times_Z Y \to Y\times_X \Gr^1,\qquad
  (y_1,y_2)\mapsto (y_1, (\s(y_1),\s(y_2))).
  \]
  Composing it with the shear map gives the map
  \[
  \varphi\colon Y\times_Z Y \to Y\times_Z Y,\qquad
  (y_1,y_2)\mapsto (y_1, y_1\cdot (\s(y_1),\s(y_2))).
  \]
  Since \(\s(y_1\cdot (\s(y_1),\s(y_2))) = \s(y_1)\) and
  \((\s(y_1),\s(y_1))\) is an identity arrow, we get
  \(\varphi^2=\varphi\).  If \(y_1,y_2\) are in the same
  \(\Gr\)\nb-orbit, that is, \(y_1\cdot g=y_2\) for some \(g\in\Gr^1 =
  X\times_Z X\), then \(\rg(g)=\s(y_1)\) and \(\s(g)=\s(y_2)\), so
  \(g=(\s(y_1),\s(y_2))\).  Thus \(\bunp(\s(y_1))=\bunp(\s(y_2))\) is
  necessary for \(y_1\) and~\(y_2\) to have the same orbit (we write
  \(y_1\Gr=y_2\Gr\)).  If this necessary condition is satisfied, then
  the only arrow in~\(\Gr\) that has a chance to map~\(y_1\)
  to~\(y_2\) is \((\s(y_1),\s(y_2))\), so \(y_1\Gr= y_2\Gr\) if and
  only if \(y_1\cdot(\s(y_1),\s(y_2)) = y_2\).  This is equivalent to
  \(\varphi(y_1,y_2)=(y_1,y_2)\).  Thus the image of~\(\varphi\) is
  the subspace \(Y\times_{Y/G} Y\) of~\(Y\times_Z Y\), and the
  restriction of~\(\psi\) to~\(Y\times_{Y/G} Y\) is a smooth inverse
  for the shear map.  More precisely, this is a smooth inverse once we
  know that \(Y\times_{Y/G} Y\) is a smooth submanifold of \(Y\times
  Y\) and hence a smooth submanifold of \(Y\times_Z Y\).  Thus the
  condition on the shear map is automatic in our case, and we only
  have to construct a smooth structure on~\(Y/G\) such that the
  projection \(\pi\colon Y\to Y/G\) is a submersion.

  We first restrict to small open subsets where~\(\bunp\) and~\(\s\)
  are of product type.  Fix \(y\in Y\).  We may choose
  neighbourhoods~\(Y_0\) of~\(y\) in~\(Y\), \(X_0\) of~\(\s(y)\)
  in~\(X\) and~\(Z_0\) of \(\bunp(\s(y))\) in~\(Z\), smooth manifolds
  \(\Omega_0\) and~\(\Xi_0\), and diffeomorphisms \(Y_0\cong
  \Omega_0\times X_0\), \(X_0\cong \Xi_0\times Z_0\), such that the
  maps~\(\s\) and~\(\bunp\) become the projections to the second
  coordinate, respectively.  We identify \(Y_0=\Omega_0\times
  \Xi_0\times Z_0\) and \(X_0=\Xi_0\times Z_0\), so that
  \(\s(\omega_0,\xi_0,z_0) = (\xi_0,z_0)\) and \(\bunp(\xi_0,z_0) =
  z_0\) for all \(\omega_0\in \Omega_0\), \(\xi_0\in \Xi_0\), \(z_0\in
  Z_0\).  Let~\((\omega,\xi,z)\) be the coordinates of the chosen
  point \(y\in Y\).

  What happens to the \(\Gr\)\nb-action on~\(Y\) in these coordinates?
  Let \((\omega_0,\xi_0,z_0)\in Y_0\).  Then
  \(\s(\omega_0,\xi_0,z_0)=(\xi_0,z_0)\).  An arrow in \(G\cap
  (X_0\times_{Z_0} X_0)\) with range \((\xi_0,z_0)\) is a pair
  \(((\xi_0,z_0),(\xi_1,z_0))\), which we identify with the triple
  \((\xi_0,z_0,\xi_1)\).  If~\(\xi_1\) is sufficiently close
  to~\(\xi_0\), then \((\omega_0,\xi_0,z_0)\cdot (\xi_0,z_0,\xi_1)\)
  belongs to~\(Y_0\) and may be written in our coordinates.  Since
  \(\s((\omega_0,\xi_0,z_0)\cdot (\xi_0,z_0,\xi_1)) = (\xi_1,z_0)\),
  we get
  \[
  (\omega_0,\xi_0,z_0)\cdot (\xi_0,z_0,\xi_1) =
  (\varphi(\omega_0,\xi_0,z_0,\xi_1),\xi_1,z_0)
  \]
  for a smooth function~\(\varphi\) defined on some
  neighbourhood~\(W_1\) of \((\omega_0,\xi_0,z_0,\xi_0)\) in
  \(\Omega_0\times \Xi_0\times Z_0\times \Xi_0\).  Let \(Y_1 \defeq
  \{y\in Y_0 \mid (\omega_0,\xi_0,z_0,\xi) \in W_1\}\), where~\(\xi\)
  is the \(\Xi_0\)\nb-coordinate of our chosen point~\(y\).  Define
  \[
  \alpha\colon Y_1\to \Omega_0\times \Xi_0\times Z_0,
  \qquad (\omega_0,\xi_0,z_0)\mapsto
  (\varphi(\omega_0,\xi_0,z_0,\xi),\xi_0,z_0).
  \]
  In other words, \(\alpha(\omega_0,\xi_0,z_0)\) takes the original
  \(\Xi_0\)-coordinate and takes the other coordinates from
  \((\omega_0,\xi_0,z_0)\cdot (\xi_0,z_0,\xi)\).  Since
  \((\xi_1,z_0,\xi_2)\cdot (\xi_2,z_0,\xi_1) = (\xi_1,z_0,\xi_1)\) is
  a unit arrow and thus acts identically on~\(Y\) for any
  \(\xi_1,\xi_2\in \Xi_0\), \(z_0\in Z_0\), we get
  \[
  \alpha^{-1}(\omega_0,\xi_0,z_0)
  = (\varphi(\omega_0,\xi,z_0,\xi_0),\xi_0,z_0)
  \]
  on some neighbourhood of~\((\omega,\xi,z)\).  Thus~\(\alpha\) is a
  diffeomorphism from some neighbourhood~\(Y_1\) of~\(y\) in~\(Y_0\)
  onto a neighbourhood \(\Omega_1\times \Xi_1\times Z_1\) of
  \(\alpha(y)=(\omega,\xi,z)\) in~\(\Omega_0\times \Xi_0\times Z_0\).

  When are \(y_1,y_2\in Y_1\) in the same \(\Gr\)\nb-orbit?  Write
  \(\alpha(y_i)= (\omega_i,\xi_i,z_i)\).  Then \(\s(y_i)=(\xi_i,z_i)\)
  and \(\bunp\s(y_i)=z_i\) in local coordinates.  Since
  \(\bunp\s(y_1)=\bunp\s(y_2)\) is necessary for \(y_1\Gr=y_2\Gr\), we
  have \(z_1=z_2\) if \(y_1\Gr=y_2\Gr\).  Assume this, then
  \((\s(y_1),\s(y_2))\in X\times_Z X=\Gr^1\), so we have
  \(y_1\Gr=y_2\Gr\) if and only if \(y_1\cdot(\s(y_1),\s(y_2))=y_2\)
  by the discussion above.  This is equivalent to \(z_1=z_2\) and
  \[
  y_1\cdot (\s(y_1),(z_1,\xi)) = y_2\cdot (\s(y_2),(z_2,\xi)),
  \]
  where~\(\xi\) is the \(\Xi_0\)-coordinate of our fixed point~\(y\).
  And this is equivalent to \(\pr_{13}\circ \alpha(y_1) =
  \pr_{13}\circ \alpha(y_2)\).

  As a consequence, the restriction of the quotient map \(Y\to Y/G\)
  to \(\alpha^{-1}(\Omega_1\times\{\xi\}\times U_1)\) is a
  homeomorphism onto its image; this image is the same as the image
  of~\(Y_1\), which is an open subset in~\(Y/G\) because~\(Y_1\) is
  open and the quotient map \(Y\to Y/G\) is open.  Furthermore, with
  this coordinate chart, the restriction of the projection \(Y\to
  Y/G\) to~\(Y_1\) becomes the coordinate projection \(\pr_{13}\colon
  \Omega_1\times \Xi_1\times Z_1\to \Omega_1\times Z_1\).  Doing the
  above construction for all \(y\in Y\), we get a covering of~\(Y\) by
  open subsets like our~\(Y_1\) with corresponding charts
  on~\(\pi(Y_1)\).  We claim that these local charts on~\(Y/G\) have
  smooth coordinate change maps.  Thus~\(Y/G\) becomes a smooth
  manifold.  The charts are defined in such a way that the projection
  \(\pi\colon Y\to Y/G\) is a smooth submersion.  We have already seen
  above that the shear map is a diffeomorphism.  Hence the proof that
  the \(\Gr\)\nb-action on~\(Y\) is basic will be finished once we
  show that coordinate change maps between our local charts on~\(Y/G\)
  are smooth.

  The charts on~\(Y/G\) are defined in such a way that the projection
  map \(\pi\colon Y\to Y/G\) is smooth for each chart and there are smooth
  sections \(Y/G\supseteq \Omega_1\times U_1\to Y\) on the chart
  neighbourhoods, defined by \((\omega_1,u_1)\mapsto
  (\omega_1,\xi,u_1)\) in local coordinates.  Now take two overlapping
  charts on~\(Y/G\) defined on open subsets \(W_1,W_2\subseteq Y/G\)
  with their smooth local sections \(\sigma_i\colon W_i\to Y\).  If
  \(w\in W_1\cap W_2\), then \(\sigma_1(w)\) and~\(\sigma_2(w)\) are
  two representatives of the same \(\Gr\)\nb-orbit.  By the discussion
  above, this means that \(\bunp\s(\sigma_1(w))=\bunp\s(\sigma_2(w))\)
  and \(\sigma_2(w)=\sigma_1(w)\cdot
  (\s(\sigma_1(w)),\s(\sigma_2(w)))\).  Thus the coordinate change
  map from~\(W_1\) to~\(W_2\) is of the form
  \[
  w\mapsto \pi\bigl(\sigma_1(w)\cdot (\s(\sigma_1(w)),\s(\sigma_2(w)))\bigr).
  \]
  Since the maps \(\pi\), \(\sigma_i\) and~\(\s\) are all smooth in
  our local coordinates, so is the composite map.  This finishes the
  proof that Assumption~\ref{assum:covering_acts_basically_weak} holds
  for \(\covers_\subm\).  The proof works for all our categories of
  infinite-dimensional manifolds.
\end{proof}

It is unclear whether Assumptions \ref{assum:two-three}
and~\ref{assum:covering_acts_basically} hold for Fréchet and locally
convex manifolds.  The problem is that we lack a general implicit
function theorem.  Such a theorem is available for Banach manifolds
and gives the following equivalent charaterisations of submersions:

\begin{proposition}
  \label{pro:submersion_Banach}
  Let \(X\) and~\(Y\) be Banach manifolds and let \(f\colon X\to Y\)
  be a smooth map.  The following are equivalent:
  \begin{enumerate}
  \item \label{enum:submersion1} \(f\) is a submersion in the sense
    of Definition~\textup{\ref{def:submersion}};
  \item \label{enum:submersion2} \(f\) has many smooth local
    sections, that is, for each \(x\in X\), there is an open
    neighbourhood~\(U\) of \(f(x)\) and a smooth map \(\sigma\colon
    U\to X\) with \(\sigma(f(x))=x\) and \(f\circ\sigma=\id_U\);
  \item \label{enum:submersion3} for each \(x\in X\), the derivative
    \(D_xf\colon T_x X\to T_{f(x)}Y\) is split surjective, that is,
    there is a continuous linear map \(s\colon T_{f(x)} Y\to T_x X\)
    with \(D_x f\circ s=\id_{T_{f(x)} Y}\).
  \end{enumerate}
  If \(X\) and~\(Y\) are Hilbert manifolds,
  then~\ref{enum:submersion3} is equivalent to~\(D_x f\) being
  surjective.
\end{proposition}

\begin{proof}
  It is trivial that~\ref{enum:submersion1}
  implies~\ref{enum:submersion2}.  Using \(D_{f(x)}\sigma\) as
  continuous linear section for~\(D_x f\), we see
  that~\ref{enum:submersion2} implies~\ref{enum:submersion3}.  The
  implication from~\ref{enum:submersion3} to~\ref{enum:submersion1} is
  \cite{Lang:Fundamentals_diffgeo}*{Proposition 2.2 in Chapter II}.
  If \(X\) and~\(Y\) are Fréchet manifolds, then the
  derivative~\(D_x f\) is open once it is surjective by the Open
  Mapping Theorem.  If \(D_x f\) is surjective and \(X\) and~\(Y\) are
  Hilbert manifolds, then the orthogonal projection onto \(\ker D_x
  f\) splits \(X\cong \ker (D_x f) \oplus D_{f(x)} Y\), so
  \ref{enum:submersion3} follows if~\(D_x f\) is surjective and \(X\)
  and~\(Y\) are Hilbert manifolds.
\end{proof}

\begin{proposition}
  \label{pro:subm_local_cover}
  On the categories \(\Mfd_\Ban\), \(\Mfd_\Hil\) and \(\Mfd_\fin\),
  the pretopology~\(\covers_\subm\) also satisfies
  Assumption~\textup{\ref{assum:two-three}}.
\end{proposition}

\begin{proof}
  Since we are working with Banach manifolds or smaller categories, we may
  use Proposition~\ref{pro:submersion_Banach} to redefine our covers
  as surjective smooth maps with many smooth local sections.  With
  this alternative definition of our covers,
  Assumption~\ref{assum:two-three} holds for the same reason as for
  the pretopology of continuous surjections with many continuous local
  sections on the category of topological spaces.  Let \(X\), \(Y\)
  and~\(Z\) be Banach manifolds and let \(f\colon X\to Y\) and
  \(g\colon Y\to Z\) be smooth maps.  We assume that \(g\circ f\)
  and~\(f\) are surjective submersions and want to show that~\(g\) is
  so to.  Of course, \(g\)~must be surjective if \(g\circ f\) is.
  Given \(y\in Y\), there is \(x\in X\) with \(f(x)=y\) because~\(g\)
  is surjective, and there is a smooth local section~\(\sigma\) for
  \(g\circ f\) near~\(g(y)\) with \(\sigma(g(y))=x\).  Then
  \(f\circ\sigma\) is a smooth local section for~\(g\) with
  \(f\circ\sigma(g(y))=y\).
\end{proof}

If we redefined submersions of locally convex or Fréchet manifolds
using one of the alternative criteria in
Proposition~\ref{pro:submersion_Banach}, then the existence of
pull-backs along submersions would become unclear.

The next lemma is needed to verify
Assumption~\ref{assum:covering_acts_basically} for Banach manifolds.

\begin{lemma}
  \label{lemma:projection-submfd}
  Let~\(Y\) be a Banach manifold and let \(\varphi\colon Y\to Y\) be a
  smooth map with \(\varphi^2=\varphi\).  The image \(\Ima(\varphi)\)
  is a closed submanifold of~\(Y\) and \(\varphi\colon Y\to
  \Ima(\varphi)\) is a surjective submersion.
\end{lemma}

\begin{proof}
  The image is closed because~\(Y\) is Hausdorff and
  \[
  \Ima(\varphi) = \{x\in Y\mid \varphi(x)=x\}.
  \]
  We need a submanifold chart for \(\Ima(\varphi)\) near \(y\in
  \Ima(\varphi)\).  We use a chart \(\gamma\colon T_yY \to Y\), that
  is, \(\gamma\) is a diffeomorphism onto an open neighbourhood
  of~\(y\) in~\(Y\) whose derivative at~\(0\) is the identity map
  on the Banach space~\(T_yY\).  It suffices to find a submanifold
  chart for
  \(\varphi'\defeq\gamma^{-1}\circ\varphi\circ\gamma\), which is an
  idempotent smooth map on a neighbourhood of~\(0\) in~\(T_yY\) with
  \(\varphi'(0)=0\).  Let \(D\colon T_yY\to T_yY\) be the derivative
  of~\(\varphi'\) at~\(0\), which is a linear map with \(D^2=D\).  We
  may assume that~\(\varphi'\) is defined on all of~\(T_yY\) to
  simplify notation.

  Let \(X=\Ima(D)\), \(W=\ker(D)\).  The map
  \[
  \Psi\colon
  X\oplus W\to T_yY,\qquad
  (\xi,\eta)\mapsto \varphi'(\xi)+\eta,
  \]
  has the identity map as derivative near~\((0,0)\).  By the Implicit
  Function Theorem for Banach manifolds, \(\Psi\) is a diffeomorphism
  between suitable open neighbourhoods of~\(0\) in~\(X\oplus W\)
  and~\(T_yY\).  We have \(\Psi(\xi,0) = \varphi'(\xi) \in
  \Ima(\varphi')\) for all \(\xi\in X\), and
  \begin{multline*}
    \varphi'(\Psi(\xi,\eta))-\Psi(\xi,\eta)
    = \varphi'(\varphi'(\xi)+\eta)-\varphi'(\xi)-\eta
    \\= D(D\xi+\eta)-D\xi-\eta + O(\norm{\xi}^2+\norm{\eta}^2)
    = - \eta + O(\norm{\xi}^2+\norm{\eta}^2).
  \end{multline*}
  This is non-zero if \(\eta\neq0\) and \(\eta,\xi\) are small enough.
  Hence, for sufficiently small \(\xi,\eta\), we have
  \(\Psi(\xi,\eta)\in\Ima(\varphi')\) if and only if \(\eta=0\).
  Thus~\(\Psi\) is a submanifold chart for \(\Ima(\varphi')\)
  near~\(0\).  Its composite with~\(\gamma\) gives the required
  submanifold chart for~\(\Ima(\varphi)\) near~\(y\).  It is clear
  from the construction that the map \(\varphi\colon Y\to
  \Ima(\varphi)\) has surjective derivative at~\(y\), hence it is a
  smooth submersion by Proposition~\ref{pro:submersion_Banach}.
\end{proof}

\begin{proposition}
  \label{prop:Mfd_basic}
  Every action of a \v{C}ech groupoid of a cover in~\((\Mfd_\Ban,
  \covers_\subm)\) is basic.
\end{proposition}

\begin{proof}
  Let \(\bunp\colon M\to N\) be a surjective submersion and let \(\Gr
  =(M\times_{\bunp,N,\bunp}M \rightrightarrows N)\) be its \v{C}ech
  groupoid (see Example~\ref{exa:covering_groupoid}).  Let~\(\Gr\) act
  on a smooth manifold~\(Y\) with anchor map \(\s\colon Y\to \Gr^0\).

  Since~\(\bunp\) has many smooth local sections, we may cover~\(N\)
  by open subsets~\(U_i\), \(i\in I\), for which there are smooth
  local sections \(\sigma_i\colon U_i\to M\) with \(\bunp\circ
  \sigma_i =\id_{U_i}\).  Then \(Y_i\defeq (\bunp
  \s)^{-1}(U_i)\subseteq Y\) is open (and thus an open submanifold
  of~\(Y\)) and \(\Gr\)\nb-invariant.

  The map
  \[
  \varphi_i\colon Y_i \to Y_i, \qquad
  y \mapsto y \cdot \bigl(\s(y), \sigma_i\circ \bunp\circ \s (y)\bigr),
  \]
  is smooth and maps~\(y\) to a point~\(y'\) in the same
  \(\Gr\)\nb-orbit with \(\s(y')\in\sigma_i(U_i)\).  If \(y'\)
  and~\(y''\) belong to the same \(\Gr\)\nb-orbit, then
  \(\s(y')=\s(y'')\) implies \(y'=y''\) because the only element
  of~\(\Gr\) that may map~\(y'\) to~\(y''\) is~\((\s(y'),\s(y''))\).
  Thus \(\varphi_i(y)\) is the unique element in the \(\Gr\)\nb-orbit
  of~\(y\) that belongs to \(\s^{-1}(\sigma_i(U_i))\).  Thus
  \(\varphi_i^2=\varphi_i\).  By Lemma~\ref{lemma:projection-submfd},
  \(\Ima(\varphi)\) is a smooth submanifold of~\(Y\) and
  \(\varphi_i\colon Y_i\to \Ima(\varphi_i)\) is a surjective
  submersion.

  If \(y_1,y_2\in Y_i\) satisfy \(\varphi_i(y_1)=\varphi_i(y_2)\),
  then \(y_2=y_1\cdot (\s(y_1),\s(y_2))\).  Hence the map
  \[
  Y_i\times_{\s,M,\pr_1} (M\times_N M) \to
  Y_i \times_{\varphi_i,Y_i,\varphi_i} Y_i,\qquad
  (y,(m_1,m_2))\mapsto (y,y\cdot (m_1,m_2)),
  \]
  is a diffeomorphism with inverse \((y_1,y_2)\mapsto
  (y_1,(\s(y_1),\s(y_2)))\).  As a consequence, the restriction of the
  \(\Gr\)\nb-action to~\(Y_i\) with the bundle projection
  \(\varphi_i\colon Y_i\to \Ima(\varphi_i)\) is a principal bundle.

  It remains to glue together these local constructions.
  Let~\(Y/\Gr\) be the quotient space with the quotient topology.  The
  quotient map \(\pi\colon Y\to Y/\Gr\) is open by
  Proposition~\ref{pro:open_orbit_space_projection} and because
  submersions are open.  Thus the subsets~\(\pi(Y_i)\) form an open
  cover of~\(Y/\Gr\).  The above argument shows that~\(\pi\) restricts
  to a homeomorphism from~\(\Ima(\varphi_i)\) onto~\(\pi(Y_i)\).  We
  claim that there is a unique smooth manifold structure on~\(Y/\Gr\)
  for which the maps \(\pi\colon \Ima(\varphi_i)\to Y/\Gr\) become
  diffeomorphisms onto open subsets of~\(Y/\Gr\).  We only have to
  check that the coordinate change maps on~\(Y/\Gr\) from
  \(\Ima(\varphi_i)\) to~\(\Ima(\varphi_j)\) are smooth maps.  This is
  so because the map
  \[
  Y_i\cap \Ima(\varphi_j)\to
  \Ima(\varphi_i)\cap Y_j,\qquad
  y\mapsto y\cdot (\sigma_i\circ\s(y),\sigma_j\circ\s(y)),
  \]
  is a diffeomorphism between submanifolds of~\(Y_i\), with inverse
  given by a similar formula.  The orbit space projection \(Y\to
  Y/\Gr\) is a surjective submersion for this smooth manifold
  structure on~\(Y/\Gr\) because this holds locally on each~\(Y_i\).
  The same argument as above shows that the map
  \[
  Y\times_{\s,M,\pr_1} (M\times_N M)\to Y\times_{Y/\Gr} Y,\qquad
  (y,(m_1,m_2))\mapsto (y,y\cdot (m_1,m_2)),
  \]
  is a diffeomorphism with inverse \((y_1,y_2)\mapsto
  (y_1,(\s(y_1),\s(y_2)))\).  Hence \(Y\prto Y/\Gr\) is a principal
  \(\Gr\)\nb-bundle.
\end{proof}

As a result, most of our theory works for locally convex manifolds
with surjective submersions as covers.  This includes the bicategories
of anafunctors, bibundle functors and bibundle equivalences, but
\emph{not} the bicategories of covering bibundle functors and bibundle
actors.  We only know that these bicategories exist for Banach,
Hilbert and finite-dimensional manifolds.

\begin{remark}
  An action is basic in the category of finite-dimensional manifolds
  if and only if it is basic in the category of Hausdorff spaces, if
  and only if it is free and proper (see
  Proposition~\ref{pro:Hausdorff_basic}).  This remains true for
  Hilbert manifolds.  For group actions, this is the main result of
  \cite{Jotz-Neeb:Closedness}; groupoid actions may be reduced to
  group actions using the proof of \cite{Zambon-Zhu:Contact}*{Lemma
    3.11}.

  For Banach manifolds and, even more generally, for locally convex
  manifolds, it remains true that basic actions are free and proper,
  but the converse fails.  Here is a counterexample (see
  also~\cite{Bourbaki:Groupes_algebres_23}*{Chapter~3}).  Let~\(E\) be
  a Banach space and let~\(F\) be a closed subspace without
  complement, for instance, \(c_0(\N)\subset\ell^\infty(\N)\).
  Let~\(F\) act on~\(E\) by inclusion and linear addition.  This
  action is free and proper, and the orbit space projection is the
  quotient map \(E\to E/F\).  The derivative of this map is the same
  map \(E\to E/F\), which has no linear section by assumption.  Hence
  the orbit space projection is not a cover, so the action is not
  basic.
\end{remark}

\end{document}